\documentclass[11pt]{article}
\usepackage[usenames,dvipsnames]{xcolor}
\usepackage{amsmath}
\usepackage{mathrsfs}
\usepackage{amssymb}
\usepackage{dcolumn}
\usepackage{graphicx}
\usepackage{fancybox}
\usepackage{amsthm}
\usepackage{tikz}
\usepackage{geometry}
\usepackage[sc]{mathpazo}    
\usepackage{bbm}
\usepackage{authblk}

\usetikzlibrary{arrows}
\usetikzlibrary{decorations.markings}
\pgfdeclarelayer{nodelayer}
\pgfdeclarelayer{edgelayer}
\pgfsetlayers{edgelayer,nodelayer,main}

\usepackage{physics}
\usepackage[UKenglish]{babel}
\usepackage[title]{appendix}
\usepackage{subcaption}
\usepackage{adjustbox}
\usepackage{ragged2e}
\usepackage{comment}
\usepackage{nicematrix}
\usepackage{longtable}

\usepackage[colorlinks=true,linkcolor=NavyBlue,urlcolor=RoyalBlue,citecolor=PineGreen,bookmarks=true,bookmarksopen=true,bookmarksopenlevel=2,unicode=true,linktocpage]{hyperref}
\usepackage{aliascnt}
\usepackage{cleveref}

\usepackage[backend=biber,natbib=true,citestyle=alphabetic,bibstyle=alphabetic,giveninits=true,maxbibnames=99,isbn=false]{biblatex}
\usepackage{xurl}
\DeclareFieldFormat[article]{eid}{art\adddot\space#1}
\AtBeginBibliography{\setlength{\emergencystretch}{2em}}

\AtEveryBibitem{
    \clearfield{issn}
    \clearfield{month}
}

\usepackage{csquotes}

\newtheorem{theorem}{Theorem}[section]

\newaliascnt{definition}{theorem}
\newtheorem{definition}[definition]{Definition}
\aliascntresetthe{definition}

\newaliascnt{proposition}{theorem}
\newtheorem{proposition}[proposition]{Proposition}
\aliascntresetthe{proposition}

\newaliascnt{lemma}{theorem}
\newtheorem{lemma}[lemma]{Lemma}
\aliascntresetthe{lemma}

\newaliascnt{corollary}{theorem}
\newtheorem{corollary}[corollary]{Corollary}
\aliascntresetthe{corollary}

\newaliascnt{example}{theorem}
\newtheorem{example}[example]{Example}
\aliascntresetthe{example}

\newaliascnt{remark}{theorem}
\newtheorem{remark}[remark]{Remark}
\aliascntresetthe{remark}

\newaliascnt{hypothesis}{theorem}

\aliascntresetthe{hypothesis}

\newaliascnt{conjecture}{theorem}
\newtheorem{conjecture}[conjecture]{Conjecture}
\aliascntresetthe{conjecture}

\newaliascnt{question}{theorem}

\aliascntresetthe{question}

\crefname{theorem}{theorem}{theorems}
\Crefname{theorem}{Theorem}{Theorems}

\crefname{definition}{definition}{definitions}
\Crefname{definition}{Definition}{Definitions}

\crefname{proposition}{proposition}{propositions}
\Crefname{proposition}{Proposition}{Propositions}

\crefname{lemma}{lemma}{lemmata}
\Crefname{lemma}{Lemma}{Lemmata}

\crefname{corollary}{corollary}{corollaries}
\Crefname{corollary}{Corollary}{Corollaries}

\crefname{example}{example}{examples}
\Crefname{example}{Example}{Examples}

\crefname{remark}{remark}{remarks}
\Crefname{remark}{Remark}{Remarks}

\crefname{hypothesis}{hypothesis}{hypotheses}
\Crefname{hypothesis}{Hypothesis}{Hypotheses}

\crefname{conjecture}{conjecture}{conjectures}
\Crefname{conjecture}{Conjecture}{Conjectures}

\crefname{question}{question}{questions}
\Crefname{question}{Question}{Questions}

\newcommand{\st}{{\bigm|}}

\newcommand{\Inv}{\mathrm{Inv}}
\newcommand{\End}{\mathrm{End}}
\newcommand{\Hom}{\mathrm{Hom}}
\newcommand{\Path}{\mathsf{Path}}
\newcommand{\State}{\mathsf{St}}
\newcommand{\Span}{\operatorname{span}}

\DeclareMathOperator{\Alg}{Alg}

\newcommand{\R}{\mathbb{R}}
\newcommand{\Z}{\mathbb{Z}}
\newcommand{\C}{\mathbb{C}}
\newcommand{\N}{\mathbb{N}}

\newcommand{\Q}{\mathbb{Q}}

\newcommand{\Net}{\mathsf{N}}
\newcommand{\Web}{\mathsf{W}}
\newcommand{\W}{\mathscr{W}}
\newcommand{\WebBasis}[1]{\mathsf{B}^{\mathrm{web}}_{#1}}
\newcommand{\GenSet}{\mathsf{G}}

\newcommand{\id}{\mathrm{id}}

\newcommand{\uqtwo}{U_q(\mathfrak{sl}_2)}
\newcommand{\uqslt}{U_q(\mathfrak{sl}_3)}
\newcommand{\eva}{\mathrm{ev}}
\newcommand{\coev}{\mathrm{coev}}
\newcommand{\wt}{\mathrm{wt}}

\newcommand{\TL}[2]{(E^{#1}_{#2})}

\newcommand{\Ydual}{\overline{\mathrm{Y}}}
\newcommand{\Fold}{\operatorname{Fold}}
\newcommand{\Unfold}{\operatorname{Unf}}
\newcommand{\Mir}{\operatorname{mir}_{\mathrm v}}
\newcommand{\WeylAlt}{\operatorname{Alt}_{A_2}}
\newcommand{\YL}[2]{(\mathrm{Y}\Ydual^{#1#2})}
\newcommand{\LY}[2]{(\Ydual\mathrm{Y}^{#1#2})}

\newcommand{\Rep}{\mathbf{Rep}^{\pm}(\mathfrak{sl}_3)}
\newcommand{\Sp}{\mathbf{Sp}_{A_2}}
\newcommand{\Fsp}{\mathbf{FSp}_{A_2}}
\newcommand{\ccup}{\mathrm{cup}}
\newcommand{\ccap}{\mathrm{cap}}

\tikzset{
    style/.style={#1}, 
    none/.style={},     
    mid arrow/.style={postaction={decorate,decoration={
            markings, mark=at position 0.55 with {\arrow{stealth}}
        }}
    },
    dotted/.style={dashed}
}

\author[2]{Azat M. Gainutdinov}
\author[1,3,4]{Jesper Lykke Jacobsen}
\author[1,3]{Yueheng Li}

\affil[1]{Laboratoire de Physique de l'\'Ecole Normale Sup\'erieure, ENS, Universit\'e PSL,
CNRS, Sorbonne Universit\'e, Universit\'e de Paris, F-75005 Paris, France}
\affil[2]{Institut Denis Poisson, CNRS, Universit\'e de Tours, Universit\'e d'Orl\'eans, Parc de Grandmont, F-37200 Tours, France}
\affil[3]{Sorbonne Universit\'e, \'Ecole Normale Sup\'erieure, CNRS, 
Laboratoire de Physique (LPENS), F-75005 Paris, France}
\affil[4]{Universit\'e Paris Saclay, CNRS, CEA, Institut de Physique Th\'eorique,  F-91191 Gif-sur-Yvette, France}

\begin{document}

\title{Cellular algebra of $A_2$ webs}
\date{}
\maketitle

\begin{abstract}
        We introduce a diagram algebra $\mathcal{K}_L$ based on Kuperberg's $A_2$-webs and prove a Schur-Weyl duality $\mathcal{K}_L \cong\End_{\uqslt}((V^+\oplus V^-)^{\otimes L})$, where $V^\pm$ are the fundamental representations of $\uqslt$. Thus, $\mathcal{K}_L$ is a natural $A_2$ analogue of the Temperley-Lieb algebra. We construct an explicit cellular structure for $\mathcal{K}_L$, with the web basis as its cellular basis; with the cell modules indexed by dominant $\uqslt$-weights, whose bases are parametrized by dominant lattice paths via a recursive growth algorithm. Using this structure, we determine local generators of $\mathcal{K}_L$ and establish web formulations of the first and second fundamental theorems of the invariant theory. In particular, we show that the braid-group image, augmented by local idempotents, does not generate $\mathcal{K}_L$, and identify an additional local generator. Finally, we derive a bimodule decomposition of  $(V^+\oplus V^-)^{\otimes L}$, yielding a complete set of irreducible $\mathcal{K}_L$-modules for formal or generic $q$.
\end{abstract}

\tableofcontents

\section{Introduction}

Let $\mathcal{A}$ be a group or a Hopf algebra over a field $\C(q)$, and $V$ a finite-dimensional $\mathcal{A}$-module. A central object in the invariant theory of $\mathcal{A}$ is the endomorphism algebra
\begin{align}\label{eqn: general endomorphism algebra}
    \mathcal{K}:=\End_{\mathcal{A}}\big( V^{\otimes L} \big).
\end{align}
The first and second fundamental theorems (henceforth FFT and SFT) of invariant theory for the pair $(\mathcal{A},V)$ give a \emph{presentation} of $\mathcal{K}$. Specifically, FFT determines a set of \emph{generators} $\GenSet$, along with a surjective algebra morphism 
\begin{align}\label{eqn: surjective morphism from free algebra}
    \Psi: \Alg_{\C(q)}(\GenSet) \to \mathcal K,
\end{align}
where $\Alg_{\C(q)}(\GenSet)$ is the free $\C(q)$-algebra generated by $\GenSet$; SFT, on the other hand, describes $\ker(\Psi)$ as the ideal generated by a list of linear relations between words in $\GenSet$ called \emph{defining relations}. The case $(\mathcal{A}=\mathrm{GL}(n,\C(q)),V=\C(q)^n)$ is the scalar extension to $\C(q)$ of the classical case $(\mathcal{A}=\mathrm{GL}(n,\C),V=\C^n)$, first considered by Schur \cite{schur_1901} and further developed by Weyl \cite{weyl_1939_classical}; it reveals deep connections between representations of $\mathrm{GL}(n)$ and the symmetric group $S_L$, now termed the Schur-Weyl duality.

In what follows, we take $\mathcal{A}=U_q(\mathfrak{g})$ to be the Drinfeld-Jimbo quantum group associated with a simple Lie algebra $\mathfrak{g}$ \cite{Jimbo:161297}, and $V$ to be one of its fundamental representations. For certain Dynkin types of $\mathfrak{g}$, a particular basis of $\mathcal{A}$-invariants in $V^{\otimes L}$ is encoded by planar graphs called \emph{webs} \cite{kuperberg_1996_spiders,cautis_kamnitzer_morrison_2014,bodish2026typebwebs,bodish2021typecwebs}. Consequently, webs encoding invariants in $V^{\otimes L}\otimes (V^*)^{\otimes L}$ give a \emph{web basis} for $\mathcal{K}$. Under the web basis, the composition of $U_q(\mathfrak{g})$-intertwiners is translated to concatenation of graphs. Linear relations between intertwiners are then encoded by local reduction rules of webs, i.e. the replacement of subgraphs by linear combinations of other subgraphs. Here, by locality, we mean that each subgraph involved in the replacements contains a finite number of vertices independent of $L$.

When a web basis exists, the presentation of the algebra can be interpreted diagrammatically. We demonstrate this perspective for $\mathcal{A}=\uqtwo$ and $V$ its fundamental representation. Here webs do not contain multivalent vertices, hence degenerate to planar non-crossing matchings of points \cite{Weyl1932,frenkel_khovanov_1997}, also called \emph{$A_1$ webs}. There is indeed an isomorphism of algebras \cite{Jimbo:161297}:
\begin{align}\label{eqn: TL centralizing property}
    \End_{\uqtwo}\big( V^{\otimes L} \big) \cong \mathrm{TL}_L(\delta),
\end{align}
where the Temperley--Lieb algebra $\mathrm{TL}_L(\delta)$ \cite{temperley_1971_relations} is a diagram algebra of non-crossing matchings of $L$ upper and $L$ lower points, multiplication being vertical concatenation. As a unital algebra, it has web generators $e_i$ $(1\leq i\leq L-1)$,
\begin{equation}\label{eqn: TL generators}
    \scalebox{0.8}{\begin{tikzpicture}[x=1cm,y=1cm,line width=0.8pt,line cap=round]

    \begin{scope}
    \node at (-1.05,0.8) {$e_i={}$};

    \draw (0,0) -- (0,1.6);
    \node at (0,-0.28) {$1$};
    \node at (0.9,0.8) {$\cdots$};
    \draw (1.8,0) -- (1.8,1.6);
    \node at (1.8,-0.28) {$i-1$};

    \draw (3,1.6) .. controls (3,0.9) and (4.2,0.9) .. (4.2,1.6);
    \draw (3,0) .. controls (3,0.7) and (4.2,0.7) .. (4.2,0);
    \node at (3,-0.28) {$i$};
    \node at (4.2,-0.28) {$i+1$};

    \draw (5.4,0) -- (5.4,1.6);
    \node at (5.4,-0.28) {$i+2$};
    \node at (6.3,0.8) {$\cdots$};
    \draw (7.2,0) -- (7.2,1.6);
    \node at (7.2,-0.28) {$L$};
    \end{scope}
\end{tikzpicture}}
\end{equation}
with only one local reduction rule: evaluate a loop to $\delta=-q-q^{-1}$. Indeed, the defining relations of $\mathrm{TL}_L(\delta)$ are recalled in \cref{fig: TL defining relations}: relation (a) uses the loop evaluation, while (b) and (c) are simply the isotopy equivalence of different embeddings of the same graph.
\begin{figure}[htbp]
    \centering
    \scalebox{0.9}{\begin{tikzpicture}[x=1cm,y=1cm,line width=0.8pt,line cap=round,line join=round]

    \begin{scope}
        \draw (0,0) .. controls (0,0.35) and (0.75,0.35) .. (0.75,0);
        \draw (0.375,1.20) ellipse [x radius=0.375,y radius=0.32];
        \draw (0,2.40) .. controls (0,2.05) and (0.75,2.05) .. (0.75,2.40);

        \node at (1.20,1.20) {$=$};
        \node at (1.62,1.20) {$\delta$};

        \draw (2.00,0) .. controls (2.00,0.55) and (2.75,0.55) .. (2.75,0);
        \draw (2.00,2.40) .. controls (2.00,1.85) and (2.75,1.85) .. (2.75,2.40);

        \node at (1.375,-0.48) { (a) $e_i^2=\delta e_i$};
    \end{scope}

    \begin{scope}[xshift=4.05cm]
        \draw (0,0) .. controls (0,0.25) and (0.65,0.25) .. (0.65,0);
        \draw (0,0.80) .. controls (0,0.55) and (0.65,0.55) .. (0.65,0.80);
        \draw (1.30,0) -- (1.30,0.80);

        \draw (0,0.80) -- (0,1.60);
        \draw (0.65,0.80) .. controls (0.65,1.05) and (1.30,1.05) .. (1.30,0.80);
        \draw (0.65,1.60) .. controls (0.65,1.35) and (1.30,1.35) .. (1.30,1.60);

        \draw (0,1.60) .. controls (0,1.85) and (0.65,1.85) .. (0.65,1.60);
        \draw (0,2.40) .. controls (0,2.15) and (0.65,2.15) .. (0.65,2.40);
        \draw (1.30,1.60) -- (1.30,2.40);

        \node at (1.77,1.20) {$=$};

        \draw (2.25,0) .. controls (2.25,0.55) and (2.90,0.55) .. (2.90,0);
        \draw (2.25,2.40) .. controls (2.25,1.85) and (2.90,1.85) .. (2.90,2.40);
        \draw (3.55,0) -- (3.55,2.40);

        \node at (1.775,-0.48) { (b) $e_i e_{i\pm 1}e_i=e_i$};
    \end{scope}

    \begin{scope}[xshift=9.05cm]
        \draw (0,0) .. controls (0,0.38) and (0.55,0.38) .. (0.55,0);
        \draw (0,1.20) .. controls (0,0.82) and (0.55,0.82) .. (0.55,1.20);
        \draw (1.25,0) -- (1.25,1.20);
        \draw (1.80,0) -- (1.80,1.20);

        \draw (0,1.20) -- (0,2.40);
        \draw (0.55,1.20) -- (0.55,2.40);
        \draw (1.25,1.20) .. controls (1.25,1.58) and (1.80,1.58) .. (1.80,1.20);
        \draw (1.25,2.40) .. controls (1.25,2.02) and (1.80,2.02) .. (1.80,2.40);

        \node at (2.22,1.20) {$=$};

        \draw (2.65,0) -- (2.65,1.20);
        \draw (3.20,0) -- (3.20,1.20);
        \draw (3.90,0) .. controls (3.90,0.38) and (4.45,0.38) .. (4.45,0);
        \draw (3.90,1.20) .. controls (3.90,0.82) and (4.45,0.82) .. (4.45,1.20);

        \draw (2.65,1.20) .. controls (2.65,1.58) and (3.20,1.58) .. (3.20,1.20);
        \draw (2.65,2.40) .. controls (2.65,2.02) and (3.20,2.02) .. (3.20,2.40);
        \draw (3.90,1.20) -- (3.90,2.40);
        \draw (4.45,1.20) -- (4.45,2.40);

        \node at (2.225,-0.48) { (c) $e_i e_j=e_j e_i$, $|i-j|>1$};
    \end{scope}

\end{tikzpicture}}
    \caption{The defining relations of the Temperley--Lieb algebra in diagrammatic form. Relations (b) are obtained by the shown figure and its mirror reflection.}
    \label{fig: TL defining relations}
\end{figure}

Web bases are known for many other quantum groups $U_q(\mathfrak{g})$, and the constructions depend on the type of $\mathfrak{g}$. Since the seminal construction of webs of type $A_2$, $B_2$ and $G_2$ by Kuperberg \cite{kuperberg_1996_spiders}, the type $A$ family has been constructed in \cite{cautis_kamnitzer_morrison_2014}, the type $C$ family in \cite{Rose2021On,Bodish2023Web,bodish2021typecwebs}, and the type $B$ family in \cite{bodish2026typebwebs}. There is also a different construction for type $A_3$ in \cite{gaetz_pechenik_pfannerer_striker_swanson_2026}, where the web basis enjoys a remarkable rotational invariance property. 

For an algebra with a web basis, we propose the following form of the FFT and SFT:

\begin{definition}\label{def: web presentation}
    Let $\mathcal K$ be an algebra with a web basis $\mathsf{B}$. A \emph{web presentation} consists of the following:
    \begin{enumerate}
        \item (FFT) A set $\GenSet\subset \mathsf{B}$, such that \eqref{eqn: surjective morphism from free algebra} is surjective.
        \item (SFT) A set of local reduction rules that reduce any word in $\GenSet$ to an element in the span of $\mathsf{B}$.
    \end{enumerate}
\end{definition}

While much is known about web bases, web presentations of the endomorphism algebras $\mathcal{K}$ remain largely unknown: beyond the aforementioned Temperley-Lieb algebra, a web presentation is available only in special cases where a presentation of $\mathcal{K}$ is known a priori through its connection to the Hecke algebra \cite{martin_1992_tlm}\cite{webimmanants}\cite[Corollary 3.14]{Lafay:2024tne}.

\subsection{Main results}

Our first main result is a web presentation for the endomorphism algebra $\mathcal{K}$ associated to the pair $(\mathcal{A}=\uqslt, V=V^+ \oplus V^-)$, where $V^\pm$ are the fundamental representations of $\uqslt$. Specifically, we define the algebra $\mathcal{K}_L$ in \cref{def: K_L}, as a $\C(q)$-algebra with basis given by trivalent bipartite graphs called \emph{$A_2$ webs}, subject to local reduction rules in \cref{fig: Kuperberg relations} that replace webs with inner faces of degree 0, 2 or 4 to linear combinations of webs without such faces. We show, as part of a stronger Schur-Weyl duality-type result in \cref{thm: Schur-Weyl duality}, that $\mathcal{K}_L$ is indeed isomorphic to the desired endomorphism algebra:
\begin{align}\label{eqn: centralizing property}
    \End_{\uqslt}\big( (V^+\oplus V^-)^{\otimes L} \big) \cong \mathcal{K}_L.
\end{align}
In view of \eqref{eqn: TL centralizing property} and \eqref{eqn: centralizing property}, \emph{self-duality} of the 6-dimensional $\uqslt$-module $V=V^+ \oplus V^-$ establishes $\mathcal{K}_L$ as the natural $A_2$ analogue of the Temperley-Lieb algebra: $V$ is the simplest non-trivial self-dual $\uqslt$-module in terms of dimension.\footnote{The 8-dimensional adjoint representation of $\uqslt$ is another self-dual candidate for $V$ that defines a natural rank two generalization of the Temperley-Lieb algebra.} The choice $V^+ \oplus V^-$ makes $\mathcal{K}_L$ well-adapted to the web calculus, as is evident in the following web formulation of first and second fundamental theorems for $\mathcal{K}_L$:
\begin{theorem}[First fundamental theorem, \Cref{thm: generators of K_L}]\label{thm: intro FFT}
The generators of $\mathcal{K}_L$ take the form of a non-trivial web $X$ on two or three consecutive sites, tensored with identity webs on the remaining sites. See \cref{fig: generating set} for a complete list of webs $X$, and \cref{fig: oriented diagram operators} for two examples of these generators. The matrix forms of the generators under the isomorphism~\eqref{eqn: centralizing property} are listed in appendix~\ref{app: matrix form of generators}.
\begin{figure}[htbp]
    \centering
    \scalebox{0.8}{\begin{tikzpicture}
	\begin{pgfonlayer}{nodelayer}
		\node (0) at (-7, 1) {};
		\node (1) at (-6, 0.5) {};
		\node (2) at (-5, 1) {};
		\node (3) at (-6, -0.5) {};
		\node (4) at (-7, -1) {};
		\node (5) at (-5, -1) {};
		\node at (-5, 1.25) {$\pm$};
		\node at (-6, -1.5) {$I^{\pm}$};

		\node (6) at (-3, 1) {};
		\node (7) at (-2.5, 0) {};
		\node (8) at (-1.5, 0) {};
		\node (9) at (-1, 1) {};
		\node (10) at (-3, -1) {};
		\node (11) at (-1, -1) {};
		\node at (-1, 1.25) {$\pm$};
		\node at (-2, -1.5) {$H^{\pm}$};

		\node (12) at (1, 1) {};
		\node (13) at (1, -1) {};
		\node (14) at (3, 1) {};
		\node (15) at (3, -1) {};
		\node at (1, 1.25) {$s_1$};
		\node at (3, 1.25) {$s_2$};
		\node at (2, -1.5) {$\id_{s_1s_2}$};

		\node (16) at (5, 1) {};
		\node (17) at (5.5, 0) {};
		\node (18) at (6, 1) {};
		\node (19) at (5, -1) {};
		\node (20) at (7, 1) {};
		\node (21) at (6.5, 0) {};
		\node (22) at (6, -1) {};
		\node (23) at (7, -1) {};
		\node at (5, 1.25) {$\pm$};
		\node at (7, 1.25) {$\pm$};
		\node at (6, -1.5) {$\YL{\pm}{\pm}$};
	\end{pgfonlayer}
	\begin{pgfonlayer}{edgelayer}
		\draw (1.center) to (0.center);
		\draw (1.center) to (2.center);
		\draw (1.center) to (3.center);
		\draw (4.center) to (3.center);
		\draw (5.center) to (3.center);
		\draw (7.center) to (6.center);
		\draw (7.center) to (10.center);
		\draw (7.center) to (8.center);
		\draw (9.center) to (8.center);
		\draw (11.center) to (8.center);
		\draw (12.center) to (13.center);
		\draw (14.center) to (15.center);
		\draw (17.center) to (16.center);
		\draw (17.center) to (18.center);
		\draw (17.center) to (19.center);
		\draw (21.center) to (20.center);
		\draw (21.center) to (23.center);
		\draw (21.center) to (22.center);
	\end{pgfonlayer}
\end{tikzpicture}}
    \caption{Generating set of $\mathcal{K}_L$. The webs are oriented so that each trivalent vertex is either a source or a sink. The signs $s_1,s_2\in \{\pm\}$ denote the orientations of the strands they sit on, with $+$ for upward orientation and $-$ for downward orientation.}
    \label{fig: generating set}
        \begin{tikzpicture}[x=1cm,y=1cm,line width=0.8pt,line cap=round]
    \begin{scope}[yshift=2.4cm]
        \node at (-1.25,0.8) {$H_i^{+}={}$};

		\draw (0,0) -- (0,1.6);
		\node at (0,-0.58) {$1$};
        \node at (0.7,0.8) {$\cdots$};
		\draw (1.4,0) -- (1.4,1.6);
		\node at (1.4,-0.58) {$i-1$};

        \coordinate (htl) at (2.8,1.6);
        \coordinate (htr) at (4.2,1.6);
        \coordinate (hl)  at (3.35,0.8);
        \coordinate (hr)  at (3.65,0.8);
        \coordinate (hbl) at (2.8,0);
        \coordinate (hbr) at (4.2,0);
        \draw[style=mid arrow] (htl) -- (hl);
        \draw[style=mid arrow] (hbl) -- (hl);
        \draw[style=mid arrow] (hr) -- (hl);
        \draw[style=mid arrow] (hr) -- (htr);
        \draw[style=mid arrow] (hr) -- (hbr);
		\node at (2.8,-0.58) {$i$};
		\node at (4.2,-0.58) {$i+1$};

		\draw (5.6,0) -- (5.6,1.6);
		\node at (5.6,-0.58) {$i+2$};
        \node at (6.3,0.8) {$\cdots$};
		\draw (7,0) -- (7,1.6);
		\node at (7,-0.58) {$L$};
    \end{scope}

    \begin{scope}[yshift=-0.5cm]
        \node at (-1.25,0.8) {$\YL{+}{+}_j={}$};

		\draw (0,0) -- (0,1.6);
		\node at (0,-0.58) {$1$};
        \node at (0.6,0.8) {$\cdots$};
		\draw (1.2,0) -- (1.2,1.6);
		\node at (1.2,-0.58) {$j-1$};

        \coordinate (ytl) at (2.5,1.6);
        \coordinate (ytm) at (3.5,1.6);
        \coordinate (ytr) at (4.5,1.6);
        \coordinate (yl)  at (3,0.8);
        \coordinate (yr)  at (4,0.8);
        \coordinate (ybl) at (2.5,0);
        \coordinate (ybm) at (3.5,0);
        \coordinate (ybr) at (4.5,0);
        \draw[style=mid arrow] (yl) -- (ytl);
        \draw[style=mid arrow] (yl) -- (ytm);
        \draw[style=mid arrow] (yl) -- (ybl);
        \draw[style=mid arrow] (yr) -- (ytr);
        \draw[style=mid arrow] (yr) -- (ybm);
        \draw[style=mid arrow] (yr) -- (ybr);
		\node at (2.5,-0.58) {$j$};
		\node at (3.5,-0.58) {$j+1$};
		\node at (4.5,-0.58) {$j+2$};

		\draw (5.8,0) -- (5.8,1.6);
		\node at (5.8,-0.58) {$j+3$};
        \node at (6.4,0.8) {$\cdots$};
		\draw (7,0) -- (7,1.6);
		\node at (7,-0.58) {$L$};
    \end{scope}
\end{tikzpicture}
        \caption{Two examples of generators of $\mathcal{K}_L$. For the generator $H^+_i$, the non-trivial web $X$ is an $H$-web on the $i$th and $(i+1)$th strands, tensored with identities represented by unoriented vertical strands, each of which is a sum over two possible orientations.}
        \label{fig: oriented diagram operators}
\end{figure}
\end{theorem}

Any word in the generators in \cref{thm: intro FFT} is a trivalent bipartite graph, which is reduced via the Kuperberg relations in \cref{fig: Kuperberg relations} to a linear combination of the web basis of $\mathcal K_L$ \cite[Sec.\,4]{kuperberg_1996_spiders}. We therefore have the following theorem.

\begin{theorem}[Second fundamental theorem]\label{thm: intro SFT}
The loop, digon, and square relations in \cref{fig: Kuperberg relations} generate a complete set of relations in the sense of \cref{def: web presentation}.
\end{theorem}

There is a natural action of the braid group algebra $\C(q)[B_L]$ on $(V^+\oplus V^-)^{\otimes L}$ via the R-matrix of $\uqslt$~\cite{kirillov_reshetikhin_1990_qweyl}. It intertwines the $\uqslt$-action. Due to \eqref{eqn: centralizing property}, its image lies in $\mathcal K_L$. The presence of the generators $\YL{+}{+}_j$ in \cref{fig: generating set} implies the following consequence:

\begin{corollary}[\Cref{prp: KL(q) is not quotient of braid group}]
    The image of $\C(q)[B_L]$ in $\mathcal{K}_L$, completed with the idempotents $\id_{s_1s_2}$, is a proper subalgebra of $\mathcal{K}_L$.
\end{corollary}

Our second main result provides an explicit \emph{cellular structure} for $\mathcal K_L$ in the sense of \cite{graham1996cellular}, where the web basis of $\mathcal K_L$ is the \emph{cellular basis}. It is also the main tool in the proof of \cref{thm: intro FFT,thm: intro SFT}. This structure defines special \emph{cell modules} of $\mathcal K_L$, such that the cellular basis is obtained by ``gluing'' cell module vectors, analogous to how the basis of a matrix algebra is obtained by multiplying column and row basis vectors. 

A practical review of our results on the cellular structure of $\mathcal K_L$ is given in \cref{sec: cellular structure guide}, while its technical aspects are left for \cref{sec: cellular structure}, and with the final result formulated in \cref{thm: cellular structure}. Now we demonstrate how the gluing mechanism works.

Let us first recall the cell modules and their gluing for $\mathrm{TL}_L(\delta)$. Let $0\leq k\leq L$ and $k= L \mod 2$. The cell module $\Delta_L(k)$ has a basis of $A_1$ webs, indexed by paths $P$ on $\Z^2$ from the origin to $(L,k)$, taking steps $(1,1)$ and $(1,-1)$, and staying above the horizontal axis. The basis vectors are generated by a map $P\mapsto v_P$: a $(1,1)$ step in $P$ opens an arc, a $(1,-1)$ step closes an arc, and arcs that are never closed are connected to the bottom. See \cref{fig: TL cellular decomposition} (a) for a basis of $\Delta_{4}(2)$, along with the paths that index them. The gluing is shown in \cref{fig: TL cellular decomposition} (b): take two vectors $v_{P_1}$, $v_{P_3}$ in $\Delta_4(2)$, place the former on the top, reflect the latter against a horizontal mirror $\dagger$ and place it at the bottom, then connect the $k$ lines. Denote the resulting cellular basis element by $C^k_{P_1,P_3}$.

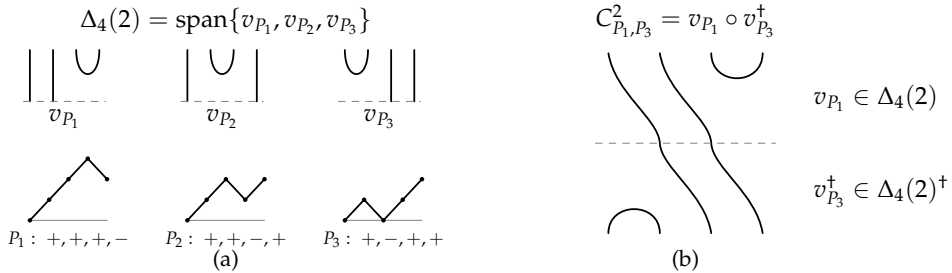
\begin{figure}[htbp]
    \centering
    \scalebox{0.85}{\begin{tikzpicture}[
    x=1cm,y=1cm,
    line width=0.8pt,
    line cap=round,
    line join=round
]

    \begin{scope}[xshift=9.00cm]
    \node at (1.20,3.30)
        {$C^2_{P_1,P_3}=v_{P_1}\circ v_{P_3}^\dagger$};

    \draw (1.60,2.80) .. controls (1.60,2.28) and (2.40,2.28) .. (2.40,2.80);

    \draw (0,0) .. controls (0,0.50) and (0.80,0.50) .. (0.80,0);

    \draw (0,2.80)
        .. controls (0.10,2.18) and (0.78,1.78) .. (0.80,1.40)
        .. controls (0.82,1.02) and (1.50,0.62) .. (1.60,0);
    \draw (0.80,2.80)
        .. controls (0.90,2.18) and (1.58,1.78) .. (1.60,1.40)
        .. controls (1.62,1.02) and (2.30,0.62) .. (2.40,0);

    \draw[dashed,gray,line width=0.35pt] (-0.20,1.40) -- (2.60,1.40);
    \node at (0.45,1.57) {};
    \node[anchor=west] at (2.75,1.40) {};

    \node[anchor=west] at (3.05,2.10) {$v_{P_1}\in\Delta_4(2)$};
    \node[anchor=west] at (3.05,0.65)
        {$v_{P_3}^\dagger\in\Delta_4(2)^\dagger$};

    \node at (1.20,-0.42) {\small (b)};
    \end{scope}

    \node at (3.05,3.30)
        {$\Delta_4(2)=\operatorname{span}\{v_{P_1},v_{P_2},v_{P_3}\}$};

    \begin{scope}[xshift=0cm]
        \draw (0,2.05) -- (0,2.85);
        \draw (0.36,2.05) -- (0.36,2.85);
        \draw (0.72,2.85) .. controls (0.72,2.35) and (1.08,2.35) .. (1.08,2.85);
        \draw[dashed,gray,line width=0.35pt] (-0.10,2.05) -- (1.18,2.05);
        \node at (0.54,1.78) {$v_{P_1}$};

        \draw[gray,line width=0.35pt] (0,0.20) -- (1.20,0.20);
        \draw (0,0.20) -- (0.30,0.52) -- (0.60,0.84)
            -- (0.90,1.16) -- (1.20,0.84);
        \foreach \x/\y in {0/0.20,0.30/0.52,0.60/0.84,0.90/1.16,1.20/0.84}
            \fill (\x,\y) circle (0.035);
        \node at (0.60,-0.10) {\scriptsize $P_1:\ +,+,+,-$};
    \end{scope}

    \begin{scope}[xshift=2.45cm]
        \draw (0,2.05) -- (0,2.85);
        \draw (1.08,2.05) -- (1.08,2.85);
        \draw (0.36,2.85) .. controls (0.36,2.35) and (0.72,2.35) .. (0.72,2.85);
        \draw[dashed,gray,line width=0.35pt] (-0.10,2.05) -- (1.18,2.05);
        \node at (0.54,1.78) {$v_{P_2}$};

        \draw[gray,line width=0.35pt] (0,0.20) -- (1.20,0.20);
        \draw (0,0.20) -- (0.30,0.52) -- (0.60,0.84)
            -- (0.90,0.52) -- (1.20,0.84);
        \foreach \x/\y in {0/0.20,0.30/0.52,0.60/0.84,0.90/0.52,1.20/0.84}
            \fill (\x,\y) circle (0.035);
        \node at (0.60,-0.10) {\scriptsize $P_2:\ +,+,-,+$};
    \end{scope}

    \begin{scope}[xshift=4.90cm]
        \draw (0.72,2.05) -- (0.72,2.85);
        \draw (1.08,2.05) -- (1.08,2.85);
        \draw (0,2.85) .. controls (0,2.35) and (0.36,2.35) .. (0.36,2.85);
        \draw[dashed,gray,line width=0.35pt] (-0.10,2.05) -- (1.18,2.05);
        \node at (0.54,1.78) {$v_{P_3}$};

        \draw[gray,line width=0.35pt] (0,0.20) -- (1.20,0.20);
        \draw (0,0.20) -- (0.30,0.52) -- (0.60,0.20)
            -- (0.90,0.52) -- (1.20,0.84);
        \foreach \x/\y in {0/0.20,0.30/0.52,0.60/0.20,0.90/0.52,1.20/0.84}
            \fill (\x,\y) circle (0.035);
        \node at (0.60,-0.10) {\scriptsize $P_3:\ +,-,+,+$};
    \end{scope}

    \node at (3.05,-0.42) {\small (a)};

\end{tikzpicture}}
    \caption{Cellular structure of $\mathrm{TL}_L(\delta)$.}
    \label{fig: TL cellular decomposition}
\end{figure}

For $\mathcal K_L$, we show that the cell modules $\Delta_L(\lambda)$ are parametrized by $\mathfrak{sl}_3$ weights $\lambda=n\lambda_1+k\lambda_2$ such that $n,k\geq 0$ and $n+k\leq L$. We construct for $\Delta_L(\lambda)$ a basis of $A_2$ webs, indexed by paths $P$ in the dominant Weyl chamber of the $\mathfrak{sl}_3$ weight lattice, with $L$ steps and final weight~$\lambda$.\footnote{The aforementioned $\Z^2$ paths for $\mathrm{TL}$ should be thought of as paths in the $\mathfrak{sl}_2$ dominant Weyl chamber $\Z_{\geq 0}$ with $L$ steps and final weight $k$.} The basis vectors are obtained from an explicit map $P\mapsto G(P)$ that generalizes the growth algorithm \cite{khovanov_1999_web} to non-zero $\mathfrak{sl}_3$ weights. The map constructs $G(P)$ recursively by concatenating basic pieces of webs; see \cref{def: growth algorithm} for a detailed account.

We demonstrate in appendix~\ref{app: cell module basis vectors} the basis vectors of all cell modules of $\mathcal K_4$, along with their indexing lattice paths. Formulae for dimensions of $\mathcal{K}_L$ and its cell modules are given in \cref{sec: dimension formulae}. 

The gluing for $\mathcal K_L$ is shown in \cref{fig: cellular structure of KL(q)}: let $\dagger$ be the mirror reflection of graphs as in the aforementioned TL case, with orientations of edges reversed after reflection. Take $G(P),G(Q)$ to be two basis vectors of $\Delta_{12}(4\lambda_1+2\lambda_2)$, and place $G(P)$, $G(Q)^\dagger$ on top and bottom, respectively. In order to connect them, we put a sequence of generators $H^\pm$ (see \cref{fig: generating set}) in between, which we call a \emph{bottleneck} and denote by $\mathcal H^S_T$. This bottleneck depends on $S,T$, which encode the edge orientations of $G(P)$ (resp. $G(Q)^\dagger$) on the bottom (resp. top). It is constructed explicitly in \cref{cor: wiring construction for H}. The resulting cellular basis element is
\begin{align*}
    C^\lambda_{P,Q} = G(P) \mathcal H^S_T G(Q)^\dagger.
\end{align*}
We remark that in the Temperley-Lieb cellular structure, the bottleneck is trivially the identity.

\begin{figure}[htbp]
    \centering
    \scalebox{1.2}{\input{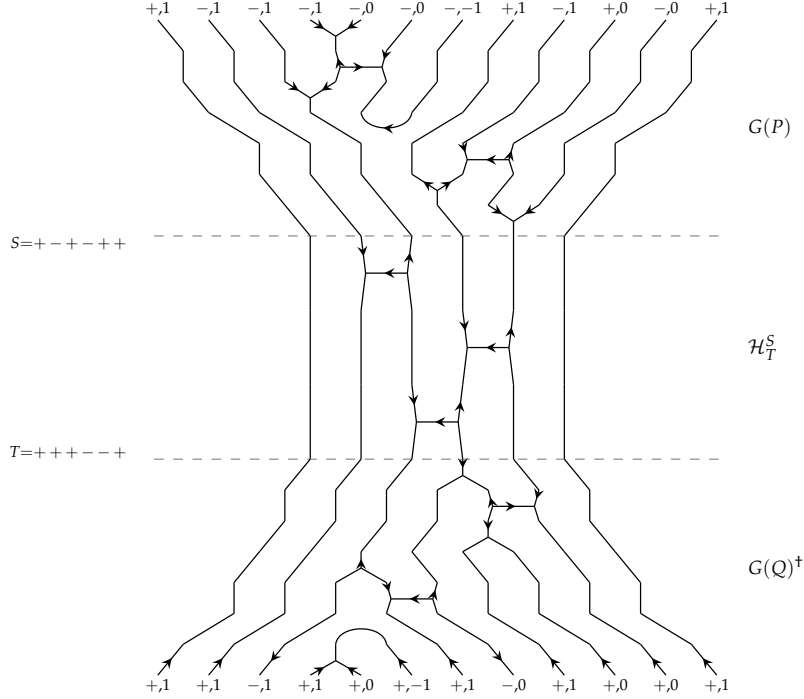}}
    \caption{An illustration of $C^\lambda_{P,Q}$. The lattice paths $P,Q$ are encoded by steps (see \cref{def: steps}), and are given at the top and bottom of the web respectively. They are read from left to right.}
    \label{fig: cellular structure of KL(q)}
\end{figure}

Now that the cellular basis of $\mathcal K_L$ and the basis vectors of its cell modules are given by graphs, the $\mathcal K_L$-action on cell modules can be analyzed combinatorially. We thereby provide an algorithm in \cref{sec: cellular structure guide} that computes the action of $\mathcal K_L$ on the aforementioned basis vectors of cell modules. Based on the algorithm, we can show that each cellular basis element $C^\lambda_{P,Q}$ is a linear combination of words in the generators $\GenSet$ of $\mathcal K_L$, which proves \cref{thm: intro FFT}. The technical aspects of this analysis are given in \cref{sec: generators of K_L}.

As an application of the cellular structure of $\mathcal K_L$, we get a bimodule decomposition:
\begin{theorem}[\Cref{thm: bimodule decomposition}]\label{thm: intro bimodule}
    The $(\uqslt,\mathcal K_L)$-bimodule on $(V^+\oplus V^-)^{\otimes L}$ decomposes as:
\begin{align}\label{eqn: intro bimodule decomposition}
        (V^+\oplus V^-)^{\otimes L}
        \cong
        \bigoplus_{\lambda\in\Lambda_+(L)}
        L_q(\lambda)\otimes \Delta_L(\lambda)^\dagger,
    \end{align}
where $L_q(\lambda)$ is the irreducible highest-weight $\uqslt$-module of highest weight $\lambda$, and $\Delta_L(\lambda)^\dagger$ is the right cell module with action $v\cdot a:=a^\dagger v$. As a consequence, the $\mathcal K_L$-cell modules form a complete set of irreducible $\mathcal K_L$-modules.
\end{theorem}

\medskip

The explicitness of the cellular structure we just described, especially the basis vectors for cell modules, is crucial in the proof of \cref{thm: intro FFT}. Let us review some related constructions of cellular structures. In \cite{elias2015lightladdersclaspconjectures,martin2026cellmodules}, a cellular structure is constructed for the $\mathfrak{sl}_n$-web category, and is used to compute Gram determinants on the $\Hom$-spaces of this category. Their cellular basis elements are linear combinations of webs with implicit coefficients. As a vector space, $\mathcal K_L$ is a direct sum of $\Hom$-spaces of the $\mathfrak{sl}_3$-web category, so it is possible to obtain a cellular structure for $\mathcal K_L$ using their cellular structure for the category. However, the cellular basis thus obtained would not be explicit enough to prove \cref{thm: intro FFT}. An $\mathfrak{sl}_3$-web algebra is defined in \cite{mackaay2014sl3web} as an algebra of $\mathfrak{sl}_3$-foams between Kuperberg webs, and is shown to be graded cellular. Its cellularity does not give a cellular structure for $\mathcal K_L$.

\medskip

\paragraph{Specializations.} In the previous discussions, $q$ is a formal parameter. Now we briefly describe which results remain when $q$ is specialized to $q_0\in\C^\times$. Put $R=\C[q,q^{-1}]$ and let $\mathcal K_{L,R}$ be the $R$-algebra with the same web basis as $\mathcal K_L$. Indeed, the coefficients that the Kuperberg relations produce lie in $R$, hence the multiplication coefficients from the cellular structure of $\mathcal K_L$ lie in $R$ as well. In particular, $\mathcal K_L\cong\mathcal K_{L,R}\otimes_R\C(q)$. Define
\[
    \mathcal K_L(q_0)
    :=\mathcal K_{L,R}\otimes_{R,\,q\mapsto q_0}\C,
\]
where $\C$ is considered as a $R$-module, with $q$ acting as multiplication by $q_0$. Observe that the arguments in \cref{sec: cellular structure} hold over $R$, in particular the anti-involution $\dagger$ preserves $\mathcal K_{L,R}$, and \cref{prp: factorized multiplication} holds as well. Thus the same cell datum from \cref{thm: cellular structure} gives a cellular structure on $\mathcal K_L(q_0)$ for every $q_0\in\C^\times$.

Furthermore, the FFT in \cref{thm: intro FFT} holds for specialization of $q\mapsto q_0$ including the cases when $q_0$ is a root of unity, except when $[2]_{q_0}= 0$ or $[3]_{q_0}=0$. Indeed, in the proof of \cref{thm: intro FFT}, \cref{lem: arc reversal,lem: triplet reversal} require $[2]_{q_0}\neq 0$ and $[3]_{q_0}\neq 0$. However, we expect that with an appropriate normalization, \cref{thm: intro FFT} holds for specialization to all roots of unity.

Finally, the bimodule decomposition in \cref{thm: intro bimodule} uses semisimplicity of $(V^+\oplus V^-)^{\otimes L}$ over $\uqslt$. Hence the result holds for specialization of $q$ to generic (non-roots of unity) $q_0$.

\subsection{Outlook}

\paragraph{Defining relations.} The second fundamental theorem of invariant theory for $\mathcal K$ from \eqref{eqn: generating set} describes the kernel $\ker(\Psi)$ of $\Psi: \Alg_{\C(q)}(\GenSet)\to \mathcal K$ as an ideal generated by a set $\mathsf R:=\{E_i\}_{i\in I}\subset \Alg_{\C(q)}(\GenSet)$, whose elements $E_i$ are called \emph{defining relations}. In a web presentation, defining relations are encoded implicitly by local reduction rules for webs, as in \cref{thm: intro SFT}. It remains an interesting problem to write $\mathsf R$ explicitly for the algebra $\mathcal K_L$ from \eqref{eqn: centralizing property}, and a web presentation is thus an intermediate step towards an ordinary presentation. Based on the analysis in \cref{sec: local surgery}, we conjecture: Based on the analysis in \cref{sec: local surgery}, we conjecture:
\begin{conjecture}
    The algebra $\mathcal K_L$ admits a set of defining relations each supported on at most six strands, i.e. every generator appearing in a relation is the identity outside a common set of at most six strand positions.
\end{conjecture}

\paragraph{Web presentation for other algebras.}
A natural generalization of \cref{thm: intro FFT} to higher rank is to find web generators for the algebras 
\[
\End_{U_q(\mathfrak{sl}_n)} \left(\bigoplus_{i=1}^{n-1}V_{\omega_i} \right)^{\otimes L},
\]
where $V_{\omega_i}$ is the fundamental $U_q(\mathfrak{sl}_n)$-module with highest weight $\omega_i$. The $\mathfrak{sl}_n$ webs \cite{cautis_kamnitzer_morrison_2014} provide the diagrammatic framework, and to find a web presentation of this algebra following the method of the current paper, a key missing ingredient is a recursive construction of $\mathfrak{sl}_n$ basis webs analogous to the growth algorithm $G$ used in the cellular structure of $\mathcal K_L$.

Another case of interest is the higher-spin algebras in rank one. Consider 
\[
\mathcal K^{(j)}:=\End_{U_q(\mathfrak{sl}_2)}(V_j^{\otimes L}),
\]
where $V_j$ is the irreducible $(j+1)$-dimensional $\uqtwo$-module. These algebras are described as quotients of the fused Hecke algebras \cite[Prp.\,9.5\,(2)]{crampe_poulain_dandecy_2023}, but a presentation is only available for $\mathcal K^{(1)}$, which is $\mathrm{TL}_L(\delta)$ by \eqref{eqn: TL centralizing property}; and for $\mathcal K^{(2)}$ which is given in \cite[Thm.\,2.6]{lehrer_zhang_2008_endomorphisms}, using a specialization of the BMW algebra. For $j>2$, finding a presentation remains an open problem, and we propose to find a web presentation as a first step, based on the valenced Temperley-Lieb diagrams in \cite{flores_peltola_2020_higher_spin}. Indeed, for $j=1$ these diagrams reduce to the usual $A_1$ webs. While for $j=2$, the local reduction rules for the valenced Temperley-Lieb diagrams are known, because $\mathcal K^{(2)}$ is isomorphic \cite[Thm.\,7.2.1\,\&\,Cor.\,8.2.4]{dallaire_2024_chromatic_category} to the \emph{chromatic algebra} \cite{fendley_krushkal_2010_link_invariants}. This algebra has a basis of trivalent graphs, and local reduction rules that preserve the chromatic polynomial. To complete the web presentation for $\mathcal K^{(2)}$, it suffices to find the web generators of the chromatic algebra using the results of \cite{lehrer_zhang_2008_endomorphisms}. For $j>2$, the local reduction rules for the valenced Temperley-Lieb diagrams will have to be worked out.

\paragraph{Applications in statistical mechanics.} The two-dimensional $3$-state Potts model is a model of magnetic spins on a lattice, where each spin takes $3$ possible values. The \emph{domain walls} are random curves such that the value of the spin changes upon crossing the curve. They are described \cite{lafay_2021_uqsln} by a model of random $\mathfrak{sl}_3$ webs specialized to $q_0=e^{i\pi/4}$. Many physically relevant properties of the Potts model, e.g.\,the correlations of spins and clusters of spins, are encoded in the spectrum of a special element $T\in \mathcal K_L(q_0)$ acting on $(V^+\oplus V^-)^{\otimes L}$, which is called the \emph{transfer matrix} in the physics literature. Here $L$ is the width of the lattice. The cellular structure of $\mathcal K_L(q_0)$ is crucial to a spectral analysis of $T$, for both technical and conceptual reasons. Technically, to study a real magnet, $L$ has to take a large value, and the large $\dim (V^+\oplus V^-)^{\otimes L}$ makes the analysis of subleading eigenvalues computationally challenging \cite{lafay_2022_uqsl3}. However, due to the bimodule decomposition \eqref{eqn: intro bimodule decomposition}, one may instead study the spectrum of $T$ on the cell modules $\Delta_L(\lambda)$, which are of much smaller dimension.\footnote{Strictly speaking, the bimodule decomposition does not hold for $q_0=e^{i\pi/4}$. However, since the eigenvalues of $T$ depend smoothly on $q$, in the spectral analysis we consider a generic $q$ close to $q_0$, and then take the limit.} Conceptually, the bimodule decomposition \eqref{eqn: intro bimodule decomposition} partitions the eigenvalues of $T$ on $(V^+\oplus V^-)^{\otimes L}$ into subsets indexed by $\mathfrak{sl}_3$ weights $\lambda$, according to the cell module they belong to. This partition has physical significance, as it allows an identification of the critical exponents, such as the thin-domain wall exponents considered in \cite{dubail2010critical}, in terms of $\mathfrak{sl}_3$-web cell modules.


\subsection{A list of notations}

Throughout the paper we use typography to record the nature of an object: sans-serif letters ($\Web$) denote combinatorial sets or distinguished bases, script letters ($\W$) denote vector spaces or modules, calligraphic letters ($\mathcal W$) denote
algebras, and bold letters ($\mathbf{W}$) denote categories. We use the quantum integers
\begin{align}\label{eqn: quantum integers}
    [n]_q := q^{n-1} + q^{n-3} + \cdots + q^{-(n-1)} = \sum_{k=0}^{n-1} q^{n-1-2k} = \frac{q^n-q^{-n}}{q-q^{-1}}.
\end{align}
We denote by $\N$ the set of nonnegative integers.
\begin{center}
\renewcommand{\arraystretch}{1.25}
\small
\begin{longtable}{|p{0.22\textwidth}|p{0.53\textwidth}|p{0.17\textwidth}|}
    \hline
    Notation & Description & Definition \\
    \hline
    $\C(q)$ & Ground field. &   \\
    \hline
    $\uqslt$ & The quantum group $U_q(\mathfrak{sl}_3)$. & appendix~\ref{app: quantum group conventions} \\
    \hline
    $V^+, V^-$ & The two fundamental representations of $\uqslt$. & appendix~\ref{app: quantum group conventions}.\\
    \hline
    $\mathscr{V}_L$ & The $\uqslt$-module $(V^+\oplus V^-)^{\otimes L}$. & \cref{thm: Schur-Weyl duality}.\\
    \hline
    $\Fsp,\ \Sp$ & The free spider and the Kuperberg spider. & \cref{def: Sp} \\
    \hline
    $s\in\{+,-\}$; $S$ and $|S|$ & A sign, a sign string, and its length. & \cref{def: nets} \\
    \hline
    $\Sigma_L$ & The set $\{+,-\}^L$ of sign strings of length $L$. & \cref{def: Omega} \\
    \hline
    $S^\vee$ & The dual sign string $(-s_{|S|},\ldots,-s_1)$. & \eqref{eqn: duality}.\\
    \hline
    $\ccup_S,\ \ccap_S$ & Coevaluation and evaluation webs for a sign string. & \cref{def: cup and cap for sign strings}\\
    \hline
    $\Web^T_S,\ \W^T_S$ & Reduced basis-web set and web vector space with boundary sign strings $S,T$. & \eqref{eqn: web set and web space}\\
    \hline
    $\Lambda,\ \Lambda_+$ & The $\mathfrak{sl}_3$ weight lattice and the dominant Weyl chamber. & \cref{sec: weight lattice and paths} \\
    \hline
    $\Lambda_+(L)$ & The dominant Weyl alcove at level $L$. & \cref{def: LambdaL}\\
    \hline
    $P=(S,J)$ & Lattice path encoded by a sign string and a state string. & \cref{def: lattice path}\\
    \hline
    $\wt(P),\ \wt(S)$ & Weight of paths and sign strings. & \cref{def: lattice path,def: weight of sign string}\\
    \hline
    $\Path_S(\lambda),\ \Path^+_S(\lambda)$ & Paths, respectively dominant paths, with sign string $S$ and endpoint $\lambda$. & \cref{def: lattice path}\\
    \hline
    $\Path_L(\lambda),\ \Path^+_L(\lambda)$ & Paths, respectively dominant paths, of length $L$ and endpoint $\lambda$. & \cref{def: lattice path}\\
    \hline
    $P_S^{\mathrm{hw}}$ & The highest-weight path with sign string $S$, meaning that every state entry is $1$. & \cref{eqn: highest-weight path} \\
    \hline
    $P^\vee$ & Dual path. & \cref{def: inversion map}\\
    \hline
    $P \mapsto G(P)$ & The growth algorithm. & \cref{def: growth algorithm}\\
    \hline
    $\WebBasis{L}$ & Set of basis webs of $\mathcal K_L$. & \eqref{def: Omega}\\
    \hline
    $C$; $C^\lambda_{P,Q}$ & Cellular indexing bijection and the corresponding cellular basis element. & \cref{def: pairing map C}\\
    \hline
    $\Delta_L(\lambda)$ & Cell module of $\mathcal{K}_L$ indexed by $\lambda$. & \cref{def: cell modules}\\
    \hline
    $\GenSet_L,\ \Alg_{\C(q)}(\GenSet_L)$ & Local generating set and the unital algebra it generates. & \cref{thm: generators of K_L}\\
    \hline
    $\Mir,\ \dagger$ & Vertical mirror automorphism and the anti-involution in the cellular datum. & \cref{def: dagger,lem: elementary geometry}\\
    \hline
\end{longtable}
\end{center}

\paragraph{Acknowledgments} 
 The authors acknowledge the hospitality of Institut Pascal in Orsay during the program {\sl CFT: Algebraic, Topological and Probabilistic Approaches in Conformal Field Theory} in 2025. A.M.G.\ acknowledges support by C.N.R.S.\ and l'Agence nationale de la recherche
(ANR grant ``New algebraic structures in quantum integrability: towards 3D" NASQI3D ANR-24-CE40-7252). 
A.M.G.\ is also grateful to the laboratory
LPENS in Paris for their kind hospitality in 2024-2026.
Y.L.\ acknowledges the hospitality of Beijing International Center for Mathematical Research during the program \emph{Peking Summer Program in Probability: Interactions with Quantum Field Theory} in 2026, where part of this paper was written.

\section{Growing Kuperberg webs}\label{sec: backgrounds}

In this section, we provide a growth algorithm constructing Kuperberg $A_2$ webs from lattice paths in the $\mathfrak{sl}_3$ dominant Weyl chamber.
We start by reviewing the $A_2$ webs and the associated category $\Sp$ in \cref{sec: kuperberg webs,sec: folding and duality}. In \cref{sec: equivalence Sp to Rep} we review the monoidal equivalence $\Gamma:\Sp\to\Rep$, where $\Rep$ is the full subcategory of finite-dimensional $\uqslt$-modules whose objects are tensor products of the fundamental representations. In \cref{sec: weight lattice and paths} we recall the $\mathfrak{sl}_3$ weight lattice and (dominant) lattice paths.
Finally, in \cref{sec: growth algorithm} we introduce the growth algorithm that generalizes the one from \cite{khovanov_1999_web} to dominant lattice paths with arbitrary weights, and give
a minimal set of growth rules. To the best of our knowledge, this generalization and the minimal growth rules have not appeared previously in the literature.

\subsection{The Kuperberg spiders}\label{sec: kuperberg webs}

Spiders are diagrammatic categories introduced by Kuperberg in \cite{kuperberg_1996_spiders} to study the representation theory of rank-two simple Lie algebras and their associated quantum groups. In this paper we focus on the spider of $A_2$ type, whose relevant diagrams are \textit{nets} and \textit{webs}, defined as follows:

\begin{definition}\label{def: nets}
    A \textup{monomial net} is an oriented planar graph embedded in a closed rectangle, considered up to isotopy relative to the boundary, such that
    \begin{enumerate}
        \item [(1)] Vertices of the graph are embedded either in the interior of the rectangle, or on the top and bottom boundaries.
        \item [(2)] Internal vertices are trivalent, and boundary vertices have valence one.
        \item [(3)] Each vertex is either a \textrm{source} (all edges oriented out of the vertex) or a \textrm{sink} (all edges oriented in).
    \end{enumerate}
\end{definition}

\begin{remark}
     As in the introduction, nets and webs (defined below) can be considered abstract bipartite graphs satisfying the following conditions: 1.\ their vertices are partitioned into three groups: internal (trivalent), top and bottom (univalent); 2.\ they admit an embedding into the rectangle, such that internal vertices are in the interior of the rectangle, and top/bottom vertices are on the top/bottom boundaries of the rectangle.
\end{remark}

The source/sink distinction endows monomial nets with a bipartite structure. We assign a $+$ sign to sources on the bottom boundary and sinks on the top boundary. 
Similarly we assign a $-$ sign to sinks on the bottom boundary and sources on the top boundary.
Reading the signs on the top (resp. bottom) boundary from left to right, we obtain the top (resp.\ bottom) \textit{sign string} of a monomial net. For a sign string $S$, we denote by $|S|$ its number of elements, and write
\begin{align*}
    \Sigma_L:=\{+,-\}^L,
    \qquad
    \Sigma_*:=\bigsqcup_{L\geq0}\Sigma_L.
\end{align*}
Thus $\Sigma_L$ is the set of sign strings of length $L$, including the
empty sign string $\emptyset\in\Sigma_0$.  \Cref{fig: web sign strings}
shows an example of a net.

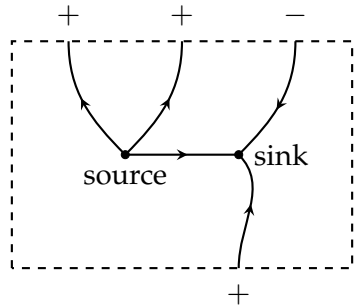
\begin{figure}[htbp]
    \centering
    \begin{tikzpicture}[scale=0.75,
        midarrow/.style={postaction={decorate,decoration={
            markings, mark=at position 0.55 with {\arrow{stealth}}
        }}}
    ]
        \draw[thick, dashed] (-3,-2) rectangle (3,2);
        
        \node[above] at (-2,2.1) {$+$};
        \node[above] at (0,2.1) {$+$};
        \node[above] at (2,2.1) {$-$};
        \node[below] at (1,-2.1) {$+$};
        
        \coordinate (v1) at (-1, 0); 
        \coordinate (v2) at (1, 0);  
        
        \draw[thick, mid arrow] (v1) to[out=135, in=270] (-2,2);
        \draw[thick, mid arrow] (v1) to[out=45, in=270] (0,2);
        
        \draw[thick, mid arrow] (v1) -- (v2);
        
        \draw[thick, style=mid arrow] (2,2) to[out=270, in=45] (v2);
        
        \draw[thick, style=mid arrow] (1,-2) to[out=90, in=315] (v2);
        
        \filldraw (v1) circle (2pt) node[below=2pt] {source};
        \filldraw (v2) circle (2pt) node[right=2pt] {sink};
    \end{tikzpicture}
    \caption{A Kuperberg net with bottom sign string $(+)$ and top sign string $(+, +, -)$.}
    \label{fig: web sign strings}
\end{figure}

For sign strings $S,T$, define the monomial-net set and its free
$\C(q)$-span by
\begin{align}\label{eqn: net set}
    \Net_S^{T}&:=\{\text{monomial nets with bottom sign string $S$ and top sign string $T$}\},\\
    \mathscr N_S^T&:=\Span_{\C(q)}\Net_S^T.
\end{align}

\begin{definition}\label{def: web}
    The \textup{web vector space} with boundary $S\to T$ is the quotient
    \begin{align*}
        \W_S^T:=\mathscr N_S^T/\sim,
    \end{align*}
    where the equivalence $\sim$ is generated by the local relations in \cref{fig: Kuperberg relations}.
\end{definition}

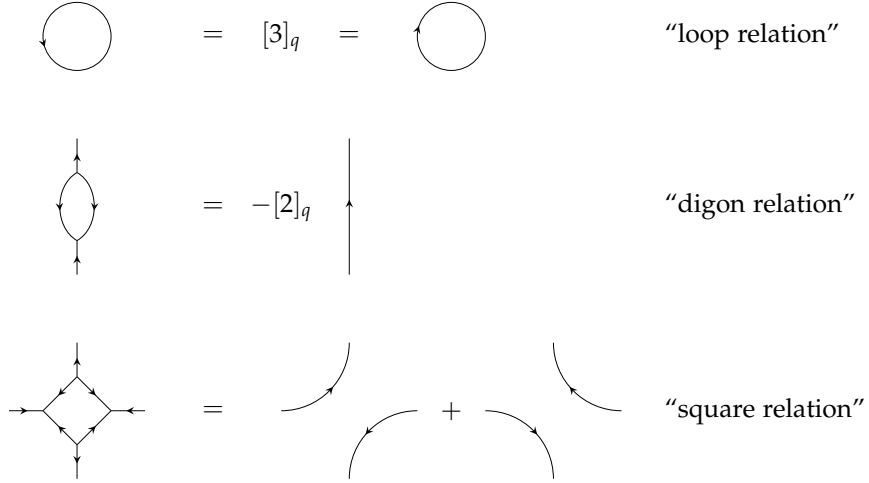
\begin{figure}[htbp]
        \centering
        \scalebox{0.9}{\begin{tikzpicture}[scale=0.5]
    \begin{scope}[shift={(0,0)}]
        \draw [style=mid arrow] (1,0) arc (0:360:1);
        \draw [style=mid arrow] (12,0) arc (360:0:1);
        
        \node at (4,0) {$=$};
        \node at (6,0) {$[3]_q$};
        \node at (8,0) {$=$};
        
        \node[right] at (17,0) {``loop relation''};
    \end{scope}

    \begin{scope}[shift={(0,-5)}]
        \draw [style=mid arrow] (0,1) -- (0,2);
        \draw [style=mid arrow] (0,-2) -- (0,-1);
        
        \draw [style=mid arrow] (0,1) to[bend right=60] (0,-1);
        \draw [style=mid arrow] (0,1) to[bend left=60] (0,-1);
        
        \node at (4,0) {$=$};
        \node at (6,0) {$-[2]_q$};
        
        \begin{scope}[shift={(8,0)}]
            \draw [style=mid arrow] (0,-2) -- (0,2);
        \end{scope}
        
        \node[right] at (17,0) {``digon relation''};
    \end{scope}

    \begin{scope}[shift={(0,-11)}]
        \draw [style=mid arrow] (0,1) -- (0,2);   
        \draw [style=mid arrow] (0,-1) -- (0,-2); 
        \draw [style=mid arrow] (-2,0) -- (-1,0); 
        \draw [style=mid arrow] (2,0) -- (1,0);   
        
        \draw [style=mid arrow] (0,1) -- (-1,0);  
        \draw [style=mid arrow] (0,1) -- (1,0);   
        \draw [style=mid arrow] (0,-1) -- (-1,0); 
        \draw [style=mid arrow] (0,-1) -- (1,0);  
        
        \node at (4,0) {$=$};
        
        \begin{scope}[shift={(8,0)}]
            \draw [style=mid arrow] (-2,0) to[out=0, in=270] (0,2);
            \draw [style=mid arrow] (2,0) to[out=180, in=90] (0,-2);
        \end{scope}
        
        \node at (11,0) {$+$};
        
        \begin{scope}[shift={(14,0)}]
            \draw [style=mid arrow] (-2,0) to[out=0, in=90] (0,-2);
            \draw [style=mid arrow] (2,0) to[out=180, in=270] (0,2);
        \end{scope}
        
        \node[right] at (17,0) {``square relation''};
    \end{scope}
\end{tikzpicture}}
        \caption{The Kuperberg relations.}
        \label{fig: Kuperberg relations}
\end{figure}

\begin{definition}\label{def: reduced web}
    A \textup{basis web}, or \textup{monomial web}, is a monomial net
    that does not contain any internal face of degree $0$ (loop), $2$
    (digon), or $4$ (square).\footnote{These are called \textit{non-elliptic
    webs} in \cite{kuperberg_1996_spiders}.}
\end{definition}

Through the Kuperberg relations, every monomial net is expressed as a linear combination of basis webs.  Distinct basis webs determine distinct
    equivalence classes of nets, so they form a basis of $\W_S^T$. We write
\begin{align}\label{eqn: web set and web space}
    \Web_{S}^{T}&:=\{\text{basis webs with bottom sign string $S$ and top sign string $T$}\},\\
    \W_S^{T}&:=\Span_{\C(q)} \Web_S^{T}.
\end{align}
Thus $w\in\Web_S^T$ denotes an individual basis web, whereas
$x\in\W_S^T$ may be an arbitrary linear combination of basis webs.

\begin{definition}\label{def: Sp}
The \textup{Kuperberg spider} $\Sp$ is a $\C(q)$-linear category with the following data:
    \begin{enumerate}
        \item Objects are finite sequences of $\pm$ signs called sign strings.
        \item Given two objects $S,T$, the $\Hom$-space is
        \begin{align}
            \Hom_{\Sp}(S,T):=\W_S^{T}.
        \end{align}
        \item For morphisms $f\in \W_{S_1}^{S_2}$ and $g\in \W_{S_2}^{S_3}$, the composition $g \circ f$ is obtained by stacking $g$ on top of $f$ and extending the operation bilinearly. \Cref{fig: composition example} gives an example. For brevity, we often write $gf$ instead of $g \circ f$. Composition in the category is defined only for matching intermediate boundaries; the zero convention for mismatched boundaries is introduced later in the algebra $\mathcal K_L$.
        \item For each object $S$, the identity morphism $\id_S$ is the web in $\Web_S^S$ that consists of $|S|$ parallel vertical strands, oriented according to $S$.
    \end{enumerate}
\end{definition}

\begin{figure}[htbp]
    \centering
    \scalebox{0.8}{\begin{tikzpicture}
    \begin{scope}[shift={(0,2.5)}]
        \draw[thick, dashed] (-3,0) rectangle (3,2);
        
        \node at (-1,1) {};
        \node at (1,1) {};
        
        \draw[thick, style=mid arrow] (1,1) -- (-1,1);
        
        \draw[thick, style=mid arrow] (1,1) -- (1,2);
        \node[above] at (1,2) {$+$};
        
        \draw[thick, style=mid arrow] (-2,0) to[out=90, in=225] (-1,1);
        \node[below] at (-2,0) {$+$};
        
        \draw[thick, style=mid arrow] (0,0) to[out=90, in=315] (-1,1);
        \node[below] at (0,0) {$+$};
        
        \draw[thick, style=mid arrow] (1,1) to[out=315, in=90] (2,0);
        \node[below] at (2,0) {$-$};
        
        \node at (-3.5, 1) {$g$};
    \end{scope}
    
    \node at (0, 1.25) {$\circ$};
    
    \begin{scope}[shift={(0,-0.5)}]
        \draw[thick, dashed] (-3,-2) rectangle (3,0);
        
        \node at (-1,-1) {};
        \node at (1,-1) {};
        
        \draw[thick, style=mid arrow] (-1,-1) -- (1,-1);
        
        \draw[thick, style=mid arrow] (1,-2) -- (1,-1);
        \node[below] at (1,-2) {$+$};
        
        \draw[thick, style=mid arrow] (-1,-1) to[out=135, in=270] (-2,0);
        \node[above] at (-2,0) {$+$};
        
        \draw[thick, style=mid arrow] (-1,-1) to[out=45, in=270] (0,0);
        \node[above] at (0,0) {$+$};
        
        \draw[thick, style=mid arrow] (2,0) to[out=270, in=45] (1,-1);
        \node[above] at (2,0) {$-$};
        
        \node at (-3.5, -1) {$f$};
    \end{scope}
    
    \node at (3.5, 1.25) {$=$};
    
    \begin{scope}[shift={(7.5,0.5)}]
        \draw[thick, dashed] (-3,-2) rectangle (3,3);
        
        \node at (-1,2) {}; 
        \node at (1,2) {};  
        \node at (-1,-1) {}; 
        \node at (1,-1) {};  
        
        \draw[thick, style=mid arrow] (-1,-1) to[out=135, in=225] (-1,2);
        \draw[thick, style=mid arrow] (-1,-1) to[out=45, in=315] (-1,2);
        
        \draw[thick, style=mid arrow] (-1,-1) -- (1,-1);
        \draw[thick, style=mid arrow] (1,2) -- (-1,2);
        
        \draw[thick, style=mid arrow] (1,2) to[out=315, in=45] (1,-1);
        
        \draw[thick, style=mid arrow] (1,-2) -- (1,-1);
        \node[below] at (1,-2) {$+$};
        
        \draw[thick, style=mid arrow] (1,2) -- (1,3);
        \node[above] at (1,3) {$+$};
        
        \node at (4.5, 0.5) {$=[2]_q^2 \id_{+}$};
    \end{scope}
    
\end{tikzpicture}}
    \caption{Composition of two webs $g \in \Hom_{\Sp}( (+,+,-), + )$ and $f \in \Hom_{\Sp}(+, (+,+,-) )$. Concatenating the top boundary of $f$ to the bottom of $g$ forms a digon and a square. One may decide to apply either the digon or the square relation first, which does not affect the final result $g \circ f = [2]_q^2 \id_{+}$. }
    \label{fig: composition example}
\end{figure}
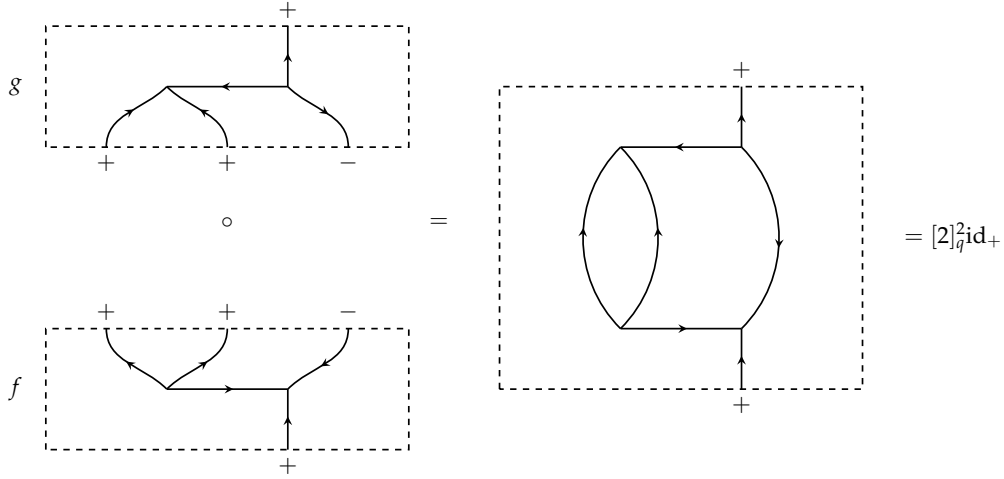

The monoidal product of $\Sp$ is defined by juxtaposition on both the object and morphism levels. The unit object for this product is the empty sign string $\emptyset$, whose identity morphism is the empty web $\id_{\emptyset}$. This makes $\Sp$ a strict monoidal $\C(q)$-linear category. Morphisms in $\Sp$ are generated by linear combinations, monoidal products, and compositions of the eight basic morphisms in \Cref{fig: elementary webs}. Therefore, the category $\Sp$ is presented by these generators, with the Kuperberg relations as defining relations \cref{fig: Kuperberg relations}. For example, the morphisms $f$ and $g$ in \cref{fig: composition example} are generated through the following identities:
\begin{align*}
    f&=(Y_{-} \otimes \id_{-})\circ Y_{+},\\
    g&=(\ccap_{+} \otimes \id_{+}) \circ (\ccap_{-}\otimes \id_{-} \otimes Y_{-}) \circ (\id_{+}\otimes Y_{+} \otimes \id_{-}).
\end{align*}

\begin{figure}[htbp]
    \centering
    \scalebox{0.8}{\begin{tikzpicture}

    \draw [dashed] (-5.5, 1.75) rectangle (-2.5, 3.75);
    \node at (-5, 4.05) {$-$};
    \node at (-3, 4.05) {$-$};
    \node at (-4, 1.45) {$+$};
    \draw [style=mid arrow] (-5, 3.75) to (-4, 3);
    \draw [style=mid arrow] (-3, 3.75) to (-4, 3);
    \draw [style=mid arrow] (-4, 1.75) to (-4, 3);
    \node at (-4, 0.75) {$Y_{+}$};

    \draw [dashed] (-1.5, 1.75) rectangle (1.5, 3.75);
    \node at (-1, 4.05) {$+$};
    \node at (1, 4.05) {$+$};
    \node at (0, 1.45) {$-$};
    \draw [style=mid arrow] (0, 3) to (0, 1.75);
    \draw [style=mid arrow] (0, 3) to (-1, 3.75);
    \draw [style=mid arrow] (0, 3) to (1, 3.75);
    \node at (0, 0.75) {$Y_{-}$};

    \draw [dashed] (2.5, 1.75) rectangle (5.5, 3.75);
    \node at (4, 4.05) {$+$};
    \node at (4, 1.45) {$+$};
    \draw [style=mid arrow] (4, 1.75) to (4, 3.75);
    \node at (4, 0.75) {$\id_{+}$};

    \draw [dashed] (6.5, 1.75) rectangle (9.5, 3.75);
    \node at (8, 4.05) {$-$};
    \node at (8, 1.45) {$-$};
    \draw [style=mid arrow] (8, 3.75) to (8, 1.75);
    \node at (8, 0.75) {$\id_{-}$};


    \draw [dashed] (-5.5, -2) rectangle (-2.5, 0);
    \node at (-5, 0.3) {$+$};
    \node at (-3, 0.3) {$-$};
    \draw [style=mid arrow, bend left=90, looseness=1.25] (-3, 0) to (-5, 0);
    \node at (-4, -2.75) {$\ccup_{+}$};

    \draw [dashed] (-1.5, -2) rectangle (1.5, 0);
    \node at (-1, 0.3) {$-$};
    \node at (1, 0.3) {$+$};
    \draw [style=mid arrow, bend right=90, looseness=1.25] (-1, 0) to (1, 0);
    \node at (0, -2.75) {$\ccup_{-}$};

    \draw [dashed] (2.5, -2) rectangle (5.5, 0);
    \node at (3, -2.3) {$+$};
    \node at (5, -2.3) {$-$};
    \draw [style=mid arrow, bend left=90, looseness=1.25] (3, -2) to (5, -2);
    \node at (4, -2.75) {$\ccap_{-}$};

    \draw [dashed] (6.5, -2) rectangle (9.5, 0);
    \node at (7, -2.3) {$-$};
    \node at (9, -2.3) {$+$};
    \draw [style=mid arrow, bend right=90, looseness=1.25] (9, -2) to (7, -2);
    \node at (8, -2.75) {$\ccap_{+}$};

\end{tikzpicture}}
    \caption{The eight basic webs. The signs of the $\ccup$ (resp.\ $\ccap$) webs are determined by the orientation of their boundary edges on the left (resp.\ right) side.}
    \label{fig: elementary webs}
\end{figure}
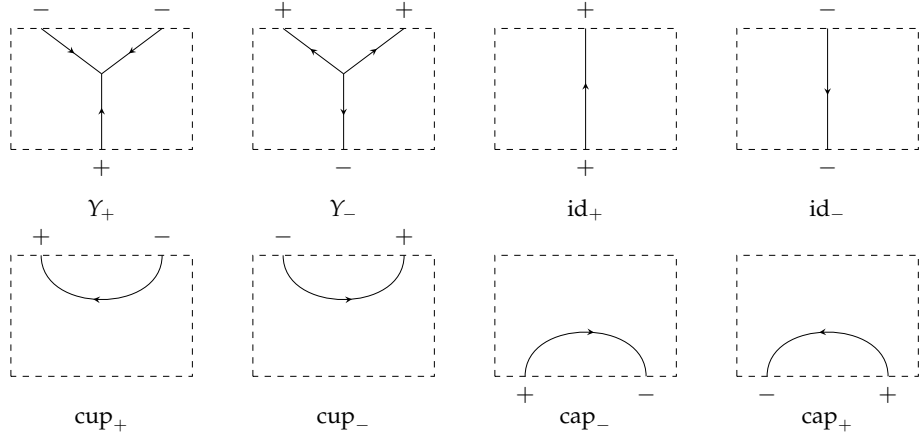

\subsection{Folding and duality}\label{sec: folding and duality}

We first define the duality $(-)^\vee$ on a sign string $(s_1,\dots,s_L)$ by 
\begin{align}\label{eqn: duality}
    (s_1,\dots,s_L)^\vee:=(-s_L,\dots,-s_1).
\end{align}
It is an anti-involution with respect to the monoidal product, i.e. for any two objects $S_1$ and $S_2$, we have $(S_1 S_2)^\vee = S_2^\vee S_1^\vee$ and $(S^\vee)^\vee = S$. Generalizing the $\ccup_{\pm}$ and $\ccap_{\mp}$ webs defined in \cref{fig: elementary webs}, we define webs $\ccup_S$ and $\ccap_S$ for any sign string $S$:

\begin{definition}[$\ccup_S$ and $\ccap_S$]\label{def: cup and cap for sign strings}
    Let $S$ be a sign string. We define $\ccup_S \in \Web^{S S^\vee}_{\emptyset}$ as the unique monomial web that consists of $|S|$ nested cups, connecting the $i$-th and the $(2|S|-i+1)$-th boundary edges on top. See \cref{fig: cupS and capS} for an example. Similarly, we define $\ccap_S \in \Web^{\emptyset}_{S^\vee S}$ as the unique monomial web consisting of $|S|$ parallel nested caps, connecting the $i$-th and the $(2|S|-i+1)$-th boundary edges on the bottom.
\end{definition}

\begin{figure}[htbp]
    \centering
    \scalebox{0.8}{\begin{tikzpicture}[scale=0.65]
    
    \node[above] at (-9, 3) {$+$};
    \node[above] at (-8, 3) {$-$};
    \node[above] at (-7, 3) {$+$};
    \node[above] at (-6, 3) {$+$};
    \node[above] at (-4, 3) {$-$};
    \node[above] at (-3, 3) {$-$};
    \node[above] at (-2, 3) {$+$};
    \node[above] at (-1, 3) {$-$};

    \draw [style=mid arrow] (-8, 3) to[out=-90, in=-90, looseness=1.25] (-2, 3);
    \draw [style=mid arrow] (-1, 3) to[out=-90, in=-90, looseness=1.25] (-9, 3);
    \draw [style=mid arrow] (-4, 3) to[out=-90, in=-90, looseness=1.25] (-6, 3);
    \draw [style=mid arrow] (-3, 3) to[out=-90, in=-90, looseness=1.25] (-7, 3);

    \node at (-5, -1.5) {$\ccup_{+-++}$};

    
    \node[below] at (1, 0) {$-$};
    \node[below] at (2, 0) {$-$};
    \node[below] at (3, 0) {$+$};
    \node[below] at (4, 0) {$-$};
    \node[below] at (6, 0) {$+$};
    \node[below] at (7, 0) {$-$};
    \node[below] at (8, 0) {$+$};
    \node[below] at (9, 0) {$+$};

    \draw [style=mid arrow] (4, 0) to[out=90, in=90, looseness=1.25] (6, 0);
    \draw [style=mid arrow] (7, 0) to[out=90, in=90, looseness=1.25] (3, 0);
    \draw [style=mid arrow] (2, 0) to[out=90, in=90, looseness=1.25] (8, 0);
    \draw [style=mid arrow] (1, 0) to[out=90, in=90, looseness=1.25] (9, 0);

    \node at (5, -1.5) {$\ccap_{+-++}$};

\end{tikzpicture}}
    \caption{The webs $\ccup_S$ and $\ccap_S$ for a sign string $S=(+,-,+,+)$.}
    \label{fig: cupS and capS}
\end{figure}
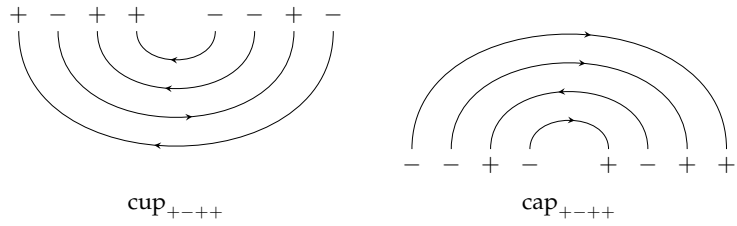

\begin{lemma}\label{lem: zigzag relations}
    Let $S$ be an object in $\Sp$. We have the following zigzag identities:
    \begin{align}\label{eqn: zigzag in Sp}
        \begin{aligned}
            (\ccap_{S^\vee} \otimes \id_{S}) \circ (\id_{S} \otimes \ccup_{S^\vee}) = \id_{S},\\
            (\id_{S^\vee} \otimes \ccap_{S^\vee}) \circ (\ccup_{S^\vee} \otimes \id_{S^\vee}) = \id_{S^\vee}.
        \end{aligned}
        \end{align}
    \end{lemma}
\begin{proof}
    This follows from the isotopy invariance of webs.
\end{proof}

\begin{definition}\label{def: folding and duality}
    Let $S_1, S_2$ be sign strings. We define the \textup{unfolding} operation $\Unfold_{S_1}$, the \textup{folding} operation $\Fold_{S_1}$, and the duality of morphisms $(-)^\vee$ as follows:
    \begin{align*}
        \Unfold_{S_1} : \W^{S_2}_{S_1} & \to \W^{S_2 S_1^\vee}_{\emptyset}, & w & \mapsto (w \otimes \id_{S_1^\vee}) \circ \ccup_{S_1}, \\
        \Fold_{S_1} : \W^{S_2 S_1^\vee}_{\emptyset} & \to \W^{S_2}_{S_1}, & w & \mapsto (\id_{S_2} \otimes \ccap_{S_1}) \circ (w \otimes \id_{S_1}), \\
        (-)^\vee : \W^{S_2}_{S_1} & \to \W^{S_1^\vee}_{S_2^\vee}, & w & \mapsto (\ccap_{S_2} \otimes \id_{S_1^\vee}) \circ (\id_{S_2^\vee} \otimes w \otimes \id_{S_1^\vee}) \circ ( \id_{S_2^\vee} \otimes \ccup_{S_1}).
    \end{align*}
    The operations are illustrated in \cref{fig: folding and duality}.
\end{definition}

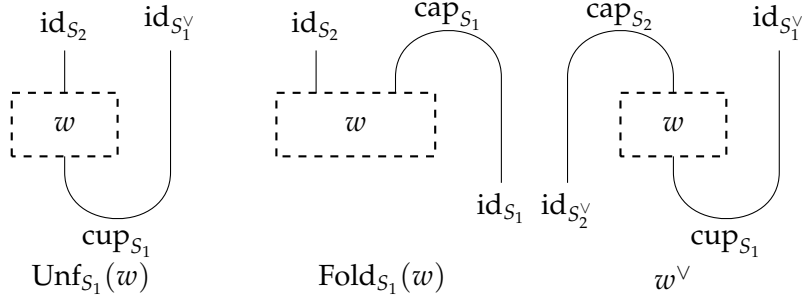
\begin{figure}[htbp]
    \centering
    \begin{tikzpicture}[scale=0.7]
        
        \begin{scope}[shift={(0,0)}]
            \draw[dashed, thick] (-1,0) rectangle (1, 1.2);
            \node at (0, 0.6) {$w$};
            
            \draw (0, 1.2) -- (0, 2);
            \node[above] at (0, 2) {$\id_{S_2}$};
            
            \draw (0, 0) -- (0, -0.3) to[out=-90, in=-90, looseness=1.5] (2, -0.3) -- (2, 2);
            \node[above] at (2, 2) {$\id_{S_1^\vee}$};
            \node[below] at (1, -1.2) {$\ccup_{S_1}$};
            
            \node at (0.5, -2.3) {$\Unfold_{S_1}(w)$};
        \end{scope}

        \begin{scope}[shift={(5.5,0)}]
            \draw[dashed, thick] (-1.5, 0) rectangle (1.5, 1.2);
            \node at (0, 0.6) {$w$};
            
            \draw (-0.75, 1.2) -- (-0.75, 2);
            \node[above] at (-0.75, 2) {$\id_{S_2}$};
            
            \draw (0.75, 1.2) -- (0.75, 1.5) to[out=90, in=90, looseness=1.5] (2.75, 1.5) -- (2.75, -0.5);
            \node[below] at (2.75, -0.5) {$\id_{S_1}$};
            \node[above] at (1.75, 2.2) {$\ccap_{S_1}$};
            
            \node at (0.5, -2.3) {$\Fold_{S_1}(w)$};
        \end{scope}

        \begin{scope}[shift={(11.5,0)}]
            \draw[dashed, thick] (-1, 0) rectangle (1, 1.2);
            \node at (0, 0.6) {$w$};
            
            \draw (0, 1.2) -- (0, 1.5) to[out=90, in=90, looseness=1.5] (-2, 1.5) -- (-2, -0.5);
            \node[below] at (-2, -0.5) {$\id_{S_2^\vee}$};
            \node[above] at (-1, 2.2) {$\ccap_{S_2}$};
            
            \draw (0, 0) -- (0, -0.3) to[out=-90, in=-90, looseness=1.5] (2, -0.3) -- (2, 2);
            \node[above] at (2, 2) {$\id_{S_1^\vee}$};
            \node[below] at (1, -1.2) {$\ccup_{S_1}$};
            
            \node at (0, -2.3) {$w^\vee$};
        \end{scope}
    \end{tikzpicture}
    \caption{Illustration of unfolding $\Unfold_{S_1}$, folding $\Fold_{S_1}$, and duality $w^\vee$. }
    \label{fig: folding and duality}
\end{figure}

As immediate consequences of the zigzag identities in \cref{lem: zigzag relations}, we obtain \cref{lem: folding is invertible}--\cref{lem: inserting zigzag into a product}:

\begin{lemma}\label{lem: folding is invertible}
    Folding and unfolding are inverses to each other. Specifically, for any $w\in \W^{S_2}_{S_1}$, we have
    \begin{align*}
        \Fold_{S_1} \circ \Unfold_{S_1}  = \id_{\W^{S_2}_{S_1}}, \quad
        \Unfold_{S_1} \circ \Fold_{S_1}  = \id_{\W^{S_2 S_1^\vee}_{\emptyset}}.
    \end{align*}
\end{lemma}

\begin{lemma}\label{lem: anti-involution property of *}
    The duality map $(-)^\vee$ is an anti-involution with respect to composition in $\Sp$, i.e. for webs $w_1,w_2$ we have
    \begin{align*}
        (w_1 w_2)^\vee = w_2^\vee w_1^\vee, \qquad (w_1^\vee)^\vee = w_1.
    \end{align*}
\end{lemma}

\begin{lemma}\label{lem: inserting zigzag into a product}
    Let $w_1\in \W_{S_1}^{S}$ and $w_2\in \W_{S}^{S_2}$. Then
    \begin{align*}
        \Unfold_{S_1}(w_2w_1)=(w_2 \otimes w_1^\vee) \circ \ccup_S.
    \end{align*}
\end{lemma}

\begin{remark}
    \Cref{lem: inserting zigzag into a product} explains the name ``unfolding'': the map $\Unfold$ unfolds the product $w_2 w_1$ into a monoidal product $w_2 \otimes w_1^\vee$, and connects the two factors with a cup web $\ccup_S$.
\end{remark}

\begin{lemma}\label{lem: duality as rotation}
    Let $w$ be a web. Then $w^\vee=O_{\pi}w$, where $O_{\pi}$ rotates a web by an angle of $\pi$.
\end{lemma}
\begin{proof}
    The fact is checked directly on the eight basic webs in \cref{fig: elementary webs}, and the general result follows from the facts that a general web $w$ is a composition of the basic webs, and that $(-)^\vee$ and $O_{\pi}$ are both anti-involutions with respect to composition.
\end{proof}

\begin{corollary}\label{cor: equivalent duality}
    The duality map $(-)^\vee$ satisfies
    \begin{align}\label{eqn: equivalent duality}
        (-)^\vee : \W^{S_2}_{S_1} & \to \W^{S_1^\vee}_{S_2^\vee}, & w & \mapsto (\id_{S_1^\vee} \otimes \ccap_{S_2^\vee} ) \circ (\id_{S_1^\vee} \otimes w \otimes \id_{S_2^\vee}) \circ (\ccup_{S_1^\vee} \otimes \id_{S_2^\vee} ).
    \end{align}
\end{corollary}
\begin{proof}
    An argument identical to the proof of \cref{lem: duality as rotation} shows that the displayed expression also coincides with $O_{\pi}w$.
\end{proof}

\subsection{The equivalence functor $\Gamma: \Sp \to \Rep$}\label{sec: equivalence Sp to Rep}
Recall the Drinfeld-Jimbo quantum group $\uqslt$. It is a Hopf algebra over $\C(q)$, with two fundamental representations $V^+$ and $V^-$, reviewed in appendix~\ref{app: quantum group conventions}. For a sign string $S=(s_1,\dots,s_n)$ with $s_i\in\{+,-\}$, define
\begin{align}\label{eqn: V^S}
    V^S := V^{s_1} \otimes \cdots \otimes V^{s_n}.
\end{align}
The action $\rho_S:\uqslt\to\End_{\C(q)}(V^S)$ defined by iterated coproduct \eqref{eqn: coproduct} makes $V^S$ a $\uqslt$ module. For the empty sign string, $V^\emptyset:=\C(q)$ is the trivial module.
\begin{definition}\label{def: Rep}
    Let $\Rep$ denote the full subcategory of the category of
    finite-dimensional $\uqslt$-modules, whose objects are $V^S$ for all sign strings $S$. It is a monoidal subcategory.
\end{definition}

Now we define a functor $\Gamma:\Sp \to \Rep$, following \cite[Sec.\,5]{kuperberg_1996_spiders}. On the object level it sends
\begin{align}\label{eqn: Gamma object}
    \Gamma: S\mapsto V^S,
\end{align}
and on the morphism level, the eight basic webs in \cref{fig: elementary webs} are sent by $\Gamma$ to the eight basic intertwiners $\id_{V^\pm}$, $\eva_{\pm}$, $\coev_{\pm}$, and $\mathcal{Y}_{\pm}$ in $\Rep$, whose precise definitions are in \eqref{eqn: coev}--\eqref{eqn: Y}:
\begin{align}\label{eqn: Gamma-morphism}
    \begin{aligned}
        \id_{\pm} &\mapsto \id_{V^{\pm}}\in \Hom_{\Rep}(V^\pm,V^\pm), \\
        Y_{\pm} &\mapsto \mathcal{Y}_{\pm}\in \Hom_{\Rep}(V^\pm,V^\mp\otimes V^\mp), \\
        \ccup_{\pm} &\mapsto \coev_{\pm}\in \Hom_{\Rep}(\C(q),V^\pm \otimes V^\mp), \\
        \ccap_{\pm} &\mapsto \eva_{\pm}\in \Hom_{\Rep}(V^\mp \otimes V^\pm,\C(q)).
    \end{aligned}
\end{align}
Hence $\Gamma$ is defined on the generators of $\Sp$. A straightforward calculation in \cref{app: intertwiners Gamma} shows that $\Gamma$ preserves the Kuperberg relations \cref{fig: Kuperberg relations} and the zigzag identities \eqref{eqn: zigzag in Sp} in $\Sp$. Therefore, $\Gamma$ is a well-defined functor. Furthermore, Kuperberg showed that it is fully faithful:
\begin{theorem}(\cite[Thms.\,5.1 \& 6.1]{kuperberg_1996_spiders})\label{thm: equivalence of Sp and Rep}
    The functor $\Gamma$ is a monoidal equivalence $\Sp\to\Rep$.
\end{theorem}

\begin{remark}
The terminology of \cite{kuperberg_1996_spiders} is very different from that in this work. Most importantly, $\Sp$ is termed the ``combinatorial spider''; $\Rep$ the ``algebraic spider''; $\Gamma$ is denoted by $\Phi$, and termed ``a morphism from combinatorial to algebraic spider''. However, the image of $\Gamma$ on the generators of $\Sp$ is not given as explicitly as in \eqref{eqn: Gamma-morphism}.

A generalization of \cref{thm: equivalence of Sp and Rep} to the entire $A_{n-1}$ family is given in \cite{cautis_kamnitzer_morrison_2014}. To do so, they generalize Kuperberg webs by introducing indices and ``tags'' on edges. To reduce these $A_2$ webs with tags and indices to the Kuperberg webs, see \cite[Sec.\,3.2 \& 3.3]{lafay_2021_uqsln} or \cite[Sec.\,1.5]{lauda_queffelec_rose_2015_skew_howe}.
\end{remark}

\subsection{The $\mathfrak{sl}_3$ weight lattice and lattice paths}\label{sec: weight lattice and paths}

Let $\Lambda = \Z\lambda_1 + \Z\lambda_2$ be the weight lattice of $\mathfrak{sl}_3$, where $\lambda_1=(-\frac{1}{2}, \frac{\sqrt{3}}{2})$ and $\lambda_2=(\frac{1}{2},\frac{\sqrt{3}}{2})$. Each $\mathfrak{sl}_3$-weight takes the form 
\begin{align*}
    \lambda=n\lambda_1+k\lambda_2,\quad  n\in \N, k\in \Z.
\end{align*}
The \textit{dominant Weyl chamber} $\Lambda_+$ \label{def: Lambda} is the sublattice of $\Lambda$ consisting of \textit{dominant weights}, which have $n,k\geq 0$ in the above expression, see \cref{fig: weight lattice}.

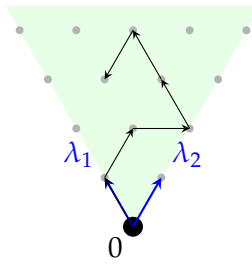
\begin{figure}[htbp]
\centering
\begin{tikzpicture}[
    scale=1.5,
    x={(-0.25cm,0.433cm)}, 
    y={(0.25cm,0.433cm)},
    >=stealth
]

\def\R{4} 

    \fill[green!10] (0,0) -- (0,\R+0.5) -- (\R+0.5,0) -- cycle;

    
    \foreach \x in {0,...,\R} {
        \foreach \y in {0,...,\R} {
            \pgfmathparse{max(abs(\x), abs(\y), abs(\x+\y)) <= \R ? 1 : 0}
            \ifnum\pgfmathresult=1
                \ifnum\y > -2
                    \fill[gray!60] (\x,\y) circle (1pt);
                \fi
            \fi
        }
    }
    
    \fill[black] (0,0) circle (2.5pt) node[below left] {$0$};

    \draw[->, blue, thick] (0,0) -- (0,1) node[above right] {$\lambda_2$};
    
    \draw[->, blue, thick] (0,0) -- (1,0) node[above left] {$\lambda_1$};

    \draw[->] (0,0) -- (1,0);
    \draw[->] (1,0) -- (1,1);
    \draw[->] (1,1) -- (0,2);
    \draw[->] (0,2) -- (1,2);
    \draw[->] (1,2) -- (2,2);
    \draw[->] (2,2) -- (2,1);

\end{tikzpicture}
\caption{A portion of the dominant $\mathfrak{sl}_3$ Weyl chamber, colored in light green. The weights $\lambda_1$ and $\lambda_2$ are shown as blue arrows. Dominant weights are shown as gray dots. The black arrows constitute the path from \cref{emp: path}.}
\label{fig: weight lattice}
\end{figure}

A partial ordering $\succ$ on $\Lambda$ is generated by the following relations: 
\begin{equation}\label{eqn: partial ordering}
    \begin{aligned}
        &a\lambda_1+b\lambda_2 \succ  (a-2)\lambda_1+(b+1)\lambda_2,\\
        &a\lambda_1+b\lambda_2 \succ  (a+1)\lambda_1+(b-2)\lambda_2.
    \end{aligned}
\end{equation}

\begin{definition}\label{def: LambdaL}
    For an integer $L\geq2$, define
    \begin{align}
        \Lambda_+(L):=\{n\lambda_1+k\lambda_2:n,k\in\Z_{\geq0},\ n+k\leq L\}.
    \end{align}
\end{definition}

\begin{definition}\label{def: steps}
    The \textup{steps} on $\Lambda$ are the six vectors $u_{s,j}:(s, j) \in \{+ , -\} \times \{1, 0, -1\}$, defined by $u_{+,1} = \lambda_1, u_{+,0} = \lambda_2 - \lambda_1$ and $u_{+,-1} = -\lambda_2$, and
    \begin{align}\label{eqn: negation of steps}
    -u_{s,j}=u_{-s,-j}, \quad s\in \{+,-\}, \ j\in \{1,0,-1\}.
    \end{align}
\end{definition}

The six steps are the vectors connecting the origin of $\Lambda$ to its nearest neighbors. It is convenient to draw these steps around a primitive triangle in $\Lambda$, so that the three steps with $s=+$ (resp.\ $s=-$) form a clockwise (resp.\ counterclockwise) contour around the triangle, see \cref{fig: steps in weight lattice}. 

\begin{figure}[htbp]
\centering
\begin{tikzpicture}
    \node at (-0.75,0.2) {$u_{+,1}$};
    \node at (0.9,0.2) {$u_{+,-1}$};
    \node at (0,1.2) {$u_{+,0}$};
    \begin{scope}[sloped, allow upside down]
        \draw  [style=mid arrow] (0,0) -- (-0.5,0.866) ;
        \draw  [style=mid arrow] (-0.5,0.866) -- (0.5,0.866) ;
        \draw  [style=mid arrow] (0.5,0.866) -- (0,0) ;
    \end{scope}
\end{tikzpicture} \quad
\begin{tikzpicture}
    \node at (-0.9,0.2) {$u_{-,-1}$};
    \node at (0.75,0.2) {$u_{-,1}$};
    \node at (0,1.2) {$u_{-,0}$};
    \begin{scope}[sloped, allow upside down]
        \draw  [style=mid arrow] (-0.5,0.866) -- (0,0) ;
        \draw  [style=mid arrow] (0,0) -- (0.5,0.866) ;
        \draw  [style=mid arrow] (0.5,0.866) -- (-0.5,0.866) ;
    \end{scope}
\end{tikzpicture}
\caption{Steps on the weight lattice.}
\label{fig: steps in weight lattice}
\end{figure}
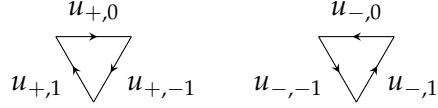

\begin{definition}\label{def: lattice path}
    A \textup{path} $P$ of length $L$ on the $\mathfrak{sl}_3$ weight lattice is a sequence of weights $(p_0, p_1, \dots, p_L)$ in $\Lambda$, such that $p_0 = 0$ and each step $(p_m - p_{m-1})$ belongs to the set of six steps defined above. Equivalently, a path is uniquely determined by its sequence of steps, i.e.\ an ordered pair $(S,J)$, where $S=(s_1,\dots,s_L)\in\Sigma_L$ is a \textup{sign string} and $J=(j_1,\dots,j_L)\in\State_L:=\{-1,0,1\}^L$ is a \textup{state string}, so that the $m$-th step of the path is $u_{s_m, j_m}$. We write $P=(S,J)$ for brevity.
    
    We adopt the following terminology for paths:
    \begin{enumerate}
        \item The \textup{length} of $P$ is the length of its sign string $S$, denoted by $|P| = L$.
        \item A path $P$ is \textup{dominant} if every point $p_m$ in the sequence lies within the dominant Weyl chamber $\Lambda_+$. 
        \item The \textup{weight} of a path $P$ is its end point $p_L$, denoted by $\wt(P)$.
        \item Let $P,Q$ be two paths. The \textup{concatenation} of $P=(S_P,J_P)$ and $Q=(S_Q,J_Q)$, denoted by $PQ$, is the path encoded by $(S_P S_Q, J_P J_Q)$, where $S_P S_Q$ and $J_P J_Q$ are the concatenations of the respective sequences.
    \end{enumerate}
    
    To distinguish paths with a fixed sign string from paths with only a fixed length, we use the following notation:
    \begin{align*}
        \Path_S(\lambda)
        &:=\{P: P\text{ has sign string } S,\wt(P)=\lambda\},
        &
        \Path_S^+(\lambda)
        &:=\{P\in\Path_S(\lambda):P\text{ is dominant}\},\\
        \Path_L(\lambda)
        &:=\bigsqcup_{S\in\Sigma_L}\Path_S(\lambda),
        &
        \Path_L^+(\lambda)
        &:=\bigsqcup_{S\in\Sigma_L}\Path_S^+(\lambda).
    \end{align*}
    We omit $\lambda$ when the endpoint is unrestricted, and omit the
    subscript when neither the sign string nor the length is fixed.  Thus
    $\Path_S$ is the set of all paths with sign string $S$, while
    $\Path^+(\lambda)$ is the set of all dominant paths ending at
    $\lambda$.
\end{definition}

\begin{example}\label{emp: path}
    An example of a dominant path $P$ of length $L=6$ and weight $\lambda=2\lambda_1 + \lambda_2$ is shown in \cref{fig: weight lattice}. The sign string of $P$ is $S=(+, -, +, +, +, +)$ and the state string is $J=(1, 1, 0, 1, 1, -1)$.
\end{example}

\begin{definition}\label{def: weight of sign string}
    For a sign string $S=(s_1,\dots,s_L)$, set
    \begin{align}
        \wt(S):=\sum_{m=1}^L u_{s_m,1}=\lambda_1 \cdot \#\{m: s_m=+\} + \lambda_2 \cdot \#\{m: s_m=-\}.
    \end{align}
\end{definition}
\begin{definition}\label{def: inversion map}
    Let $P=(S,J)$ be a path with length $L$, $J=(j_1,\dots,j_L)$. We extend the definition of $(-)^\vee$ from sign strings to paths by
    \begin{align*}
        P=(S,J) &\mapsto P^\vee=(S^\vee,J^\vee).
    \end{align*}
    Here $S^\vee$ is defined in \eqref{eqn: duality}, and $J^\vee := (-j_L, \dots, -j_1)$.
\end{definition}

\begin{definition}\label{def: order on paths}
    Paths of greater length have higher order. For paths with the same sign string $S$, we order equal-length paths lexicographically by their state strings $J$, using $-1<0<1$. For example, with fixed $S$,
    \[
        (0,-1,1) < (0,0,-1) < (1,-1,-1).
    \]
\end{definition}

\begin{lemma}\label{lem: weight and existence of paths}
    Let $L\geq 2$ and let $\lambda\in \Lambda$. The set $\Path_L^+(\lambda)$ is non-empty if and only if $\lambda \in \Lambda_+(L)$.
\end{lemma}
\begin{proof}
Suppose that $P\in \Path_L^+(\lambda)$ for some $\lambda\in \Lambda$. By dominance of $P$, we have $\lambda\in\Lambda_+$, so that $\lambda=n\lambda_1+k\lambda_2$ for some $n,k\geq0$. Since each allowed step changes the sum of the two fundamental-weight coordinates by at most one, and the path starts at the origin, we have $n+k\leq L$.

Conversely, suppose that $\lambda\in\Lambda_+(L)$, so that $\lambda=n\lambda_1+k\lambda_2$ with $n,k\geq0$ and $n+k\leq L$. We construct a path in $\Path_L^+(\lambda)$ as follows. If $L-n-k$ is even, take the first $n$ steps to be $u_{+,1}$, the next $k$ steps to be $u_{-,1}$, and the remaining steps to be $u_{+,1}$ and $u_{-,-1}$ in turn. If $L-n-k$ is odd, there are two cases. If $n+k>0$, replace one of the first $n+k$ steps $u_{\pm,1}$ by the two steps $u_{\mp,1},u_{\mp,0}$, reducing to the even case. If $n=k=0$, then $L\geq3$; take the first three steps to be
\[
    u_{+,1},\quad u_{+,0},\quad u_{+,-1},
\]
and take the remaining steps to be $u_{+,1}$ and $u_{-,-1}$ in turn. In each case the resulting path is dominant and has endpoint $\lambda$.
\end{proof}

For a sign string $S$, define its highest-weight path by
\begin{equation}\label{eqn: highest-weight path}
    P_S^{\mathrm{hw}}:=(S,(1,1,\dots,1)).
\end{equation}

\begin{lemma}\label{lem: weight domination}
    Every path $P=(S,J)$ satisfies $\wt(P)\preceq \wt(S)$, with equality if and only if $P=P_S^{\mathrm{hw}}$.
\end{lemma}
\begin{proof}
    By definition, $\wt(S)=\wt(P_S^{\mathrm{hw}})$. We then observe that any path $P=(S,J)$ is obtained from $P_S^{\mathrm{hw}}$ by replacing states in the state string. It is easy to check from the definition of steps in \cref{def: steps} that replacement of state $1$ by $0$ or $-1$ always decreases the weight of the path, and in particular does not produce weights incomparable with $\wt(S)$. Hence $\wt(P)\preceq \wt(S)$, and equality holds if and only if $P=P_S^{\mathrm{hw}}$.
\end{proof}

\begin{lemma}\label{lem: PQ-dual is an excursion}
    Let $P=(p_0,p_1,\dots,p_L)$, $Q=(q_0,q_1,\dots,q_M)$ be two paths. The following statements are equivalent:
    \begin{enumerate}
        \item $P$ and $Q$ are dominant and $p_L=q_M$.
        \item $PQ^\vee\in \Path_{L+M}^+(0)$.
    \end{enumerate}
\end{lemma}
\begin{proof}
    Let $Q=(S,J)$ and let $PQ^\vee=(c_0,c_1,\dots,c_{L+M})$. By definition of $PQ^\vee$, the first $L$ steps of are those of $P$, and the last $M$ steps are those of $Q^\vee$. Hence 
    \begin{align}\label{eqn: expression for c_k}
        c_k = \begin{cases}
            p_k & \text{if } k \leq L, \\
            p_L + \sum_{n=1}^m u_{(S^\vee)_{n},(J^\vee)_{n}} = p_L -\sum_{n=1}^m u_{s_{M-n+1},j_{M-n+1}} & \text{if } k=L+m, \qquad m=1,\dots,M.
        \end{cases}
    \end{align}
    where in the second case we used the definition of $(-)^\vee$ and \eqref{eqn: negation of steps}.

    By definition of $Q$, for each $m=0,\dots,M$ we have $q_{M-m}=\sum_{n=1}^{M-m} u_{s_n,j_n}$, from which we deduce
    \begin{align}\label{eqn: expression for q_M-m}
        q_{M-m}=q_M - \sum_{n=M-m+1}^M u_{s_n,j_n} = q_M - \sum_{n=1}^m u_{s_{M-n+1},j_{M-n+1}}.
    \end{align}

    \textbf{1 implies 2:} We now have $p_L=q_M$, hence by comparing \eqref{eqn: expression for c_k} and \eqref{eqn: expression for q_M-m}, we have $c_{L+m}=q_{M-m}$. The dominance of $Q$ implies $c_{L+m}=q_{M-m}\in \Lambda_+$ for $m=1,\dots,M$, and setting $m=M$ yields $c_{L+M}=q_0=0$. It remains to show that $c_k\in \Lambda_+$ for $0\leq k \leq L$, which is true by the dominance of $P$ and the first case of \eqref{eqn: expression for c_k}.
    
    \textbf{2 implies 1:} Now $C$ is dominant and $c_{L+M}=0$. Choosing $m=M$ in \eqref{eqn: expression for c_k} yields $0=c_{L+M}=p_L-q_M$, hence comparison of \eqref{eqn: expression for c_k} and \eqref{eqn: expression for q_M-m} gives again $c_{L+m}=q_{M-m}$ for $m=1,\dots,M$, from which the dominance of $Q$ follows. The dominance of $P$ follows from the first case in \eqref{eqn: expression for c_k}.
\end{proof}

\begin{corollary}\label{cor: involution preserves excursions}
    Let $Q\in \Path_L^+(0)$. Then $Q^\vee \in \Path_L^+(0)$.
\end{corollary}
\begin{proof}
    We may set $P$ to be the trivial path of length $0$ in \cref{lem: PQ-dual is an excursion}, so that $PQ^\vee=Q^\vee$. The result follows immediately.
\end{proof}

\subsection{The growth algorithm}\label{sec: growth algorithm}

Let $\lambda\in \Lambda_+$ be an arbitrary dominant weight. The growth algorithm $G$ maps each $P\in \Path^+(\lambda)$ to a web $G(P)$. It was introduced in \cite{khovanov_1999_web} in the case $\lambda=0$, and we generalize it to arbitrary $\lambda\in \Lambda_+$. The map $G$ is defined recursively, based on a set of defining growth rules:

\begin{definition}\label{def: defining growth rules}
    A \textup{defining growth rule} is one of the triples $(w,P^i,P^f)$ displayed in \cref{fig: defining growth rules}. Here $w\in \W^{S^i}_{S^f}$ is a web, and $P^i=(S^i,J^i)$ and $P^f=(S^f,J^f)$ are paths whose steps are written next to the top and bottom boundary vertices of the web $w$, respectively.
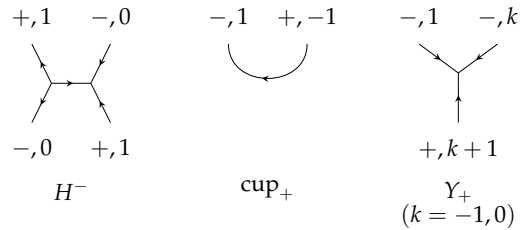
\begin{figure}[htbp]
    \centering
    \scalebox{0.8}{\begin{tikzpicture}


        \node [label={$+,1$}] (h1topl) at (-4.2, 1.3) {};
        \node [label=below:{$-,0$}] (h1botl) at (-4.2, 0) {};
        \node [label={$-,0$}] (h1topr) at (-2.9, 1.3) {};
        \node [label=below:{$+,1$}] (h1botr) at (-2.9, 0) {};
        \node (h1l) at (-3.875, 0.65) {};
        \node (h1r) at (-3.225, 0.65) {};
        \node [label={$H^{-}$}] (h1label) at (-3.55, -1.55) {};

        \draw [style=mid arrow] (h1l.center) to (h1topl.center);
        \draw [style=mid arrow] (h1l.center) to (h1botl.center);
        \draw [style=mid arrow] (h1l.center) to (h1r.center);
        \draw [style=mid arrow] (h1topr.center) to (h1r.center);
        \draw [style=mid arrow] (h1botr.center) to (h1r.center);

        \node [label={$-,1$}] (cuptopl) at (-0.95, 1.3) {};
        \node [label={$+,-1$}] (cuptopr) at (0.35, 1.3) {};
        \node [label={$\ccup_+$}] (cuplabel) at (-0.3, -1.55) {};

        \draw [style=mid arrow, bend left=90, looseness=1.50] (cuptopr.center) to (cuptopl.center);

        \node [label={$-,1$}] (y1topl) at (2.2, 1.3) {};
        \node [label={$-,k$}] (y1topr) at (3.5, 1.3) {};
        \node [label=below:{$+,k+1$}] (y1bot) at (2.85, 0) {};
        \node (y1mid) at (2.85, 0.82) {};
        \node at (2.85, -1.15) {$Y_+$};
        \node at (2.85, -1.55) {$(k=-1,0)$};

        \draw [style=mid arrow] (y1topl.center) to (y1mid.center);
        \draw [style=mid arrow] (y1topr.center) to (y1mid.center);
        \draw [style=mid arrow] (y1bot.center) to (y1mid.center);

\end{tikzpicture}}
    \caption{Four defining growth rules. The notations $H^-$, $Y_+$ and $\ccup_+$ for the involved web $w$ are given below the rules. For example, the defining growth rule on the left is $(H^-,(u_{+,1}u_{-,0}),(u_{-,0}u_{+,1}))$. Four similar rules are obtained by reversing all orientations and signs.}
    \label{fig: defining growth rules}
\end{figure}
\end{definition}

\begin{definition}[Growth algorithm]\label{def: growth algorithm}\label{def: extension of growth algorithm}
    Let $\lambda\in \Lambda_+$ and let $P\in \Path^+(\lambda)$. The web $G(P)$ is defined recursively as follows:
    \begin{enumerate}
        \item If $P=P_T^{\mathrm{hw}}=(T,(1,1,\dots,1))$ where $\wt(T)=\lambda$, then set $G(P):=\id_T$.
        \item Otherwise, write $P=P_1P^iP_2$, where $P^i$ belongs to a defining growth rule $(w,P^i,P^f)$ from \cref{def: defining growth rules},\footnote{In the defining growth rules, distinct rules have distinct $P^i$ Hence for each $P^i$, the rule $(w,P^i,P^f)$ is defined unambiguously.} and $P_1,P_2$ are the remaining steps of the path. Then define
        \begin{align}\label{eqn: recursion}
            G(P) := (\id_{S_1} \otimes w \otimes \id_{S_2}) \circ G(P_1P^fP_2),
        \end{align}
        where $\circ$ is composition in $\Sp$, and $S_1$ and $S_2$ are the sign strings of $P_1$ and $P_2$, respectively.
    \end{enumerate}
\end{definition}

\begin{remark}
    In the case $\lambda=0$ studied in \cite{khovanov_1999_web}, the sign string $T$ with $\wt(T)=0$ is the empty string, and the path $P_T^{\mathrm{hw}}$ is the trivial path $\emptyset$ of length $0$.

    The growth rules in \cref{fig: defining growth rules} form a proper subset of the growth rules in \cite{khovanov_1999_web}. We show in \cref{prp: growth algorithm is well-defined} that this subset suffices to define the maps $G$, and in \cref{rmk: minimality of growth rules} we show that removal of any one of them would make $G$ ill-defined. Therefore they form a minimal set of growth rules.
\end{remark}

\begin{example}\label{exm: example of growth algorithm}
    The web $G(P)$ for the path $P=(S,J)$ with $S=(-,+,-,+,+,-)$ and $J=(1,1,0,0,-1,-1)$ is constructed in \cref{fig: example of growth algorithm}. We first write $P=P_1P^iP_2$, where $P_1=(u_{-,1})$, $(P^i=u_{+,1}u_{-,0})$ and $P_2=(u_{+,0}u_{+,-1}u_{-,-1})$. By the growth rule $(H^-,P^i,P^f)$ where $P^f=(u_{-,0}u_{+,1})$, we define $G(P)=(\id_{-} \otimes H^- \otimes \id_{++-}) \circ G(P_1P^fP_2)$, thus yielding the first recursion. The remaining recursive steps are carried out similarly, and use growth rules whose relevant webs are $Y_-$, $\ccup_-$, $Y_+$ and $\ccup_+$, respectively.
\end{example}

\begin{figure}[htbp]
        \centering
        \scalebox{0.7}{\begin{tikzpicture}
	\begin{pgfonlayer}{nodelayer}
		\node (p10) at (-5, 2) {};
		\node (p11) at (-3, 2) {};
		\node (p12) at (-1, 2) {};
		\node (p13) at (1, 2) {};
		\node (p14) at (3, 2) {};
		\node (p15) at (5, 2) {};
		\node at (-4.5, 2.25) {$-,1$};
		\node at (-2.5, 2.25) {$+,1$};
		\node at (-0.5, 2.25) {$-,0$};
		\node at (1.5, 2.25) {$+,0$};
		\node at (3.5, 2.25) {$+,-1$};
		\node at (5.5, 2.25) {$-,-1$};

		\node (p20) at (-5, 0) {};
		\node (p21) at (-3, 0) {};
		\node (p22) at (-1, 0) {};
		\node (p23) at (1, 0) {};
		\node (p24) at (3, 0) {};
		\node (p25) at (5, 0) {};
		\node at (-4.5, 0.25) {$-,1$};
		\node at (-2.5, 0.25) {$-,0$};
		\node at (-0.5, 0.25) {$+,1$};
		\node at (1.5, 0.25) {$+,0$};
		\node at (3.5, 0.25) {$+,-1$};
		\node at (5.5, 0.25) {$-,-1$};

		\node (p30) at (-5, -2) {};
		\node (p31) at (-3, -2) {};
		\node (p32) at (0, -2) {};
		\node (p33) at (3, -2) {};
		\node (p34) at (5, -2) {};
		\node at (-4.5, -1.75) {$-,1$};
		\node at (-2.5, -1.75) {$-,0$};
		\node at (0.5, -1.75) {$-,1$};
		\node at (3.5, -1.75) {$+,-1$};
		\node at (5.5, -1.75) {$-,-1$};

		\node (p40) at (-5, -4) {};
		\node (p41) at (-3, -4) {};
		\node (p42) at (5, -4) {};
		\node at (-4.5, -3.75) {$-,1$};
		\node at (-2.5, -3.75) {$-,0$};
		\node at (5.5, -3.75) {$-,-1$};

		\node (p50) at (-4, -6) {};
		\node (p51) at (5, -6) {};
		\node at (-3.5, -5.75) {$+,1$};
		\node at (5.5, -5.75) {$-,-1$};

		\node (h1l) at (-2.5, 1) {};
		\node (h1r) at (-1.5, 1) {};
		\node (y2) at (0, -0.75) {};
		\node (y4) at (-4, -5) {};

		\node (d10) at (-7, 0) {};
		\node (d11) at (7, 0) {};
		\node (d20) at (-7, -2) {};
		\node (d21) at (7, -2) {};
		\node (d30) at (-7, -4) {};
		\node (d31) at (7, -4) {};
		\node (d40) at (-7, -6) {};
		\node (d41) at (7, -6) {};
		\node at (-9, 1) {Recursion 1};
		\node at (-9, -1) {Recursion 2};
		\node at (-9, -3) {Recursion 3};
		\node at (-9, -5) {Recursion 4};
		\node at (-9, -7) {Recursion 5};
	\end{pgfonlayer}
	\begin{pgfonlayer}{edgelayer}
		\draw (p10.center) to (p20.center);
		\draw (p11.center) to (h1l.center);
		\draw (p21.center) to (h1l.center);
		\draw (h1l.center) to (h1r.center);
		\draw (p12.center) to (h1r.center);
		\draw (p22.center) to (h1r.center);
		\draw (p13.center) to (p23.center);
		\draw (p14.center) to (p24.center);
		\draw (p15.center) to (p25.center);

		\draw (p20.center) to (p30.center);
		\draw (p21.center) to (p31.center);
		\draw (p22.center) to (y2.center);
		\draw (p23.center) to (y2.center);
		\draw (p32.center) to (y2.center);
		\draw (p24.center) to (p33.center);
		\draw (p25.center) to (p34.center);

		\draw (p30.center) to (p40.center);
		\draw (p31.center) to (p41.center);
		\draw (p34.center) to (p42.center);
		\draw [looseness=1.25, bend left=90] (p33.center) to (p32.center);

		\draw (p40.center) to (y4.center);
		\draw (p41.center) to (y4.center);
		\draw (p50.center) to (y4.center);
		\draw (p42.center) to (p51.center);

		\draw [looseness=0.75, bend left=90] (p51.center) to (p50.center);

		\draw [dashed] (d10.center) to (d11.center);
		\draw [dashed] (d20.center) to (d21.center);
		\draw [dashed] (d30.center) to (d31.center);
		\draw [dashed] (d40.center) to (d41.center);
	\end{pgfonlayer}
\end{tikzpicture}}
        \caption{An example of the algorithm, with orientations encoded by the signs.}
        \label{fig: example of growth algorithm}
\end{figure}
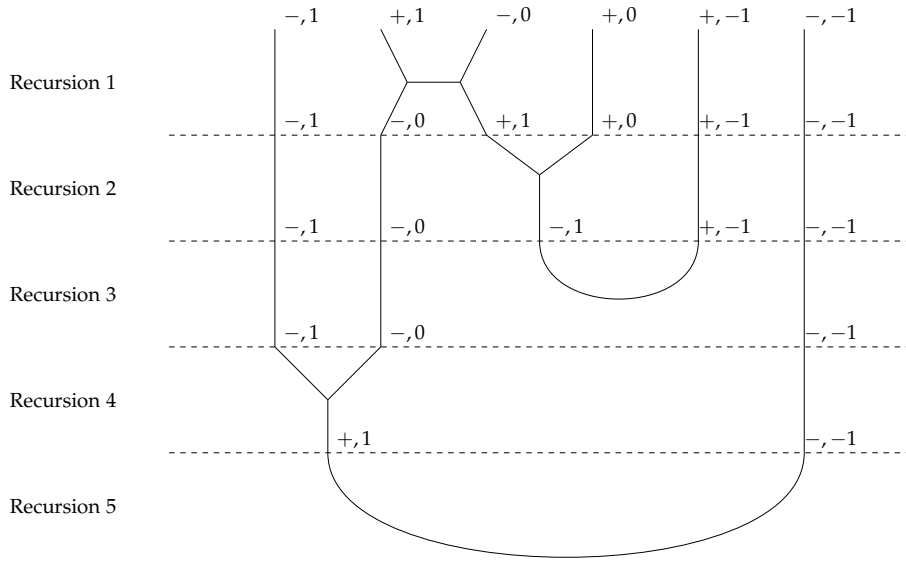


The growth algorithm is well-defined:

\begin{proposition}\label{prp: growth algorithm is well-defined}
    Let $\lambda\in \Lambda_+$ and $P\in \Path^+(\lambda)$. The following properties ensure that the map $P\mapsto G(P)$ is well-defined:
    \begin{enumerate}
        \item The algorithm terminates after a finite number of recursions. \label{cond: termination}
        \item The result $G(P)$ does not depend on the choice of growth rules used in each recursion.\label{cond: order-independence}
    \end{enumerate}
\end{proposition}
\begin{proof}
    Order the paths in $\Path^+(\lambda)$ by \cref{def: order on paths}. The proof of \cref{cond: termination} is based on the following observations: 
    \begin{enumerate}
        \item For each fixed $P\in \Path^+(\lambda)$, there are only finitely many lower-order paths in $\Path^+(\lambda)$.
        \item A path $P\in \Path^+(\lambda)$ is of minimal order if and only if $P=P_T^{\mathrm{hw}}$ where $\wt(T)=\lambda$.
        \item A growth rule is applicable to $P\in\Path^+(\lambda)$ unless $P=P_T^{\mathrm{hw}}$ for some $\wt(T)=\lambda$.
        \item The application of a growth rule strictly lowers the order of the path.
        \item If $P=P_1P^iP_2\in \Path^+(\lambda)$ and $(w,P^i,P^f)$ is a growth rule, then $P_1P^fP_2 \in \Path^+(\lambda)$.
    \end{enumerate}
    The first observation follows since only finitely many paths in $\Path^+(\lambda)$ have length less than that of $P$, and for each fixed length there are only finitely many paths. The second observation follows from \cref{lem: weight domination}. The third observation uses the fact that, for a dominant path $P=(S,J)$, one has $j_1=1$, and if $J$ contains a state $0$ or $-1$, then there is a minimal $m$ such that $j_m=1,j_{m+1}\in \{0,-1\}$, and a growth rule from \cref{fig: defining growth rules} is applicable at the $m$-th and $(m+1)$-th steps of $P$. The fourth observation simply follows since all growth rules either reduce the length of the path, or reduce the lexicographical order of the state string. The fifth observation is technical and is therefore deferred to \cref{lem: dominance} in appendix~\ref{sec: well-definedness of the growth algorithm}.
    
    Now we prove \cref{cond: termination}: given a path $P\in \Path^+(\lambda)$, growth strictly lowers the order of the path at each step, and does not terminate before reaching a path of the form $P_T^{\mathrm{hw}}$ for some $\wt(T)=\lambda$, for which $G(P)$ is defined, and $P_T^{\mathrm{hw}}$ is indeed reached after a finite number of recursions.
    
    To show \cref{cond: order-independence}, suppose that two distinct growth rules are applicable for a path $P$. We claim that the steps of $P$ on which the two rules are applicable must be disjoint. This implies \cref{cond: order-independence} since the two growth rules can be applied in either order. 
    
    To prove the claim, we first observe that the steps admitting the two distinct rules cannot overlap completely: Two consecutive steps admit at most one growth rule. Moreover, the steps admitting the two growth rules cannot overlap by one step: for each growth rule in \cref{fig: defining growth rules}, the path $P^i=u_{s_1,1}u_{s_2,j}$, where $j<1$ and $s_1,s_2\in\{+,-\}$. Now suppose that $P=P_1u_{s_1,j_1}u_{s_2,j_2}u_{s_3,j_3}P_2$, where a growth rule is applicable at the first two steps $u_{s_1,j_1}u_{s_2,j_2}$, and another growth rule is applicable at the last two steps $u_{s_2,j_2}u_{s_3,j_3}$. This is however impossible, since the first growth rule implies $j_2<1$, while the second growth rule implies $j_2=1$, a contradiction. Hence the steps admitting the two growth rules are disjoint, and the claim is proved.
\end{proof}

\begin{remark}\label{rmk: minimality of growth rules}
    Removing any of the defining growth rules in \cref{fig: defining growth rules} makes the map $G$ undefined for some dominant paths: for example, the paths
    \begin{align}
        P_1=u_{+,1}u_{-,1}u_{+,0}, \quad P_2=u_{+,1}u_{-,-1}, \quad P_3=u_{+,1}u_{+,0}, \quad P_4=u_{-,1}u_{+,1}u_{+,-1}
    \end{align}
    each admit one of the rules in \cref{fig: defining growth rules} as the only applicable rule. As a consequence, the removal of any one of these rules would make $G(P_i)$ undefined for some path $P_i,i=1,2,3,4$. 
\end{remark}

\begin{lemma}\label{lem: growth preserves weight}
    Let $\lambda\in \Lambda_+$ and $P\in \Path^+_L(\lambda)$. Let $S$ be the bottom sign string of $G(P)$. Then $\wt(S)=\lambda$.
\end{lemma}
\begin{proof}
    Note that for each defining growth rule from \cref{def: defining growth rules} $(w,P^i,P_f)$, $\wt(P^i)=\wt(P^f)$. Hence in each recursion \eqref{eqn: recursion} the weight is preserved. The algorithm constructing $G(P)$ terminates at the path $P^{\mathrm{hw}_S}$, which has the same weight as $S$ by \cref{lem: weight domination}. Hence $\lambda=\wt(P)=\wt(P)=\wt(P^{\mathrm{hw}_S})=\wt(S)$.
\end{proof}

\begin{definition}\label{def: growth rule}
    A \textup{growth rule} is a triple $(w,P^i,P^f)$, where $w$ is a web, and $P^i$, $P^f$ are paths, such that \eqref{eqn: recursion} holds for every $\lambda\in \Lambda_+$ and $P\in \Path^+(\lambda)$ of the form $P=P_1P^iP_2$.
\end{definition}

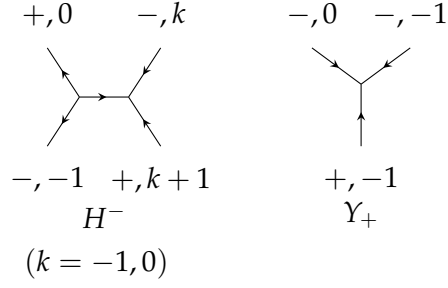
\begin{figure}
    \centering
    \scalebox{1.0}{\begin{tikzpicture}


        \node [label={$+,0$}] (h1topl) at (-2.3, 1.3) {};
        \node [label=below:{$-,-1$}] (h1botl) at (-2.3, 0) {};
        \node [label={$-,k$}] (h1topr) at (-0.8, 1.3) {};
        \node [label=below:{$+,k+1$}] (h1botr) at (-0.8, 0) {};
        \node (h1l) at (-1.875, 0.65) {};
        \node (h1r) at (-1.225, 0.65) {};
        \node [label={$H^{-}$}] (h1label) at (-1.55, -1.35) {};
        \node at (-1.65, -1.55) {$(k=-1,0)$};

        \draw [style=mid arrow] (h1l.center) to (h1topl.center);
        \draw [style=mid arrow] (h1l.center) to (h1botl.center);
        \draw [style=mid arrow] (h1l.center) to (h1r.center);
        \draw [style=mid arrow] (h1topr.center) to (h1r.center);
        \draw [style=mid arrow] (h1botr.center) to (h1r.center);

        \node [label={$-,0$}] (y1topl) at (1.2, 1.3) {};
        \node [label={$-,-1$}] (y1topr) at (2.5, 1.3) {};
        \node [label=below:{$+,-1$}] (y1bot) at (1.85, 0) {};
        \node (y1mid) at (1.85, 0.82) {};
        \node [label={$Y_+$}] (y1label) at (1.85, -1.35) {};

        \draw [style=mid arrow] (y1topl.center) to (y1mid.center);
        \draw [style=mid arrow] (y1topr.center) to (y1mid.center);
        \draw [style=mid arrow] (y1bot.center) to (y1mid.center);

\end{tikzpicture}}
    \caption{Three growth rules, with three more obtained by arrow and sign reversals.}
    \label{fig: growth rules}
\end{figure}

The defining growth rules in \cref{fig: defining growth rules} are growth rules by definition, and we have the following additional growth rules:

\begin{proposition}\label{prp: extra relations}
    The triples in \cref{fig: growth rules} are growth rules.
\end{proposition}
\begin{proof}
    The proposition is proved in the $\lambda=0$ case in \cite{khovanov_1999_web}, and we generalize the proof to arbitrary $\lambda\in \Lambda_+$ with a similar argument in appendix~\ref{sec: well-definedness of the growth algorithm}. The main idea is to show that the result of the growth algorithm does not depend on the choice of growth rules.
\end{proof}

We conclude this subsection with several properties of the webs $G(P)$ obtained from the growth algorithm.

\begin{theorem}[\cite{khovanov_1999_web}]\label{thm: growth algorithm}
Let $S$ be a sign string. The growth algorithm $G$ defines a bijection
\begin{align}
    G: \Path_S^+(0) \to \Web^S_{\emptyset}.
\end{align}
In particular, $G(P)$ is a monomial web for every $P\in \Path_S^+(0)$.
\end{theorem}

\begin{remark}
    The finiteness of the set $\Web^S_{\emptyset}$ is not immediately clear from its definition, but indeed follows from \cref{thm: growth algorithm} since $\Path_S^+(0)$ is a finite set.
\end{remark}

Seen as invariants through the equivalence $\Gamma:\Sp \to \Rep$, the webs obtained from the growth algorithm have remarkable properties. Let $P=(S,J)$ be a path of length $L$. Let $\{v^\pm_{j} \st j=-1,0,1\}$ be the basis vectors of $V^\pm$ defined in appendix~\ref{app: fundamental reps of uqslt}. We define
\begin{align}\label{def(eqn): vP}
    v_P := v_{j_1}^{s_1} \otimes v_{j_2}^{s_2} \otimes \cdots \otimes v_{j_L}^{s_L} \in V^{S}.
\end{align}

\begin{theorem}(\cite[Thm.\,2]{khovanov_1999_web})\label{thm: unitriangularity}
    For $P\in\Path_S^+(0)$, the invariant $\Gamma(G(P))$ has an expansion in the tensor basis of the weight-zero subspace of $V^S$ of the form
    \begin{align}\label{condition: unitriangular}
        \Gamma(G(P))
        =v_P+\sum_{\substack{Q\in\Path_S(0)\\Q<P}}\alpha_{P,Q}v_Q,
    \end{align}
    where $\alpha_{P,Q}\in \N[q,q^{-1}]$ are scalars. The order $<$ is defined in \cref{def: order on paths}.
\end{theorem}
\begin{proof}
    The notation differs considerably between \cite{khovanov_1999_web} and this work. In order to use the theorem, we identify $\uqslt$, $v_P$, $G(P)$ and $\Gamma$ in this paper with the corresponding objects in \cite{khovanov_1999_web}. First, in \cite{khovanov_1999_web} $\uqslt$ is a $\C(q_{KK}^{1/2})$-algebra. After identifying $q_{KK}^{1/2}=q$ in this work, the algebras are the same. Second, for a path $P=(S,J)$, the tensor $e^S_J\in V^S$ in \cite{khovanov_1999_web} is identified with $v_P$ in this work. Third, the web $G(P)$ is the same as the web $w^S_J$ in \cite{khovanov_1999_web}. Finally, comparing the images under $\Gamma:\Sp\to \Rep$ of the web generators in \cref{fig: elementary webs}, we see that the functors $\Gamma$ are the same in \cite{khovanov_1999_web} and in this work. So their theorem 2 applies.
\end{proof}

\begin{lemma}\label{lem: web in web}
    Let $P,Q$ be paths such that the concatenation $PQ^\vee\in \Path^+(\lambda)$ is a dominant path. Let $R$ be the bottom sign string of $G(P)$, $P^{\mathrm{hw}}_R$ be as in \eqref{eqn: highest-weight path}, and let $T$ be the sign string of $Q$. Then
    \begin{align}\label{eqn: relation}
        G(PQ^\vee) = (G(P) \otimes \id_{T^\vee}) \circ G(P_R^{\mathrm{hw}}Q^\vee).
    \end{align}
\end{lemma}
\begin{proof}
    Since $PQ^\vee$ is dominant, its initial segment $P$ is dominant, so $G(P)$ is defined. By \cref{lem: growth preserves weight}, $\wt(P_R^{\mathrm{hw}})=\wt(P)$. Since $P_R^{\mathrm{hw}}$ has only $\lambda_1$- and $\lambda_2$-steps, while the subsequent $Q^\vee$ segment visits the same vertices as in the dominant path $PQ^\vee$, the whole path $P_R^{\mathrm{hw}}Q^\vee$ is dominant, and the right-hand side of \eqref{eqn: relation} is well-defined.

    Let $G(P) = (\id_{S_1} \otimes w \otimes \id_{S_2}) \circ G(P_1P^fP_2)$ be the first recursion in the definition of $G(P)$, with $P=P_1P^iP_2$ and $(w,P^i,P^f)$ being a growth rule. We choose accordingly for $G(PQ^\vee)$ the same growth rule at the same steps of $P$, so that
    \begin{align*}
        G(PQ^\vee) & = (\id_{S_1} \otimes w \otimes \id_{S_2} \otimes \id_{T^\vee}) \circ G(P_1P^fP_2Q^\vee).
    \end{align*}
    Repeating this process until the $P$ segment has become the highest path $P_R^{\mathrm{hw}}$ gives \eqref{eqn: relation}.
\end{proof}

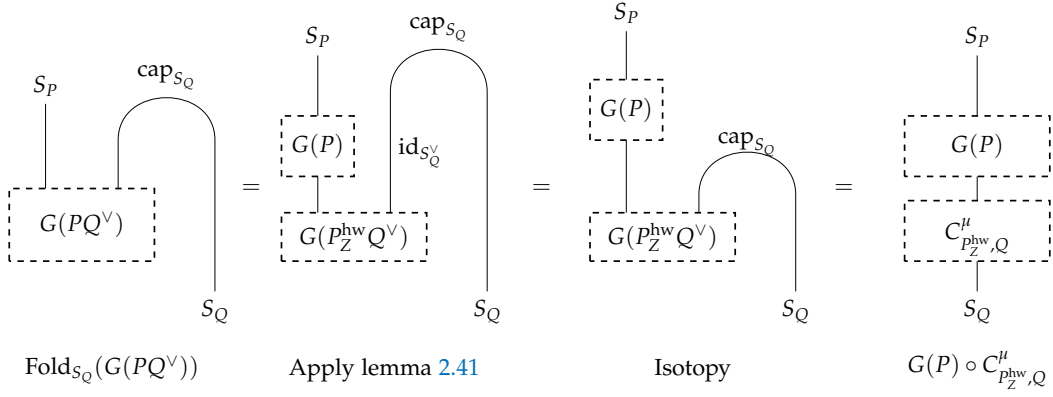
\begin{figure}[htbp]
    \centering
    \begingroup
    \scalebox{0.8}{\begin{tikzpicture}
        \begin{scope}[shift={(0,0)}]
            \draw[thick, dashed] (-1.2, 0) rectangle (1.2, 1.2);
            \node at (0, 0.6) {$G(PQ^\vee)$};

            \draw (-0.6, 1.2) -- (-0.6, 2.6);
            \node[above] at (-0.6, 2.6) {$S_P$};

            \draw (0.6, 1.2) -- (0.6, 2.0) to[out=90, in=90, looseness=1.5] (2.2, 2.0) -- (2.2, -0.5);
            \node[above] at (1.4, 2.7) {$\ccap_{S_Q}$};
            \node[below] at (2.2, -0.5) {$S_Q$};

            \node at (0.5, -1.8) {$\Fold_{S_Q}(G(PQ^\vee))$};
        \end{scope}

        \node at (2.8, 1.2) {$=$};

        \begin{scope}[shift={(4.5,0)}]
            \draw[thick, dashed] (-1.2, 0) rectangle (1.2, 0.8);
            \node at (0, 0.4) {$G(P_Z^{\mathrm{hw}} Q^\vee)$};

            \draw[thick, dashed] (-1.2, 1.4) rectangle (0, 2.4);
            \node at (-0.6, 1.9) {$G(P)$};

            \draw (-0.6, 0.8) -- (-0.6, 1.4);

            \draw (-0.6, 2.4) -- (-0.6, 3.4);
            \node[above] at (-0.6, 3.4) {$S_P$};

            \draw (0.6, 0.8) -- (0.6, 2.8) to[out=90, in=90, looseness=1.5] (2.2, 2.8) -- (2.2, -0.5);
            \node[right] at (0.6, 1.8) {$\id_{S_Q^\vee}$};
            \node[above] at (1.4, 3.5) {$\ccap_{S_Q}$};
            \node[below] at (2.2, -0.5) {$S_Q$};

            \node at (0.5, -1.8) {Apply \cref{lem: web in web}};
        \end{scope}

        \node at (7.6, 1.2) {$=$};

        \begin{scope}[shift={(9.6,0)}]
            \draw[thick, dashed] (-1.2, 0) rectangle (1.2, 0.8);
            \node at (0, 0.4) {$G(P_Z^{\mathrm{hw}} Q^\vee)$};

            \draw[thick, dashed] (-1.2, 2.0) rectangle (0, 3.0);
            \node at (-0.6, 2.5) {$G(P)$};

            \draw (-0.6, 0.8) -- (-0.6, 2.0);

            \draw (-0.6, 3.0) -- (-0.6, 3.8);
            \node[above] at (-0.6, 3.8) {$S_P$};

            \draw (0.6, 0.8) -- (0.6, 1.1) to[out=90, in=90, looseness=1.5] (2.2, 1.1) -- (2.2, -0.5);
            \node[above] at (1.4, 1.6) {$\ccap_{S_Q}$};
            \node[below] at (2.2, -0.5) {$S_Q$};

            \node at (0.5, -1.8) {Isotopy};
        \end{scope}

        \node at (12.6, 1.2) {$=$};

        \begin{scope}[shift={(14.8,0)}]
            \draw[thick, dashed] (-1.2, 0) rectangle (1.2, 1.0);
            \node at (0, 0.4) {$C^\mu_{P_Z^{\mathrm{hw}}, Q}$};

            \draw (0, -0.5) -- (0, 0);
            \node[below] at (0, -0.5) {$S_Q$};

            \draw[thick, dashed] (-1.2, 1.4) rectangle (1.2, 2.4);
            \node at (0, 1.9) {$G(P)$};

            \draw (0, 1.0) -- (0, 1.4);

            \draw (0, 2.4) -- (0, 3.4);
            \node[above] at (0, 3.4) {$S_P$};

            \node at (0, -1.8) {$G(P) \circ C^\mu_{P_Z^{\mathrm{hw}},Q}$};
        \end{scope}

    \end{tikzpicture}}
    \endgroup
    \caption{Illustration of \eqref{eqn: relation}.}
    \label{fig: relation between old and new grown web}
\end{figure}

\begin{corollary}\label{cor: extended growth algorithm monomial}
    Let $\lambda\in \Lambda_+$. For every $P\in \Path^+(\lambda)$, the web $G(P)$ is monomial.
\end{corollary}
\begin{proof}
    Choose any $Q\in \Path^+(\lambda)$. By \eqref{eqn: relation}, any reduction of $G(P)$ by Kuperberg relations would induce a reduction of $G(PQ^\vee)$, which contradicts \cref{thm: growth algorithm} stating that $G(PQ^\vee)$ is monomial.
\end{proof}

\section{The algebra $\mathcal{K}_L$}\label{sec: K_L}

In this section, we define the algebra $\mathcal{K}_L$ and state the main results of this paper. In \cref{sec: def of K_L} we define $\mathcal{K}_L$ and prove a Schur-Weyl duality between $\mathcal{K}_L$ and $\uqslt$. In \cref{sec: cellular structure guide}, we describe the cellular structure constructed in \cref{sec: cellular structure} in practical terms and give examples and a bimodule decomposition result that enhances the Schur-Weyl duality. In \cref{sec: generators guide} we identify the generators of $\mathcal{K}_L$ in \cref{thm: generators of K_L} and discuss its consequences. In \cref{sec: dimension formulae} we give formulae for the dimension of $\mathcal{K}_L$ and its cell modules using the reflection principle.

\subsection{Definition and Schur-Weyl duality}\label{sec: def of K_L}
Let $ L \in \N$ and $\Sigma_L$ be the set of sign strings of length $L$. We define the basis web set
\begin{align}\label{def: Omega}
    \WebBasis{L}:=\bigsqcup_{S,T\in \Sigma_L} \Web^{T}_{S},
\end{align}
where $\Web^T_S$ is the set of monomial webs with boundary orientation $T$ and $S$, defined in \eqref{eqn: web set and web space} and \cref{def: reduced web}.

\begin{definition}\label{def: K_L}
    The algebra $\mathcal K_L$ is a $\C(q)$-vector space
    \begin{align}\label{def: KLq}
        \mathcal{K}_L
        :=\Span_{\C(q)}\WebBasis{L}.
    \end{align}
    Its multiplication is defined by concatenation: for basis webs
    $w_1\in \Web^{S_1}_{S_2}$ and $w_2\in \Web^{S_3}_{S_4}$, set
    $w_1w_2=0$ if $S_2\neq S_3$; otherwise stack $w_1$ on top of $w_2$
    and apply the Kuperberg relations in \cref{fig: Kuperberg relations}. Extend this multiplication
    $\C(q)$-bilinearly. Its unit is
    \begin{align*}
        1_L:=\sum_{S\in\Sigma_L}\id_S,
    \end{align*}
    where $\id_S$ is the web in $\Web_S^S$ that consists of $|S|$ parallel vertical strands, oriented according to $S$.
\end{definition}

\begin{remark}\label{rmk: identity in KL(q)}
 The unit $1_L$ is not a basis web unless $L=0$.
\end{remark}

Consider the vector space
\begin{align}\label{eqn: def of V_L}
    \mathscr{V}_L:=(V^+\oplus V^-)^{\otimes L}\cong\bigoplus_{S\in\Sigma_L}V^S, 
\end{align}
where each summand $V^S$ defined in \eqref{eqn: V^S} admits the $\uqslt$-action $\rho_S$. Define an action $\rho_L:\uqslt\to \End_{\C(q)}(\mathscr V_L)$ by setting $\rho_L|_{V^S}=\rho_S$, which makes $\mathscr V_L$ a $\uqslt$-module. Let
\begin{align}\label{eqn: inclusion-projection}
    p_S: \mathscr{V}_L  \to V^S, \qquad e_S: V^S  \to \mathscr{V}_L
\end{align}
be the projection and inclusion maps of the $V^S$-summand, respectively.
These maps are $\uqslt$-equivariant and satisfy
\begin{align}\label{eqn: orthogonality}
    p_{T}\circ e_S = \delta_{S,T} \id_{V^S}, \qquad \forall S,T\in \Sigma_L.
\end{align}
Recall the equivalence $\Gamma$ from \cref{thm: equivalence of Sp and Rep}, and define the $\C(q)$-linear map
\begin{align}\label{eqn: def of Phi}
    \Phi: \mathcal{K}_L & \to \End_{\uqslt}(\mathscr{V}_L), \\
    \W^{S_2}_{S_1} \ni w & \mapsto e_{S_2}\circ \Gamma(w) \circ p_{S_1}.
\end{align}

\begin{theorem}\label{thm: Schur-Weyl duality}
    The map $\Phi$ is a faithful action of
    $\mathcal{K}_L$ on $\mathscr{V}_L$, and the images of $\uqslt$ and
    $\mathcal{K}_L$ are mutual centralizers:
    \begin{align}\label{eqn: double centralizer for KL}
        \Phi(\mathcal{K}_L)
        &=\End_{\rho_L(\uqslt)}(\mathscr{V}_L),
        &
        \rho_L(\uqslt)
        &=\End_{\Phi(\mathcal{K}_L)}(\mathscr{V}_L).
    \end{align}
    In particular, $\mathcal{K}_L$ is a semisimple $\C(q)$-algebra, and there is an isomorphism of $\C(q)$-algebras
    \begin{align*}
        \mathcal{K}_L\cong\End_{\uqslt}(\mathscr{V}_L).
    \end{align*}
\end{theorem}
\begin{proof}
By \eqref{eqn: def of V_L}, there is an $\C(q)$-linear isomorphism 
\begin{align*}
    F:\bigoplus_{S,T\in\Sigma_L}\Hom_{\uqslt}(V^S,V^T) &\to \End_{\uqslt}(\mathscr{V}_L),\\
     \Hom_{\uqslt}(V^S,V^T)\ni f &\mapsto e_T\circ f\circ p_S.
\end{align*}
The equivalence $\Gamma:\W_S^T\to \Hom_{\uqslt}(V^S,V^T)$ is a $\C(q)$-linear isomorphism, hence $\Phi=F\circ \Gamma$ is a $\C(q)$-linear isomorphism. It remains to check that $\Phi$ is an algebra
    morphism, which amounts to checking \eqref{eqn: Phi preserves multiplication}:
    \begin{align}\label{eqn: Phi preserves multiplication}
        \Phi(w_1w_2) & = \Phi(w_1) \circ \Phi(w_2), \qquad \forall w_1,w_2\in \WebBasis{L}.
    \end{align}
    If the target of $w_2$ is not the source of $w_1$, then both sides are zero: On the right-hand side it follows from the orthogonality relation \eqref{eqn: orthogonality} that $\Phi(w_1)\circ \Phi(w_2) = 0$; on the left-hand side, $w_1w_2=0$ by definition of multiplication in $\mathcal{K}_L$, hence $\Phi(w_1w_2)=0$ by linearity of $\Phi$. If the target of $w_2$ is the source of $w_1$, then
    \begin{align*}
        \Phi(w_1w_2) =& e_{T}\circ \Gamma(w_1w_2) \circ p_S \\
        =& e_{T}\circ \Gamma(w_1)\circ \Gamma(w_2) \circ p_S \\
        =& \Phi(w_1)\circ \Phi(w_2)
    \end{align*}
    where the second equality follows from the functoriality of $\Gamma$
    and the third equality uses the orthogonality relation
    \eqref{eqn: orthogonality}. Therefore, $\Phi$ is an isomorphism of
    $\C(q)$-algebras, and its image is
    \begin{align*}
        \Phi(\mathcal{K}_L)=\End_{\rho_L(\uqslt)}(\mathscr{V}_L),
    \end{align*}
    which proves the first equality in
    \eqref{eqn: double centralizer for KL}.

    Set $A_L:=\rho_L(\uqslt)\subseteq\End_{\C(q)}(\mathscr{V}_L)$. Since $q$ is an indeterminate, each $V^S$ is semisimple as a $\uqslt$-module over $\C(q)$~\cite[Thm.\,10.1.7]{chari_pressley_1994}. Hence their direct sum $\mathscr{V}_L$ is semisimple as an $A_L$-module, and $A_L$ acts faithfully by definition. The finite-dimensional double centralizer theorem, valid over the field $\C(q)$, now gives
    \begin{align*}
        \rho_L(\uqslt)
        &=\End_{\End_{A_L}(\mathscr{V}_L)}(\mathscr{V}_L)
        =\End_{\Phi(\mathcal{K}_L)}(\mathscr{V}_L),
    \end{align*}
    proving the second equality in \eqref{eqn: double centralizer for KL}, and the semisimplicity of $\mathcal K_L$.
\end{proof}

\subsection{Cellular structure}\label{sec: cellular structure guide}

We first recall the definition of a cellular algebra following \cite{graham1996cellular}.
\begin{definition}\label{def: cellular algebra}
    A \emph{cellular algebra} over $\C(q)$ is a unital associative algebra $A$, with a \emph{cell datum} $(\Lambda, M, C, \dagger)$, where
\begin{enumerate}
    \item $\Lambda$ is a finite partially ordered set (poset). \label{cond: poset}
    \item For each $\lambda\in \Lambda$, $M(\lambda)$ is a finite set. \label{cond: finiteness}
    \item $C$ is an injective map $\coprod_{\lambda\in \Lambda} M(\lambda)\times M(\lambda)\to A$, with $(s,t)\mapsto C^\lambda_{s,t}$. The image of $C$ is a basis of $A$ called the \emph{cellular basis}. \label{cond: C}
    \item $\dagger$ is an algebra anti-automorphism of $A$ such that $(C_{s,t}^\lambda)^\dagger=C_{t,s}^\lambda$ for all $s,t\in M(\lambda)$ and $\lambda\in \Lambda$. \label{cond: dagger}
\end{enumerate}
Moreover, the following condition is satisfied: For each $\lambda\in \Lambda$, let $A(\prec \lambda)$ be the $\C(q)$-subspace of $A$ generated by $\{C_{s,t}^\mu\st s,t\in M(\mu), \mu\prec \lambda\}$. Then for any $\lambda\in \Lambda$, $s\in M(\lambda)$ and $a\in A$, there exist $r_a(s',s)\in \C(q)$ such that for all $t\in M(\lambda)$,
\begin{align}\label{eqn: factorized multiplication}
    aC_{s,t}^\lambda \equiv \sum_{s'\in M(\lambda)} r_a(s',s)C_{s',t}^\lambda \mod A(\prec \lambda).
\end{align}
\end{definition}
The cellular structure of $\mathcal{K}_L$ for $L\geq 2$ is stated in \cref{thm: cellular structure}, whose construction involves various ingredients from \cref{sec: cellular structure}. In this subsection, we provide a practical description of the cellular structure of $\mathcal{K}_L$ and its consequences.

\vspace{2ex}

We recall the set $\Lambda_+(L)$ of $\mathfrak{sl}_3$-dominant weights from \cref{def: LambdaL}, equipped with the partial ordering $\prec$ generated by \eqref{eqn: partial ordering}. The poset $(\Lambda_+(L),\prec)$ is $\Lambda$ in \cref{def: cellular algebra}. We furthermore recall, for each $\lambda\in \Lambda_+(L)$, the set $\Path^+_L(\lambda)$ of paths in the $\mathfrak{sl}_3$-dominant Weyl chamber from \cref{def: lattice path}, which will be $M(\lambda)$ in \cref{def: cellular algebra}.

\paragraph{The cellular basis $\WebBasis{L}$.} Recall the monomial web basis $\WebBasis{L}$ of $\mathcal K_L$ in \eqref{def: Omega}. We index $\WebBasis{L}$ by pairs of dominant $\mathfrak{sl}_3$-lattice paths on $\Lambda_+(L)$. Specifically, we have a map $C$ from \cref{def: pairing map C}:
\begin{equation}\label{eqn: bijection C}
    C: \bigsqcup_{\lambda\in \Lambda_+(L)}\Path_L^+(\lambda)\times \Path_L^+(\lambda) \longrightarrow \WebBasis{L}, \qquad (P,Q) \longmapsto C^\lambda_{P,Q},
\end{equation}
which is a bijection by \cref{prp: structure theorem}. For $\lambda\in \Lambda_+(L)$ and $P,Q\in \Path_L^+(\lambda)$, by \cref{lem: practical C}, the web $C^\lambda_{P,Q}$ is the concatenation of three webs:
\begin{align}\label{eqn: practical C}
    C^\lambda_{P,Q} = G(P) \circ \mathcal{H}^S_T \circ G(Q)^\dagger,
\end{align}
where $G$ is the growth algorithm from \cref{def: growth algorithm}; $\dagger$ is the reflection of webs against a horizontal axis, followed by reversal of orientations on edges from \cref{def: dagger}; and $\mathcal{H}^S_T$ is a concatenation of $H^\pm$-webs, constructed in \cref{cor: wiring construction for H} and \cref{exp: H^S_T}. An example of $C^\lambda_{P,Q}$ is given in \cref{fig: cellular structure of KL(q)}. The map $\dagger$ is an anti-involution, and a $\mathcal K_L$-anti-automorphism, see \cref{lem: dagger is anti-involution}. Moreover, for all $\lambda\in \Lambda_+(L)$ and $P,Q\in \Path^+_L(\lambda)$, by \cref{prp: dagger swaps PQ} we have: 
\begin{align}\label{eqn: anti-involution}
    (C^\lambda_{P,Q})^\dagger = C^\lambda_{Q,P}.
\end{align}

\paragraph{The filtration.} Let $a\in \mathcal{K}_L$ and $C^\lambda_{P,Q}$ be a cellular basis element, by \cref{prp: factorized multiplication} the following holds:
\begin{align}\label{eqn: filtered multiplication}
        aC^\lambda_{P,Q} = \sum_{P'} r_a(P,P') C^\lambda_{P',Q} \mod \mathcal K_L^{\prec \lambda},
\end{align}
where $\mathcal K_L^{\prec \lambda}:=\Span_{\C(q)}\{C^\mu_{P,Q}:P,Q\in\Path_L^+(\mu),\ \mu\prec\lambda\}$, and $r_a(P,P')\in \C(q)$. As a consequence of \eqref{eqn: filtered multiplication}, $(\mathcal K_L^{\prec \lambda})_{\lambda\in \Lambda_+(L)}$ is a filtration by two-sided ideals in $\mathcal K_L$.

\paragraph{The cell modules.} For each $\lambda\in\Lambda_+(L)$, the \emph{cell module} $\Delta_L(\lambda)$ is the $\mathcal{K}_L$-left module with basis $\{v_P : P\in \Path_L^+(\lambda)\}$, and $\mathcal K_L$-action defined by
\begin{align}\label{eqn: cell module action}
    a\cdot v_P := \sum_{P'} r_a(P,P') v_{P'}, \quad \forall a\in \mathcal K_L.
\end{align}
This indeed defines an action due to \eqref{eqn: filtered multiplication} and the associativity of multiplication in $\mathcal K_L$.

Now fix $a\in \mathcal K_L$, $\lambda\in\Lambda_+(L)$, and $v_P\in \Delta_L(\lambda)$. We give an algorithm to compute the coefficients $r_a(P,P')$. Identify
\begin{align*}
    \{v_P : P\in \Path_L^+(\lambda)\} \to \{G(P):P\in\Path_L^+(\lambda)\}, \quad v_P \mapsto G(P),
\end{align*}
and consider the concatenation of webs $aG(P)$. Let $U$ be the bottom sign string of $G(P)$. Apply the Kuperberg relations to express $aG(P)$ as a linear combination of reduced webs. By \cref{lem: basic structure theorem}, each reduced
summand is of the form $C^\mu_{P',Q}$ for some $P'\in\Path_L^+(\mu)$ and $Q\in\Path_U^+(\mu)$. Note that $C^\mu_{P',Q}$ is a web with bottom sign string $U$, whose length is not necessarily $L$, see \cref{def: pairing map C}. It should not be confused with a basis element of $\mathcal K_L$ in \eqref{eqn: bijection C}.

Since $Q$ has sign string $U$, \cref{lem: weight domination}
gives $\mu\preceq\wt(U)=\lambda$. Thus the reduced summands fall into two types:
\begin{enumerate}
    \item If $\mu=\lambda$, then
    $Q=P_U^{\mathrm{hw}}$ by \cref{lem: weight domination}. \Cref{lem: basis decomposition} gives
    \begin{align}
        C^\lambda_{P',P_U^{\mathrm{hw}}}
        =G(P')\circ\mathcal H^{U'}_U,
    \end{align}
    where $U'$ is the bottom sign string of $G(P')$. By \cref{lem: H^S_T is concatenation of H web}, $\mathcal H^{U'}_U$ is a concatenation of $H^\pm$-webs.

    \item If $\mu\prec\lambda$, then
    \cref{lem: basis decomposition} gives
    \[
        C^\mu_{P',Q}
        =G(P')\circ C^\mu_{P_M^{\mathrm{hw}},Q}\,,
    \]
    where $M$ is the bottom sign string of $G(P')$.
    By \cref{lem: growth preserves weight},
    $\wt(M)=\mu\prec\lambda$. Hence this summand factors through a sign string of weight strictly below $\lambda$, like a ``bottleneck'' of an hourglass, see \cref{fig: cellular structure of KL(q)}.
\end{enumerate}

We then discard all summands of the second type, as explained below \cref{def: cell modules} of the cell modules. The remaining linear combination reads
\[
    \sum_{P'\in\Path_L^+(\lambda)}
    \widetilde r_a(P,P')\,
    G(P')\circ\mathcal H^{U'}_U,
\]
where $\widetilde r_a(P,P')$ is the coefficient of the corresponding retained web. Finally, remove $\mathcal H^{U'}_U$, and identify $r_a(P,P')=\tilde{r}_a(P,P')$. We now give four examples of the local diagrammatic reductions used in this action.

\begin{example}
    Consider the cell module $\Delta_3(\lambda_1+\lambda_2)$. Define the paths $P_{\mathrm L}=u_{+,1}u_{+,0}u_{+,1}$ and $P_{\mathrm R}=u_{+,1}u_{+,1}u_{+,0}$. We consider the action $I^+_2\cdot G(P_{\mathrm L})$. The growth algorithm gives $G(P_{\mathrm L})=Y_-\otimes\id_+$ and $G(P_{\mathrm R})=\id_+\otimes Y_-$, where $Y_\pm$ is in \cref{fig: elementary webs}. The action by $I^+_2$ is shown in \cref{fig: cell module action example}, giving
    \begin{equation}
        I^+_2\cdot G(P_{\mathrm L})=G(P_{\mathrm R}).
    \end{equation}
\begin{figure}[htbp]
    \centering
    \scalebox{1.02}{
\tikzset{
    style/.style={#1},
    none/.style={},
    mid arrow/.style={postaction={decorate,decoration={
            markings, mark=at position 0.55 with {\arrow{stealth}}
        }}
    },
    dotted/.style={dashed}
}
\begin{tikzpicture}[
    x=1cm, y=1cm,
    web/.style={line width=0.5pt},
    every node/.style={font=\small, inner sep=0.5pt},
    labelnode/.style={inner sep=0.3pt},
    relationlabel/.style={font=\scriptsize},
    arrowlabel/.style={font=\scriptsize,fill=white,inner sep=1.2pt},
    descriptionlabel/.style={font=\scriptsize},
    htail/.style={draw=gray!75, dashed, rounded corners=1.5pt,
        fill=gray!8}
]
\begin{scope}[shift={(-5,0)}]
    \draw[gray!70,densely dotted,line width=0.5pt] (-0.92,-0.68) -- (0.92,-0.68);
\begin{scope}[shift={(0,-0.68)}]
    \draw [style=mid arrow] (0.36,-1.72) -- (0.72,0);
    \draw [style=mid arrow] (-0.36,-0.912) -- (-0.72,0);
    \draw [style=mid arrow] (-0.36,-0.912) -- (0,0);
    \draw [style=mid arrow] (-0.36,-0.912) -- (-0.36,-1.72);
\end{scope}
    \draw [style=mid arrow] (-0.72,-0.68) -- (-0.72,0);
    \draw [style=mid arrow] (0.36,-0.19) -- (0,0);
    \draw [style=mid arrow] (0,-0.68) -- (0.36,-0.49);
    \draw [style=mid arrow] (0.36,-0.19) -- (0.72,0);
    \draw [style=mid arrow] (0.72,-0.68) -- (0.36,-0.49);
    \draw [style=mid arrow] (0.36,-0.19) -- (0.36,-0.49);
    \node[descriptionlabel] at (0,-2.83) {$I^+_2\circ G(P_{\mathrm L})$};
\end{scope}
\node at (-2.5,-1.18) {$=$};
\begin{scope}
    \fill[gray!8,rounded corners=1.5pt] (-0.7,-2.52) rectangle (0.7,-1);
    \draw [style=none] (-0.36,-1.08) -- (-0.72,0);
    \draw [style=mid arrow] (0.36,-0.572) -- (0,0);
    \draw [style=mid arrow] (0.36,-0.572) -- (0.72,0);
    \draw [style=none] (0.36,-0.572) -- (0.36,-1.08);
    \draw [style=mid arrow] (-0.23,-1.74) -- (-0.36,-1.08);
    \draw [style=mid arrow] (-0.23,-1.74) -- (-0.36,-2.4);
    \draw [style=mid arrow] (-0.23,-1.74) -- (0.23,-1.74);
    \draw [style=mid arrow] (0.36,-1.08) -- (0.23,-1.74);
    \draw [style=mid arrow] (0.36,-2.4) -- (0.23,-1.74);
    \draw[gray!75,dashed,rounded corners=1.5pt] (-0.7,-2.52) rectangle (0.7,-1);
    \draw[gray!75,dashed] (-0.74,-1) -- (0.74,-1);
    \node[anchor=west,relationlabel] at (0.88,-1.76) {$H^-$-tail};
    \node[descriptionlabel] at (0,-2.83) {$G(P_{\mathrm R})\circ H^-$};
\end{scope}
\draw[->,>=stealth] (1.35,-1.18) -- (3.65,-1.18) node[midway,above,arrowlabel,align=center] {remove the bottom $H^-$};
\begin{scope}[shift={(5,0)}]
    \draw [style=mid arrow] (-0.36,-2.4) -- (-0.72,0);
    \draw [style=mid arrow] (0.36,-1.272) -- (0,0);
    \draw [style=mid arrow] (0.36,-1.272) -- (0.72,0);
    \draw [style=mid arrow] (0.36,-1.272) -- (0.36,-2.4);
    \node[descriptionlabel] at (0,-2.83) {$G(P_{\mathrm R})$};
\end{scope}
\end{tikzpicture}}
    \caption{}
    \label{fig: cell module action example}
\end{figure}
\end{example}

\begin{example}
    For an example involving a reduction by Kuperberg relations, consider the cell module $\Delta_6(3\lambda_1)$, and define three paths
    \begin{align*}
        P_{\mathrm{sq}}
        &=u_{+,1}u_{+,1}u_{+,1}u_{+,0}u_{+,1}u_{+,-1},\\
        Q_{\mathrm L}
        &=u_{+,1}u_{+,1}u_{+,1}u_{+,0}u_{+,-1}u_{+,1},\\
    Q_{\mathrm R}
    &=u_{+,1}u_{+,1}u_{+,1}u_{+,1}u_{+,0}u_{+,-1}.
\end{align*}
We consider the action $I^+_4\cdot G(P_{\mathrm{sq}})$. The web $I^+_4G(P_{\mathrm{sq}})$ contains the internal square highlighted in \cref{fig: square cell module action example}, so it resolves into two webs $G(Q_{\mathrm L})$ and $G(Q_{\mathrm R})$. The action is thus
\begin{equation}
    I^+_4\cdot G(P_{\mathrm{sq}})=G(Q_{\mathrm L})+G(Q_{\mathrm R}).
\end{equation}

\begin{figure}[htbp]
    \centering
    \scalebox{1.02}{
\tikzset{
    style/.style={#1},
    none/.style={},
    mid arrow/.style={postaction={decorate,decoration={
            markings, mark=at position 0.55 with {\arrow{stealth}}
        }}
    },
    dotted/.style={dashed}
}
\begin{tikzpicture}[
    x=1cm, y=1cm,
    web/.style={line width=0.5pt},
    every node/.style={font=\small, inner sep=0.5pt},
    labelnode/.style={inner sep=0.3pt},
    relationlabel/.style={font=\scriptsize},
    arrowlabel/.style={font=\scriptsize,fill=white,inner sep=1.2pt},
    descriptionlabel/.style={font=\scriptsize}
]
\begin{scope}[shift={(-5.25,0)}]
    \fill[gray!12] (0,-0.967) -- (0,-1.24) -- (0.31,-1.82) -- (0.62,-2.127) -- (0.93,-1.82) -- (0.93,-1.547) -- (0.62,-1.24) -- (0.93,-0.66) -- (0.62,-0.475) -- (0.31,-0.66) -- cycle;
    \draw[gray!70,densely dotted,line width=0.5pt] (-1.75,-0.66) -- (1.75,-0.66);
\begin{scope}[shift={(0,-0.66)}]
    \draw [style=none] (-1.24,-0.58) -- (-1.55,0);
    \draw [style=none] (-0.62,-0.58) -- (-0.93,0);
    \draw [style=none] (0.62,-0.58) -- (0.93,0);
    \draw [style=none] (1.24,-0.58) -- (1.55,0);
    \draw [style=mid arrow] (0,-0.307) -- (-0.31,0);
    \draw [style=mid arrow] (0,-0.307) -- (0.31,0);
    \draw [style=none] (0,-0.307) -- (0,-0.58);
    \draw [style=none] (-0.93,-1.16) -- (-1.24,-0.58);
    \draw [style=none] (-0.31,-1.16) -- (-0.62,-0.58);
    \draw [style=none] (0,-0.58) -- (0.31,-1.16);
    \draw [style=mid arrow] (0.93,-0.887) -- (0.62,-0.58);
    \draw [style=mid arrow] (0.93,-0.887) -- (1.24,-0.58);
    \draw [style=none] (0.93,-0.887) -- (0.93,-1.16);
    \draw [style=mid arrow] (-0.62,-1.74) -- (-0.93,-1.16);
    \draw [style=mid arrow] (0,-1.74) -- (-0.31,-1.16);
    \draw [style=mid arrow] (0.31,-1.16) -- (0.62,-1.467);
    \draw [style=mid arrow] (0.93,-1.16) -- (0.62,-1.467);
    \draw [style=mid arrow] (0.62,-1.74) -- (0.62,-1.467);
\end{scope}
    \draw [style=mid arrow] (-1.55,-0.66) -- (-1.55,0);
    \draw [style=mid arrow] (-0.93,-0.66) -- (-0.93,0);
    \draw [style=mid arrow] (-0.31,-0.66) -- (-0.31,0);
    \draw [style=mid arrow] (1.55,-0.66) -- (1.55,0);
    \draw [style=mid arrow] (0.62,-0.185) -- (0.31,0);
    \draw [style=mid arrow] (0.31,-0.66) -- (0.62,-0.475);
    \draw [style=mid arrow] (0.62,-0.185) -- (0.93,0);
    \draw [style=mid arrow] (0.93,-0.66) -- (0.62,-0.475);
    \draw [style=mid arrow] (0.62,-0.185) -- (0.62,-0.475);
    \draw[gray!70,dashed,line width=0.5pt] (0,-0.967) -- (0,-1.24) -- (0.31,-1.82) -- (0.62,-2.127) -- (0.93,-1.82) -- (0.93,-1.547) -- (0.62,-1.24) -- (0.93,-0.66) -- (0.62,-0.475) -- (0.31,-0.66) -- cycle;
    \node[descriptionlabel] at (0,-2.82) {$I^{+}_{4}\circ G(P_{\mathrm{sq}})$};
\end{scope}
\draw[->,>=stealth] (-3.42,-1.16) -- (-1.97,-1.16) node[midway,above,arrowlabel] {square relation};
\begin{scope}[shift={(0,0)}]
    \draw [style=none] (-1.24,-1.2) -- (-1.55,0);
    \draw [style=none] (-0.62,-1.2) -- (-0.93,0);
    \draw [style=none] (0,-1.2) -- (-0.31,0);
    \draw [style=none] (1.24,-1.2) -- (1.55,0);
    \draw [style=mid arrow] (0.62,-0.636) -- (0.31,0);
    \draw [style=mid arrow] (0.62,-0.636) -- (0.93,0);
    \draw [style=none] (0.62,-0.636) -- (0.62,-1.2);
    \draw [style=mid arrow] (-0.62,-2.4) -- (-1.24,-1.2);
    \draw [style=mid arrow] (0,-2.4) -- (-0.62,-1.2);
    \draw [style=mid arrow] (0.62,-2.4) -- (0,-1.2);
    \draw [style=mid arrow, looseness=1.05, bend right=82] (0.62,-1.2) to (1.24,-1.2);
    \node[descriptionlabel] at (0,-2.82) {$G(Q_{\mathrm R})$};
\end{scope}
\begin{scope}[shift={(5.25,0)}]
    \draw [style=none] (-1.24,-1.2) -- (-1.55,0);
    \draw [style=none] (-0.62,-1.2) -- (-0.93,0);
    \draw [style=none] (0.62,-1.2) -- (0.93,0);
    \draw [style=none] (1.24,-1.2) -- (1.55,0);
    \draw [style=mid arrow] (0,-0.636) -- (-0.31,0);
    \draw [style=mid arrow] (0,-0.636) -- (0.31,0);
    \draw [style=none] (0,-0.636) -- (0,-1.2);
    \draw [style=mid arrow] (-0.62,-2.4) -- (-1.24,-1.2);
    \draw [style=mid arrow] (0,-2.4) -- (-0.62,-1.2);
    \draw [style=mid arrow] (0.62,-2.4) -- (1.24,-1.2);
    \draw [style=mid arrow, looseness=1.05, bend right=82] (0,-1.2) to (0.62,-1.2);
    \node[descriptionlabel] at (0,-2.82) {$G(Q_{\mathrm L})$};
\end{scope}
\node at (2.625,-1.16) {$+$};
\end{tikzpicture}}
    \caption{}
    \label{fig: square cell module action example}
\end{figure}
\end{example}

\begin{example}
    An example where some webs are discarded: consider the cell module $\Delta_3(2\lambda_2)$, and define paths $P=u_{+,1}u_{-,1}u_{+,0}$, $Q=u_{+,1}u_{+,0}u_{-,1}$. We consider the action $H^-_2\cdot G(P)$. As shown in \cref{fig: partially killed cell module action}, the square relation gives $H^-_2 G(P)=G(Q)+w_{\mathrm{low}}$. The first term $G(Q)$ is kept, while the second term $w_{\mathrm{low}}$ admits a cut with weight $\lambda_1\prec2\lambda_2$, hence is discarded. We thus obtain the action
    \begin{equation}
        H^-_2\cdot G(P)=G(Q).
    \end{equation}
\begin{figure}[htbp]
    \centering
    \scalebox{1.02}{
\tikzset{
    style/.style={#1},
    none/.style={},
    mid arrow/.style={postaction={decorate,decoration={
            markings, mark=at position 0.55 with {\arrow{stealth}}
        }}
    },
    dotted/.style={dashed}
}
\begin{tikzpicture}[
    x=1cm, y=1cm,
    web/.style={line width=0.5pt},
    every node/.style={font=\small, inner sep=0.5pt},
    labelnode/.style={inner sep=0.3pt},
    relationlabel/.style={font=\scriptsize,align=center},
    arrowlabel/.style={font=\scriptsize,fill=white,inner sep=1.2pt,align=center},
    descriptionlabel/.style={font=\scriptsize,align=center}
]
\begin{scope}[shift={(-5,0)}]
    \fill[gray!12] (0.13,-0.31) -- (0.59,-0.31) -- (0.72,-0.58) -- (0.59,-1.102) -- (0.13,-1.102) -- (0,-0.58) -- cycle;
    \draw[gray!70,densely dotted,line width=0.5pt] (-0.92,-0.58) -- (0.92,-0.58);
    \draw [style=mid arrow] (-0.72,-0.58) -- (-0.72,0);
    \draw [style=mid arrow] (0.13,-0.31) -- (0,0);
    \draw [style=mid arrow] (0.13,-0.31) -- (0,-0.58);
    \draw [style=mid arrow] (0.13,-0.31) -- (0.59,-0.31);
    \draw [style=mid arrow] (0.72,0) -- (0.59,-0.31);
    \draw [style=mid arrow] (0.72,-0.58) -- (0.59,-0.31);
\begin{scope}[shift={(0,-0.58)}]
    \draw [style=none] (-0.72,-0.985) -- (-0.72,0);
    \draw [style=mid arrow] (0,0) -- (0.13,-0.522);
    \draw [style=none] (0,-0.985) -- (0.13,-0.522);
    \draw [style=mid arrow] (0.59,-0.522) -- (0.13,-0.522);
    \draw [style=mid arrow] (0.59,-0.522) -- (0.72,0);
    \draw [style=none] (0.59,-0.522) -- (0.72,-0.985);
    \draw [style=mid arrow] (0.72,-0.985) -- (0.36,-1.97);
    \draw [style=mid arrow] (-0.36,-1.507) -- (-0.72,-0.985);
    \draw [style=mid arrow] (-0.36,-1.507) -- (0,-0.985);
    \draw [style=mid arrow] (-0.36,-1.507) -- (-0.36,-1.97);
\end{scope}
    \draw[gray!65,dashed] (0.13,-0.31) -- (0.59,-0.31) -- (0.72,-0.58) -- (0.59,-1.102) -- (0.13,-1.102) -- (0,-0.58) -- cycle;
    \node[descriptionlabel] at (0,-3.08) {$H^-_2\circ G(P)$};
\end{scope}
\draw[->,>=stealth] (-3.85,-0.72) -- (-1.25,-0.72) node[midway,above,arrowlabel]{square relation};
\begin{scope}
    \draw [style=mid arrow] (0.72,0) -- (0.36,-2.55);
    \draw [style=mid arrow] (-0.36,-1.351) -- (-0.72,0);
    \draw [style=mid arrow] (-0.36,-1.351) -- (0,0);
    \draw [style=mid arrow] (-0.36,-1.351) -- (-0.36,-2.55);
    \node[relationlabel] at (0,-3.08) {$G(Q)$\\$\operatorname{wt}(-,-)=2\lambda_2$;\enspace kept};
\end{scope}
\node at (2.5,-1.18) {$+$};
\begin{scope}[shift={(5,0)}]
    \draw [style=mid arrow] (-0.72,-1.15) -- (-0.72,0);
    \draw [style=mid arrow,looseness=1.05,bend left=82] (0.72,0) to (0,0);
    \draw [style=mid arrow, looseness=1.05, bend left=82] (0,-1.15) to (0.72,-1.15);
    \draw [style=mid arrow] (-0.36,-1.88) -- (-0.72,-1.15);
    \draw [style=mid arrow] (-0.36,-1.88) -- (0,-1.15);
    \draw [style=mid arrow] (-0.36,-1.88) -- (-0.36,-2.55);
    \draw [style=mid arrow] (0.72,-1.15) -- (0.36,-2.55);
    \draw[gray!70,densely dashed] (-1.02,-0.58) -- (1.02,-0.58) node[right,relationlabel] {$(+)$};
    \node[relationlabel] at (0,-3.08) {lower-weight resolution\\$\operatorname{wt}(+)=\lambda_1\prec2\lambda_2$;\enspace discarded};
\end{scope}
\end{tikzpicture}}
    \caption{}
    \label{fig: partially killed cell module action}
\end{figure}
\end{example}

\begin{example}\label{exmp: cascade of squares}
    For an example requiring two successive square reductions, consider the cell module $\Delta_7(0)$ and define
    \begin{align*}
        P&=u_{+,1}u_{-,1}u_{+,0}u_{+,1}u_{+,-1}u_{+,-1}u_{-,-1},\\
        Q_1&=u_{+,1}u_{-,-1}u_{+,1}u_{+,0}u_{+,-1}u_{+,1}u_{-,-1},\\
        Q_2&=u_{+,1}u_{-,1}u_{+,1}u_{+,0}u_{+,-1}u_{+,-1}u_{-,-1},\\
        Q_3&=u_{+,1}u_{-,1}u_{+,0}u_{+,-1}u_{+,1}u_{+,-1}u_{-,-1}.
    \end{align*}
    We consider the action $I^+_3\cdot G(P)$.
    As shown in \cref{fig: seven site double square cell action}, the action $I^+_3G(P)$ creates a square, which resolves into two webs: $I^+_3G(P)=W+G(Q_3)$. However, the web $W$ has another square, whose two resolutions give $W=G(Q_1)+G(Q_2)$. We thus obtain
    \begin{equation}
        I^+_3\cdot G(P)=G(Q_1)+G(Q_2)+G(Q_3).
    \end{equation}
    \begin{center}
    \begin{figure}[htbp]
        \centering
        \begingroup
        \scalebox{1.02}{\input{tikz_drawings/cell_module_action_I3_seven_sites.tikz}}
        \endgroup
        \caption{}
        \label{fig: seven site double square cell action}
    \end{figure}
    \end{center}
\end{example}

\begin{remark}
The ``cascade of squares'' phenomenon in \cref{exmp: cascade of squares} has no analogue in the web basis of Temperley--Lieb cell modules, usually called link states: each $\mathrm{TL}$ generator $e_i$ sends a link state to zero or to a scalar multiple of a single link state. By contrast, the $A_2$ square relation, applied recursively, can produce several basis webs during the action of a generator of $\mathcal K_L$ on a web basis vector of its cell module.
\end{remark}

\paragraph{Bimodule decomposition.} Let $L\geq 2$. Recall $\mathscr{V}_L=(V^+\oplus V^-)^{\otimes L}$ from \eqref{eqn: def of V_L}. As an application of the cellular structure, we obtain a decomposition of $\mathscr{V}_L$ as a $(\uqslt,\mathcal{K}_L)$-bimodule.\footnote{We restrict the statement to $L\geq2$, as in \cref{def: LambdaL}. The smaller cases are elementary: $\mathcal K_0\cong\C(q)$ and $\mathcal K_1\cong\C(q)\oplus\C(q)$. Their occurring highest weights are $\{0\}$ and $\{\lambda_1,\lambda_2\}$, respectively; in particular, weight $0$ does not occur for $L=1$.}
\begin{theorem}\label{thm: bimodule decomposition}
    Let $L\geq 2$, and give $\mathscr{V}_L=(V^+\oplus V^-)^{\otimes L}$ the right $\mathcal{K}_L$-action
    \begin{align*}
        h\cdot a:=\Phi(a^\dagger)h.
    \end{align*}
    Here $\Phi$ is the isomorphism defined in \eqref{eqn: def of Phi}.
    Then there is an isomorphism of $(\uqslt,\mathcal{K}_L)$-bimodules
    \begin{align}\label{eqn: Schur-Weyl bimodule decomposition}
        \mathscr{V}_L
        \cong
        \bigoplus_{\lambda\in\Lambda_+(L)}
        L_q(\lambda)\otimes \Delta_L(\lambda)^\dagger,
    \end{align}
    where $L_q(\lambda)$ is the irreducible highest-weight $\uqslt$-module of highest weight $\lambda$, and $\Delta_L(\lambda)^\dagger$ is the right cell module with action $v\cdot a=a^\dagger v$. In particular, the cell modules
    $\Delta_L(\lambda)$ are irreducible and pairwise non-isomorphic, and they form a complete set of irreducible $\mathcal{K}_L$-modules.
\end{theorem}
The proof is postponed to \cref{sec: proof of bimodule decomposition}.

\subsection{Generators}\label{sec: generators guide}

In this section, we fix an integer $L\geq 2$, and determine a set of generators for $\mathcal{K}_L$. For a web $X$ with $k\leq L$ boundary edges on top and bottom, we define
\begin{align}
    X_i := 1_{i-1} \otimes X \otimes 1_{L-i-k+1}, \qquad i=1,\ldots, L-k+1,
\end{align}
where $1_m$ denotes the unit of $\mathcal K_m$ (and $1_0=\id_{\emptyset}$). The definition ensures that for any $w$ with $L$ top boundary vertices, $X_iw$ is obtained by stacking $X$ on top of $w$ with the leftmost boundary edge of $X$ at the $i$-th boundary edge of $w$. With the webs in \cref{fig: generating set}, we formulate the following theorem on the generators of $\mathcal K_L$, with its proof in \cref{sec: outline for FFT}.
\begin{theorem}\label{thm: generators of K_L}
    The algebra $\mathcal{K}_L$ is generated by the following set of webs:
    \begin{align}\label{eqn: generating set}
        \GenSet_L:=\{I_i^{s_1},H_i^{s_1},(\id_{s_1s_2})_i, \YL{s_1}{s_1}_j: s_1,s_2 \in \{+,-\}, i=1,\ldots,L-1, j=1,\ldots L-2 \}.
    \end{align}
\end{theorem}
Let $B_L$ denote the Artin braid group on $L$ strands. The $R$-matrix of $\uqslt$ \cite{rosso_1989_pbw} gives a braid group representation on $\mathscr V_L$, whose image lies in $\End_{\uqslt}(\mathscr V_L)$.\footnote{The $\check R$-matrices are written as linear combinations of webs in \cite[Sec.\,4]{kuperberg_1996_spiders}.} The isomorphism $\Phi$ in \cref{thm: Schur-Weyl duality} induces an algebra homomorphism $\rho_R:\C(q)[B_L]\to\mathcal K_L$. We now show that $\rho_R$ is not surjective:
\begin{proposition}\label{prp: KL(q) is not quotient of braid group}
    For $L\geq3$, the subalgebra generated by $\rho_R(\C(q)[B_L])$ and the idempotents $(\id_{s_1s_2})_i$, with $s_1,s_2\in\{+,-\}$ and $i=1,\ldots,L-1$, is a proper subalgebra of $\mathcal K_L$.
\end{proposition}
\begin{proof}
    Let $s,t\in \{\pm\}$. Since $V^\pm$ are invariant under $\uqslt$, and the universal $R$-matrix acts through $\uqslt$ operators, the action preserves each $V^s\otimes V^t$. Hence the braid operator $\check R=\tau\circ R$, where $\tau$
    flips the tensor factors, maps $V^s\otimes V^t$ to
    $V^t\otimes V^s$. Thus braiding only permutes the signs of
    a sign string and preserves its weight. The embedded idempotents satisfy the same condition, which
    is preserved under sums and products. However, the nonzero generator $\YL{+}{+}_1$ changes the first three signs from $---$ to $+++$, changing the weight by $3(\lambda_1-\lambda_2)$. It therefore cannot belong to the subalgebra generated by braidings and the idempotents.
\end{proof}

\begin{remark}
    It is proved in \cite{lehrer_2006_strongly} that, for any simply laced Lie algebra $\mathfrak{g}$, the endomorphism algebra $\End_{U_q(\mathfrak{g})}(V^{\otimes L})$ is a quotient of the braid group algebra if $V$ is a strongly multiplicity free $U_q(\mathfrak{g})$-module. In the case of $\mathcal{K}_L$, the representation $V=V^+\oplus V^-$ is not strongly multiplicity free, so the result of \cite{lehrer_2006_strongly} does not apply, as witnessed by \cref{prp: KL(q) is not quotient of braid group}. This feature of $\mathcal{K}_L$ differs from the Temperley-Lieb algebra, which is generated by the R-matrix on $U_q(\mathfrak{sl}_2)$-fundamental modules and idempotents.
\end{remark}

\begin{figure}[htbp]
    \centering
    \scalebox{0.8}{\begin{tikzpicture}
	\begin{pgfonlayer}{nodelayer}
		\node (0) at (-7, 1) {};
		\node (1) at (-6.5, 0) {};
		\node (2) at (-6, 1) {};
		\node (3) at (-7, -1) {};
		\node (4) at (-5, 1) {};
		\node (5) at (-5.5, 0) {};
		\node (6) at (-6, -1) {};
		\node (7) at (-5, -1) {};
		\node at (-7, 1.25) {$\pm$};
		\node at (-5, 1.25) {$\mp$};
		\node at (-6, -1.5) {$\YL{\pm}{\mp}$};

		\node (8) at (-3, 1) {};
		\node (9) at (-1, 1) {};
		\node (10) at (-3, -1) {};
		\node (11) at (-1, -1) {};
		\node at (-3, 1.25) {$s_1$};
		\node at (-3, -1.25) {$s_2$};
		\node at (-2, -1.5) {$\TL{s_1}{s_2}$};

		\node (12) at (1, 1) {};
		\node (13) at (3, 1) {};
		\node (14) at (2, 1) {};
		\node (15) at (1, -1) {};
		\node (16) at (3, -1) {};
		\node (17) at (2, -1) {};
		\node (18) at (2, 0.25) {};
		\node (19) at (2, -0.25) {};
		\node at (1, 1.25) {$s_1$};
		\node at (1, -1.25) {$s_2$};
		\node at (2, -1.5) {$T^{s_1}_{s_2}$};

		\node (20) at (6, 1) {};
		\node (21) at (6.5, 0) {};
		\node (22) at (7, 1) {};
		\node (23) at (7, -1) {};
		\node (24) at (5, 1) {};
		\node (25) at (5.5, 0) {};
		\node (26) at (5, -1) {};
		\node (27) at (6, -1) {};
		\node at (5, 1.25) {$s_1$};
		\node at (7, 1.25) {$s_2$};
		\node at (6, -1.5) {$\LY{s_1}{s_2}$};
	\end{pgfonlayer}
	\begin{pgfonlayer}{edgelayer}
		\draw (1.center) to (0.center);
		\draw (1.center) to (2.center);
		\draw (1.center) to (3.center);
		\draw (5.center) to (4.center);
		\draw (5.center) to (7.center);
		\draw (5.center) to (6.center);
		\draw [bend right=90, looseness=1.25] (8.center) to (9.center);
		\draw [bend left=90, looseness=1.25] (10.center) to (11.center);
		\draw [bend right=90, looseness=1.25] (12.center) to (13.center);
		\draw [bend left=90, looseness=1.25] (15.center) to (16.center);
		\draw (14.center) to (18.center);
		\draw (19.center) to (17.center);
		\draw (21.center) to (20.center);
		\draw (21.center) to (22.center);
		\draw (21.center) to (23.center);
		\draw (25.center) to (24.center);
		\draw (25.center) to (27.center);
		\draw (25.center) to (26.center);
	\end{pgfonlayer}
\end{tikzpicture}}
    \caption{Auxiliary basic webs, here $s_1,s_2\in \{\pm\}$ are signs.}
    \label{fig: auxiliary basic webs}
\end{figure}
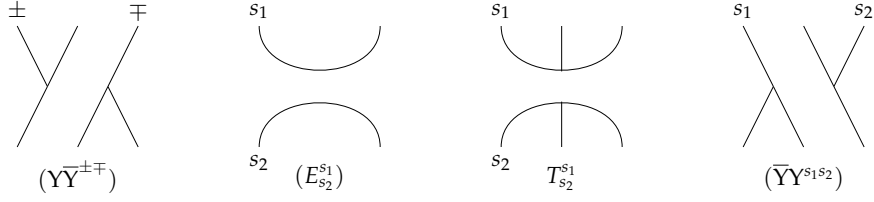

The webs in \cref{fig: auxiliary basic webs} are not needed to generate $\mathcal{K}_L$, but they are useful in future discussions. Therefore, we provide relations in \eqref{eqn: relations in KL(q)} to show that they can be expressed in terms of the generators in \eqref{eqn: generating set}. The relations are obtained by applying Kuperberg relations to the auxiliary basic webs.
\begin{equation}\label{eqn: relations in KL(q)}
    \begin{aligned}
        &\YL{\pm}{\mp}=[2]_q^{-1}(I^{\pm}_1 H^{\mp}_{2}H^{\mp}_1I^\pm_{2}-I^{\pm}_1\TL{\pm}{\mp}_{2}H^\mp_1),\\
        &\TL{\pm}{\pm}=H^\mp H^\pm - \id_{\pm\mp},\\
        &\TL{\pm}{\mp}=[2]_q^{-1}\TL{\pm}{\pm}H^\mp,\\
        &T^{\pm}_{\pm}=I^{\pm}_1 I^\mp_{2} I_1^{\pm} - I_1^{\pm},\\
        &T^{\pm}_{\mp}=\YL{\pm}{\pm} \YL{\mp}{\mp} \YL{\pm}{\pm}-\YL{\pm}{\pm},\\
        &\LY{\pm}{\pm} =I^\pm_2 \YL{\pm}{\pm} I_1^\mp-T^{\pm\pm},\\
        &\LY{\pm}{\mp} =I^{\mp}_2 H^{\mp}_1 H^{\mp}_2 - H^{\mp}_1 \TL{\pm}{\mp}_2.
    \end{aligned}
\end{equation}

\subsection{Dimension formulae}\label{sec: dimension formulae}

We now give formulae for the dimension of $\mathcal{K}_L$ and its cell modules in terms of enumeration formulae of $\mathfrak{sl}_3$-lattice paths. Let $S$ be a sign string with $\wt(S)=p\lambda_1+q\lambda_2$. Let $\lambda=x\lambda_1+y\lambda_2$ be an $\mathfrak{sl}_3$-weight. Introduce
\begin{align}
    U_{p,q}(x,y)
    &:=
    \abs{\Path_S(\lambda)}
\end{align}
which clearly depends only on $p,q$ and $x,y$. This function admits simple multinomial formulae:
\begin{align}\label{eqn: Upq multinomial}
    U_{p,q}(x,y)
    =
    \sum_{\substack{m_1+m_2+m_3=q\\m_i\geq 0\\m_i+d_i\geq 0}}
    \binom{p}{m_1+d_1,\;m_2+d_2,\;m_3+d_3}
    \binom{q}{m_1,m_2,m_3},
\end{align}
where $\binom{n}{a,b,c}:=\frac{n!}{a!b!c!}$, with $n=a+b+c$, is the multinomial coefficient, and the integer shifts $d_i$ are given by
\begin{align}\label{eqn: R coefficients}
    d_1&=\frac{p-q+2x+y}{3},&
    d_2&=\frac{p-q-x+y}{3},&
    d_3&=\frac{p-q-x-2y}{3}.
\end{align}
If $d_i\notin\Z$ for some $i$, then the right hand side of \eqref{eqn: Upq multinomial} is understood as zero. We also define
\begin{align}\label{eqn: UL coefficient}
    U_L(x,y)
    &:=\abs{\Path_L(\lambda)}=
    \sum_{p+q=L}\binom{L}{p}U_{p,q}(x,y),
\end{align}
where the second equality follows because there are $\binom{L}{p}$ sign strings of length $L$ and with $p$ plus signs.

To enumerate paths in the dominant $A_2$ Weyl chamber, we apply the reflection principle \cite{gessel_1992_random,grabiner_1993_random}. For a function $F:\Z^2\to\Z$, define its shifted Weyl alternation $\WeylAlt F$ by
\begin{align}\label{eqn: reflection operator}
    (\WeylAlt F)(x,y)
    :={}&F(x,y)
    -F(-x-2,\;x+y+1)
    -F(x+y+1,\;-y-2) \notag\\
    &+F(-x-y-3,\;x)
    +F(y,\;-x-y-3)
    -F(-y-2,\;-x-2).
\end{align}
This is the reflection principle with the usual $\rho=\lambda_1+\lambda_2$
shift. In particular, for $\lambda=x\lambda_1+y\lambda_2$,
\begin{align}\label{eqn: dominant path count}
        \abs{\Path_S^+(\lambda)}
        =
        (\WeylAlt U_{p,q})(x,y),
        \qquad
        \abs{\Path_L^+(\lambda)}
        =
        (\WeylAlt U_L)(x,y).
\end{align}

We now relate the formulae \eqref{eqn: dominant path count} to the dimensions of $\mathcal{K}_L$ and its cell modules.

\begin{proposition}\label{prp: dimension formula for KL(q)}
    The algebra $\mathcal{K}_L$ is finite-dimensional, with
    \begin{align}\label{eqn: dim KL}
        \dim \mathcal{K}_L=(\WeylAlt U_{2L})(0,0).
    \end{align}
    The first few dimensions are
    \begin{align*}
        1,\;2,\;12,\;130,\;1946,\;34776,\;694056,\ldots
    \end{align*}
    for $L=0,1,2,3,4,5,6,\ldots$.  Equivalently, this is the even-indexed
    subsequence of the numbers $\abs{\Path_n^+(0)}$, which are
    recorded as OEIS A151366 \cite{a2019_a151366}.
\end{proposition}
\begin{proof}
    The algebra $\mathcal{K}_L$ has a basis $\WebBasis{L} = \bigsqcup_{S,T\in \Sigma_L} \Web^S_T$ by its definition. For each fixed $\Web^S_T$, the folding bijection $F:\Web^{ST^\vee}_{\emptyset}\to \Web^S_T$ from \cref{lem: folding is invertible} and the growth algorithm bijection from \cref{thm: growth algorithm} together give a bijection $\Web^S_T \to \Path_{ST^\vee}^+(0)$. Taking the union over all $S,T\in \Sigma_L$ gives a bijection $\WebBasis{L} \to \Path_{2L}^+(0)$, which proves \eqref{eqn: dim KL}.
\end{proof}

\begin{proposition}\label{prp: dimension formula for VLlambda}
    Let $L\geq 0$ and $\lambda=x\lambda_1+y\lambda_2\in\Lambda_+$.  Then
    the cell module indexed by $(L,\lambda)$ has dimension
    \begin{align}\label{eqn: dim VLlambda}
        \dim \Delta_L(\lambda)
        =
        \abs{\Path_L^+(\lambda)}
        =
        (\WeylAlt U_L)(x,y).
    \end{align}
\end{proposition}
\begin{proof}
    A basis of the cell module $\Delta_L(\lambda)$ indexed by $\Path_L^+(\lambda)$ follows from \cref{thm: basis of cell modules}.
\end{proof}

\section{Construction of the cellular structure}\label{sec: cellular structure}
In this section, we construct the cell datum of $\mathcal K_L$ for $L\geq 2$. After defining the map $C$ and the involutive anti-automorphism $\dagger$ in \cref{sec: proof of structure theorem,sec: the anti-involution dagger}, we construct the clasped web spaces, show that they are $\mathcal K_L$-modules, and determine their web bases in \cref{sec: cell modules}. In \cref{sec: alternative approach to cell modules} we construct special isomorphisms between the clasped web spaces. In \cref{sec: cellularity proof} we give the cell datum, and identify the clasped web spaces as cell modules. We prove the bimodule decomposition, \cref{thm: bimodule decomposition}, in \cref{sec: proof of bimodule decomposition}.

\subsection{The map $C$}\label{sec: proof of structure theorem} 
\begin{definition}\label{def: pairing map C}
    Let $\lambda\in \Lambda_+$, and $S,T$ be sign strings of arbitrary length. We define the map
    \begin{align}\label{eqn: definition of C}
        \begin{aligned}
            C^\lambda: \Path^+_S(\lambda) \times \Path^+_T(\lambda) &\to \Web^S_T, \\
            (P,Q) &\mapsto \Fold_T(G(PQ^\vee))=: C^\lambda_{P,Q},
        \end{aligned}
        \end{align}
    where $PQ^\vee$ is the concatenated path (see \cref{def: lattice path}), $G$ is the growth algorithm (see \cref{def: growth algorithm}), $\Web^S_T$ is the set from \eqref{eqn: web set and web space}, and $\Fold_T$ is folding (see \cref{def: folding and duality}). See the leftmost of \cref{fig: cellular basis decomposition proof} for an illustration of $\Fold_T(G(PQ^\vee))$.
\end{definition}
Since $P$ and $Q$ are dominant paths ending at weight $\lambda$, the concatenation $PQ^\vee\in \Path^+_{|P|+|Q|}(0)$ by \cref{lem: PQ-dual is an excursion}, hence $G(PQ^\vee)$ is well-defined. Moreover, $G(PQ^\vee)$ has top sign string $ST^\vee$, where $S$ is the sign string of $P$, so $\Fold_T(G(PQ^\vee))$ is well-defined. Therefore, the map $C$ is well-defined.

\begin{lemma}\label{lem: basic structure theorem}
    Let $S,T$ be sign strings. With $C^\lambda_{P,Q}$ as defined in \eqref{eqn: definition of C}, we have the following bijection:
    \begin{align}
        \bigsqcup_{\lambda \in \Lambda_+} \Path_S^+(\lambda) \times \Path_T^+(\lambda)  \to \Web^{S}_{T} \qquad
        (P,Q)  \mapsto C^\lambda_{P,Q}.
    \end{align}
\end{lemma}
\begin{proof}
    We can decompose the map $C$ into a composition of three bijections:
    \begin{enumerate}
        \item \textbf{Path concatenation:} The map 
        \begin{align*}
            \bigsqcup_{\lambda\in \Lambda_+} \Path_S^+(\lambda) \times \Path_T^+(\lambda)   \to \Path_{ST^\vee}^+(0) \qquad
            (P,Q)  \mapsto PQ^\vee
        \end{align*}
        is a bijection.
        For surjectivity, let $D\in \Path_{ST^\vee}^+(0)$, by \cref{lem: PQ-dual is an excursion}, $D$ can be written as $PQ^\vee$, where $P\in \Path_S^+(\lambda)$ and $Q\in \Path_T^+(\lambda)$ for a dominant weight $\lambda$. Injectivity is clear from the \cref{def: lattice path} of path concatenation.

        \item \textbf{Growth algorithm:} The map
        \begin{align*}
            \Path_{ST^\vee}^+(0)  \to  \Web^{ST^\vee}_\emptyset \qquad
            P  \mapsto G(P)
        \end{align*}
        is a bijection by \cref{thm: growth algorithm}.
        \item \textbf{Folding:} The folding
        \begin{align*}
            \Web^{ST^\vee}_\emptyset  \to \Web^{S}_{T} \qquad
            w  \mapsto \Fold_T(w)
        \end{align*}
        is a bijection by \cref{lem: folding is invertible}. 
    \end{enumerate}
    The composition of these three bijections is exactly the map $C$, hence $C$ is a bijection as well.
\end{proof}

As a direct consequence, we obtain the following proposition.

\begin{proposition}\label{prp: structure theorem}
    Recall $\WebBasis{L}=\bigsqcup_{S,T\in\Sigma_L}\Web^S_T$ from \eqref{def: Omega}. The following map is a bijection:
    \begin{align*}
        C: \bigsqcup_{\lambda\in \Lambda_+(L)}  \Path_L^+(\lambda) \times \Path_L^+(\lambda) & \to \WebBasis{L}.
    \end{align*}
\end{proposition}
\begin{proof}
    Using $\Path_L^+(\lambda) = \bigsqcup_{S\in \Sigma_L} \Path_S^+(\lambda)$, we get a bijection by \cref{lem: basic structure theorem}:
    \begin{align}\label{eqn: large bijection}
        C: \bigsqcup_{\lambda\in \Lambda_+}  \Path_L^+(\lambda) \times \Path_L^+(\lambda) & \to \WebBasis{L}.
    \end{align}
    By \cref{lem: weight and existence of paths}, $\Path_L^+(\lambda)$ is non-empty if and only if $\lambda\in \Lambda_+(L)$. Hence, we replace the condition $\lambda\in \Lambda_+$ by $\lambda\in \Lambda_+(L)$ in \eqref{eqn: large bijection}, and thus conclude the proof.
\end{proof}

\subsection{The anti-involution $\dagger$}\label{sec: the anti-involution dagger}

In this subsection we define an involutive anti-automorphism $\dagger$ for the cellular structure of $\mathcal{K}_L$.

\begin{definition}[$\dagger$]\label{def: dagger}
    The map $\dagger$ on $\mathcal{K}_L$ is defined on webs as a reflection against the horizontal axis, followed by the reversal of all edge orientations. The map $\dagger$ is then linearly extended.
\end{definition}

\begin{lemma}\label{lem: dagger is anti-involution}
    The map $\dagger$ is an involutive anti-automorphism on $\mathcal{K}_L$.
\end{lemma}
\begin{proof}
    Reflection against the horizontal axis is clearly involutive, and reverses the order of concatenation of webs, which is the multiplication in $\mathcal{K}_L$.
\end{proof}

The next property of $\dagger$ is required by the cellular structure of $\mathcal K_L$:
\begin{proposition}\label{prp: dagger swaps PQ}
    For any $\lambda\in \Lambda_+(L)$ and $P,Q\in \Path_L^+(\lambda)$, we have
    \begin{align*}
        (C^\lambda_{P,Q})^\dagger = C^\lambda_{Q,P}.
    \end{align*}
\end{proposition}

The remainder of the subsection is devoted to the proof of \cref{prp: dagger swaps PQ}. 

\begin{definition}
    Let $\Mir$ be reflection of a basis web against a vertical axis, followed by reversal of all edge orientations:
    \begin{align*}
        \Mir: \Web^{S_2}_{S_1} & \to \Web^{S_2^\vee}_{S_1^\vee}.
    \end{align*}
    It is then extended $\C(q)$-linearly.
\end{definition}
It is straightforward to check several simple properties of $\Mir$:
\begin{enumerate}
    \item $\Mir$ preserves composition in $\Sp$:
    \begin{align*}
        \Mir(w_1\circ w_2)=\Mir(w_1)\circ\Mir(w_2).
    \end{align*}
    In particular, $\Mir$ is a $\mathcal{K}_L$-automorphism.
    \item The webs $\ccup_S$ and $\ccap_S$ from \cref{def: cup and cap for sign strings} satisfy
    \begin{align*}
        \Mir(\ccup_S)=\ccup_S\;,
        \qquad
        \Mir(\ccap_S)=\ccap_S.
    \end{align*}
    \item $\Mir$ reverses the monoidal product:
    \begin{align*}
        \Mir(a\otimes b)=\Mir(b)\otimes\Mir(a).
    \end{align*}
\end{enumerate}

\begin{lemma}\label{lem: elementary geometry}
    Let $w$ be a web. The map $\dagger$ factorizes as the composition of $\Mir$ and duality $\vee$:
    \begin{align*}
        w^\dagger=(\Mir(w))^\vee = \Mir(w^\vee).
    \end{align*}
\end{lemma}
\begin{proof}
    By \cref{lem: duality as rotation}, $w^\vee=O_{\pi}(w)$, where $O_{\pi}$ is the counter-clockwise rotation by an angle of $\pi$. The result then follows from elementary geometry.
\end{proof}

To prove the main result  of this subsection, \cref{prp: dagger swaps PQ}, we establish the following symmetry of the growth algorithm:

\begin{lemma}\label{lem: growth algorithm symmetry}
    For $P\in \Path_L^+(0)$, we have
    \begin{align*}
        \Mir(G(P)) = G(P^\vee).
    \end{align*}
\end{lemma}
\begin{proof}
    By \cref{cor: involution preserves excursions}, $P^\vee\in \Path_L^+(0)$, and the web $G(P^\vee)$ is well-defined. We proceed by induction on the order of paths from \cref{def: order on paths}. The claimed result indeed holds for the empty path with $0$ steps. Assume that the lemma holds for all dominant paths with weight $0$, with order strictly lower than a given path $P$. We will show that it also holds for $P$, thus concluding the proof.
    
    We apply a growth rule $(w, P^i, P^f)$ to $P$. Write $P=P_1P^iP_2$, and $P'=P_1P^fP_2$. By the growth algorithm from \cref{def: growth algorithm}, we apply this growth rule to $P$ to get
    \begin{align}\label{eqn: growth on P}
        G(P) &= (\id_{S_1}\otimes w \otimes \id_{S_2}) \circ G(P').
    \end{align}
    As seen in \cref{prp: growth algorithm is well-defined}, applying growth rules lowers the order of paths, hence $P'$ has order strictly lower than $P$. Hence, by the induction hypothesis we have
\begin{align}\label{eqn: symmetry of P'}
    \Mir(G(P')) = G((P')^\vee).
\end{align}
The dual paths decompose as
\begin{align*}
    P^\vee&=(P_2)^\vee(P^i)^\vee(P_1)^\vee,
    &
    (P')^\vee&=(P_2)^\vee(P^f)^\vee(P_1)^\vee.
\end{align*}
Moreover, $(\Mir(w),(P^i)^\vee,(P^f)^\vee)$ is also a growth rule, which yields
\begin{align}\label{eqn: growth on dual P}
    G(P^\vee)= (\id_{S_2^\vee}\otimes \Mir(w) \otimes \id_{S_1^\vee}) \circ G((P')^\vee).
\end{align}
Since $\Mir$ reverses the monoidal product, it is easy to check that
\begin{align}\label{eqn: symmetry on growth rule}
    \id_{S_2^\vee}\otimes \Mir(w) \otimes \id_{S_1^\vee}=\Mir(\id_{S_1}\otimes w \otimes \id_{S_2}).
\end{align}
Hence, we have
\begin{align*}
    \Mir(G(P))
    &=\Mir(\id_{S_1}\otimes w \otimes \id_{S_2})\Mir(G(P'))\\
    &=(\id_{S_2^\vee}\otimes \Mir(w) \otimes \id_{S_1^\vee})G((P')^\vee)\\
    &=G(P^\vee),
\end{align*}
where the first equality uses \eqref{eqn: growth on P} and the fact that $\Mir$ preserves composition in $\Sp$, the second equality uses \eqref{eqn: symmetry of P'} and \eqref{eqn: symmetry on growth rule}, and the third equality follows from \eqref{eqn: growth on dual P}. This concludes the induction step and hence the proof.
\end{proof}

\begin{proof}[Proof of \Cref{prp: dagger swaps PQ}]
    Let $S,T$ be the sign strings of $P,Q$ respectively. By definition
    \begin{align}
        C^\lambda_{P,Q} = \Fold_T(G(PQ^\vee))=(\id_{S}\otimes \ccap_{T}) \circ (G(PQ^\vee) \otimes \id_{T}).
    \end{align}
    Applying $\Mir$ to both sides and using \cref{lem: growth algorithm symmetry}, we have
    \begin{align}
        \Mir(C^\lambda_{P,Q}) =& (\ccap_{T} \otimes \id_{S^\vee}) \circ (\id_{T^\vee} \otimes G(QP^\vee)).
    \end{align}
    Apply $\vee$ to both sides using the realization in \cref{cor: equivalent duality}; we get
    \begin{align}
        (\Mir(C^\lambda_{P,Q}))^\vee = (\id_{T} \otimes \ccap_{T} \otimes \ccap_{S}) \circ (\id_{T} \otimes \id_{T^\vee} \otimes G(QP^\vee) \otimes \id_{S}) \circ (\ccup_{T} \otimes \id_{S}).
    \end{align}
Applying the zig-zag relation, we get
\begin{align}
    (\Mir(C^\lambda_{P,Q}))^\vee = (\id_{T} \otimes \ccap_{S}) \circ ( G(QP^\vee) \otimes \id_{S}) = C^\lambda_{Q,P},
\end{align}
where the last equality follows from the definition of $C^\lambda_{Q,P}$. We conclude the proof by \cref{lem: elementary geometry}.
\end{proof}

\subsection{Clasped web spaces}\label{sec: cell modules}
Recall that $\Sigma_L$ denotes the set of sign strings of length $L$. For a sign string $S$ of arbitrary length, set
\begin{align}\label{eqn: web module}
    \Web^L_S:=\bigsqcup_{T\in\Sigma_L}\Web^T_S,
    \qquad
    \W^L_S:=\bigoplus_{T\in\Sigma_L}\W^T_S
    =\Span_{\C(q)}\Web^L_S.
\end{align}
Postcomposition makes $\W^L_S$ a left $\mathcal K_L$-module. For sign strings $S,T$ and any weight $\lambda\in\Lambda_+$, define
\begin{align*}
    \W^T_S(\prec\lambda)
    &:=\Span_{\C(q)}\{uv:v\in\W^M_S,\ u\in\W^T_M, M \text{ is a sign string}, \wt(M)\prec\lambda\}.
\end{align*}
Thus $\W^T_S(\prec\lambda)$ consists of linear combinations of webs that factor through an intermediate boundary $M$ of weight strictly below $\lambda$.
Define the following subspace of $\W^L_S$:
\begin{align}\label{eqn: lower weight submodule}
    \W^L_{\prec S}:=\bigoplus_{T\in\Sigma_L}\W^T_S(\prec\wt(S)).
\end{align}

Fix $\lambda\in \Lambda_+$. The next lemma shows that the subcategory of $\Sp$ with the same objects, and $\Hom$-spaces defined by $\Hom(S,T):=\W^T_S(\prec \lambda)$ is a categorical ideal:

\begin{lemma}\label{lem: ideal property}
    Let $\lambda\in \Lambda_+$. Let $S,T,U,V$ be sign strings. Then $\W^U_T \circ \W^T_S(\prec \lambda) \circ \W^S_V \subset \W^U_V(\prec \lambda)$.
\end{lemma}
\begin{proof}
    By linearity, it suffices to take $w=xy\in\W^T_S(\prec\lambda)$, where $x\in\W^T_M$, $y\in\W^M_S$ and $\wt(M)\prec\lambda$. For $a\in\W^U_T$ and $b\in\W^S_V$, we have $awb=(ax)(yb)$ with $ax\in\W^U_M$ and $yb\in\W^M_V$. This factorization shows that $awb\in\W^U_V(\prec\lambda)$.
\end{proof}

\begin{definition}\label{def: clasped web space}
    Let $S$ be a sign string. The \emph{clasped web space} with bottom sign string $S$ is 
    \begin{align}\label{eqn: clasped quotient}
        \W^L_{[S]}:=\W^L_S/\W^L_{\prec S}.
    \end{align}
    For $w\in\W^L_S$, write $[w]:=w+\W^L_{\prec S}$ for its class in this quotient.
\end{definition}

\begin{corollary}
    For each sign string $S$, the clasped web space $\W^L_{[S]}$ is a left $\mathcal K_L$-module.
\end{corollary}
\begin{proof}
    Since $\W^L_S$ is a left $\mathcal K_L$-module and $\W^L_{\prec S}$ is its subspace, it suffices to show that $\W^L_{\prec S}$ is preserved by the left $\mathcal K_L$-action. By \cref{lem: ideal property}, the left $\mathcal K_L$-action on $\W^L_S$ indeed preserves $\W^L_{\prec S}$.
\end{proof}

\begin{remark}
    Let $T$ and $S$ be two sign strings. The term ``clasped web space'' was introduced in \cite{kuperberg_1996_spiders} to describe the image of a special idempotent called the clasp (a rank-two analogue of the Jones-Wenzl projector) on the space $\W^T_S$. In fact, this image can be identified with $\W^T_{[S]}:=\W^T_S/\W^T_S(\prec\lambda)$, but we omit the proof.
\end{remark}

\begin{lemma}\label{lem: isomorphism by H}
    Let $S=(s_1,\dots,s_n)$ where $s_i \neq s_{i+1}$ for some $i$, and let 
    $$\sigma_i(S)=(s_1,\dots,s_{i-1},s_{i+1},s_i,s_{i+2},\dots,s_n)$$
    be the sign string with the $i$-th and $(i+1)$-th entries permuted. Then the map 
    \begin{align}
        \theta_i^S: \W^L_S & \to \W^L_{\sigma_i(S)} \\
        w & \mapsto w \circ (\id_{(s_1,\dots,s_{i-1})} \otimes H^{s_{i+1}} \otimes \id_{(s_{i+2},\dots,s_n)})
    \end{align}
    is a morphism of $\mathcal{K}_L$-modules, where the webs $H^\pm$ are introduced in \Cref{fig: generating set}. Moreover, it descends to an isomorphism $\W^L_{[S]} \to \W^L_{[\sigma_i(S)]}$ of $\mathcal{K}_L$-modules.
\end{lemma}
\begin{proof}
    Precomposition naturally commutes with the $\mathcal{K}_L$-action on $\W^L_S$ because the latter is defined by postcomposition, so $\theta_i^S$ is a $\mathcal{K}_L$-module morphism. Set $\lambda=\wt(S)=\wt(\sigma_i(S))$. By \cref{lem: ideal property}
    $\theta_i^S(\W_S^{L}(\prec \lambda))\subset
    \W_{\sigma_i(S)}^{L}(\prec \lambda)$,
    so $\theta_i^S$ descends to a $\mathcal{K}_L$-morphism
    $\W^L_{[S]} \to \W^L_{[\sigma_i(S)]}$. It remains to check that it
    is invertible, to this end we have:
    \begin{align}\label{eqn: involutive property of theta}
        \theta_i^{\sigma_i(S)}(\theta_i^S(w)) & = w \circ (\id_{(s_1,\dots,s_{i-1})} \otimes (H^{s_{i+1}}H^{s_{i}}) \otimes \id_{(s_{i+2},\dots,s_n)}) \notag \\
        & = w + w \circ (\id_{(s_1,\dots,s_{i-1})} \otimes \TL{s_i}{-s_i} \otimes \id_{(s_{i+2},\dots,s_n)}),
    \end{align}
where the first equality is by definition of $\theta_i^S$, and the second by the Kuperberg square relation $H^{s_{i+1}}H^{s_i} = \id_{s_i} + \TL{s_i}{-s_i}$. Since the operator $\TL{s_i}{s_{i+1}}$ factorizes into $\ccup_{s_i} \circ \ccap_{s_{i+1}}$, it is straightforward to see that the second term of \eqref{eqn: involutive property of theta} factorizes through a sign string
\begin{align*}
    T=(s_1,\dots,s_{i-1},s_{i+2},\dots,s_n),
\end{align*}
where the $i$-th and $(i+1)$-th entries are removed. By definition \eqref{eqn: partial ordering} of the partial ordering on weights, we have $\wt(T) \prec \wt(S) = \lambda$. Hence, the second term of \eqref{eqn: involutive property of theta} sits in the kernel $\W^L_S(\prec \lambda)$ of the quotient map $\W^L_S \to \W^L_{[S]}$. Therefore, $\theta_i^{\sigma_i(S)}\theta_i^S=\id$ on $\W^L_{[S]}$.
\end{proof}

\begin{corollary}\label{cor: most clasped modules are isomorphic}
    Let $S,T$ be sign strings such that $\wt(S)=\wt(T)$. Then $\W^L_{[S]}\cong \W^L_{[T]}$ as $\mathcal K_L$-modules.
\end{corollary}
\begin{proof}
    Note that $S$ and $T$ in this case have the same number of plus and minus signs, hence they are related by a sequence of permutations. We thus apply the isomorphism in \cref{lem: isomorphism by H} accordingly to obtain the claimed isomorphism.
\end{proof}

In the remainder of the subsection, we determine a basis of $\W^L_{[S]}$. 

\begin{lemma}\label{lem: basis decomposition}
    Let $P\in \Path_T^+(\mu)$ and $Q\in \Path_S^+(\mu)$. Let $Z$ be the bottom sign string of $G(P)$. Then $C^\mu_{P,Q} = G(P) \circ C^\mu_{P_Z^{\mathrm{hw}},Q}$, where $P_Z^{\mathrm{hw}}=(Z,(1,\dots,1))$.
\end{lemma}
\begin{proof}

By definition $C^\mu_{P,Q}=\Fold_S(G(PQ^\vee))$, and the proof is illustrated in \cref{fig: cellular basis decomposition proof}.

\begin{figure}[htbp]
    \centering
    \begingroup
    \scalebox{0.8}{\begin{tikzpicture}
        \begin{scope}[shift={(0,0)}]
            \draw[thick, dashed] (-1.2, 0) rectangle (1.2, 1.2);
            \node at (0, 0.6) {$G(PQ^\vee)$};

            \draw (-0.6, 1.2) -- (-0.6, 2.6);
            \node[above] at (-0.6, 2.6) {$S_P$};

            \draw (0.6, 1.2) -- (0.6, 2.0) to[out=90, in=90, looseness=1.5] (2.2, 2.0) -- (2.2, -0.5);
            \node[above] at (1.4, 2.7) {$\ccap_{S_Q}$};
            \node[below] at (2.2, -0.5) {$S_Q$};

            \node at (0.5, -1.8) {$\Fold_{S_Q}(G(PQ^\vee))$};
        \end{scope}

        \node at (2.8, 1.2) {$=$};

        \begin{scope}[shift={(4.5,0)}]
            \draw[thick, dashed] (-1.2, 0) rectangle (1.2, 0.8);
            \node at (0, 0.4) {$G(P_Z^{\mathrm{hw}} Q^\vee)$};

            \draw[thick, dashed] (-1.2, 1.4) rectangle (0, 2.4);
            \node at (-0.6, 1.9) {$G(P)$};

            \draw (-0.6, 0.8) -- (-0.6, 1.4);

            \draw (-0.6, 2.4) -- (-0.6, 3.4);
            \node[above] at (-0.6, 3.4) {$S_P$};

            \draw (0.6, 0.8) -- (0.6, 2.8) to[out=90, in=90, looseness=1.5] (2.2, 2.8) -- (2.2, -0.5);
            \node[right] at (0.6, 1.8) {$\id_{S_Q^\vee}$};
            \node[above] at (1.4, 3.5) {$\ccap_{S_Q}$};
            \node[below] at (2.2, -0.5) {$S_Q$};

            \node at (0.5, -1.8) {Apply \cref{lem: web in web}};
        \end{scope}

        \node at (7.6, 1.2) {$=$};

        \begin{scope}[shift={(9.6,0)}]
            \draw[thick, dashed] (-1.2, 0) rectangle (1.2, 0.8);
            \node at (0, 0.4) {$G(P_Z^{\mathrm{hw}} Q^\vee)$};

            \draw[thick, dashed] (-1.2, 2.0) rectangle (0, 3.0);
            \node at (-0.6, 2.5) {$G(P)$};

            \draw (-0.6, 0.8) -- (-0.6, 2.0);

            \draw (-0.6, 3.0) -- (-0.6, 3.8);
            \node[above] at (-0.6, 3.8) {$S_P$};

            \draw (0.6, 0.8) -- (0.6, 1.1) to[out=90, in=90, looseness=1.5] (2.2, 1.1) -- (2.2, -0.5);
            \node[above] at (1.4, 1.6) {$\ccap_{S_Q}$};
            \node[below] at (2.2, -0.5) {$S_Q$};

            \node at (0.5, -1.8) {Isotopy};
        \end{scope}

        \node at (12.6, 1.2) {$=$};

        \begin{scope}[shift={(14.8,0)}]
            \draw[thick, dashed] (-1.2, 0) rectangle (1.2, 1.0);
            \node at (0, 0.4) {$C^\mu_{P_Z^{\mathrm{hw}}, Q}$};

            \draw (0, -0.5) -- (0, 0);
            \node[below] at (0, -0.5) {$S_Q$};

            \draw[thick, dashed] (-1.2, 1.4) rectangle (1.2, 2.4);
            \node at (0, 1.9) {$G(P)$};

            \draw (0, 1.0) -- (0, 1.4);

            \draw (0, 2.4) -- (0, 3.4);
            \node[above] at (0, 3.4) {$S_P$};

            \node at (0, -1.8) {$G(P) \circ C^\mu_{P_Z^{\mathrm{hw}},Q}$};
        \end{scope}

    \end{tikzpicture}}
    \endgroup
    \caption{Proof of $C^\mu_{P,Q} = G(P) \circ C^\mu_{P_Z^{\mathrm{hw}},Q}$.}
    \label{fig: cellular basis decomposition proof}
\end{figure}
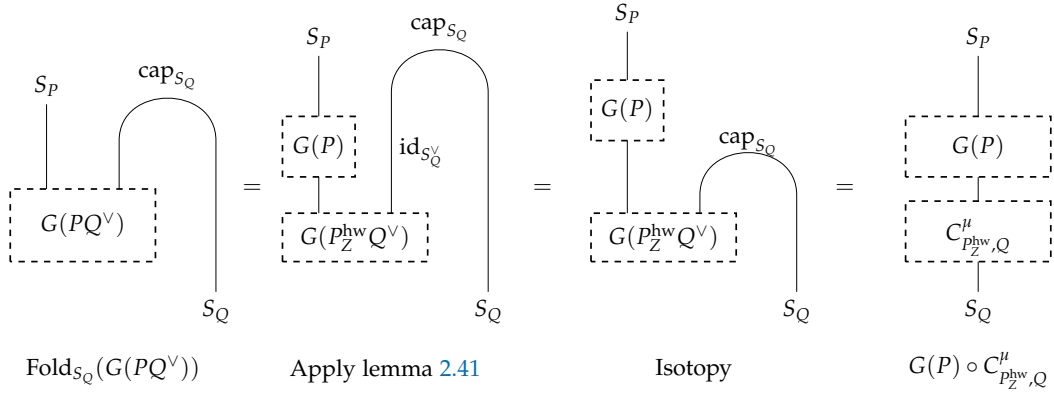
\end{proof}
Recall the equivalence functor $\Gamma$ from \cref{sec: equivalence Sp to Rep} and $\Unfold$ from \cref{def: folding and duality}.
\begin{lemma}\label{lem: lower-weight tensor support}
    Let $S,T$ be sign strings, $\lambda\in\Lambda$, and
    $w\in\W^T_S(\prec\lambda)$. Write
    \begin{align}\label{eqn: unfold expansion}
        \Gamma(\Unfold_S(w))
        =\sum_{\substack{A\in\Path_T\\
                        B\in\Path_{S^\vee}}}
          \kappa_{A,B}\,v_A\otimes v_B,
    \end{align}
    where $v_A\in V^T$ and $v_B\in V^{S\vee}$ are defined in \eqref{def(eqn): vP}. Then
    \begin{align}\label{eqn: lower-weight tensor support}
        \kappa_{A,B}=0
        \qquad\text{unless}\qquad
        \wt(A)\prec\lambda.
    \end{align}
\end{lemma}

\begin{proof}
    By the definition of $\W^T_S(\prec\lambda)$ and
    linearity, it suffices to consider
    \[
        w=x\circ y,\qquad
        y\in\W^M_S,\quad x\in\W^T_M,\quad
        \wt(M)\prec\lambda.
    \]
    By \cref{lem: inserting zigzag into a product},
    \begin{align}\label{eqn: unfold uv}
        \Gamma(\Unfold_S(w))
        =(\Gamma(x)\otimes\Gamma(y^\vee))
          \Gamma(\ccup_M).
    \end{align}
    Since $\ccup_M$ is a nested family of $\ccup_\pm$, and $\Gamma(\ccup_\pm)$ is the coevaluation $\C(q)\to V^\pm\otimes V^\mp$ we can write
    \[
        \Gamma(\ccup_M)
        =\sum_{D\in\Path_M}v_D\otimes z_D,
        \qquad z_D\in V^{M^\vee},
    \]
    where $v_D$ is from \eqref{def(eqn): vP}. Consequently,
    \[
        \Gamma(\Unfold_S(w))
        =\sum_{D\in\Path_M}
          \Gamma(x)(v_D)\otimes\Gamma(y^\vee)(z_D).
    \]
    Since $\Gamma(x):V^M\to V^T$ is an intertwiner,
    it preserves $\mathfrak{sl}_3$-weight spaces. Thus each nonzero
    $\Gamma(x)(v_D)$ has the $\mathfrak{sl}_3$-weight equal $\wt(D)$ satisfying
    \[
        \wt(D)\preceq\wt(M)\prec\lambda,
    \]
    where the first inequality follows from
    \cref{lem: weight domination}.
    Expanding these images in the tensor basis of
    $V^T$ proves \eqref{eqn: lower-weight tensor support} and concludes the proof.
\end{proof}

\begin{proposition}\label{prp: lower-weight web basis}
    Let $S,T$ be sign strings and let $\lambda\in\Lambda_+$. A basis of $\W^T_S(\prec\lambda)$ from \eqref{eqn: lower weight submodule} is
    \begin{align*}
        \Web^T_S(\prec\lambda):=\{C^\mu_{P,Q}:P\in\Path_T^+(\mu),\ Q\in\Path_S^+(\mu),\ \mu\prec\lambda\}.
    \end{align*}
\end{proposition}
\begin{proof}
    Every web $C^\mu_{P,Q}\in\Web^T_S(\prec\lambda)$ belongs to $\W^T_S(\prec\lambda)$. Indeed, \cref{lem: basis decomposition} gives a factorization $C^\mu_{P,Q}=G(P)\circ C^\mu_{P_Z^{\mathrm{hw}},Q}$, where $Z$ is the bottom boundary of $G(P)$. By \cref{lem: growth preserves weight}, $\wt(Z)=\wt(P)=\mu\prec\lambda$. This is a factorization of the form defining $\W^T_S(\prec\lambda)$. The set $\Web^T_S(\prec\lambda)$ is linearly independent in $\W^T_S(\prec\lambda)$, because it is a subset of the monomial web basis $\Web^T_S$ from \cref{lem: basic structure theorem}.

    Now we prove spanning. Let $w\in\W^T_S(\prec\lambda)\subset \W^T_S$, and expand it in basis $W^T_S$ from \cref{lem: basic structure theorem}:
\begin{align}\label{eqn: expansion of w}
    w =
    \sum_{\substack{\mu\in\Lambda_+\\ \mu\prec\lambda}}
    \sum_{\substack{P\in\Path_T^+(\mu)\\
                    Q\in\Path_S^+(\mu)}}
    \gamma_{P,Q}C^\mu_{P,Q}
    \quad+
    \sum_{\substack{\eta\in\Lambda_+\\ \eta\not\prec\lambda}}
    \sum_{\substack{P\in\Path_T^+(\eta)\\
                    Q\in\Path_S^+(\eta)}}
    \varepsilon_{P,Q}C^\eta_{P,Q}.
\end{align}
    Here $\eta \not\prec \lambda$ means either $\lambda\preceq\eta$, or they are incomparable.
    \textbf{Claim:} $\varepsilon_{P,Q}=0$ for all $(P,Q)$ in the second sum in \eqref{eqn: expansion of w}. This claim proves spanning.

    By \cref{def: pairing map C,lem: folding is invertible}, for all $(\eta,P,Q)$ appearing in \eqref{eqn: expansion of w}, $\Unfold_S(C^\eta_{P,Q})=G(PQ^\vee)$. By \cref{lem: PQ-dual is an excursion}, $PQ^\vee\in \Path^+_{TS^\vee}(0)$, hence \cref{thm: unitriangularity} applied to $\Gamma(\Unfold_S(C^\eta_{P,Q}))$ gives
    \begin{align}\label{triangularity of summand}
        \Gamma(\Unfold_S(C^\eta_{P,Q}))
        =v_P\otimes v_{Q^\vee}
        +\sum_{\substack{R\in\Path_{TS^\vee}(0)\\
                        R<PQ^\vee}}
          \alpha_Rv_R,
    \end{align}
    where $\alpha_R\in \N[q,q^{-1}]$ are scalars. We prove the claim by contradiction. Assume there are coefficients $\varepsilon_{P,Q}\neq0$. Among all such pairs $(P,Q)$, choose $(P_0,Q_0)$ for which $P_0Q_0^\vee$ is largest in the lexicographical order
    of \cref{def: order on paths}. Since the lexicographical order is a total order on $\Path^+_{TS^\vee}(0)$, there is a unique choice of $(P_0,Q_0)$.

    Apply $\Gamma\circ\Unfold_S$ to \eqref{eqn: expansion of w}, we get an identity in $\Inv_{\uqslt}(V^T\otimes V^{S^\vee})$:
    \begin{align}\label{eqn: unfolded expansion of w}
    \Gamma\circ\Unfold_S(w) =
    \sum_{\substack{\mu\in\Lambda_+\\ \mu\prec\lambda}}
    \sum_{\substack{P\in\Path_T^+(\mu)\\
                    Q\in\Path_S^+(\mu)}}
    \gamma_{P,Q}\Gamma\circ\Unfold_S(C^\mu_{P,Q})
    \quad+
    \sum_{\substack{\eta\in\Lambda_+\\ \eta\not\prec\lambda}}
    \sum_{\substack{P\in\Path_T^+(\eta)\\
                    Q\in\Path_S^+(\eta)}}
    \varepsilon_{P,Q}\Gamma\circ\Unfold_S(C^\eta_{P,Q}).
\end{align}
    We consider the coefficient of $v_{P_0}\otimes v_{Q_0^\vee}\in V^T\otimes V^{S^\vee}$ on both sides: 
    \begin{itemize}
        \item On the right hand side, this coefficient is exactly $\varepsilon_{P_0,Q_0}$. Indeed, the first sum in \eqref{eqn: unfolded expansion of w} contributes zero to the coefficient, because $C^\mu_{P,Q}\in \W^T_S(\prec\lambda)$ as we have shown, and \cref{lem: lower-weight tensor support} shows that $\Gamma\circ\Unfold_S(C^\mu_{P,Q})$ contributes zero to $v_{P_0}\otimes v_{Q_0^\vee}$ since $\wt(P_0)\not\prec\lambda$. Maximality of $P_0Q_0^\vee$ and \eqref{triangularity of summand} then show that the second sum contributes $\varepsilon_{P_0,Q_0}\neq0$ to the coefficient.

        \item On the left hand side, since $w\in \W^T_S(\prec \lambda)$ and $\wt(P_0)\not\prec\lambda$, \cref{lem: lower-weight tensor support} shows that this coefficient is zero. This contradicts equation \eqref{eqn: unfolded expansion of w}.
    \end{itemize}
    Finally, the contradiction shows that the claim holds, which gives spanning.
\end{proof}

Recall that $P^{\mathrm{hw}}_S$ is defined by \eqref{eqn: highest-weight path} and $\Path^+_L(\lambda)$ by \cref{def: lattice path}. 
\begin{proposition}\label{thm: basis of cell modules}
    Let $S$ be a sign string with weight $\lambda=\wt(S)$.  A basis of $\W^L_{[S]}$ defined in \eqref{eqn: clasped quotient} is
    \begin{align}\label{eqn: clasped web basis}
        \{[C^\lambda_{P,P_S^{\mathrm{hw}}}]:P\in\Path_L^+(\lambda)\},
    \end{align}
    where $C^\lambda_{P,P^{\mathrm{hw}}_S}\in \W^L_{S}$ is from \cref{def: pairing map C}.
\end{proposition}
\begin{proof}
    The basis of $\W^L_S$ is $\{C^\mu_{P,Q}: \mu \in \Lambda_+(L), \mu \preceq \wt(S), P\in\Path_L^+(\mu),Q\in\Path_S^+(\mu)\}$. Here the constraint $\mu\preceq \wt(S)$ follows from \cref{lem: weight domination}, and guarantees the existence of the path $Q$. By \cref{prp: lower-weight web basis} with $\lambda=\wt(S)$, $\W^L_{\prec S}$ has a basis $\{C^\mu_{P,Q}: \mu \in \Lambda_+(L), \mu \prec \lambda, P\in\Path_L^+(\mu),Q\in\Path_S^+(\mu)\}$. For every remaining element, \cref{lem: weight domination} gives $\mu\preceq\wt(S)=\lambda$, so $\mu=\lambda$ and $Q=P_S^{\mathrm{hw}}$. The remaining basis classes are therefore \eqref{eqn: clasped web basis}.
\end{proof}

For each sign string $S$ with $\wt(S)=\lambda$, \cref{thm: basis of cell modules} gives a basis for the $\mathcal K_L$-module $\W^L_{[S]}$. Under this basis, for every $a\in\mathcal K_L$ and $P\in\Path_L^+(\lambda)$, we define the coefficients $r_{a,S}(P,P')\in\C(q)$ that satisfy
\begin{align}\label{eqn: boundary dependent action coefficients}
    aC^\lambda_{P,P_S^{\mathrm{hw}}}
    \equiv
    \sum_{P'\in\Path_L^+(\lambda)}
    r_{a,S}(P,P')C^\lambda_{P',P_S^{\mathrm{hw}}}
    \mod\W^L_{\prec S}.
\end{align}

\subsection{Action coefficients}\label{sec: alternative approach to cell modules}

The goal of this subsection is to show that the action coefficients in \eqref{eqn: boundary dependent action coefficients} do not depend on the sign string $S$. Hence we define
\begin{align}\label{eqn: boundary independent action coefficients}
    r_a(P,P'):=r_{a,S}(P,P'),
    \qquad P,P'\in\Path_L^+(\lambda),
\end{align}
where $\lambda=\wt(S)$. It eventually allows the identification of $\W^L_{[S]}$ as cell modules of $\mathcal K_L$.

\begin{definition}\label{def: H^S_T}
    Let $S,T$ be two sign strings such that $\wt(S)=\wt(T)$. We define $\mathcal{H}^S_{T}\in \Web^S_{T}$ by
    \begin{align}
    \mathcal{H}^S_{T} := C^{\wt(S)}_{P_S^{\mathrm{hw}}, P_T^{\mathrm{hw}}}.
    \end{align}
    Here $C^\lambda_{P,Q}$ is defined in \cref{def: pairing map C}, and $P_S^{\mathrm{hw}}$ from \eqref{eqn: highest-weight path}. See \cref{exp: H^S_T} for an example of $\mathcal H^S_T$.
\end{definition}

The main results of this subsection are \cref{thm: explicit form of H-isomorphism,prp: boundary independent action coefficients}:

\begin{proposition}\label{thm: explicit form of H-isomorphism}
    Let $\wt(S)=\wt(T)=\lambda$. There is a $\mathcal K_L$-module isomorphism
    \begin{align}\label{eqn: connector isomorphism}
        \theta_{S,T}:\W^L_{[S]}&\xrightarrow{\ \cong\ }\W^L_{[T]},
        & [w]&\longmapsto[w\mathcal H^S_T].
    \end{align}
    Moreover, for every $P\in\Path_L^+(\lambda)$,
    \begin{align}\label{eqn: H transport on basis}
        \theta_{S,T}\bigl([C^\lambda_{P,P_S^{\mathrm{hw}}}]\bigr)
        =[C^\lambda_{P,P_T^{\mathrm{hw}}}].
    \end{align}
\end{proposition}

\begin{proposition}\label{prp: boundary independent action coefficients}
    Let $\lambda\in\Lambda_+(L)$, and let $S,T$ be sign strings with $\wt(S)=\wt(T)=\lambda$. For every $a\in\mathcal K_L$ and $P,P'\in\Path_L^+(\lambda)$, we have
    \begin{align*}
        r_{a,S}(P,P')=r_{a,T}(P,P').
    \end{align*}
\end{proposition}
This proposition justifies the definition of $r_a$ in \eqref{eqn: boundary independent action coefficients}. We prove \cref{thm: explicit form of H-isomorphism,prp: boundary independent action coefficients} at the end of the subsection, before which we establish several technical lemmata. 

Recall the duality of web $\vee$ from \cref{def: folding and duality}, and define the web
\begin{align}\label{eqn: definition of Lambda web}
    \Ydual^{\mp}:=(Y_{\pm})^\vee.
\end{align}

Recall \cref{def: growth rule} for growth rules, and \cref{def: steps} the definition of the steps $u_{s,j}$, $s\in \pm$ and $j\in \{-1,0,1\}$.
\begin{lemma}\label{lem: E-lemma}
    The triples $(\Ydual^{\pm}, u_{\pm,0}, u_{\mp,-1}u_{\mp,1})$ are growth rules.
\end{lemma}
\begin{proof}
Let $P=P_1 u_{+,0} P_2$ and $P'=P_1 u_{-,-1} u_{-,1} P_2$ be two dominant paths. It is easy to check that $\wt(P)=\wt(P')$. We prove the identity
\begin{align}\label{eqn: attach Lambda}
    G(P)= (\id_{|P_1|}  \otimes \Ydual^{+} \otimes \id_{|P_2|}) G(P').
\end{align}
Let $\lambda=\wt(P_1)=:n\lambda_1+k\lambda_2$. By dominance of $P$, we have 
$\wt(P_1)+u_{+,0}=\lambda+(\lambda_2-\lambda_1)\in \Lambda_+$, hence $n\geq 1$. In particular, the web $G(P_1)$ has a non-empty bottom sign string $S$. By \cref{lem: web in web} it suffices to prove the relation \eqref{eqn: attach Lambda} in the case $P_1=P_S^{\mathrm{hw}}$. Write $S=S's$ where $s$ is the last sign in $S$ and $S'$ is the rest of the sign string. We prove \eqref{eqn: attach Lambda} case by case according to the value of $s$:

    \textbf{Case $s=-$:} Let $P'':=P_{S'}^{\mathrm{hw}} u_{+,0} u_{-,1}P_2$, it follows from the growth algorithm that
    \begin{align*}
        G(P)=G(P_{S'}^{\mathrm{hw}} u_{-,1} u_{+,0} P_2)=& (\id_{S'} \otimes H^+ \otimes \id_{|P_2|}) \circ G(P''),\\
        G(P')=G(P_{S'}^{\mathrm{hw}} u_{-,1} u_{-,-1} u_{-,1} P_2)=& (\id_{S'} \otimes Y_{+} \otimes \id_{|P_2|}) \circ G(P''),
    \end{align*}
    and the desired relation follows from identity \eqref{eqn: identity 1 in E} illustrated in \cref{fig: relations in E lemma}:
    \begin{align}\label{eqn: identity 1 in E}
        H^+ = (\id_- \otimes \Ydual^{+}) \circ (Y_{+} \otimes \id_-).
    \end{align}
    \textbf{Case $s=+$:} Let $P'':=P_{S'}^{\mathrm{hw}} u_{-,1} P_2$, it follows from the growth algorithm that
    \begin{align*}
        G(P)=G(P_{S'}^{\mathrm{hw}} u_{+,1} u_{+,0} P_2)=& (\id_{S'} \otimes Y_{-} \otimes \id_{|P_2|}) \circ G(P''),\\
        G(P')=G(P_{S'}^{\mathrm{hw}} u_{+,1} u_{-,-1} u_{-,1} P_2)=& (\id_{S'} \otimes \ccup_{+} \otimes \id_- \otimes \id_{|P_2|}) \circ G(P'').
    \end{align*}
    so the desired relation follows from identity \eqref{eqn: identity 2 in E} illustrated in \cref{fig: relations in E lemma}:
    \begin{align}\label{eqn: identity 2 in E}
         Y_{-} = (\id_{+} \otimes \Ydual^{+}) \circ (\ccup_{+} \otimes \id_-).
    \end{align}
    \begin{figure}[htbp]
        \centering
        \begingroup
\tikzset{
    style/.style={#1},
    none/.style={},
    mid arrow/.style={postaction={decorate,decoration={
            markings, mark=at position 0.55 with {\arrow{stealth}}
        }}
    },
    webline/.style={line width=0.55pt,line cap=round,line join=round},
    cutline/.style={dashed,gray,line width=0.35pt},
    labelnode/.style={font=\scriptsize,inner sep=0.4pt},
    piecelabel/.style={font=\tiny,inner sep=0.5pt,fill=white}
}
\begin{tikzpicture}[
    x=1cm, y=1cm,
    every node/.style={font=\scriptsize,inner sep=0.6pt}
]


\begin{scope}[shift={(-3.2,0)}]
    \node at (0,1.28) {$H^{+}$};
    \node[labelnode] at (-0.65,1.02) {$-$};
    \node[labelnode] at (0.65,1.02) {$+$};
    \node[labelnode] at (-0.65,-1.02) {$+$};
    \node[labelnode] at (0.65,-1.02) {$-$};

    \coordinate (hy) at (-0.28,0);
    \coordinate (hl) at (0.28,0);

    \draw[style=mid arrow,webline] (-0.65,0.9) -- (hy);
    \draw[style=mid arrow,webline] (-0.65,-0.9) -- (hy);
    \draw[style=mid arrow,webline] (hl) -- (hy);
    \draw[style=mid arrow,webline] (hl) -- (0.65,0.9);
    \draw[style=mid arrow,webline] (hl) -- (0.65,-0.9);
\end{scope}

\node at (-1.35,0) {$=$};

\begin{scope}[shift={(1.55,0)}]
    \node[labelnode] at (-0.8,1.02) {$-$};
    \node[labelnode] at (0.8,1.02) {$+$};
    \node[labelnode] at (-0.8,-1.02) {$+$};
    \node[labelnode] at (0.8,-1.02) {$-$};

    \draw[cutline] (-1.15,0) -- (1.15,0);
    \node[piecelabel] at (0,1.35)
        {$\mathrm{id}_{-}\otimes\Ydual^{+}$};
    \node[piecelabel] at (0,-1.25)
        {$Y_{+}\otimes\mathrm{id}_{-}$};

    \coordinate (b1) at (-0.8,0);
    \coordinate (b2) at (0,0);
    \coordinate (b3) at (0.8,0);
    \coordinate (yplus) at (-0.42,-0.55);
    \coordinate (lambdaH) at (0.42,0.55);

    \node[labelnode] at (-0.8,-0.14) {$-$};
    \node[labelnode] at (0,-0.14) {$-$};
    \node[labelnode] at (0.8,-0.14) {$-$};

    \draw[style=mid arrow,webline] (-0.8,0.9) -- (b1);
    \draw[style=mid arrow,webline] (lambdaH) -- (0.8,0.9);
    \draw[style=mid arrow,webline] (lambdaH) -- (b2);
    \draw[style=mid arrow,webline] (lambdaH) -- (b3);

    \draw[style=mid arrow,webline] (b1) -- (yplus);
    \draw[style=mid arrow,webline] (b2) -- (yplus);
    \draw[style=mid arrow,webline] (-0.8,-0.9) -- (yplus);
    \draw[style=mid arrow,webline] (b3) -- (0.8,-0.9);
\end{scope}


\begin{scope}[shift={(-3.2,-3.3)}]
    \node at (0,1.28) {$Y_{-}$};
    \node[labelnode] at (-0.55,1.02) {$+$};
    \node[labelnode] at (0.55,1.02) {$+$};
    \node[labelnode] at (0,-1.02) {$-$};

    \coordinate (ym) at (0,0);
    \draw[style=mid arrow,webline] (ym) -- (-0.55,0.9);
    \draw[style=mid arrow,webline] (ym) -- (0.55,0.9);
    \draw[style=mid arrow,webline] (ym) -- (0,-0.9);
\end{scope}

\node at (-1.35,-3.3) {$=$};

\begin{scope}[shift={(1.55,-3.3)}]
    \node[labelnode] at (-0.75,1.02) {$+$};
    \node[labelnode] at (0.75,1.02) {$+$};
    \node[labelnode] at (0.75,-1.02) {$-$};

    \draw[cutline] (-1.15,0) -- (1.15,0);
    \node[piecelabel] at (0,1.25)
        {$\mathrm{id}_{+}\otimes\Ydual^{+}$};
    \node[piecelabel] at (0,-1.25)
        {$\mathrm{cup}_{+}\otimes\mathrm{id}_{-}$};

    \coordinate (c1) at (-0.75,0);
    \coordinate (c2) at (0,0);
    \coordinate (c3) at (0.75,0);
    \coordinate (lambdaY) at (0.42,0.55);

    \node[labelnode] at (-0.75,-0.14) {$+$};
    \node[labelnode] at (0,-0.14) {$-$};
    \node[labelnode] at (0.75,-0.14) {$-$};

    \draw[style=mid arrow,webline] (c1) -- (-0.75,0.9);
    \draw[style=mid arrow,webline] (lambdaY) -- (0.75,0.9);
    \draw[style=mid arrow,webline] (lambdaY) -- (c2);
    \draw[style=mid arrow,webline] (lambdaY) -- (c3);

    \draw[style=mid arrow,webline,looseness=1.25,bend left=88] (c2) to (c1);
    \draw[style=mid arrow,webline] (c3) -- (0.75,-0.9);
\end{scope}

\end{tikzpicture}
        \endgroup
        \caption{The identities \eqref{eqn: identity 1 in E} and \eqref{eqn: identity 2 in E}.}
        \label{fig: relations in E lemma}
    \end{figure}
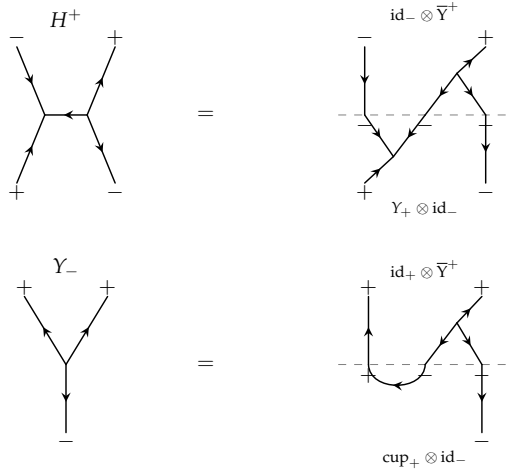
\end{proof}
Recall the webs $I^\pm$ defined in \cref{fig: generating set}. 
\begin{proposition}\label{prp: I-rule}
    The triples $(I^\pm,(u_{\pm,1}u_{\pm,-1}), (u_{\pm,-1}u_{\pm,1}))$ are growth rules.
\end{proposition}
\begin{proof}
    Let $P=P_1u_{s,1}u_{s,-1}P_2$ and $P'=P_1u_{s,-1}u_{s,1}P_2$ be dominant, and set $P''=P_1u_{-s,0}P_2$. Apply the growth rule $(Y_{-s}, (u_{s,1}u_{s,-1}),u_{-s,0})$ from \cref{fig: defining growth rules} and \cref{lem: E-lemma} give, respectively,
    \begin{align}\label{eqn: I-rule 1}
        G(P)&=(\id\otimes Y_{-s}\otimes\id)G(P''),\\
        G(P'')&=(\id\otimes\Ydual^{-s}\otimes\id)G(P').
        \label{eqn: I-rule 2}
    \end{align}
    Substitution and $I^s=Y_{-s}\circ\Ydual^{-s}$ therefore yield
    \begin{align}\label{eqn: attach I}
        G(P)=(\id_{|P_1|}\otimes I^s\otimes\id_{|P_2|})G(P'),
    \end{align}
    as required.
\end{proof}

\begin{lemma}\label{lem: practical C}
    Let $\lambda \in \Lambda_+(L)$, $P,Q\in \Path_L^+(\lambda)$, and $S,T$ be the bottom sign strings of $G(P)$ and $G(Q)$ respectively. Then
    \begin{equation}
        C^\lambda_{P,Q}=G(P) \circ \mathcal{H}^S_T \circ G(Q)^\dagger.
    \end{equation} 
\end{lemma}
\begin{proof}
    Applying \cref{lem: basis decomposition} and \cref{prp: dagger swaps PQ}, we get 
\begin{align*}
        C^\lambda_{P,Q} = G(P) \circ C^{\lambda}_{P_S^{\mathrm{hw}}, Q}=G(P) \circ (C^{\lambda}_{Q,P_S^{\mathrm{hw}}})^\dagger
        =G(P) \circ (G(Q) \circ C^{\lambda}_{{P_T^{\mathrm{hw}},P_S^{\mathrm{hw}}}})^\dagger =G(P) \circ \mathcal{H}^S_T \circ G(Q)^\dagger.
\end{align*}
\end{proof}

\begin{lemma}\label{lem: H^S_T is concatenation of H web}
    The web $\mathcal{H}^S_{T}$ from \cref{def: H^S_T} is a concatenation of $H^\pm$-webs defined in \cref{fig: generating set}.
\end{lemma}
\begin{proof}
    We do induction on $n=|S|=|T|$. For $n=0$, the web is the empty identity; for $n=1$, equal weights force $S=T$ and the web is $\id_S$. Both are products of zero $H^\pm$-webs.

    Suppose $n\geq2$, write $S=(s_1,\ldots,s_n)$ and $T=(t_1,\ldots,t_n)$, and let $M$ be the last occurrence of $t_n$ in $S$. Such an occurrence exists because $S$ and $T$ have the same numbers of each sign. Let $\widehat S$ be the sign string obtained by removing $s_M$ from $S$, and similarly define $\widehat T$ by removing $t_n$ from $T$. In the path 
    \begin{align*}
        P_S^{\mathrm{hw}}(P_T^{\mathrm{hw}})^\vee
        =u_{s_1,1}\cdots u_{s_n,1}
        u_{-t_n,-1}\cdots u_{-t_1,-1},
    \end{align*}
    move $u_{-t_n,-1}$ left past the steps with indices $n,\ldots,M+1$ using \cref{prp: I-rule}, then apply the cup rule at $M$. These steps have sign $-t_n$, by the choice of $M$. The intermediate paths remain dominant: their changed vertices have weight at least $\wt(s_1,\ldots,s_M)-\lambda_{t_n}$ in fundamental-weight coordinates, where $\lambda_+=\lambda_1$ and $\lambda_-=\lambda_2$. This weight is dominant because the prefix contains $s_M=t_n$.

    Fold the $T^\vee$ boundary. The local isotopy $I_{k+1}\circ(\ccup)_k=H_k\circ(\ccup)_{k+1}$ moves the resulting cup to the last position, where folding turns it into $\id_{t_n}$. Here we omit the signs of $I^\pm$, $H^\pm$ and $\ccup_\pm$ for simplicity of notations. Thus
    \begin{align}\label{eqn: induction for H^S_T}
        \mathcal H^S_T
        =H_M H_{M+1}\cdots H_{n-1}
        \circ(\mathcal H^{\widehat S}_{\widehat T}\otimes\id_{t_n}).
    \end{align}
    Here identities on unaffected strands are understood, and each $H_k$ has the orientations required to move the final $t_n$ strand left to position $M$. If $M=n$, the product is empty. The induction hypothesis proves the claim. The same recursion gives $\mathcal H^S_S=\id_S$.
\end{proof}

As a corollary of the induction relation \eqref{eqn: induction for H^S_T}, we get a convenient diagrammatic construction of $\mathcal{H}^S_{T}$:
\begin{corollary}\label{cor: wiring construction for H}
    The web $\mathcal{H}^S_T$ has the following direct diagrammatic construction. Place the entries of $S$ in order along the top boundary and those of $T$ in order along the bottom boundary. For each $\varepsilon\in\{+,-\}$, join the $k$-th occurrence of $\varepsilon$ in $S$ to the $k$-th occurrence of $\varepsilon$ in $T$ by a line segment. Segments carrying the same sign do not cross. Replace every crossing, which is therefore between a $+$-segment and a $-$-segment, by the oriented $H^\pm$-web with the same four local boundary signs. The resulting web is $\mathcal{H}^S_T$.
\end{corollary}

\begin{example}\label{exp: H^S_T}
    Let $S=(+,+,-,-)$ and $T=(-,-,+,+)$. The web $\mathcal H^S_T$ is constructed in the following diagram.
    
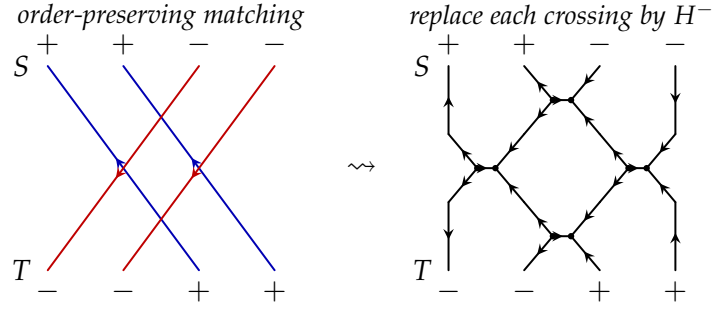
\begin{center}
        \begin{tikzpicture}[x=1cm,y=1cm,line width=0.75pt,line cap=round]
    \def\Hcell#1#2{%
        \begin{scope}[shift={(#1,#2)}]
            \draw[style=mid arrow] (0.38,-0.45) -- (0,0);
            \draw[style=mid arrow] (0.38,-0.45) -- (0,-0.9);
            \draw[style=mid arrow] (0.38,-0.45) -- (0.62,-0.45);
            \draw[style=mid arrow] (1,0) -- (0.62,-0.45);
            \draw[style=mid arrow] (1,-0.9) -- (0.62,-0.45);
            \fill (0.38,-0.45) circle (1.2pt);
            \fill (0.62,-0.45) circle (1.2pt);
        \end{scope}%
    }

    \begin{scope}
        \node at (1.5,3.05) {\small order-preserving matching};
        \node at (-0.35,2.4) {$S$};
        \node at (-0.35,-0.3) {$T$};
        \foreach \x/\s in {0/+,1/+,2/-,3/-}
            \node at (\x,2.68) {$\s$};
        \foreach \x/\s in {0/-,1/-,2/+,3/+}
            \node at (\x,-0.58) {$\s$};

        \draw[blue!70!black,style=mid arrow] (2,-0.3) -- (0,2.4);
        \draw[blue!70!black,style=mid arrow] (3,-0.3) -- (1,2.4);
        \draw[red!75!black,style=mid arrow] (2,2.4) -- (0,-0.3);
        \draw[red!75!black,style=mid arrow] (3,2.4) -- (1,-0.3);
    \end{scope}

    \node at (4.15,1.05) {$\rightsquigarrow$};

    \begin{scope}[xshift=5.3cm]
        \node at (1.5,3.05) {\small replace each crossing by $H^-$};
        \node at (-0.35,2.4) {$S$};
        \node at (-0.35,-0.3) {$T$};
        \foreach \x/\s in {0/+,1/+,2/-,3/-}
            \node at (\x,2.68) {$\s$};
        \foreach \x/\s in {0/-,1/-,2/+,3/+}
            \node at (\x,-0.58) {$\s$};

        \draw[style=mid arrow] (0,1.5) -- (0,2.4);
        \Hcell{1}{2.4}
        \draw[style=mid arrow] (3,2.4) -- (3,1.5);

        \Hcell{0}{1.5}
        \Hcell{2}{1.5}

        \draw[style=mid arrow] (0,0.6) -- (0,-0.3);
        \Hcell{1}{0.6}
        \draw[style=mid arrow] (3,-0.3) -- (3,0.6);
    \end{scope}
\end{tikzpicture}
    \end{center}
\end{example}

\begin{lemma}\label{lem: H connector composition}
    Let $S,T,V$ be sign strings of the same weight $\lambda$ and length $l$. Then
    \begin{align}\label{eqn: composition of H connectors}
        \mathcal H^V_S\mathcal H^S_T-\mathcal H^V_T\in\W^l_{\prec T}.
    \end{align}
\end{lemma}
\begin{proof}
    Set $F=\mathcal H^V_S\mathcal H^S_T$ and $H=\mathcal H^V_T$. By \cref{thm: basis of cell modules,lem: weight domination}, $[H]$ is the unique basis vector of$\W^l_{[T]}$ with top sign string $V$. Hence
    \begin{align}\label{eqn: FH relation}
        F=cH+\ell,\qquad c\in\C(q),\quad\ell\in\W^V_T(\prec\lambda).
    \end{align}
    By \cref{lem: H^S_T is concatenation of H web}, both $F$ and $H$ are concatenations of $H^\pm$-webs. Successively canceling their reflected elementary factors, using the square relation in \eqref{eqn: relations in KL(q)} and \cref{lem: ideal property}, gives
    \begin{align*}
        F^\dagger F\equiv H^\dagger H\equiv\id_T
        \mod\W^T_T(\prec\lambda).
    \end{align*}
    By \eqref{eqn: FH relation}, 
    \begin{align*}
        [\id_T]=[F^\dagger F]=c^2[H^\dagger H]=c^2[\id_T],
    \end{align*}
    where $[w]\in \W^l_{[T]}=\W^l_T/\W^T_T(\prec\lambda)$ as in \cref{def: clasped web space}. Since $\mathcal H^T_T=\id_T$ by \eqref{eqn: induction for H^S_T}, its class is nonzero by \cref{thm: basis of cell modules}; thus $c=\pm1$.

    Furthermore, $c=c(q)$ is a Laurant polynomial in $q$, because the loop, digon, and square reduction coefficients in \cref{fig: Kuperberg relations} are Laurant polynomials in $q$. At $q=-1$, the loop, digon, and square coefficients are $3$, $2$, and $1$, respectively. Therefore $c(-1)\geq0$, excluding $c=-1$. On the other hand, since $c=\pm 1$, we conclude that $c=1$, and
    \begin{align*}
        F-H\in\W^V_T(\prec\lambda)\subseteq\W^l_{\prec T}.
    \end{align*}
\end{proof}

\begin{proof}[Proof of \cref{thm: explicit form of H-isomorphism}]
    By \cref{lem: H^S_T is concatenation of H web}, $\mathcal H^S_T$ is a concatenation of $H^\pm$-webs. By \cref{lem: isomorphism by H}, $\theta_{S,T}$, which is precomposition of $H^\pm$ webs, is an isomorphism $\W^L_{[S]} \to \W^L_{[T]}$. 

    Let $V$ be the bottom boundary of $G(P)$. By \cref{lem: basis decomposition,def: H^S_T},
    \begin{align}\label{eqn: boundary growth representative}
    C^\lambda_{P,P_S^{\mathrm{hw}}}=G(P)\mathcal H^V_S.
    \end{align}
    Consequently, by \cref{lem: H connector composition,lem: ideal property},
    \begin{align*}
        C^\lambda_{P,P_S^{\mathrm{hw}}}\mathcal H^S_T
        -C^\lambda_{P,P_T^{\mathrm{hw}}}
        &=G(P)\bigl(\mathcal H^V_S\mathcal H^S_T-\mathcal H^V_T\bigr)
        \in\W^L_{\prec T},
    \end{align*}
    proving \eqref{eqn: H transport on basis}.
\end{proof}

\begin{proof}[Proof of \Cref{prp: boundary independent action coefficients}]
    Let $S,T$ have weight $\lambda$. By \cref{thm: explicit form of H-isomorphism}, the map $\theta_{S,T}:\W^L_{[S]}\to\W^L_{[T]}$ is $\mathcal K_L$-linear and takes $[C^\lambda_{P,P_S^{\mathrm{hw}}}]$ to $[C^\lambda_{P,P_T^{\mathrm{hw}}}]$ for each $P$. Applying this map to \eqref{eqn: boundary dependent action coefficients} gives
    \begin{align*}
        a[C^\lambda_{P,P_T^{\mathrm{hw}}}]
        &=\theta_{S,T}\bigl(a[C^\lambda_{P,P_S^{\mathrm{hw}}}]\bigr)\\
        &=\sum_{P'\in\Path_L^+(\lambda)}r_{a,S}(P,P')\,
        \theta_{S,T}\bigl([C^\lambda_{P',P_S^{\mathrm{hw}}}]\bigr)\\
        &=\sum_{P'\in\Path_L^+(\lambda)}r_{a,S}(P,P')
        [C^\lambda_{P',P_T^{\mathrm{hw}}}],
    \end{align*}
    hence $r_{a,S}=r_{a,T}$ by definition of the latter in \eqref{eqn: boundary dependent action coefficients}.
\end{proof}

\subsection{Cellular structure of $\mathcal K_L$}\label{sec: cellularity proof}

For $\lambda\in\Lambda_+(L)$, set
\begin{align}\label{eqn: lower-weight algebra ideal}
    \mathcal K_L^{\prec \lambda}
    =\Span_{\C(q)}\{C^\mu_{P,Q}:P,Q\in\Path_L^+(\mu),\ \mu\prec\lambda\}.
\end{align}

Recall $C^\lambda_{P,Q}$ from \cref{def: pairing map C} and the action coefficients $r_a(P,P')$ from \eqref{eqn: boundary independent action coefficients}.
\begin{proposition}\label{prp: factorized multiplication}
    For $a\in\mathcal K_L$, $\lambda\in\Lambda_+(L)$, and $P,Q\in\Path_L^+(\lambda)$, the coefficients $r_a$ in \eqref{eqn: boundary independent action coefficients} satisfy
    \begin{align*}
        aC^\lambda_{P,Q}
        \equiv\sum_{P'\in\Path_L^+(\lambda)}r_a(P,P')C^\lambda_{P',Q}
        \mod\mathcal K_L^{\prec\lambda}.
    \end{align*}
\end{proposition}
\begin{proof}
    Let $U$ be the bottom sign string of $G(Q)$, by \cref{lem: growth preserves weight} $\wt(U)=\lambda$. Apply \cref{prp: dagger swaps PQ} and then \cref{lem: basis decomposition} for every $R\in\Path_L^+(\lambda)$,
    \begin{align}\label{eqn: decomp}
        C^\lambda_{R,Q}=C^\lambda_{R,P_U^{\mathrm{hw}}}G(Q)^\dagger.
    \end{align}
    Using the defining equation \eqref{eqn: boundary dependent action coefficients} for $r_{a,U}(P,P')$ and the equality $r_{a,U}(P,P')=r_a(P,P')$ from \cref{prp: boundary independent action coefficients}, we obtain
    \begin{align*}
        E:=aC^\lambda_{P,P_U^{\mathrm{hw}}}
        -\sum_{P'\in\Path_L^+(\lambda)}r_a(P,P')C^\lambda_{P',P_U^{\mathrm{hw}}}
        \in\W^L_{\prec U}.
    \end{align*}
    Let $S_Q$ be the sign string of $Q$. Since $G(Q)^\dagger\in\W^U_{S_Q}$, \cref{lem: ideal property} and \eqref{eqn: lower weight submodule} give
    \begin{align*}
        EG(Q)^\dagger\in\bigoplus_{T\in\Sigma_L}\W^T_{S_Q}(\prec\lambda)
        \subseteq\mathcal K_L^{\prec\lambda}.
    \end{align*}
    The inclusion follows from \cref{prp: lower-weight web basis}. Together with \eqref{eqn: decomp}, this concludes the proof.
\end{proof}

Now we are ready to give the cell datum of $\mathcal{K}_L$, in the sense of \cref{def: cellular algebra}.

\begin{theorem}\label{thm: cellular structure}
    Let $L\geq 2$. The algebra $\mathcal{K}_L$ is a cellular algebra with cell datum $(\Lambda_+(L),M,C,\dagger)$, where
    \begin{enumerate}
        \item $\Lambda_+(L)$ is defined in \cref{def: LambdaL}, and is partially ordered by  $\preceq$ from \eqref{eqn: partial ordering}.
        \item For each $\lambda\in \Lambda_+(L)$, $M(\lambda)=\Path_L^+(\lambda)$ is the set of dominant lattice paths from \cref{def: lattice path}. 
        \item $C$ is defined in \cref{def: pairing map C}.
        \item $\dagger$ is defined in \cref{def: dagger}.
    \end{enumerate}
\end{theorem}
\begin{proof}
    All properties of a cell datum follow from \cref{prp: structure theorem}, \cref{lem: dagger is anti-involution}, \cref{prp: dagger swaps PQ} and \cref{prp: factorized multiplication}.
\end{proof}

\begin{definition}\label{def: cell modules}
    Let $\lambda\in\Lambda_+(L)$. The cell module $\Delta_L(\lambda)$ is the left $\mathcal K_L$-module with basis $\{v_P:P\in\Path_L^+(\lambda)\}$, and action defined by
    \begin{align}\label{eqn: constructed cell module action}
        a\cdot v_P:=\sum_{P'\in\Path_L^+(\lambda)}r_a(P,P')v_{P'},
        \qquad a\in\mathcal K_L.
    \end{align}
\end{definition}
Since the coefficients $r_a$ are defined as the action coefficients of $\mathcal K_L$ on the clasped web space $\W^L_{[S]}$, see \eqref{eqn: boundary independent action coefficients}, $\Delta_L(\lambda)$ is indeed a $\mathcal K_L$-module, isomorphic to $\W^L_{[S]}$ for any sign string $S$ with $\wt(S)=\lambda$.

\subsection{Proof of bimodule decomposition}\label{sec: proof of bimodule decomposition}
We end this section with the proof of \cref{thm: bimodule decomposition}, which is an application of the cellular structure constructed in the previous subsections.

Recall the $\uqslt$-module $\mathscr{V}_L$ from
\eqref{eqn: def of V_L}.
Let $\Xi_L$ be the set of dominant weights $\mu$ for which $L_q(\mu)$ occurs in $\mathscr V_L$.
As a finite-dimensional semisimple $\uqslt$-module, it has an isotypic decomposition
    \begin{align}\label{eqn: isotypic decomposition for bimodule}
        \mathscr{V}_L
        \cong
        \bigoplus_{\mu\in\Xi_L}
        L_q(\mu)\otimes M_{L,\mu},
        \qquad
        M_{L,\mu}
        :=
        \Hom_{\uqslt}\bigl(L_q(\mu),\mathscr{V}_L\bigr).
    \end{align}
    The isomorphism from the direct sum to $\mathscr V_L$ sends $x\otimes f$ from each direct summand to $f(x)$, where $x\in L_q(\mu)$ and $f\in M_{L,\mu}$.
    By \cref{thm: Schur-Weyl duality}, there is a left $\mathcal{K}_L$-action $\Phi$ on $\mathscr{V}_L$ that centralizes the $\uqslt$-action. Hence, for each $\mu\in\Xi_L$, $\mathcal{K}_L$ acts on $M_{L,\mu}$ by the
    left $\mathcal K_L$-action
    \begin{align}\label{eqn: left action on multiplicity space}
        a\cdot f:=\Phi(a)\circ f.
    \end{align}
    Now fix a dominant weight $\lambda$ and a sign string $S$ with $\wt(S)=\lambda$. The $\uqslt$-module $V^S$ decomposes as a direct sum \cite[Thm.\,10.1.7]{chari_pressley_1994}:
    \begin{align}\label{eqn: top and lower decomposition of VS}
        V^S=L_q(\lambda)\oplus V^S_{\prec\lambda},
    \end{align}
    where $L_q(\lambda)$ is the highest weight summand with multiplicity one, and $V^S_{\prec\lambda}$ is the sum of the lower weight
    components \cite[Chap.\,10,Sec.\,5.2,Prp.\,3]{procesi2006lie}.  Set
    components \cite[Chap.\,10,Sec.\,5.2,Prp.\,3]{procesi2006lie}.  Set
    \begin{align*}
        \mathscr E_S
        :=
        \Hom_{\uqslt}(V^S,\mathscr{V}_L),\quad
        \mathscr E_{S,\prec\lambda}
        :=
        \{f\in\mathscr E_S:
        \left.f\right|_{L_q(\lambda)}=0\}.
    \end{align*}
    By \eqref{eqn: top and lower decomposition of VS}, restriction to $V^S_{\prec \lambda}$
    identifies $\mathscr E_{S,\prec\lambda}$ with
    $\Hom_{\uqslt}(V^S_{\prec\lambda},\mathscr{V}_L)$ as $\mathcal K_L$-left modules, endowed with the $\mathcal K_L$-action \eqref{eqn: left action on multiplicity space}. The subspace $\mathscr E_{S,\prec\lambda}$ is obviously stable under this action, so it defines a left $\mathcal{K}_L$-action on the quotient
    \begin{align*}
        \mathscr E_S/\mathscr E_{S,\prec\lambda} \cong \Hom_{\uqslt}(L_q(\lambda),\mathscr{V}_L),
    \end{align*}
    and therefore induces an
    isomorphism of left $\mathcal{K}_L$-modules
    \begin{align}\label{eqn: multiplicity space as a quotient}
        \mathscr E_S/\mathscr E_{S,\prec\lambda}
        &\xrightarrow{\ \cong\ }M_{L,\lambda},&
        [f]&\longmapsto\left.f\right|_{L_q(\lambda)}.
    \end{align}

    Recall the $\mathcal{K}_L$-module $\W^L_S$ from \eqref{eqn: web module}. We define the $\C(q)$-linear isomorphism $\widetilde{\Gamma}_S$ using the equivalence $\Gamma:\Sp \to \Rep$ and the inclusion $e_T:V^T \hookrightarrow\mathscr{V}_L$ from \eqref{eqn: inclusion-projection}:
    \begin{equation}\label{eqn: Gamma tilde isomorphism}
        \begin{aligned}
        \widetilde\Gamma_S:\W^L_S
        =\bigoplus_{T\in\Sigma_L}\W^T_S
        &\xrightarrow{\ \cong\ }\mathscr E_S, \\
        w\in\W^T_S
        &\longmapsto e_T\circ\Gamma(w).
        \end{aligned}
    \end{equation}
    Moreover, functoriality of $\Gamma$ and the definition of $\Phi$ give
    \begin{align*}
        \widetilde\Gamma_S(a\cdot w)
        =
        \Phi(a)\circ\widetilde\Gamma_S(w)
        =
        a\cdot\widetilde\Gamma_S(w), \quad \forall a\in \mathcal K_L.
    \end{align*}
    Thus $\widetilde\Gamma_S:\W^L_S \to \mathscr{E}_S$ is an isomorphism of left
    $\mathcal{K}_L$-modules. 

    By \cref{def: cell modules}, $\W^L_{[S]}
        =
        \W^L_S/\W^L_{\prec S}
        \cong\Delta_L(\lambda)$, where $\W^L_{\prec S}$ is the $\mathcal{K}_L$-submodule of $\W^L_S$ defined in \eqref{eqn: lower weight submodule}. We state the following lemma, proved at the end of this subsection.
\begin{lemma}\label{lem: lower ideal as Hom}
    Let $S$ be a sign string with $\wt(S)=\lambda$.  The $\mathcal{K}_L$-isomorphism
    $\widetilde\Gamma_S$ maps
    $\W^L_{\prec S}$ onto $\mathscr E_{S,\prec\lambda}$; that is,
    \begin{align}\label{eqn: lower ideal as Hom}
        \widetilde\Gamma_S\bigl(\W^L_{\prec S}\bigr)
        =
        \mathscr E_{S,\prec\lambda}.
    \end{align}
\end{lemma}
    By \cref{lem: lower ideal as Hom},
    $\widetilde\Gamma_S$ descends to an isomorphism from $\W^L_{[S]}$ to $\mathscr E_S/\mathscr E_{S,\prec\lambda}$. Using $\W^L_{[S]}\cong\Delta_L(\lambda)$ and \eqref{eqn: multiplicity space as a quotient}, we obtain an isomorphism of left $\mathcal{K}_L$-modules
    \begin{align}\label{eqn: cell module as multiplicity space}
        \Delta_L(\lambda)
        \cong 
        M_{L,\lambda}.
    \end{align}
    Recall from \cref{def: cell modules} that $\Delta_L(\lambda)$ has a basis indexed by $\Path_L^+(\lambda)$. Hence
    \cref{lem: weight and existence of paths} and
    \eqref{eqn: cell module as multiplicity space} show that
    $M_{L,\lambda}\neq0$ precisely when
    $\lambda\in\Lambda_+(L)$.  Thus $\Xi_L=\Lambda_+(L)$ in \eqref{eqn: isotypic decomposition for bimodule}.

    Finally, recall $\dagger$ from \cref{def: dagger}, we define the right $\mathcal{K}_L$-action on $\mathscr{V}_L$
    by:
    \begin{align*}
        h\cdot a:=a^\dagger\cdot h=\Phi(a^\dagger)h.
    \end{align*}
    Under the isomorphism \eqref{eqn: isotypic decomposition for bimodule}, this replaces each left multiplicity space $M_{L,\lambda}$ by the right module $(M_{L,\lambda})^\dagger$, which is $\Delta_L(\lambda)^\dagger$ by \eqref{eqn: cell module as multiplicity space}. Hence \eqref{eqn: isotypic decomposition for bimodule} yields an isomorphism of $(\uqslt,\mathcal{K}_L)$-bimodules
    \begin{align*}
        \mathscr{V}_L
        \cong
        \bigoplus_{\lambda\in\Lambda_+(L)}
        L_q(\lambda)\otimes(M_{L,\lambda})^\dagger
        \cong
        \bigoplus_{\lambda\in\Lambda_+(L)}
        L_q(\lambda)\otimes\Delta_L(\lambda)^\dagger,
    \end{align*}
    which is \eqref{eqn: Schur-Weyl bimodule decomposition}. This concludes the proof of \cref{thm: bimodule decomposition}.

\begin{proof}[Proof of \Cref{lem: lower ideal as Hom}]
    By \eqref{eqn: lower weight submodule},
    $\W^L_{\prec S}$ is spanned
    by webs which factor through a sign string $R$ with
    $\wt(R)\prec\lambda$.  Under $\widetilde\Gamma_S$, such a web becomes
    a morphism which factors as
    \begin{align*}
        V^S\longrightarrow V^R\longrightarrow\mathscr{V}_L.
    \end{align*}
    Every irreducible summand of $V^R$ has highest weight
    $\nu\preceq\wt(R)\prec\lambda$.  Hence
    \begin{align*}
        \Hom_{\uqslt}(L_q(\lambda),V^R)=0,
    \end{align*}
    so the resulting morphism vanishes on $L_q(\lambda)$.  It therefore
    lies in $\mathscr E_{S,\prec\lambda}$, proving the inclusion
    ``$\subseteq$'' in \eqref{eqn: lower ideal as Hom}.

    To prove the converse inclusion, by additivity it suffices to consider morphisms in $\mathscr E_{S, \prec \lambda}$ of the form
    \begin{align*}
        V^S\xrightarrow{\pi}L_q(\mu)
        \xrightarrow{g}\mathscr{V}_L,
        \qquad \mu\prec\lambda,
    \end{align*}
    where $\pi$ is the projection onto one irreducible summand of
    $V^S_{\prec\lambda}$.  Choose a sign string $R$ with $\wt(R)=\mu$.
    The module $L_q(\mu)$ occurs in $V^R$ as its multiplicity-one
    highest-weight summand, so complete reducibility provides maps
    \begin{align*}
        \iota_\mu:L_q(\mu)\longrightarrow V^R,
        \qquad
        p_\mu:V^R\longrightarrow L_q(\mu),
        \qquad
        p_\mu\circ\iota_\mu=\id_{L_q(\mu)}.
    \end{align*}
    Hence
    \begin{align*}
        g\circ\pi
        =
        (g\circ p_\mu)\circ(\iota_\mu\circ\pi)
        :V^S\longrightarrow V^R\longrightarrow\mathscr{V}_L.
    \end{align*}
    Since $\Gamma:\Sp\to\Rep$ is an equivalence, it maps surjectively onto the $\Hom$-spaces. Hence there exist $w_1\in\W^R_S$ and $w_2\in\W^L_R$ such that
    \begin{align*}
        \Gamma(w_1)=\iota_\mu\circ\pi,
        \qquad
        \widetilde\Gamma_R(w_2)=g\circ p_\mu.
    \end{align*}
    Since $\wt(R)=\mu\prec\lambda$, their composite
    $w_2w_1$ belongs to $\W^L_{\prec S}$ and is mapped to $g\circ\pi$. This proves the reverse inclusion in
    \eqref{eqn: lower ideal as Hom}.
\end{proof}

\section{The determination of generators of $\mathcal K_L$}\label{sec: generators of K_L}

We prove \cref{thm: generators of K_L}. 

\subsection{An outline of the strategy}\label{sec: outline for FFT}

We outline the crucial steps of the proof to \cref{thm: generators of K_L} in this subsection, which are then fulfilled in the subsequent subsections. We define the following paths:
\begin{align*}
    P_c^{\pm}:=u_{\pm,1}u_{\mp,-1}, \qquad P_t^{\pm}:=u_{\pm,1}u_{\pm,0}u_{\pm,-1}, \qquad P^{\pm}_Y:=u_{\mp,1}u_{\mp,0}.
\end{align*}
Clearly the paths are dominant, and through the growth algorithm, they correspond to the following webs:
\begin{center}
    \begin{tikzpicture}[scale=0.9]
        \begin{scope}[shift={(0,0)}]
            \node[above] at (-0.6, 1) {$+$};
            \node[above] at (0.6, 1) {$-$};
            
            \draw[style=mid arrow] (0.6, 1) to[out=-90, in=-90, looseness=1.5] (-0.6, 1);
            
            \node at (0, -0.8) {$G(P_c^+) = \ccup_+$};
        \end{scope}
        
        \begin{scope}[shift={(4.5,0)}]
            \node[above] at (-0.8, 1) {$+$};
            \node[above] at (0, 1) {$+$};
            \node[above] at (0.8, 1) {$+$};
            
            \coordinate (V) at (0, 0.2);
            
            \draw[style=mid arrow] (V) -- (-0.8, 1);
            \draw[style=mid arrow] (V) -- (0, 1);
            \draw[style=mid arrow] (V) -- (0.8, 1);
            
            \node at (0, -0.8) {$G(P_t^+) =: t_+$};
        \end{scope}
        
        \begin{scope}[shift={(9,0)}]
            \node[above] at (-0.6, 1) {$-$};
            \node[above] at (0.6, 1) {$-$};
            
            \node[below] at (0, 0) {$+$};
            
            \coordinate (V2) at (0, 0.6);
            
            \draw[style=mid arrow] (-0.6, 1) -- (V2);
            \draw[style=mid arrow] (0.6, 1) -- (V2);
            \draw[style=mid arrow] (0,0) -- (V2);
            
            \node at (0, -0.8) {$G(P_Y^+) = Y_+$};
        \end{scope}
    \end{tikzpicture}
\end{center}

\begin{definition}\label{def: cyclic path and web}
    Let $S=(s_1,\dots,s_n)$ be a sign string such that $|S|\leq L$. We define the path $P^S_{cyc}\in \Path_L^+(\wt(S))$ as follows:
    \begin{align}\label{eqn: cyclic path}
        P^S_{cyc}:=\begin{aligned}
            \begin{cases}
                (P_c^+)^{(L-|S|)/2}P_S^{\mathrm{hw}} & L-|S|\in 2\N, \\
                P_Y^{s_1} P_{(s_2,\dots,s_n)}^{\mathrm{hw}} & L-|S|=1, \\
                (P_c^+)^{(L-|S|-3)/2} P_t^+ P_S^{\mathrm{hw}}, & L-|S|\in 2\N+3.
            \end{cases}
        \end{aligned}
    \end{align}
    We define the \textup{cyclic web} as the monomial web $w^S_{cyc}:=G(P^S_{cyc})$. 
\end{definition}

Evidently, the cyclic web $w^S_{cyc}$ is given by:
\begin{align}\label{eqn: cyclic web}
        w^S_{cyc}:=\begin{aligned}
            \begin{cases}
                (\ccup_{+})^{\otimes (L-|S|)/2} \otimes \id_{S} & L-|S|\in 2\N, \\
                Y_{s_1} \otimes \id_{(s_2,\dots s_n)} & L-|S|=1, \\
                (\ccup_{+})^{\otimes (L-|S|-3)/2} \otimes t_+ \otimes \id_S, & L-|S|\in 2\N+3.
            \end{cases}
        \end{aligned}
\end{align}

The cyclic webs are so named because they are $\Alg_{\C(q)}(\GenSet_L)$-cyclic vectors in the cell modules. Specifically, we have:

\begin{proposition}\label{prp: cyclic vector is indeed cyclic}
    Let $P\in \Path_L^+(\lambda)$, let $S$ be the bottom sign string of $G(P)$. Then there exists $g\in \Alg_{\C(q)}(\GenSet_L)$ such that
    \begin{align*}
        G(P) = g G(P^S_{cyc}).
    \end{align*}
\end{proposition}

The proof of \cref{prp: cyclic vector is indeed cyclic} takes two steps: First, in \cref{sec: local surgery} we identify many new growth rules. Then in \cref{sec: cyclic property}, using these growth rules we provide an algorithm that determines for any $P\in \Path_L^+(\lambda)$, an element $g_P\in \Alg_{\C(q)}(\GenSet_L)$ such that $G(P)=g G(P^S_{cyc})$. We finally show that the map $C$ glues two cyclic webs into an element in $\Alg_{\C(q)}(\GenSet_L)$:
The proof of \cref{prp: cyclic vector is indeed cyclic} takes two steps: First, in \cref{sec: local surgery} we identify many new growth rules. Then in \cref{sec: cyclic property}, using these growth rules we provide an algorithm that determines for any $P\in \Path_L^+(\lambda)$, an element $g_P\in \Alg_{\C(q)}(\GenSet_L)$ such that $G(P)=g G(P^S_{cyc})$. We finally show that the map $C$ glues two cyclic webs into an element in $\Alg_{\C(q)}(\GenSet_L)$:

\begin{proposition}\label{prp: basic web in the cyclic subalgebra}
    Let $S,T$ be sign strings such that their lengths $|S|\leq L$, $|T|\leq L$, with weights $\lambda=\wt(S)=\wt(T)$. Let $P^S_{cyc}$ and $P^T_{cyc}$ be the cyclic paths associated to $S$ and $T$, respectively. Then $C^\lambda_{P^S_{cyc}, P^T_{cyc}} \in \Alg_{\C(q)}(\GenSet_L)$.
\end{proposition}

Based on \cref{prp: cyclic vector is indeed cyclic} and \cref{prp: basic web in the cyclic subalgebra}, we can prove \cref{thm: generators of K_L}:

\begin{proof}[Proof of \Cref{thm: generators of K_L}]
    Let $C^\lambda_{P,Q}$ be a basis element in $\mathcal{K}_L$, where $\lambda\in \Lambda_+(L)$ and $P,Q\in \Path_L^+(\lambda)$. Let $S$, $T$ be the bottom sign strings of $G(P)$ and $G(Q)$ respectively. By \cref{prp: cyclic vector is indeed cyclic}, there exist $g,h\in \Alg_{\C(q)}(\GenSet_L)$ such that
    \begin{align}\label{eqn: decomposition for GP and GQ}       
        G(P) = g G(P^S_{cyc}), \qquad G(Q) = h G(P^T_{cyc}).
    \end{align}
    Now consider $C^\lambda_{P,Q}$:
    \begin{align*}
        C^\lambda_{P,Q} = C^\lambda_{P,P_T^{\mathrm{hw}}} G(Q)^\dagger =C^\lambda_{P,P_T^{\mathrm{hw}}} (h G(P^T_{cyc}))^\dagger = C^\lambda_{P,P_T^{\mathrm{hw}}} (G(P^T_{cyc}))^\dagger h^\dagger = C^\lambda_{P,P^T_{cyc}} h^\dagger,
    \end{align*}
where the first equality uses \cref{lem: basis decomposition}, the second equality uses \eqref{eqn: decomposition for GP and GQ}, the third equality uses the fact that $\dagger$ is an anti-involution, and the last equality uses \cref{lem: basis decomposition} again. A similar treatment of $P$ gives
    \begin{align*}
        C^\lambda_{P,Q} =g C^\lambda_{P^{S}_{cyc},P^T_{cyc}} h^\dagger \in \Alg_{\C(q)}(\GenSet_L),
    \end{align*}
where the last inclusion uses \cref{prp: basic web in the cyclic subalgebra}, and the fact that $\Alg_{\C(q)}(\GenSet_L)$ is closed under the anti-involution $\dagger$.
\end{proof}

\subsection{Some growth rules}\label{sec: local surgery}

Recall \cref{def: growth algorithm} for the growth algorithm and \cref{def: growth rule} of (general) growth rules. 

\begin{lemma}\label{lem: arc reversal}
    For each sign $s\in\{+,-\}$, the triple $(-[2]_q^{-1}H^{-s},P_c^s,P_c^{-s})$ is a growth rule. Equivalently, if the paths
    \begin{align*}
        P=P_1 P_c^\pm P_2,\qquad P'=P_1 P_c^\mp P_2
    \end{align*}
    are dominant, for some paths $P_1$ and $P_2$, then $G(P)=g G(P')$ for $g=-[2]_q^{-1}H_{|P_1|+1}^{\mp}\in \Alg_{\C(q)}(\GenSet_L)$.
\end{lemma}
\begin{proof}
Due to \cref{prp: growth algorithm is well-defined}(2),  
it follows from the growth algorithm that
    \begin{align*}
        G(P) & = (\id_{|P_1|} \otimes \ccup_{\pm} \otimes \id_{|P_2|}) \circ G(P_1P_2),\\
        G(P') & = (\id_{|P_1|} \otimes \ccup_{\mp} \otimes \id_{|P_2|}) \circ G(P_1P_2).
    \end{align*}
    On the other hand, by the Kuperberg digon relation we have $\ccup_{\pm}=-[2]_q^{-1}H^\mp \ccup_{\mp}$, illustrated in \cref{fig: arc-reversal}.
    \begin{figure}[htbp]
        \centering
        \scalebox{0.95}{\begin{tikzpicture}
	\begin{pgfonlayer}{nodelayer}
		\node [style=none] (9) at (1, -1) {};
		\node [style=none] (10) at (3, -1) {};
		\node [style=none] (11) at (1, 1) {};
		\node [style=none] (12) at (1.5, 0) {};
		\node [style=none] (13) at (2.5, 0) {};
		\node [style=none] (14) at (3, 1) {};
		\node [style=none] (15) at (0, 0) {$=-[2]_q^{-1}$};
		\node [style=none] (16) at (-3, -1) {};
		\node [style=none] (17) at (-1, -1) {};
		\node [style=none] (20) at (-3, 1) {};
		\node [style=none] (21) at (-1, 1) {};
		\node [style=none] (22) at (-5, -1) {};
		\node [style=none] (23) at (5, -1) {};
	\end{pgfonlayer}
	\begin{pgfonlayer}{edgelayer}
		\draw [style=mid arrow, bend right=90, looseness=1.25] (9.center) to (10.center);
		\draw [style=mid arrow] (10.center) to (13.center);
		\draw [style=mid arrow] (12.center) to (13.center);
		\draw [style=mid arrow] (14.center) to (13.center);
		\draw [style=mid arrow] (12.center) to (11.center);
		\draw [style=mid arrow] (12.center) to (9.center);
		\draw [style=mid arrow, bend left=90, looseness=1.25] (17.center) to (16.center);
		\draw [style=mid arrow] (21.center) to (17.center);
		\draw [style=mid arrow] (16.center) to (20.center);
		\draw[gray!65,densely dotted,line width=0.6pt]
			(0.65,-1) -- (3.35,-1);
	\end{pgfonlayer}
\end{tikzpicture}}
        \caption{Diagram identity $\ccup_+=-[2]_q^{-1}H^-\ccup_-$.}
        \label{fig: arc-reversal}
    \end{figure}
    Combining the above identities gives the desired result.
\end{proof}

The cyclic webs have $\ccup$ webs as their leftmost parts. The next lemma shows that, if a general web $G(P)$ contains $\ccup$ as a piece, then we can move the $\ccup$ to the leftmost part of $G(P)$, at the cost of an extra factor in $\Alg_{\C(q)}(\GenSet_L)$.

\begin{lemma}\label{lem: arc switching}
Let $s_1,s_2$ be signs, and $u=u_{s_1,j}$ be a step from \cref{def: steps}. Then there exists a web $g_{s_1,s_2}\in\Alg_{\C(q)}(\GenSet_3)$ such that $(g_{s_1,s_2},uP_c^{s_2},P_c^{s_2}u)$ is a growth rule. Equivalently, if the paths
    \begin{align*}
        P=P_1 u P_c^{s_2} P_2,\qquad P'=P_1 P_c^{s_2} u P_2
    \end{align*}
    are dominant, for some paths $P_1$ and $P_2$, then
    \begin{align*}
        G(P)= (\id_{|P_1|} \otimes g_{s_1,s_2} \otimes \id_{|P_2|}) G(P').
    \end{align*}
\end{lemma}
\begin{proof}
    We may assume without loss of generality that $s_1=s_2$. Indeed, if $s_1=-s_2$, we apply \cref{lem: arc reversal}. Now write $s=s_1=s_2$. By the growth algorithm, we have
    \begin{align*}
        G(P) & = (\id_{|P_1|} \otimes \id_{s}\otimes  \ccup_{s} \otimes \id_{|P_2|}) \circ G(P_1uP_2),\\
        G(P') & = (\id_{|P_1|} \otimes \ccup_{s} \otimes \id_{s} \otimes \id_{|P_2|}) \circ G(P_1uP_2).
    \end{align*}
    Let $g_{s,s}=\TL{s}{-s}_{2}$, where $\TL{s}{-s}$ is defined in \cref{fig: auxiliary basic webs}. The desired identity follows from
    \begin{align}\label{eqn: arc switching}
        \id_{s}\otimes\ccup_{s} = \TL{s}{-s}_2 \circ (\ccup_{s}\otimes\id_{s})
    \end{align}
    which follows by isotopy and is illustrated in \cref{fig: arc-switching}.
    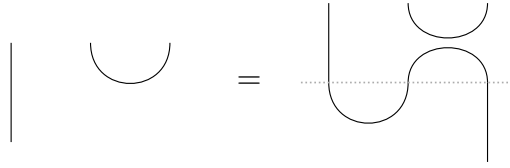
\begin{figure}[htbp]
        \centering
        \scalebox{1.05}{\begin{tikzpicture}
	\begin{pgfonlayer}{nodelayer}
		\node [style=none] (0) at (0, 0) {$=$};
		\node [style=none] (3) at (-3, 0.5) {};
		\node [style=none] (4) at (-3, -0.75) {};
		\node [style=none] (5) at (-2, 0.5) {};
		\node [style=none] (6) at (-1, 0.5) {};
		\node [style=none] (7) at (1, 0) {};
		\node [style=none] (8) at (2, 0) {};
		\node [style=none] (9) at (3, 0) {};
		\node [style=none] (10) at (3, -1) {};
		\node [style=none] (11) at (2, 1) {};
		\node [style=none] (12) at (3, 1) {};
		\node [style=none] (13) at (1, 1) {};
	\end{pgfonlayer}
	\begin{pgfonlayer}{edgelayer}
		\draw (3.center) to (4.center);
		\draw [looseness=1.75, bend right=90] (5.center) to (6.center);
		\draw [looseness=1.75, bend right=90] (7.center) to (8.center);
		\draw [looseness=1.5, bend right=90] (11.center) to (12.center);
		\draw [looseness=1.5, bend left=90] (8.center) to (9.center);
		\draw (9.center) to (10.center);
		\draw (13.center) to (7.center);
		\draw[gray!65,densely dotted,line width=0.6pt]
			(0.65,0) -- (3.35,0);
	\end{pgfonlayer}
\end{tikzpicture}}
        \caption{Illustration of the identity \eqref{eqn: arc switching}.}
        \label{fig: arc-switching}
    \end{figure}
\end{proof}

Similar to \cref{lem: arc reversal} and \cref{lem: arc switching}, one may move a \textit{triplet web} $t_\pm$ to the leftmost part of a web, or reverse its orientation, at the cost of stacking an extra web $T$ from~\cref{fig: auxiliary basic webs}.

\begin{lemma}\label{lem: triplet reversal}
    For each sign $s\in\{+,-\}$, the triple $(-[2]_q^{-1}[3]_q^{-1}T^s_{-s},P_t^s,P_t^{-s})$ is a growth rule. Equivalently, if $P=P_1 P_t^+ P_2$ and $P'=P_1 P_t^- P_2$ are dominant, for some paths $P_1$ and $P_2$, then $G(P)=gG(P')$ for $g=-[2]_q^{-1}[3]_q^{-1}(T^+_{-})_{|P_1|+1}\in \Alg_{\C(q)}(\GenSet_L)$.
\end{lemma}

\begin{proof}
    Suppose first that $s_1=s_2=s$. By the growth algorithm,
    \begin{align*}
        G(P)
        &=
        (\id_{|P_1|}\otimes\id_s\otimes t_s
        \otimes\id_{|P_2|})
        \circ G(P_1uP_2),\\
        G(P')
        &=
        (\id_{|P_1|}\otimes t_s\otimes\id_s
        \otimes\id_{|P_2|})
        \circ G(P_1uP_2).
    \end{align*}
    The Kuperberg relations give
    \begin{align}\label{eqn: triplet reversal}
        t_+ = -[2]_q^{-1}[3]_q^{-1} T^+_- \circ t_-,
    \end{align}
    and the orientation-reversed identity gives the other triple. The identity \eqref{eqn: triplet reversal} is illustrated in \cref{fig: triplet-reversal}.
    \begin{figure}[htbp]
        \centering
        \scalebox{1.05}{\begin{tikzpicture}
	\begin{pgfonlayer}{nodelayer}
		\node [style=none] (0) at (1.25, 1) {};
		\node [style=none] (1) at (2.25, 1) {};
		\node [style=none] (2) at (3.25, 1) {};
		\node [style=none] (3) at (1.25, -1) {};
		\node [style=none] (4) at (2.25, -1) {};
		\node [style=none] (5) at (3.25, -1) {};
		\node [style=none] (6) at (2.25, 0.25) {};
		\node [style=none] (7) at (2.25, -0.25) {};
		\node [style=none] (8) at (1.25, -1) {};
		\node [style=none] (9) at (2.25, -1.75) {};
		\node [style=none] (10) at (2.25, -1) {};
		\node [style=none] (11) at (3.25, -1) {};
		\node [style=none] (12) at (-3.25, 0) {};
		\node [style=none] (13) at (-2.25, 0) {};
		\node [style=none] (14) at (-1.25, 0) {};
		\node [style=none] (15) at (-2.25, -0.75) {};
		\node [style=none] (16) at (0, 0) {$=-[2]_q^{-1}[3]_q^{-1}$};
	\end{pgfonlayer}
	\begin{pgfonlayer}{edgelayer}
		\draw [style=mid arrow, bend left, looseness=1.25] (6.center) to (0.center);
		\draw [style=mid arrow] (6.center) to (1.center);
		\draw [style=mid arrow, bend right] (6.center) to (2.center);
		\draw [style=mid arrow, bend right, looseness=1.25] (8.center) to (9.center);
		\draw [style=mid arrow] (10.center) to (9.center);
		\draw [style=mid arrow, bend left] (11.center) to (9.center);
		\draw [style=mid arrow, bend right=45] (7.center) to (3.center);
		\draw [style=mid arrow] (7.center) to (4.center);
		\draw [style=mid arrow, bend left=45] (7.center) to (5.center);
		\draw [style=mid arrow, bend left, looseness=1.25] (15.center) to (12.center);
		\draw [style=mid arrow] (15.center) to (13.center);
		\draw [style=mid arrow, bend right] (15.center) to (14.center);
		\draw[gray!65,densely dotted,line width=0.6pt]
			(0.85,-1) -- (3.65,-1);
	\end{pgfonlayer}
\end{tikzpicture}}
        \caption{Illustration of the identity \eqref{eqn: triplet reversal}.}
        \label{fig: triplet-reversal}
    \end{figure}
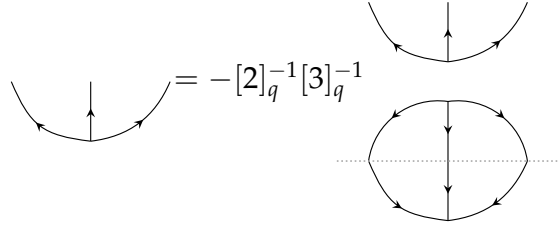
\end{proof}

\begin{lemma}\label{lem: triplet switching}
    For each step $u$, there exists a web $g_u\in\Alg_{\C(q)}(\GenSet_4)$ such that $(g_u,uP_t^+,P_t^+u)$ is a growth rule.  Equivalently, if $P=P_1uP_t^+ P_2$ and $P'=P_1  P_t^+ u P_2$ are dominant, for some paths $P_1$ and $P_2$, then $G(P)= (\id_{|P_1|} \otimes g_{u} \otimes \id_{|P_2|}) G(P')$.
    For each step $u$, there exists a web $g_u\in\Alg_{\C(q)}(\GenSet_4)$ such that $(g_u,uP_t^+,P_t^+u)$ is a growth rule.  Equivalently, if $P=P_1uP_t^+ P_2$ and $P'=P_1  P_t^+ u P_2$ are dominant paths, for some paths $P_1$ and $P_2$, then $G(P)= (\id_{|P_1|} \otimes g_{u} \otimes \id_{|P_2|}) G(P')$.
\end{lemma}
\begin{proof}
    Applying the growth rule for the triplet piece gives
    \begin{align*}
        G(P) & = \id_{|P_1|}\otimes \id_u \otimes t_+ \otimes \id_{|P_2|}\circ G(P_1uP_2),\\
        G(P') & = \id_{|P_1|}\otimes t_+ \otimes \id_{u} \otimes \id_{|P_2|}\circ G(P_1uP_2).
    \end{align*}
    The local web identity
    \begin{align}\label{eqn: triplet switching}
        (\id_u\otimes t_+)=g_u\circ (t_+\otimes \id_u)
    \end{align}
    follows from isotopy and the triplet-reversal relation when the sign of $u$ is reversed. This identity is illustrated in \cref{fig: triplet-switching}, and proves the auxiliary rule.
    \begin{figure}[htbp]
        \centering
        \scalebox{0.88}{\begin{tikzpicture}
	\begin{pgfonlayer}{nodelayer}
		\node [style=none] (0) at (1, 0) {};
		\node [style=none] (1) at (2, 0) {};
		\node [style=none] (3) at (4, 0) {};
		\node [style=none] (4) at (4, -2) {};
		\node [style=none] (5) at (2, -1) {};
		\node [style=none] (6) at (3, 0) {};
		\node [style=none] (7) at (0, 0) {$=-[2]_q^{-1}$};
		\node [style=none] (8) at (-3, 1) {};
		\node [style=none] (9) at (-2, 1) {};
		\node [style=none] (10) at (-2, 0.25) {};
		\node [style=none] (11) at (-1, 1) {};
		\node [style=none] (12) at (-4, 1) {};
		\node [style=none] (13) at (-4, -1) {};
		\node [style=none] (14) at (3, 0.75) {};
		\node [style=none] (15) at (1, 2) {};
		\node [style=none] (16) at (2, 2) {};
		\node [style=none] (17) at (3, 2) {};
		\node [style=none] (18) at (3, 1.25) {};
		\node [style=none] (19) at (4, 2) {};
	\end{pgfonlayer}
	\begin{pgfonlayer}{edgelayer}
		\draw [bend left=45] (5.center) to (0.center);
		\draw (5.center) to (1.center);
		\draw [bend right=45] (5.center) to (6.center);
		\draw (3.center) to (4.center);
		\draw [bend left=45] (10.center) to (8.center);
		\draw (10.center) to (9.center);
		\draw [bend right=45] (10.center) to (11.center);
		\draw (12.center) to (13.center);
		\draw [looseness=1.25, bend right=45] (14.center) to (1.center);
		\draw (14.center) to (6.center);
		\draw (14.center) to (14.center);
		\draw [looseness=1.25, bend left=45] (14.center) to (3.center);
		\draw [bend left=45] (18.center) to (16.center);
		\draw (18.center) to (17.center);
		\draw [bend right=45] (18.center) to (19.center);
		\draw (15.center) to (0.center);
		\draw[gray!65,densely dotted,line width=0.6pt]
			(0.6,0) -- (4.4,0);
	\end{pgfonlayer}
\end{tikzpicture}}
        \caption{Illustration of the identity \eqref{eqn: triplet switching}.}
        \label{fig: triplet-switching}
    \end{figure}
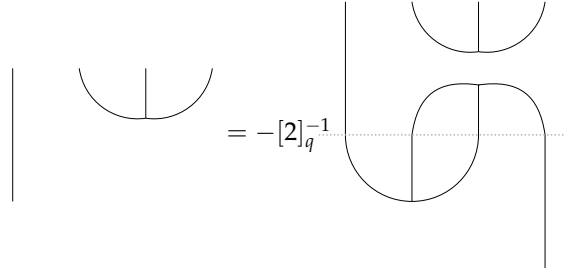
\end{proof}

\begin{lemma}\label{lem: Y-switch}
    For signs $s,s'\in\{+,-\}$, the triple $(-[2]_q^{-1}\LY{s}{s'},u_{s,1}P_Y^{s'},P_Y^s u_{s',1})$ is a growth rule.  Equivalently, if the paths
    \begin{align*}
        P=P_1 u_{s,1} P^{s'}_Y P_2, \qquad P'=P_1 P^{s}_Y u_{s',1} P_2
    \end{align*}
    are dominant, for some paths $P_1$ and $P_2$, then $G(P)=gG(P')$ for $g=[2]_q^{-1}\LY{s}{s'}_{|P_1|+1}$.
\end{lemma}
\begin{proof}
    The desired identity follows from the next identity, verified directly using the Kuperberg relations, illustrated in \cref{fig: Y-switch}:
    \begin{align}\label{eqn: Y switch}
        \id_{s} \otimes Y_{s'} = -[2]_q^{-1} \LY{s}{s'} \circ (Y_{s} \otimes \id_{s'}).
    \end{align}
    \begin{figure}[htbp]
        \centering
        \scalebox{0.88}{\begin{tikzpicture}
	\begin{pgfonlayer}{nodelayer}
		\node [style=none] (0) at (0, 0) {$=-[2]_q^{-1}$};
		\node [style=none] (18) at (-4, 1) {};
		\node [style=none] (19) at (-4, -1) {};
		\node [style=none] (20) at (-3, 1) {};
		\node [style=none] (21) at (-1, 1) {};
		\node [style=none] (22) at (-2, 0.25) {};
		\node [style=none] (23) at (-2, -1) {};
		\node [style=none] (24) at (1, 0) {};
		\node [style=none] (25) at (3, 0) {};
		\node [style=none] (26) at (2, -0.75) {};
		\node [style=none] (27) at (2, -2) {};
		\node [style=none] (28) at (4, 0) {};
		\node [style=none] (29) at (4, -2) {};
		\node [style=none] (30) at (2, 0.75) {};
		\node [style=none] (31) at (2, 2) {};
		\node [style=none] (32) at (4, 1.25) {};
		\node [style=none] (33) at (3, 2) {};
		\node [style=none] (34) at (5, 2) {};
	\end{pgfonlayer}
	\begin{pgfonlayer}{edgelayer}
		\draw (22.center) to (23.center);
		\draw (22.center) to (21.center);
		\draw (22.center) to (20.center);
		\draw (18.center) to (19.center);
		\draw (26.center) to (27.center);
		\draw (26.center) to (25.center);
		\draw (26.center) to (24.center);
		\draw (28.center) to (29.center);
		\draw (30.center) to (24.center);
		\draw (30.center) to (25.center);
		\draw (30.center) to (31.center);
		\draw (33.center) to (32.center);
		\draw (34.center) to (32.center);
		\draw (32.center) to (28.center);
		\draw[gray!65,densely dotted,line width=0.6pt]
			(0.6,0) -- (4.4,0);
	\end{pgfonlayer}
\end{tikzpicture}}
        \caption{Illustration of the identity \eqref{eqn: Y switch}.}
        \label{fig: Y-switch}
    \end{figure}
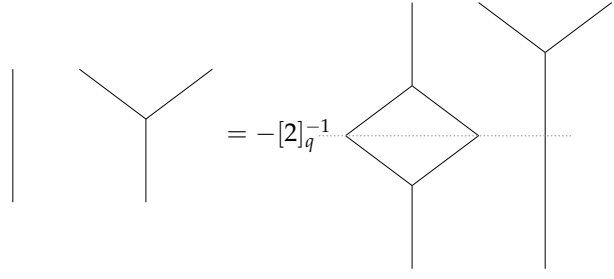
\end{proof}

\begin{lemma}\label{lem: merge Y}
    For signs $s,s'\in\{+,-\}$, the triple $(\LY{-s}{-s'},P_Y^sP_Y^{s'},P_c^{-s}u_{s,1}u_{s',1})$ is a growth rule.  Equivalently, if the paths
    \begin{align*}
        P=P_1 P^{s}_Y P^{s'}_Y P_2, \qquad P'=P_1 P_c^{-s} u_{s,1} u_{s',1} P_2
    \end{align*}
    are dominant, for some paths $P_1$ and $P_2$, then $G(P)=gG(P')$ for $g=\LY{-s}{-s'}_{|P_1|+2}\in \Alg_{\C(q)}(\GenSet_L)$.
    are dominant, for some paths $P_1$ and $P_2$, then $G(P)=gG(P')$ for $g=\LY{-s}{-s'}_{|P_1|+2}\in \Alg_{\C(q)}(\GenSet_L)$.
\end{lemma}
\begin{proof}
    This follows from the identity
    \begin{align}\label{eqn: merge Y}
        Y_{s} \otimes Y_{s'} =  \LY{-s}{-s'}_{2} \circ (\ccup_{-s} \otimes \id_{ss'})
    \end{align}
    which is illustrated in \cref{fig: Y-merge surgery}.
    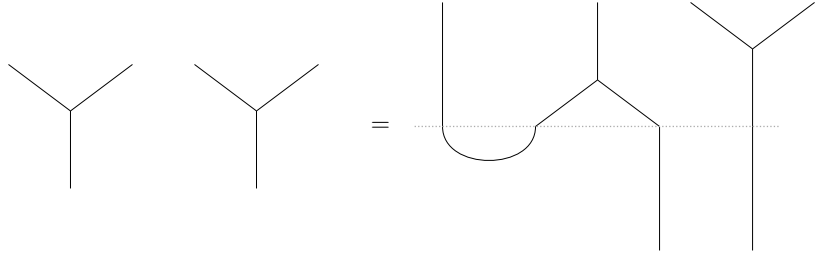
\begin{figure}[htbp]
        \centering
        \scalebox{0.82}{\begin{tikzpicture}
	\begin{pgfonlayer}{nodelayer}
		\node [style=none] (0) at (-6, 1) {};
		\node [style=none] (1) at (-5, 0.25) {};
		\node [style=none] (2) at (-4, 1) {};
		\node [style=none] (3) at (-5, -1) {};
		\node [style=none] (4) at (-3, 1) {};
		\node [style=none] (5) at (-1, 1) {};
		\node [style=none] (6) at (-2, 0.25) {};
		\node [style=none] (7) at (-2, -1) {};
		\node [style=none] (8) at (0, 0) {$=$};
		\node [style=none] (9) at (1, 0) {};
		\node [style=none] (10) at (2.5, 0) {};
		\node [style=none] (11) at (4.5, 0) {};
		\node [style=none] (12) at (6, 0) {};
		\node [style=none] (13) at (6, 1.25) {};
		\node [style=none] (14) at (5, 2) {};
		\node [style=none] (15) at (7, 2) {};
		\node [style=none] (16) at (3.5, 0.75) {};
		\node [style=none] (17) at (3.5, 2) {};
		\node [style=none] (18) at (6, -2) {};
		\node [style=none] (19) at (1, 2) {};
		\node [style=none] (20) at (4.5, -2) {};
	\end{pgfonlayer}
	\begin{pgfonlayer}{edgelayer}
		\draw (0.center) to (1.center);
		\draw (2.center) to (1.center);
		\draw (1.center) to (3.center);
		\draw (4.center) to (6.center);
		\draw (5.center) to (6.center);
		\draw (6.center) to (7.center);
		\draw (12.center) to (18.center);
		\draw [bend right=90, looseness=1.25] (9.center) to (10.center);
		\draw (16.center) to (10.center);
		\draw (16.center) to (11.center);
		\draw (17.center) to (16.center);
		\draw (19.center) to (9.center);
		\draw (14.center) to (13.center);
		\draw (15.center) to (13.center);
		\draw (13.center) to (12.center);
		\draw (11.center) to (20.center);
		\draw[gray!65,densely dotted,line width=0.6pt]
			(0.55,0) -- (6.45,0);
	\end{pgfonlayer}
\end{tikzpicture}}
        \caption{Illustration of the identity \eqref{eqn: merge Y}.}
        \label{fig: Y-merge surgery}
    \end{figure}
\end{proof}

\begin{lemma}\label{lem: YH-EY}
    For each sign $s\in\{+,-\}$, the triple
    \[
        (I^s\otimes\id_s,\ u_{s,1}u_{s,0}u_{s,0},\ u_{s,0}u_{s,1}u_{s,0})
    \]
    is a growth rule. Equivalently, if the paths
    \begin{align*}
        P=P_1 u_{s,1} u_{s,0} u_{s,0} P_2, \qquad P'=P_1 u_{s,0} u_{s,1} u_{s,0} P^s_Y P_2
    \end{align*}
    are dominant, for some paths $P_1$ and $P_2$, then $G(P)= I^s_{|P_1|+1} G(P')$.
    are dominant, for some paths $P_1$ and $P_2$, then $G(P)= I^s_{|P_1|+1} G(P')$.
\end{lemma}
\begin{proof}
    By applying growth rules from \cref{fig: defining growth rules} and \cref{fig: growth rules},
    \begin{align}\label{eqn: growth temporary}
        \begin{aligned}
            G(P)=& \id_{|P_1|} \otimes ((Y_{-s} \otimes  \id_{s}) \circ H^s ) \otimes \id_{|P_2|} \circ G(P_1u_{s,0} u_{-s,1}P_2),\\
            G(P')=& \id_{|P_1|} \otimes (\id_{s} \otimes Y_{-s}) \otimes \id_{|P_2|} \circ G(P_1 u_{s,0} u_{-s,1} P_2).
        \end{aligned}
    \end{align}
On the other hand, we have the following relation
\begin{align}\label{eqn: YH-EY relation}
    (Y_{-s} \otimes  \id_{s}) \circ H^s = I^s_1 \circ (\id_{s} \otimes Y_{-s})
\end{align}
illustrated in \cref{fig: YH-EY}.
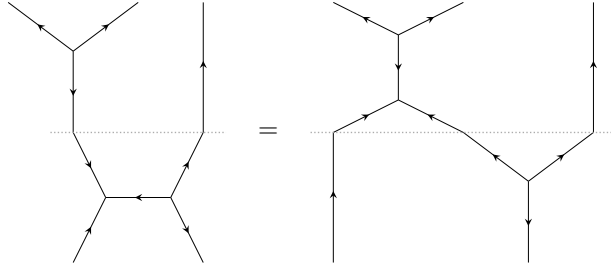
\begin{figure}[htbp]
    \centering
    \scalebox{0.86}{\begin{tikzpicture}
	\begin{pgfonlayer}{nodelayer}
		\node [style=none] (0) at (-4, 2) {};
		\node [style=none] (1) at (-2, 2) {};
		\node [style=none] (2) at (-3, 1.25) {};
		\node [style=none] (3) at (-3, 0) {};
		\node [style=none] (4) at (-1, 2) {};
		\node [style=none] (5) at (-1, 0) {};
		\node [style=none] (6) at (-2.5, -1) {};
		\node [style=none] (7) at (-1.5, -1) {};
		\node [style=none] (8) at (-3, -2) {};
		\node [style=none] (9) at (-1, -2) {};
		\node [style=none] (10) at (0, 0) {$=$};
		\node [style=none] (11) at (1, 2) {};
		\node [style=none] (12) at (2, 1.5) {};
		\node [style=none] (13) at (2, 0.5) {};
		\node [style=none] (14) at (3, 2) {};
		\node [style=none] (15) at (1, 0) {};
		\node [style=none] (16) at (3, 0) {};
		\node [style=none] (17) at (5, 2) {};
		\node [style=none] (18) at (5, 0) {};
		\node [style=none] (19) at (4, -0.75) {};
		\node [style=none] (20) at (4, -2) {};
		\node [style=none] (21) at (1, -2) {};
	\end{pgfonlayer}
	\begin{pgfonlayer}{edgelayer}
		\draw [style=mid arrow] (2.center) to (0.center);
		\draw [style=mid arrow] (2.center) to (1.center);
		\draw [style=mid arrow] (2.center) to (3.center);
		\draw [style=mid arrow] (5.center) to (4.center);
		\draw [style=mid arrow] (7.center) to (5.center);
		\draw [style=mid arrow] (7.center) to (6.center);
		\draw [style=mid arrow] (7.center) to (9.center);
		\draw [style=mid arrow] (3.center) to (6.center);
		\draw [style=mid arrow] (8.center) to (6.center);
		\draw [style=mid arrow] (15.center) to (13.center);
		\draw [style=mid arrow] (16.center) to (13.center);
		\draw [style=mid arrow] (12.center) to (13.center);
		\draw [style=mid arrow] (12.center) to (14.center);
		\draw [style=mid arrow] (12.center) to (11.center);
		\draw [style=mid arrow] (19.center) to (16.center);
		\draw [style=mid arrow] (19.center) to (18.center);
		\draw [style=mid arrow] (18.center) to (17.center);
		\draw [style=mid arrow] (21.center) to (15.center);
		\draw [style=mid arrow] (19.center) to (20.center);
		\draw[gray!65,densely dotted,line width=0.6pt]
			(-3.35,0) -- (-0.65,0);
		\draw[gray!65,densely dotted,line width=0.6pt]
			(0.65,0) -- (5.35,0);
	\end{pgfonlayer}
\end{tikzpicture}}
    \caption{Illustration of \eqref{eqn: YH-EY relation}.}
    \label{fig: YH-EY}
\end{figure}
Combining \eqref{eqn: growth temporary} and \eqref{eqn: YH-EY relation}, we get the desired identity.
\end{proof}

\begin{lemma}\label{lem: triplet merge}
    For signs $s_1,s_2\in\{+,-\}$, there exists a web $g_{s_1,s_2}\in\Alg_{\C(q)}(\GenSet_6)$ such that $$(g_{s_1,s_2},P_t^{s_1}P_t^{s_2},P_c^+P_c^+P_c^+)$$ is a growth rule. Equivalently, if the paths
    \begin{align*}
        P=P_1P_t^{s_1}P_t^{s_2}P_2, \qquad P'=P_1P_c^+ P_c^+ P_c^+ P_2
    \end{align*}
    are dominant, for some paths $P_1$ and $P_2$, then $G(P)=gG(P')$ for some $g_{s_1,s_2}\in \Alg_{\C(q)}(\GenSet_L)$.
\end{lemma}
\begin{proof}
    Write
    \begin{align*}
        &P_t^{s_1}P_t^{s_2}  \\
        \longrightarrow& u_{s_1,1} u_{s_1,0} P_t^{s_2} u_{s_1,-1} \\
        \longrightarrow& P_c^{-s_1} u_{s_1,1} u_{s_2,1} u_{s_2,-1} u_{s_1,-1} \\
        \longrightarrow& P_c^+ P_c^+ P_c^+.
    \end{align*}
where the first arrow uses \cref{lem: triplet switching}, the second arrow uses \cref{lem: merge Y} on the first four steps, and the last arrow uses \cref{lem: arc reversal} and \cref{lem: arc switching}. Composing these auxiliary rules gives the required web $g_{s_1,s_2}$.
\end{proof}

\begin{lemma}\label{lem: triplet-Y merge}
    For each sign $s\in\{+,-\}$, the triple
    \[
        (\LY{s}{-s}_3\circ\TL{s}{-s}_2,\ P_t^sP_Y^s,\ P_c^+P_c^+u_{s,1})
    \]
    is a growth rule. Equivalently, if the paths
    \begin{align*}
        P=P_1 P_t^s P_Y^s P_2, \qquad P'=P_1 P_c^+ P_c^+ u_{s,1} P_2
    \end{align*}
    are dominant, for some paths $P_1$ and $P_2$, then $G(P)=gG(P')$ for $g= \LY{s}{-s}_{|P_1|+3} \circ \TL{s}{-s}_{|P_1|+2}$.
    are dominant, for some paths $P_1$ and $P_2$, then $G(P)=gG(P')$ for $g= \LY{s}{-s}_{|P_1|+3} \circ \TL{s}{-s}_{|P_1|+2}$.
\end{lemma}
\begin{proof}
    This follows from the identity
    \begin{align}\label{eqn: merge triplet and Y}
        t_s \otimes Y_s = \LY{s}{-s}_3 \circ \TL{s}{-s}_2 \circ (\ccup_+ \otimes \ccup_+ \otimes \id_{s}),
    \end{align}
    illustrated in \cref{fig: triplet-Y surgery}.
    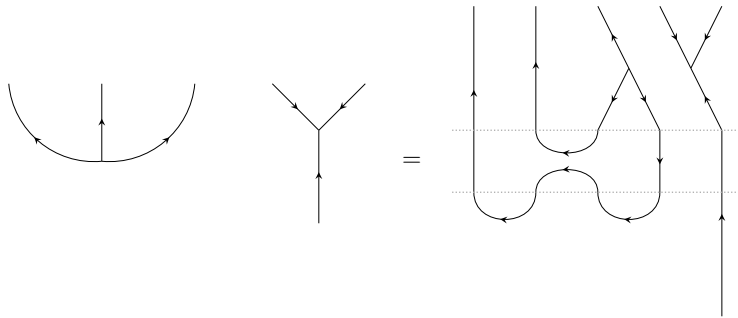
\begin{figure}[htbp]
        \centering
        \scalebox{0.82}{\begin{tikzpicture}
	\begin{pgfonlayer}{nodelayer}
		\node [style=none] (0) at (-6.5, 1.25) {};
		\node [style=none] (1) at (-5, 1.25) {};
		\node [style=none] (2) at (-3.5, 1.25) {};
		\node [style=none] (3) at (-5, 0) {};
		\node [style=none] (4) at (-2.25, 1.25) {};
		\node [style=none] (5) at (-1.5, -1) {};
		\node [style=none] (6) at (0, 0) {$=$};
		\node [style=none] (7) at (1, -0.5) {};
		\node [style=none] (8) at (2, -0.5) {};
		\node [style=none] (9) at (3, -0.5) {};
		\node [style=none] (10) at (4, -0.5) {};
		\node [style=none] (11) at (5, -0.5) {};
		\node [style=none] (12) at (5, -2.5) {};
		\node [style=none] (13) at (2, 0.5) {};
		\node [style=none] (14) at (3, 0.5) {};
		\node [style=none] (15) at (5, 0.5) {};
		\node [style=none] (16) at (4, 0.5) {};
		\node [style=none] (17) at (3, 2.5) {};
		\node [style=none] (18) at (4, 2.5) {};
		\node [style=none] (19) at (5, 2.5) {};
		\node [style=none] (20) at (2, 2.5) {};
		\node [style=none] (21) at (1, 2.5) {};
		\node [style=none] (22) at (4.5, 1.5) {};
		\node [style=none] (23) at (3.5, 1.5) {};
		\node [style=none] (24) at (-1.5, 0.5) {};
		\node [style=none] (25) at (-0.75, 1.25) {};
	\end{pgfonlayer}
	\begin{pgfonlayer}{edgelayer}
		\draw [style=mid arrow, bend left=45] (3.center) to (0.center);
		\draw [style=mid arrow] (3.center) to (1.center);
		\draw [style=mid arrow, bend right=45] (3.center) to (2.center);
		\draw [style=mid arrow, bend left=90, looseness=1.50] (8.center) to (7.center);
		\draw [style=mid arrow, bend left=90, looseness=1.50] (10.center) to (9.center);
		\draw [style=mid arrow, bend right=90, looseness=1.25] (9.center) to (8.center);
		\draw [style=mid arrow, bend left=90, looseness=1.25] (14.center) to (13.center);
		\draw [style=mid arrow] (12.center) to (15.center);
		\draw [style=mid arrow] (15.center) to (22.center);
		\draw [style=mid arrow] (19.center) to (22.center);
		\draw [style=mid arrow] (18.center) to (22.center);
		\draw [style=mid arrow] (23.center) to (14.center);
		\draw [style=mid arrow] (23.center) to (16.center);
		\draw [style=mid arrow] (23.center) to (17.center);
		\draw [style=mid arrow] (16.center) to (10.center);
		\draw [style=mid arrow] (13.center) to (20.center);
		\draw [style=mid arrow] (7.center) to (21.center);
		\draw [style=mid arrow] (5.center) to (24.center);
		\draw [style=mid arrow] (4.center) to (24.center);
		\draw [style=mid arrow] (25.center) to (24.center);
		\draw[gray!65,densely dotted,line width=0.6pt]
			(0.65,-0.5) -- (5.35,-0.5);
		\draw[gray!65,densely dotted,line width=0.6pt]
			(0.65,0.5) -- (5.35,0.5);
	\end{pgfonlayer}
\end{tikzpicture}}
        \caption{Illustration of the identity \eqref{eqn: merge triplet and Y}.}
        \label{fig: triplet-Y surgery}
    \end{figure}
\end{proof}

\subsection{Cyclicity of cell modules}\label{sec: cyclic property}
In this subsection we prove \cref{prp: cyclic vector is indeed cyclic}. Specifically, we construct an algorithm that determines, for each $P\in \Path_L^+(\lambda)$, an element $g_P\in \Alg_{\C(q)}(\GenSet_L)$ such that $G(P)=g_P G(P^S_{cyc})$. We now construct this algorithm.

\begin{definition}\label{def: active area}
    Let $P$ be a dominant path encoded by $(S,J)$. The \textup{cup stack} of $P$ consists of the initial $2k$ steps of $P$ that form a concatenation of $k$ pieces of $P_c^\pm$, for the maximal possible $k\geq 0$. The \textup{triplet stack} of $P$ consists of the subsequent $3l$ steps that form a concatenation of $l$ pieces of $P_t^\pm$, with $l$ being maximal. The \textup{1 stack} of $P$ is the final $m$ steps that are all $1$-steps, for maximal possible $m\geq 0$. The \textup{Y stack} of $P$ consists of the $2n$ steps before the 1 stack that form a concatenation of $n$ pieces of $P_Y^\pm$, $n$ being maximal. The \textup{active area} of $P$ consists of the remaining steps in between the triplet stack and the Y stack.
\end{definition}

\begin{example}\label{ex: path decomposition matrix}
Below is an example of a path of length $L=15$ partitioned into these five zones. The top row represents the sign string $S$, and the bottom row represents the state string $J$:
    \begin{equation}
        \begin{matrix}
            S = \\[0.5em]
            J = 
        \end{matrix}
        \quad
        \underbrace{
        \begin{matrix}
            + & - & + & - \\[0.5em]
            1 & -1 & 1 & -1
        \end{matrix}
        }_{\text{Cup stack}}
        \quad\ \ 
        \underbrace{
        \begin{matrix}
            + & + & + \\[0.5em]
            1 & 0 & -1
        \end{matrix}
        }_{\text{Triplet stack}}
        \quad\ \ 
        \underbrace{
        \begin{matrix}
            - & - & - & - \\[0.5em]
            1 & 1 & 0 & 0
        \end{matrix}
        }_{\text{Active area}}
        \quad\ \ 
        \underbrace{
        \begin{matrix}
            + & + \\[0.5em]
            1 & 0
        \end{matrix}
        }_{\text{Y stack}}
        \quad\ \ 
        \underbrace{
        \begin{matrix}
            - & +  \\[0.5em]
            1 & 1 
        \end{matrix}
        }_{\text{1 stack}}
    \end{equation}
    
\end{example}

\begin{lemma}\label{lem: active area reduction}
    Let $P$ be a dominant path with a non-empty active area. Then there exists an element $g\in \Alg_{\C(q)}(\GenSet_L)$ such that $G(P)=g G(P')$, where $P'$ has an active area with strictly smaller length than that of $P$.
\end{lemma}
\begin{proof}
    Recall the growth rules with a web $H^\pm$ involved from \cref{fig: defining growth rules}, and \cref{prp: I-rule}. Note that the $H^\pm$-webs and the $I$-webs are in the generating set $\GenSet_L$, hence by applying these rules, we assume that the active area of $P$ satisfies:
    \begin{enumerate}
        \item [\textbf{a}.] No two consecutive steps in $P$ take the form $u_{\pm,1} u_{\mp,0}$, $u_{\pm,0} u_{\mp,0}$ or $u_{\pm,0} u_{\mp,-1}$; otherwise, an $H^\pm$-web growth rule can be applied.
        \item [\textbf{b}.] No consecutive steps in $P$ take the form $u_{\pm,1} u_{\pm,-1}$; otherwise, \cref{prp: I-rule} can be applied.
    \end{enumerate}
    Now consider the last $1$-step in the active area of $P$. Based on the next step in the active area, there are three cases to be considered:
    \begin{enumerate}
        \item There is no such next step: The last $1$-step is also the last step in the active area, so it is followed by the Y-stack, in which case we apply \cref{lem: Y-switch} to move this $1$-step into the $1$-stack, so the active area of the resulting path is smaller by one step.
        \item The next step in the active area is a $-1$-step. By assumption \textbf{b}, this $-1$-step must have a different sign from the $1$-step. In this case the two steps together form $P_c^\pm$, and we can apply \cref{lem: arc switching} to move this $P_c^\pm$ to the leftmost into the cup stack, so the active area of the resulting path is smaller by two steps.
        \item The next step in the active area is a $0$-step. By assumption \textbf{a}, this $0$-step must take the same sign as the $1$-step. The following step is still in the active area, otherwise the last $1$-step and the subsequent $0$-step should have been in the Y-stack, contradicting the maximality in \cref{def: active area}. We further divide into two cases based on this following step:
        \begin{enumerate}
            \item [3.1.] The following step is a $-1$-step, which takes the same sign as the $1$-step by assumption \textbf{a}. In this case, the three steps together form a $P_t^\pm$, and we can apply \cref{lem: triplet switching} to move this $P_t^\pm$ into the triplet-stack, so the active area of the resulting path is reduced by three steps.
            \item [3.2.] The following step is a $0$-step, which still has the same sign by assumption \textbf{a}. In this case the three steps together form a substring $u_{s,1} u_{s,0} u_{s,0}$, and we can apply \cref{lem: YH-EY} to move the last $1$-step in the active area one step to the right. Since the active area is finite, the last $1$-step in it can only move finitely many times to the right. Eventually we reach one of the previous two cases that reduce the active area.
        \end{enumerate}
    \end{enumerate}
\end{proof}

\begin{proof}[Proof of \Cref{prp: cyclic vector is indeed cyclic}]
    Let $P\in \Path_L^+(\lambda)$. Since the active area of $P$ has finite length, by \cref{lem: active area reduction}, we can apply a sequence of elements in $\Alg_{\C(q)}(\GenSet_L)$ to $G(P)$ and get a web $G(P')$ where $P'$ has an empty active area. Such a path $P'$ could differ from a cyclic path $P^S_{cyc}$ by containing more than one copy of $P_c^\pm$ and $P_Y^{\pm}$ in the cup stack and the Y stack. But \cref{lem: triplet merge}, \cref{lem: merge Y} and \cref{lem: triplet-Y merge} allow us to reduce to having at most one copy of $P_t^\pm$ and $P_Y^{\pm}$, so $P'$ is a cyclic path. Hence we have shown that there exists $g\in \Alg_{\C(q)}(\GenSet_L)$ such that $G(P)=g G(P^S_{cyc})$ for some sign string $S$, as desired.
\end{proof}

\subsection{Proof of \Cref{prp: basic web in the cyclic subalgebra}}\label{sec: proof of generators}

By \cref{lem: basis decomposition}, we may write
    \begin{align*}
        C^\lambda_{P^S_{cyc},P^T_{cyc}} = w^S_{cyc} \circ \mathcal{H}^S_T \circ (w^T_{cyc})^\dagger,
    \end{align*}
Note that $\ccup \circ (\ccup)^\dagger=\TL{}{}$ for appropriate signs, and $t_+ \circ (t_+)^\dagger=T$ for appropriate signs. Hence, to show that $C^\lambda_{P^S_{cyc},P^T_{cyc}} \in \Alg_{\C(q)}(\GenSet_L)$, it suffices to show that
\begin{align}\label{eqn: Y-Lambda}
    &(Y \otimes \id) \circ \mathcal{H}^S_T \circ (Y \otimes \id)^\dagger \notag \\
    =&(Y \otimes \id) \circ \mathcal{H}^S_T \circ (\Ydual \otimes \id)\in \Alg_{\C(q)}(\GenSet_L).
\end{align}
By repeatedly applying the commutation relation between $Y$ and $\mathcal{H}$ 
\begin{align}
    Y_k \circ H_k = I_k \circ Y_{k+1},
\end{align}
which is evident from \cref{fig: YH-EY}, \eqref{eqn: Y-Lambda} boils down to showing that
\begin{align}\label{eqn: Y-Lambda 2}
    Y \otimes \id \otimes \Ydual \otimes \id \in \Alg_{\C(q)}(\GenSet_L),
\end{align}
which is true by applying \eqref{eqn: Y switch}, which shows that the left hand side of \eqref{eqn: Y-Lambda 2} is a composition of $\LY{}{}$ and $\LY{}{}$, both of which are in $\Alg_{\C(q)}(\GenSet_L)$. We therefore conclude the proof.

\newpage
\begin{appendices}

\section{$\uqslt$ and its representations}\label{app: quantum group conventions}

\subsection{The quantum algebra $\uqslt$}\label{app: quantum sl3}
In what follows, $q$ is a formal variable. Let $\alpha:=\left(\begin{smallmatrix}
    2 & -1 \\ -1 & 2\end{smallmatrix}\right)$ be the Cartan matrix of $\mathfrak{sl}_3$. The quantum group $\uqslt$ is a $\C(q)$-algebra generated by $E_i$, $F_i$, $K_i^{\pm 1}$ for $1\leq i \leq 2$, subject to the following defining relations:
\begin{subequations}
\begin{align}
    K_iK_i^{-1} &= K_i^{-1}K_i = 1 \,, \\
    K_iK_j &= K_jK_i \,, \\
    K_jE_i &= q^{\alpha_{ij}}E_iK_j \,, \\
    K_jF_i &= q^{-\alpha_{ij}}F_iK_j \,, \\
    [E_i,F_j] &= \delta_{ij}\frac{K_i-K_i^{-1}}{q-q^{-1}} \,, \\
    [2]_q E_iE_jE_i &= E_i^2E_j+E_jE_i^2 \,, \quad |i-j|=1 \,, \\
    [2]_q F_iF_jF_i &= F_i^2F_j+F_jF_i^2 \,, \ \quad |i-j|=1,
\end{align}
\end{subequations}
where the quantum integers $[n]_q$ are defined in \eqref{eqn: quantum integers}.

The algebra $\uqslt$ is a Hopf algebra with coproduct
\begin{align}\label{eqn: coproduct}
    \Delta(E_i)=E_i\otimes 1 + K_i^{-1}\otimes E_i \,, \qquad \Delta(F_i)=F_i\otimes K_i + 1\otimes F_i \,, \qquad \Delta(K_i)=K_i\otimes K_i \,,
\end{align}
 antipode
\begin{align}
    S(E_i)=-K_iE_i\,, \qquad S(F_i)=-F_iK_i^{-1} \,, \qquad S(K_i)=K_i^{-1} \,,
\end{align}
and counit
\begin{align}
    \epsilon(E_i)=0 \,, \qquad \epsilon(F_i)=0 \,, \qquad \epsilon(K_i)=1 \,.
\end{align}

\subsection{Fundamental representations of $\uqslt$}\label{app: fundamental reps of uqslt}

Let $\alpha_1,\alpha_2$ denote the simple roots of $\mathfrak{sl}_3$, normalized by the $\mathfrak{sl}_3$-Cartan matrix $(\alpha_i,\alpha_j)=\alpha_{ij}$. The two fundamental weights $\lambda_1,\lambda_2$ satisfy $(\alpha_i,\lambda_j)=\delta_{ij}$. Let $V^+$ be the fundamental $\uqslt$-representation with weight $\lambda_1$, with basis $\{v^+_1,v^+_0,v^+_{-1}\}$. The action of $\uqslt$ on $V^+$ is given by the following matrices:

\begin{align*}
    K_1=\begin{pmatrix}
        q & 0 & 0 \\
        0 & q^{-1} & 0 \\
        0 & 0 & 1
\end{pmatrix}
, \qquad
E_1=\begin{pmatrix}
        0 & 1 & 0 \\
        0 & 0 & 0 \\
        0 & 0 & 0
\end{pmatrix}
, \qquad
    F_1=\begin{pmatrix}
        0 & 0 & 0 \\
        1 & 0 & 0 \\
        0 & 0 & 0
\end{pmatrix},
\end{align*}
\begin{align*}
    K_2=\begin{pmatrix}
        1 & 0 & 0 \\
        0 & q & 0 \\
        0 & 0 & q^{-1}
\end{pmatrix}
, \qquad
    E_2=\begin{pmatrix}
        0 & 0 & 0 \\
        0 & 0 & 1 \\
        0 & 0 & 0
\end{pmatrix}
, \qquad
    F_2=\begin{pmatrix}
        0 & 0 & 0 \\
        0 & 0 & 0 \\
        0 & 1 & 0
\end{pmatrix}.
\end{align*}

Let $V^-:=(V^+)^*$ be the vector space dual of $V^+$. The antipode of $\uqslt$ induces a representation of $\uqslt$ on $V^-$ by setting $hf(\cdot):=f(S(h) \cdot)$ for $h\in \uqslt$ and $f\in V^-$. We choose a basis of $V^-$ by setting
 \begin{align*}
    v^-_i:=(-q)^{i}(v^+_{-i})^* \,, \quad i=1,0,-1,
 \end{align*}
where $(v^+_i)^*$ satisfies $(v^+_i)^*(v^+_j)=\delta_{ij}$. In this basis, the action of $\uqslt$ on $V^-$ is given by

\begin{align*}
    K_1 = \begin{pmatrix}
        1 & 0 & 0 \\
        0 & q & 0 \\
        0 & 0 & q^{-1}
\end{pmatrix}, \qquad
    E_1 = \begin{pmatrix}
        0 & 0 & 0 \\
        0 & 0 & 1 \\
        0 & 0 & 0
\end{pmatrix}, \qquad
    F_1 = \begin{pmatrix}
        0 & 0 & 0 \\
        0 & 0 & 0 \\
        0 & 1 & 0
\end{pmatrix},
\end{align*}
\begin{align*}
    K_2 = \begin{pmatrix}
        q & 0 & 0 \\
        0 & q^{-1} & 0 \\
        0 & 0 & 1
\end{pmatrix}, \qquad
    E_2 = \begin{pmatrix}
        0 & 1 & 0 \\
        0 & 0 & 0 \\
        0 & 0 & 0
\end{pmatrix}, \qquad
    F_2 = \begin{pmatrix}
        0 & 0 & 0 \\
        1 & 0 & 0 \\
        0 & 0 & 0
\end{pmatrix},
\end{align*}
so $V^-$ is the fundamental $\uqslt$-representation with weight $\lambda_2$.

\subsection{The image of $\Gamma$ for small webs}\label{app: intertwiners Gamma}
We give the definition of $\id_{V^\pm},\eva_{\pm},\coev_{\pm}$ and $\mathcal Y_{\pm}$ that are used in \cref{sec: equivalence Sp to Rep} in the definition of $\Gamma:\Sp\to \Rep$.

Firstly, we define $\id_{V^\pm}$ as the identities on $V^\pm$, which are obviously morphisms in $\Rep$. In what follows we let $z:=-q$. Recall that $\uqslt$ has a \textit{pivotal element} $g=K_1^2K_2^2$, which satisfies
\begin{align*}
    \Delta(g)&=g\otimes g, \\
    S^2(h)&=ghg^{-1}, \quad \forall h\in \uqslt.
\end{align*}
Using this, we define the morphisms
\begin{equation*}
\begin{alignedat}{3}
\eva_V :&\; V^* \otimes V \to \C(q),
&\quad& f\otimes v &\ \mapsto \ & f(v), \\
\coev_V :&\; \C(q) \to V \otimes V^*,
&\quad& 1 &\ \mapsto\ & \sum_i v_i \otimes v_i^*,\\
\widetilde{\eva}_V :&\; V \otimes V^* \to \C(q),
&\quad& v\otimes f &\ \mapsto\ & f(gv), \\
\widetilde{\coev}_V :&\; \C(q) \to V^* \otimes V,
&\quad& 1 &\ \mapsto\ & \sum_i v_i^* \otimes g^{-1}v_i,
\end{alignedat}
\end{equation*}
which are morphisms in $\Rep$ by definition of the pivotal element. We define $\coev_\pm:=z^{-1}\coev_{V^\pm}$ and $\eva_\pm:=z\cdot\eva_{V^\pm}$, hence they are also morphisms in $\Rep$. In the basis $\{v^\pm_{1},v^\pm_{0},v^\pm_{-1}\}$ they read
\begin{align}\label{eqn: coev}
    \coev_\pm: 1 \mapsto v_1^{\pm} \otimes v_{-1}^{\mp}+z^{-1}v_0^{\pm} \otimes v_0^{\mp}+z^{-2}v_{-1}^{\pm} \otimes v_1^{\mp},
\end{align}
and
\begin{align}\label{eqn: eva}
    \eva_{\pm}(v_1^{\mp} \otimes v_{-1}^{\pm})  = z^2, \qquad
    \eva_{\pm}(v_0^{\mp} \otimes v_0^{\pm}) = z, \qquad
    \eva_{\pm}(v_{-1}^{\mp} \otimes v_1^{\pm})  = 1,
\end{align}
and all other values on basis vectors are zero. 

\vspace{2ex}

We define the morphisms $\mathcal{Y}_{-}^{++}: V^- \to V^+ \otimes V^+$ in $\Rep$ by
\begin{align}\label{eqn: Y}
\mathcal{Y}_{-}^{++} v_1^{-} & =v_1^{+} \otimes v_0^{+}+z^{-1}v_0^{+} \otimes v_1^{+} \notag \\
\mathcal{Y}_{-}^{++} v_0^{-} & =v_1^{+} \otimes v_{-1}^{+}+z^{-1} v_{-1}^{+} \otimes v_1^{+} \\
\mathcal{Y}_{-}^{++} v_{-1}^{-} & =v_0^{+} \otimes v_{-1}^{+}+z^{-1} v_{-1}^{+} \otimes v_0^{+}\notag
\end{align}
\begin{align}
\mathcal{Y}_{+}^{--} v_1^{+} & =v_1^{-} \otimes v_0^{-}+z^{-1}v_0^{-} \otimes v_1^{-} \notag \\
\mathcal{Y}_{+}^{--} v_0^{+} & =v_1^{-} \otimes v_{-1}^{-}+z^{-1} v_{-1}^{-} \otimes v_1^{-} \\
\mathcal{Y}_{+}^{--} v_{-1}^{+} & =v_0^{-} \otimes v_{-1}^{-}+z^{-1} v_{-1}^{-} \otimes v_0^{-}.
\end{align}
It is easy to verify that they are also morphisms in $\Rep$.

With $\id_{V^\pm},\eva_{\pm},\coev_{\pm}$ and $\mathcal Y_{\pm}$ defined, it is a simple exercise to check that $\Gamma$, defined by \eqref{eqn: Gamma object}--\eqref{eqn: Gamma-morphism}, preserves the three Kuperberg relations \cref{fig: Kuperberg relations} and the zigzag identities \eqref{eqn: zigzag in Sp} in $\Sp$; hence it fulfills the conditions of \cref{thm: equivalence of Sp and Rep}, so $\Gamma$ is a well-defined equivalence. For example, the zigzag identities are checked by the equations
\begin{align*}
    (\eva_{\pm}\otimes\id_{V^{\mp}})\circ(\id_{V^{\mp}}\otimes\coev_{\pm})&=\id_{V^{\mp}},\\
    (\id_{V^{\pm}}\otimes\eva_{\pm})\circ(\coev_{\pm}\otimes\id_{V^{\pm}})&=\id_{V^{\pm}}.
\end{align*}

For the square relation, consider \cref{fig: square verification Gamma}. The square on the left-hand side is the concatenation of $\ccap$ and $Y$ webs:
\begin{figure}[htbp]
    \centering
    \begingroup
    \scalebox{0.8}{\begin{tikzpicture}[scale=0.65, line width=0.7pt]
    \fill[gray!12]
        (0.7,0.7) -- (1.5,1.2)
        to[out=90,in=90] (2.5,1.2) -- (3.3,0.7)
        -- (4,1.2) -- (4,1.8) -- (3.3,2.3)
        -- (2.5,2.8) to[out=90,in=90] (1.5,2.8)
        -- (0.7,2.3) -- (0,1.8) -- (0,1.2) -- cycle;
    \draw[mid arrow] (0,0) -- (0.7,0.7);
    \draw[mid arrow] (3.3,0.7) -- (4,0);
    \draw[mid arrow] (3.3,0.7) -- (2.5,1.2);
    \draw[mid arrow] (2.5,1.2) to[out=90,in=90] (1.5,1.2);
    \draw[mid arrow] (1.5,1.2) -- (0.7,0.7);
    \draw[mid arrow] (0.7,2.3) -- (0,1.8) -- (0,1.2) -- (0.7,0.7);
    \draw[mid arrow] (3.3,0.7) -- (4,1.2) -- (4,1.8) -- (3.3,2.3);
    \draw[mid arrow] (0.7,2.3) -- (1.5,2.8);
    \draw[mid arrow] (1.5,2.8) to[out=90,in=90] (2.5,2.8);
    \draw[mid arrow] (2.5,2.8) -- (3.3,2.3);
    \draw[mid arrow] (0.7,2.3) -- (0,2.8) -- (0,3.8);
    \draw[mid arrow] (4,3.8) -- (4,2.8) -- (3.3,2.3);
    \foreach \x/\y in {0.7/0.7,3.3/0.7,0.7/2.3,3.3/2.3}
        \fill (\x,\y) circle (2pt);
    \node[below] at (0,0) {$+$};
    \node[below] at (4,0) {$-$};
    \node[above] at (0,3.8) {$+$};
    \node[above] at (4,3.8) {$-$};
    \node at (5.1,1.9) {$=$};
    \begin{scope}[shift={(6.2,0)}]
        \draw[mid arrow] (0,0) -- (0,3.8);
        \draw[mid arrow] (2,3.8) -- (2,0);
        \node[below] at (0,0) {$+$};
        \node[below] at (2,0) {$-$};
        \node[above] at (0,3.8) {$+$};
        \node[above] at (2,3.8) {$-$};
    \end{scope}
    \node at (9.3,1.9) {$+$};
    \begin{scope}[shift={(10.4,0)}]
        \draw[mid arrow] (0,0) to[out=90,in=90] (2,0);
        \draw[mid arrow] (2,3.8) to[out=270,in=270] (0,3.8);
        \node[below] at (0,0) {$+$};
        \node[below] at (2,0) {$-$};
        \node[above] at (0,3.8) {$+$};
        \node[above] at (2,3.8) {$-$};
    \end{scope}
\end{tikzpicture}}
    \endgroup
    \caption{The square relation.}
    \label{fig: square verification Gamma}
\end{figure}
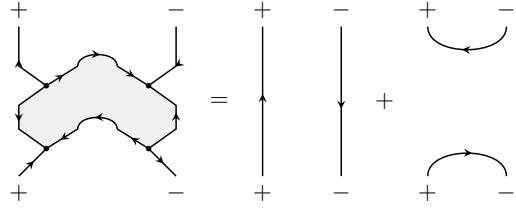

Applying $\Gamma$ to \cref{fig: square verification Gamma}, we obtain: 
\begin{align}\label{eqn: square verification Gamma}
    & (\id_{V^{\pm}}\otimes\eva_{\mp}\otimes\id_{V^{\mp}})
    \circ(\mathcal Y_{\mp}\otimes\mathcal Y_{\pm}) \circ(\id_{V^{\mp}}\otimes\eva_{\pm}\otimes\id_{V^{\pm}})
    \circ(\mathcal Y_{\pm}\otimes\mathcal Y_{\mp})\notag\\
    =&\id_{V^{\pm}}\otimes\id_{V^{\mp}}+\coev_{\pm}\circ\eva_{\mp}.
\end{align}
Substituting \eqref{eqn: coev}--\eqref{eqn: Y} verifies \eqref{eqn: square verification Gamma} on each basis vector $v_i^{\pm}\otimes v_j^{\mp}$, hence $\Gamma$ indeed preserves the square relation. The loop and digon Kuperberg relations are verified similarly.

\subsection{Matrix forms of generators of $\mathcal K_L$}\label{app: matrix form of generators}
This subsection records the matrix form of the generators of $\mathcal K_L$.
For a two-strand web $X$ with bottom sign string $S$ and top sign string $T$, the corresponding operator on $\mathscr V_2=(V^+\oplus V^-)^{\otimes 2}$ has a $36\times36$ matrix. By \eqref{eqn: def of Phi}, it is
\begin{align*}
    \Phi(X)=e_T\circ\Gamma(X)\circ p_S,
\end{align*}
where $p_S$ projects onto the $V^S$-summand and $e_T$ includes the $V^T$-summand. Hence the only potentially nonzero block is the $(T,S)$-block $\Gamma(X)$, of size $9\times9$. We give these blocks below; the corresponding generators of $\mathcal K_L$ act on the indicated consecutive sites and as the identity on the remaining sites. For the three-strand generator $\YL{s_1}{s_2}$, the full local matrix has size $216\times216$; its $27\times27$ block is given at the end of this subsection as a tensor product.

Recall the web $\Ydual^{\pm}$ defined in \cref{eqn: definition of Lambda web}. Note that $I^\pm=Y_{\mp}\circ \Ydual^{\mp}$ and $H^\pm=(\Ydual^{\mp}\otimes \id_\pm)\circ (\id_{\pm}\otimes Y_{\mp})$. Note also $\YL{s_1}{s_2}=Y\otimes\Ydual$. Since $\Gamma(Y_{\pm})=\mathcal Y_{\pm}$ has matrix form given in \eqref{eqn: Y}, it remains to compute the matrix form of $\Gamma(\Ydual^\pm)$. By its definition, we give the following formula:
\begin{align}\label{eqn: lambda}
    \Ydual^{-}(v_1^{+}\otimes v_0^{+})&=zv_1^{-},
    &\Ydual^{-}(v_0^{+}\otimes v_1^{+})&=v_1^{-},\\
    \Ydual^{-}(v_1^{+}\otimes v_{-1}^{+})&=zv_0^{-},
    &\Ydual^{-}(v_{-1}^{+}\otimes v_1^{+})&=v_0^{-},\\
    \Ydual^{-}(v_0^{+}\otimes v_{-1}^{+})&=zv_{-1}^{-},
    &\Ydual^{-}(v_{-1}^{+}\otimes v_0^{+})&=v_{-1}^{-},
\end{align}
and the matrix form of $\Gamma(\Ydual^{+})$ is obtained by reversing all the signs in \eqref{eqn: lambda}.

Now we give the matrix form of $\Gamma(I^\pm)$ and $\Gamma(H^\pm)$. 
The morphism $\Gamma(I^+): V^+ \otimes V^+ \to V^+ \otimes V^+$ takes the following matrix form in the basis $\{v^{+}_{i} \otimes v^{+}_{j} \st i,j=1,0,-1\}$, ordered lexicographically (i.e. $v^+_1 \otimes v^+_1$ is the first basis vector, $v^+_1 \otimes v^+_0$ is the second, and so on):

\begin{equation}\label{mat: spider generator}
    \Gamma(I^+)=\left[
        \begin{NiceMatrix}
        0 & & & & & & & & \\
        & z& & 1& & & & & \\
        & & z& & & & 1& & \\
        & 1 & & z^{-1}& & & & & \\
        & & & & 0& & & & \\
        & & & & & z& & 1& \\
        & & 1 & & & & z^{-1}& & \\
        & & & & & 1& &z^{-1}& \\
        & & & & & & & & 0\\
    \end{NiceMatrix}
    \right]
\end{equation}

The morphism $\Gamma(H^+): V^+ \otimes V^- \to V^- \otimes V^+$ takes the following matrix form in the basis $\{v^{\pm}_{i} \otimes v^{\mp}_{j} \st i,j=1,0,-1\}$, ordered lexicographically (i.e. $v^+_1 \otimes v^-_1$ is the first basis vector, $v^+_1 \otimes v^-_0$ is the second, and so on):

\begin{equation}\label{H matrix}
\Gamma(H^{+})=\left[
    \begin{NiceMatrix}
    1 & & & & & & & & \\
    & & & 1& & & & & \\
    & & z& & 1& & & & \\
    & 1 & & & & & & & \\
    & & 1& & & & 1& & \\
    & & & & & & & 1& \\
    & & & & 1 & & z^{-1}& & \\
    & & & & & 1 & & & \\
    & & & & & & & & 1\\
\end{NiceMatrix}
\right]
\end{equation}

For completeness, we also give the matrix form of the idempotent generator. For $s_1,s_2\in\{+,-\}$, $\Gamma(\id_{s_1s_2})=\id_{V^{s_1}\otimes V^{s_2}}=I_9$. Its image under $\Phi$ vanishes on the summands $V^{s_3}\otimes V^{s_4}$ whenever $(s_1,s_2)\neq (s_3,s_4)$.

Finally, the generator 
$\YL{s_1}{s_2}=Y\otimes\Ydual$ on three consecutive tensor factors is the tensor product of two maps:
\begin{align}
    \Gamma\!\left(\YL{s_1}{s_2}\right)
    =
    \Gamma(Y)\otimes \Gamma(\Ydual).
\end{align}
The matrix forms of $\Gamma(Y)=\mathcal Y_\pm$ and
$\Gamma(\Ydual)=\overline{\mathcal Y}_\pm$ are given in \eqref{eqn: Y}--\eqref{eqn: lambda}. The matrix form of $\YL{s_1}{s_2}$ is thus obtained by taking the tensor product of these matrices.

\section{Well-definedness of the growth algorithm}\label{sec: well-definedness of the growth algorithm}

These arguments were sketched in \cite{khovanov_1999_web} in the $\lambda=0$ case. We do not claim originality, but would like to provide details for the sake of completeness and convenience of the reader.

\begin{lemma}\label{lem: dominance}
    If $P=P_1P^iP_2\in \Path^+(\lambda)$ and $(w,P^i,P^f)$ is a growth rule, then $P'=P_1P^fP_2\in \Path^+(\lambda)$.
\end{lemma}
\begin{proof}
    That $\wt(P)=\wt(P')$ follows by checking for each growth rule $(w,P^i,P^f)$, $\wt(P^i)=\wt(P^f)$. We then observe that the weight of paths is additive under concatenation, so $\wt(P)=\wt(P_1)+\wt(P^i)+\wt(P_2)$ and $\wt(P')=\wt(P_1)+\wt(P^f)+\wt(P_2)$.

    Now we check that $P'$ is dominant. Since $P$ is dominant by assumption, all its vertices lie in $\Lambda_+$. Outside the local replacement, the vertices of $P_1P^fP_2$ agree with vertices of $P_1P^iP_2$: before the replacement this is immediate, and after the replacement it follows from $\wt(P^i)=\wt(P^f)$. Thus the only vertices which require attention are the proper intermediate vertices of $P^f$. If $|P^f|\leq 1$, there are no such vertices. If $|P^f|=2$, there is exactly one such vertex, and the following local picture illustrates the check.

    Consider the growth rule $P^i=(u_{+,1}u_{-,0})$ and $P^f=(u_{-,0}u_{+,1})$. Let $\mu=\wt(P_1)=n\lambda_1+k\lambda_2$, $n,k\geq 0$. The local parts of $P$ and $P'$ have the same endpoint, but they pass through different intermediate weights:
    
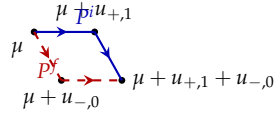
\begin{center}
    \begin{tikzpicture}[
        scale=0.8,
        x={(1cm,0cm)},
        y={(0.55cm,0.8cm)},
        every node/.style={font=\scriptsize}
    ]
        \coordinate (mu) at (0,0);
        \coordinate (oldmid) at (1,0);
        \coordinate (newmid) at (1,-1);
        \coordinate (end) at (2,-1);

        \fill (mu) circle (1.6pt) node[below left] {$\mu$};
        \fill (oldmid) circle (1.6pt) node[above] {$\mu+u_{+,1}$};
        \fill (newmid) circle (1.6pt) node[below] {$\mu+u_{-,0}$};
        \fill (end) circle (1.6pt) node[right] {$\mu+u_{+,1}+u_{-,0}$};

        \draw[blue!70!black, thick, style=mid arrow] (mu) -- (oldmid);
        \draw[blue!70!black, thick, style=mid arrow] (oldmid) -- (end);
        \draw[red!70!black, thick, dashed, style=mid arrow] (mu) -- (newmid);
        \draw[red!70!black, thick, dashed, style=mid arrow] (newmid) -- (end);

        \node[blue!70!black] at (0.65,0.35) {$P^i$};
        \node[red!70!black] at (0.65,-0.75) {$P^f$};
    \end{tikzpicture}
    \end{center}
    The weight $\mu'$ of $P'$ that is not visited by $P$ is obtained by adding the first step of $P^f$ to $\mu$:
    \begin{align*}
    \mu' = \mu + u_{-,0} = (n+1)\lambda_1+(k-1)\lambda_2.
    \end{align*}
    Thus $\mu'\in \Lambda_+$ if $k-1\geq 0$. By dominance of $P$, we know that
    \begin{align*}
        \wt(P_1)+\wt(P^i) = \mu+ u_{+,1} +u_{-,0} = (n+2)\lambda_1 + (k-1)\lambda_2 \in \Lambda_+,
    \end{align*}
    in particular $k-1\geq 0$ by definition of $\Lambda_+$. Therefore $\mu' \in \Lambda_+$ as well. The remaining growth rules with $|P^f|=2$ are checked in the same way: the only new intermediate vertex has any potentially negative coordinate controlled by a vertex already visited by the dominant path $P$.
\end{proof}

\begin{proof}[Proof of \Cref{prp: extra relations}]
    To show that the rules in \cref{fig: growth rules} are indeed growth rules, we consider an algorithm $G'$ with both the defining growth rules in \cref{fig: defining growth rules} and \cref{fig: growth rules}. We claim that it suffices to show that, for all $P\in\Path^+(\lambda)$, the resulting web $G'(P)$ does not depend on the choice of growth rules. Indeed, since we have the defining rules in this algorithm, \cref{prp: growth algorithm is well-defined} implies that if $G'(P)$ is well-defined, then $G'(P)=G(P)$ because we may decide to use only the defining rules in the growth process of $G'$.

    The claim is proved by contradiction: Let $P$ be the minimally ordered path in $\Path^+(\lambda)$ such that $G'(P)$ is inconsistently defined, so there must be two growth rules $r_1,r_2$ applicable to $P$, such that different choices of growth rules yield different results. The paths obtained by applying $r_1$ and $r_2$ to $P$ have strictly lower order than $P$, so by the inductive hypothesis, the resulting webs are well-defined.

    Consider the steps of $P$ that are involved in the growth rules $r_1,r_2$: The steps for the two rules must overlap, otherwise we can apply the two growth rules in either order and get the same result, contradicting the assumption that $G'(P)$ is inconsistently defined. On the other hand, the steps for the two rules cannot be the same, since there is at most one growth rule applicable to a fixed pair of steps. Therefore the steps for the two growth rules $r_1,r_2$ must overlap by one step, giving a finite list of possibilities:
\begin{center}\begin{tabular}{c|c|c}
    involved steps \\
    \hline
    $(u_{\pm,1}u_{\mp,0}u_{\pm,0})$\\
    $(u_{\pm,1}u_{\pm,0}u_{\mp,0})$\\
    $(u_{\pm,1}u_{\pm,0}u_{\pm,-1})$\\
    $(u_{\pm,1}u_{\mp,0}u_{\pm,-1})$\\
    $(u_{\pm,1}u_{\pm,0}u_{\mp,-1})$\\
    $(u_{\pm,1}u_{\mp,0}u_{\mp,-1})$\\
    $(u_{\pm,0}u_{\mp,0}u_{\pm,0})$\\
    $(u_{\pm,0}u_{\mp,0}u_{\pm,-1})$\\
    $(u_{\pm,0}u_{\mp,0}u_{\mp,-1})$
\end{tabular}\end{center}

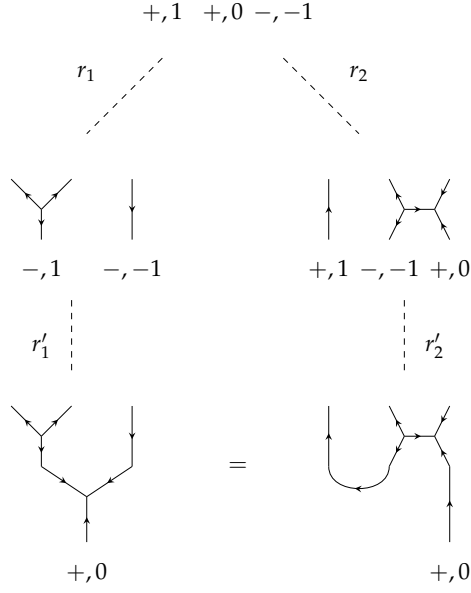
\begin{figure}[htbp]
    \centering
    \scalebox{0.8}{\begin{tikzpicture}
	\begin{pgfonlayer}{nodelayer}
		\node [style=none] (3) at (-1, 3) {$+,1$};
		\node [style=none] (4) at (0, 3) {$+,0$};
		\node [style=none] (5) at (1, 3) {$-,-1$};
		\node [style=none] (6) at (-3.5, 0.25) {};
		\node [style=none] (7) at (-2.5, 0.25) {};
		\node [style=none] (8) at (-1.5, 0.25) {};
		\node [style=none] (9) at (-1, 2.25) {};
		\node [style=none] (10) at (-2.25, 1) {};
		\node [style=none] (11) at (-3, -0.25) {};
		\node [style=none] (12) at (-3, -0.75) {};
		\node [style=none] (13) at (-1.5, -0.75) {};
		\node [style=none] (14) at (1.75, 0.25) {};
		\node [style=none] (15) at (2.75, 0.25) {};
		\node [style=none] (16) at (3.75, 0.25) {};
		\node [style=none] (17) at (3, -0.25) {};
		\node [style=none] (18) at (3.5, -0.25) {};
		\node [style=none] (19) at (2.75, -0.75) {};
		\node [style=none] (20) at (3.75, -0.75) {};
		\node [style=none] (21) at (1.75, -0.75) {};
		\node [style=none] (22) at (1, 2.25) {};
		\node [style=none] (23) at (2.25, 1) {};
		\node [style=none] (24) at (3.75, -1.25) {$+,0$};
		\node [style=none] (25) at (2.75, -1.25) {$-,-1$};
		\node [style=none] (26) at (1.75, -1.25) {$+,1$};
		\node [style=none] (27) at (-3, -1.25) {$-,1$};
		\node [style=none] (28) at (-2.25, 2) {$r_1$};
		\node [style=none] (29) at (2.25, 2) {$r_2$};
		\node [style=none] (30) at (1.75, -3.5) {};
		\node [style=none] (31) at (2.75, -3.5) {};
		\node [style=none] (32) at (3.75, -3.5) {};
		\node [style=none] (33) at (3, -4) {};
		\node [style=none] (34) at (3.5, -4) {};
		\node [style=none] (35) at (2.75, -4.5) {};
		\node [style=none] (36) at (3.75, -4.5) {};
		\node [style=none] (37) at (1.75, -4.5) {};
		\node [style=none] (39) at (-1.5, -1.25) {$-,-1$};
		\node [style=none] (40) at (3.75, -5.75) {};
		\node [style=none] (41) at (3.75, -6.25) {$+,0$};
		\node [style=none] (42) at (-3.5, -3.5) {};
		\node [style=none] (43) at (-2.5, -3.5) {};
		\node [style=none] (44) at (-1.5, -3.5) {};
		\node [style=none] (45) at (-3, -4) {};
		\node [style=none] (46) at (-3, -4.5) {};
		\node [style=none] (47) at (-1.5, -4.5) {};
		\node [style=none] (48) at (-2.25, -5) {};
		\node [style=none] (49) at (-2.25, -5.75) {};
		\node [style=none] (50) at (-2.25, -6.25) {$+,0$};
		\node [style=none] (51) at (-2.5, -1.75) {};
		\node [style=none] (52) at (-2.5, -3) {};
		\node [style=none] (53) at (3, -1.75) {};
		\node [style=none] (54) at (3, -3) {};
		\node [style=none] (55) at (-3, -2.5) {$r_1'$};
		\node [style=none] (56) at (3.5, -2.5) {$r_2'$};
		\node [style=none] (57) at (0.25, -4.5) {$=$};
	\end{pgfonlayer}
	\begin{pgfonlayer}{edgelayer}
		\draw [style=mid arrow] (11.center) to (6.center);
		\draw [style=mid arrow] (11.center) to (7.center);
		\draw [style=mid arrow] (11.center) to (12.center);
		\draw [style=mid arrow] (8.center) to (13.center);
		\draw [style=mid arrow] (21.center) to (14.center);
		\draw [style=mid arrow] (17.center) to (15.center);
		\draw [style=mid arrow] (17.center) to (18.center);
		\draw [style=mid arrow] (17.center) to (19.center);
		\draw [style=mid arrow] (16.center) to (18.center);
		\draw [style=mid arrow] (20.center) to (18.center);
		\draw [style=mid arrow] (37.center) to (30.center);
		\draw [style=mid arrow] (33.center) to (31.center);
		\draw [style=mid arrow] (33.center) to (34.center);
		\draw [style=mid arrow] (33.center) to (35.center);
		\draw [style=mid arrow] (32.center) to (34.center);
		\draw [style=mid arrow] (36.center) to (34.center);
		\draw [style=mid arrow, bend left=90, looseness=1.25] (35.center) to (37.center);
		\draw [style=mid arrow] (40.center) to (36.center);
		\draw [style=mid arrow] (45.center) to (42.center);
		\draw [style=mid arrow] (45.center) to (43.center);
		\draw [style=mid arrow] (45.center) to (46.center);
		\draw [style=mid arrow] (44.center) to (47.center);
		\draw [style=mid arrow] (46.center) to (48.center);
		\draw [style=mid arrow] (47.center) to (48.center);
		\draw [style=mid arrow] (49.center) to (48.center);
		\draw [style=dotted] (9.center) to (10.center);
		\draw [style=dotted] (22.center) to (23.center);
		\draw [style=dotted] (51.center) to (52.center);
		\draw [style=dotted] (53.center) to (54.center);
	\end{pgfonlayer}
\end{tikzpicture}}
    \caption{Different choices of growth rules end up with the same web.}
    \label{fig: diamond}
\end{figure}

In each case, a growth rule is applicable to either the first two steps or the last two steps. However, no inconsistency arises because whichever growth rule is applied first, a suitable choice of complementary growth rule can be applied to bring the two choices back to the same web, which completes the proof. We demonstrate the existence of the complementary growth rules via an explicit example: the sequence $( u_{+,1}u_{+,0}u_{-,-1} )$. On the first two steps, the applicable growth rule is 
\begin{align*}
    r_1=(Y_{-}, (u_{+,1}u_{+,0}), (u_{-,1})),
\end{align*}
and on the last two steps, the applicable growth rule is 
\begin{align*}
    r_2=(H^{-}, (u_{+,0}u_{-,-1}), (u_{-,-1}u_{+,0}))
\end{align*}
The complementary growth rules are
\begin{align*}
    r_1'&=(Y_{+}, (u_{-,1}u_{-,0}), (u_{+,1})), \\
    r_2'&=(\ccup_{+}, ((u_{+,0}u_{-,-1}), (u_{-,-1}u_{+,0})),\emptyset).
\end{align*}
Applying either $r_1$ and $r_1'$, or $r_2$ and $r_2'$ gives the same web (by isotopy invariance) as shown in \cref{fig: diamond}. The remaining cases are similar and are omitted.
\end{proof}

\clearpage
\section{Cell modules for $\mathcal K_4$}\label{app: cell module basis vectors}

In this appendix, we list basis vectors for all the cell modules of $\mathcal K_4$. These modules $\Delta_4(\lambda)$ are indexed by $\lambda \in \Lambda_+(4)$. We write $\lambda=(x,y)$, and list cell modules with $x\geq y$. Bases for $\Delta_4(x,y)$ with $x<y$ are obtained by reversing the orientations of all webs in the basis for $\Delta_4(y,x)$.

The basis vectors are organized in rows and columns. To the left of each row, we indicate the top sign string of the vectors in this row.
The catalogue in \cref{fig: cell modules K4 plate} is arranged over two pages, with all webs shown at the same scale.

\begingroup
\captionsetup[subfigure]{font=small,skip=6pt,position=top}

\begin{figure}[!ht]
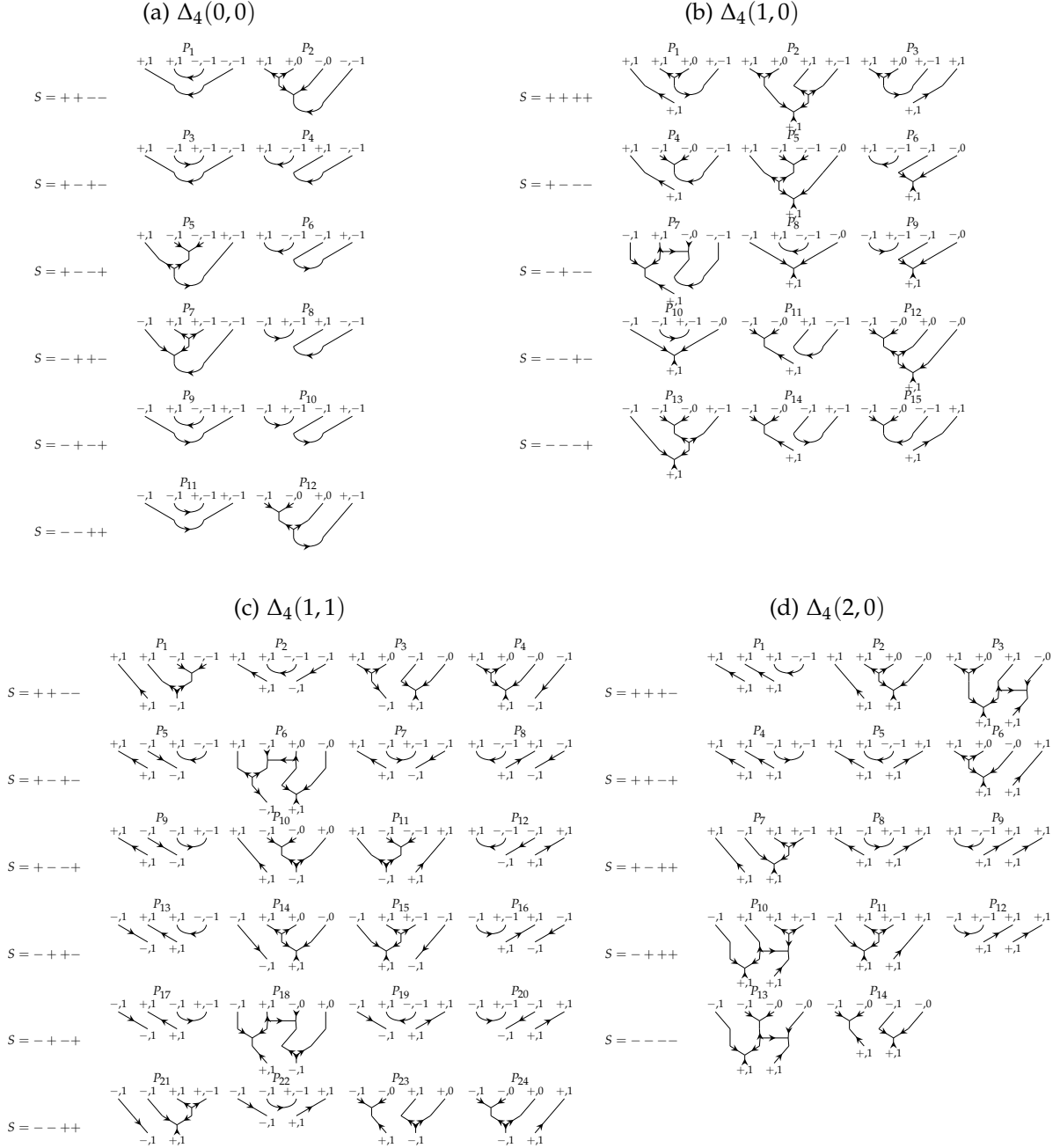

    \centering
    \begin{subfigure}[t]{0.38\linewidth}
        \centering
        \caption{$\Delta_4(0,0)$}
        \label{fig: cell module 4 00}
        \scalebox{0.92}{
\tikzset{
    style/.style={#1},
    none/.style={},
    mid arrow/.style={postaction={decorate,decoration={
            markings, mark=at position 0.55 with {\arrow{stealth}}
        }}
    },
    dotted/.style={dashed}
}
\begin{tikzpicture}[
    x=1cm, y=1cm,
    every node/.style={font=\tiny, inner sep=0.5pt},
    labelnode/.style={font=\tiny, inner sep=0.3pt}
]
\node[anchor=east] at (-1.3, -0.47) {$S=++--$};
\begin{scope}[shift={(0,0)}]
    \node at (0,0.32) {$P_{1}$};
    \node[labelnode] at (-0.72,0.12) {$ \scriptscriptstyle +,1 $};
    \node[labelnode] at (-0.24,0.12) {$ \scriptscriptstyle +,1 $};
    \node[labelnode] at (0.24,0.12) {$ \scriptscriptstyle -,-1 $};
    \node[labelnode] at (0.72,0.12) {$ \scriptscriptstyle -,-1 $};
    \draw [style=none] (-0.24,-0.28) -- (-0.72,0);
    \draw [style=none] (0.72,0) -- (0.24,-0.28);
    \draw [style=mid arrow, looseness=1.05, bend left=82] (0.24,0) to (-0.24,0);
    \draw [style=mid arrow, looseness=1.05, bend left=82] (0.24,-0.28) to (-0.24,-0.28);
\end{scope}
\begin{scope}[shift={(1.95,0)}]
    \node at (0,0.32) {$P_{2}$};
    \node[labelnode] at (-0.72,0.12) {$ \scriptscriptstyle +,1 $};
    \node[labelnode] at (-0.24,0.12) {$ \scriptscriptstyle +,0 $};
    \node[labelnode] at (0.24,0.12) {$ \scriptscriptstyle -,0 $};
    \node[labelnode] at (0.72,0.12) {$ \scriptscriptstyle -,-1 $};
    \draw [style=none] (0.24,0) -- (0,-0.28);
    \draw [style=none] (0.72,0) -- (0.48,-0.28);
    \draw [style=mid arrow] (-0.48,-0.148) -- (-0.72,0);
    \draw [style=mid arrow] (-0.48,-0.148) -- (-0.24,0);
    \draw [style=none] (-0.48,-0.148) -- (-0.48,-0.28);
    \draw [style=none] (0.48,-0.28) -- (0.24,-0.56);
    \draw [style=mid arrow] (-0.48,-0.28) -- (-0.24,-0.428);
    \draw [style=mid arrow] (0,-0.28) -- (-0.24,-0.428);
    \draw [style=none] (-0.24,-0.56) -- (-0.24,-0.428);
    \draw [style=mid arrow, looseness=1.05, bend left=82] (0.24,-0.56) to (-0.24,-0.56);
\end{scope}
\node[anchor=east] at (-1.3, -1.89) {$S=+-+-$};
\begin{scope}[shift={(0,-1.42)}]
    \node at (0,0.32) {$P_{3}$};
    \node[labelnode] at (-0.72,0.12) {$ \scriptscriptstyle +,1 $};
    \node[labelnode] at (-0.24,0.12) {$ \scriptscriptstyle -,1 $};
    \node[labelnode] at (0.24,0.12) {$ \scriptscriptstyle +,-1 $};
    \node[labelnode] at (0.72,0.12) {$ \scriptscriptstyle -,-1 $};
    \draw [style=none] (-0.24,-0.28) -- (-0.72,0);
    \draw [style=none] (0.72,0) -- (0.24,-0.28);
    \draw [style=mid arrow, looseness=1.05, bend right=82] (-0.24,0) to (0.24,0);
    \draw [style=mid arrow, looseness=1.05, bend left=82] (0.24,-0.28) to (-0.24,-0.28);
\end{scope}
\begin{scope}[shift={(1.95,-1.42)}]
    \node at (0,0.32) {$P_{4}$};
    \node[labelnode] at (-0.72,0.12) {$ \scriptscriptstyle +,1 $};
    \node[labelnode] at (-0.24,0.12) {$ \scriptscriptstyle -,-1 $};
    \node[labelnode] at (0.24,0.12) {$ \scriptscriptstyle +,1 $};
    \node[labelnode] at (0.72,0.12) {$ \scriptscriptstyle -,-1 $};
    \draw [style=none] (-0.24,-0.28) -- (0.24,0);
    \draw [style=none] (0.72,0) -- (0.24,-0.28);
    \draw [style=mid arrow, looseness=1.05, bend left=82] (-0.24,0) to (-0.72,0);
    \draw [style=mid arrow, looseness=1.05, bend left=82] (0.24,-0.28) to (-0.24,-0.28);
\end{scope}
\node[anchor=east] at (-1.3, -3.31) {$S=+--+$};
\begin{scope}[shift={(0,-2.84)}]
    \node at (0,0.32) {$P_{5}$};
    \node[labelnode] at (-0.72,0.12) {$ \scriptscriptstyle +,1 $};
    \node[labelnode] at (-0.24,0.12) {$ \scriptscriptstyle -,1 $};
    \node[labelnode] at (0.24,0.12) {$ \scriptscriptstyle -,-1 $};
    \node[labelnode] at (0.72,0.12) {$ \scriptscriptstyle +,-1 $};
    \draw [style=none] (-0.48,-0.28) -- (-0.72,0);
    \draw [style=none] (0.48,-0.28) -- (0.72,0);
    \draw [style=mid arrow] (-0.24,0) -- (0,-0.148);
    \draw [style=mid arrow] (0.24,0) -- (0,-0.148);
    \draw [style=none] (0,-0.28) -- (0,-0.148);
    \draw [style=none] (0.24,-0.56) -- (0.48,-0.28);
    \draw [style=mid arrow] (-0.24,-0.428) -- (-0.48,-0.28);
    \draw [style=mid arrow] (-0.24,-0.428) -- (0,-0.28);
    \draw [style=none] (-0.24,-0.428) -- (-0.24,-0.56);
    \draw [style=mid arrow, looseness=1.05, bend right=82] (-0.24,-0.56) to (0.24,-0.56);
\end{scope}
\begin{scope}[shift={(1.95,-2.84)}]
    \node at (0,0.32) {$P_{6}$};
    \node[labelnode] at (-0.72,0.12) {$ \scriptscriptstyle +,1 $};
    \node[labelnode] at (-0.24,0.12) {$ \scriptscriptstyle -,-1 $};
    \node[labelnode] at (0.24,0.12) {$ \scriptscriptstyle -,1 $};
    \node[labelnode] at (0.72,0.12) {$ \scriptscriptstyle +,-1 $};
    \draw [style=none] (0.24,0) -- (-0.24,-0.28);
    \draw [style=none] (0.24,-0.28) -- (0.72,0);
    \draw [style=mid arrow, looseness=1.05, bend left=82] (-0.24,0) to (-0.72,0);
    \draw [style=mid arrow, looseness=1.05, bend right=82] (-0.24,-0.28) to (0.24,-0.28);
\end{scope}
\node[anchor=east] at (-1.3, -4.73) {$S=-++-$};
\begin{scope}[shift={(0,-4.26)}]
    \node at (0,0.32) {$P_{7}$};
    \node[labelnode] at (-0.72,0.12) {$ \scriptscriptstyle -,1 $};
    \node[labelnode] at (-0.24,0.12) {$ \scriptscriptstyle +,1 $};
    \node[labelnode] at (0.24,0.12) {$ \scriptscriptstyle +,-1 $};
    \node[labelnode] at (0.72,0.12) {$ \scriptscriptstyle -,-1 $};
    \draw [style=none] (-0.72,0) -- (-0.48,-0.28);
    \draw [style=none] (0.72,0) -- (0.48,-0.28);
    \draw [style=mid arrow] (0,-0.148) -- (-0.24,0);
    \draw [style=mid arrow] (0,-0.148) -- (0.24,0);
    \draw [style=none] (0,-0.148) -- (0,-0.28);
    \draw [style=none] (0.48,-0.28) -- (0.24,-0.56);
    \draw [style=mid arrow] (-0.48,-0.28) -- (-0.24,-0.428);
    \draw [style=mid arrow] (0,-0.28) -- (-0.24,-0.428);
    \draw [style=none] (-0.24,-0.56) -- (-0.24,-0.428);
    \draw [style=mid arrow, looseness=1.05, bend left=82] (0.24,-0.56) to (-0.24,-0.56);
\end{scope}
\begin{scope}[shift={(1.95,-4.26)}]
    \node at (0,0.32) {$P_{8}$};
    \node[labelnode] at (-0.72,0.12) {$ \scriptscriptstyle -,1 $};
    \node[labelnode] at (-0.24,0.12) {$ \scriptscriptstyle +,-1 $};
    \node[labelnode] at (0.24,0.12) {$ \scriptscriptstyle +,1 $};
    \node[labelnode] at (0.72,0.12) {$ \scriptscriptstyle -,-1 $};
    \draw [style=none] (-0.24,-0.28) -- (0.24,0);
    \draw [style=none] (0.72,0) -- (0.24,-0.28);
    \draw [style=mid arrow, looseness=1.05, bend right=82] (-0.72,0) to (-0.24,0);
    \draw [style=mid arrow, looseness=1.05, bend left=82] (0.24,-0.28) to (-0.24,-0.28);
\end{scope}
\node[anchor=east] at (-1.3, -6.15) {$S=-+-+$};
\begin{scope}[shift={(0,-5.68)}]
    \node at (0,0.32) {$P_{9}$};
    \node[labelnode] at (-0.72,0.12) {$ \scriptscriptstyle -,1 $};
    \node[labelnode] at (-0.24,0.12) {$ \scriptscriptstyle +,1 $};
    \node[labelnode] at (0.24,0.12) {$ \scriptscriptstyle -,-1 $};
    \node[labelnode] at (0.72,0.12) {$ \scriptscriptstyle +,-1 $};
    \draw [style=none] (-0.72,0) -- (-0.24,-0.28);
    \draw [style=none] (0.24,-0.28) -- (0.72,0);
    \draw [style=mid arrow, looseness=1.05, bend left=82] (0.24,0) to (-0.24,0);
    \draw [style=mid arrow, looseness=1.05, bend right=82] (-0.24,-0.28) to (0.24,-0.28);
\end{scope}
\begin{scope}[shift={(1.95,-5.68)}]
    \node at (0,0.32) {$P_{10}$};
    \node[labelnode] at (-0.72,0.12) {$ \scriptscriptstyle -,1 $};
    \node[labelnode] at (-0.24,0.12) {$ \scriptscriptstyle +,-1 $};
    \node[labelnode] at (0.24,0.12) {$ \scriptscriptstyle -,1 $};
    \node[labelnode] at (0.72,0.12) {$ \scriptscriptstyle +,-1 $};
    \draw [style=none] (0.24,0) -- (-0.24,-0.28);
    \draw [style=none] (0.24,-0.28) -- (0.72,0);
    \draw [style=mid arrow, looseness=1.05, bend right=82] (-0.72,0) to (-0.24,0);
    \draw [style=mid arrow, looseness=1.05, bend right=82] (-0.24,-0.28) to (0.24,-0.28);
\end{scope}
\node[anchor=east] at (-1.3, -7.57) {$S=--++$};
\begin{scope}[shift={(0,-7.1)}]
    \node at (0,0.32) {$P_{11}$};
    \node[labelnode] at (-0.72,0.12) {$ \scriptscriptstyle -,1 $};
    \node[labelnode] at (-0.24,0.12) {$ \scriptscriptstyle -,1 $};
    \node[labelnode] at (0.24,0.12) {$ \scriptscriptstyle +,-1 $};
    \node[labelnode] at (0.72,0.12) {$ \scriptscriptstyle +,-1 $};
    \draw [style=none] (-0.72,0) -- (-0.24,-0.28);
    \draw [style=none] (0.24,-0.28) -- (0.72,0);
    \draw [style=mid arrow, looseness=1.05, bend right=82] (-0.24,0) to (0.24,0);
    \draw [style=mid arrow, looseness=1.05, bend right=82] (-0.24,-0.28) to (0.24,-0.28);
\end{scope}
\begin{scope}[shift={(1.95,-7.1)}]
    \node at (0,0.32) {$P_{12}$};
    \node[labelnode] at (-0.72,0.12) {$ \scriptscriptstyle -,1 $};
    \node[labelnode] at (-0.24,0.12) {$ \scriptscriptstyle -,0 $};
    \node[labelnode] at (0.24,0.12) {$ \scriptscriptstyle +,0 $};
    \node[labelnode] at (0.72,0.12) {$ \scriptscriptstyle +,-1 $};
    \draw [style=none] (0,-0.28) -- (0.24,0);
    \draw [style=none] (0.48,-0.28) -- (0.72,0);
    \draw [style=mid arrow] (-0.72,0) -- (-0.48,-0.148);
    \draw [style=mid arrow] (-0.24,0) -- (-0.48,-0.148);
    \draw [style=none] (-0.48,-0.28) -- (-0.48,-0.148);
    \draw [style=none] (0.24,-0.56) -- (0.48,-0.28);
    \draw [style=mid arrow] (-0.24,-0.428) -- (-0.48,-0.28);
    \draw [style=mid arrow] (-0.24,-0.428) -- (0,-0.28);
    \draw [style=none] (-0.24,-0.428) -- (-0.24,-0.56);
    \draw [style=mid arrow, looseness=1.05, bend right=82] (-0.24,-0.56) to (0.24,-0.56);
\end{scope}
\end{tikzpicture}}
    \end{subfigure}\hfill
    \begin{subfigure}[t]{0.60\linewidth}
        \centering
        \caption{$\Delta_4(1,0)$}
        \label{fig: cell module 4 10}
        \scalebox{0.92}{\input{tikz_drawings/kuperberg_cell_basis_L4_10.tikz}}
    \end{subfigure}

    \vspace{0.65cm}

    \begin{subfigure}[t]{0.55\linewidth}
        \centering
        \caption{$\Delta_4(1,1)$}
        \label{fig: cell module 4 11}
        \scalebox{0.92}{\input{tikz_drawings/kuperberg_cell_basis_L4_11.tikz}}
    \end{subfigure}\hfill
    \begin{subfigure}[t]{0.44\linewidth}
        \centering
        \caption{$\Delta_4(2,0)$}
        \label{fig: cell module 4 20}
        \scalebox{0.92}{\input{tikz_drawings/kuperberg_cell_basis_L4_20.tikz}}
    \end{subfigure}
    \caption{Basis vectors for the cell modules of $\mathcal K_4$ with $x\geq y$ (first part).
    The top sign string of each row is displayed to its left.}
    \label{fig: cell modules K4 plate}
\end{figure}

\clearpage

\begin{figure}[!ht]
    \ContinuedFloat
    \centering
    \begin{subfigure}[t]{0.38\linewidth}
        \centering
        \caption{$\Delta_4(3,0)$}
        \label{fig: cell module 4 30}
        \scalebox{0.92}{
\tikzset{
    style/.style={#1},
    none/.style={},
    mid arrow/.style={postaction={decorate,decoration={
            markings, mark=at position 0.55 with {\arrow{stealth}}
        }}
    },
    dotted/.style={dashed}
}
\begin{tikzpicture}[
    x=1cm, y=1cm,
    every node/.style={font=\tiny, inner sep=0.5pt},
    labelnode/.style={font=\tiny, inner sep=0.3pt}
]
\node[anchor=east] at (-1.3, -0.47) {$S=++--$};
\begin{scope}[shift={(0,0)}]
    \node at (0,0.32) {$P_{1}$};
    \node[labelnode] at (-0.72,0.12) {$ \scriptscriptstyle +,1 $};
    \node[labelnode] at (-0.24,0.12) {$ \scriptscriptstyle +,1 $};
    \node[labelnode] at (0.24,0.12) {$ \scriptscriptstyle -,1 $};
    \node[labelnode] at (0.72,0.12) {$ \scriptscriptstyle -,0 $};
    \draw [style=mid arrow] (-0.48,-0.28) -- (-0.72,0);
    \draw [style=mid arrow] (0,-0.28) -- (-0.24,0);
    \draw [style=mid arrow] (0.24,0) -- (0.48,-0.148);
    \draw [style=mid arrow] (0.72,0) -- (0.48,-0.148);
    \draw [style=mid arrow] (0.48,-0.28) -- (0.48,-0.148);
    \node[labelnode] at (-0.48,-0.4) {$ \scriptscriptstyle +,1 $};
    \node[labelnode] at (0,-0.4) {$ \scriptscriptstyle +,1 $};
    \node[labelnode] at (0.48,-0.4) {$ \scriptscriptstyle +,1 $};
\end{scope}
\node[anchor=east] at (-1.3, -1.89) {$S=+-+-$};
\begin{scope}[shift={(0,-1.42)}]
    \node at (0,0.32) {$P_{2}$};
    \node[labelnode] at (-0.72,0.12) {$ \scriptscriptstyle +,1 $};
    \node[labelnode] at (-0.24,0.12) {$ \scriptscriptstyle -,1 $};
    \node[labelnode] at (0.24,0.12) {$ \scriptscriptstyle +,1 $};
    \node[labelnode] at (0.72,0.12) {$ \scriptscriptstyle -,0 $};
    \draw [style=none] (-0.72,-0.28) -- (-0.72,0);
    \draw [style=none] (-0.24,0) -- (-0.24,-0.28);
    \draw [style=mid arrow] (0.25,-0.148) -- (0.24,0);
    \draw [style=none] (0.25,-0.148) -- (0.24,-0.28);
    \draw [style=mid arrow] (0.25,-0.148) -- (0.71,-0.148);
    \draw [style=mid arrow] (0.72,0) -- (0.71,-0.148);
    \draw [style=none] (0.72,-0.28) -- (0.71,-0.148);
    \draw [style=mid arrow] (-0.48,-0.56) -- (-0.72,-0.28);
    \draw [style=mid arrow] (0.48,-0.56) -- (0.72,-0.28);
    \draw [style=mid arrow] (-0.24,-0.28) -- (0,-0.428);
    \draw [style=mid arrow] (0.24,-0.28) -- (0,-0.428);
    \draw [style=mid arrow] (0,-0.56) -- (0,-0.428);
    \node[labelnode] at (-0.48,-0.68) {$ \scriptscriptstyle +,1 $};
    \node[labelnode] at (0,-0.68) {$ \scriptscriptstyle +,1 $};
    \node[labelnode] at (0.48,-0.68) {$ \scriptscriptstyle +,1 $};
\end{scope}
\node[anchor=east] at (-1.3, -3.31) {$S=+--+$};
\begin{scope}[shift={(0,-2.84)}]
    \node at (0,0.32) {$P_{3}$};
    \node[labelnode] at (-0.72,0.12) {$ \scriptscriptstyle +,1 $};
    \node[labelnode] at (-0.24,0.12) {$ \scriptscriptstyle -,1 $};
    \node[labelnode] at (0.24,0.12) {$ \scriptscriptstyle -,0 $};
    \node[labelnode] at (0.72,0.12) {$ \scriptscriptstyle +,1 $};
    \draw [style=mid arrow] (-0.48,-0.28) -- (-0.72,0);
    \draw [style=mid arrow] (0.48,-0.28) -- (0.72,0);
    \draw [style=mid arrow] (-0.24,0) -- (0,-0.148);
    \draw [style=mid arrow] (0.24,0) -- (0,-0.148);
    \draw [style=mid arrow] (0,-0.28) -- (0,-0.148);
    \node[labelnode] at (-0.48,-0.4) {$ \scriptscriptstyle +,1 $};
    \node[labelnode] at (0,-0.4) {$ \scriptscriptstyle +,1 $};
    \node[labelnode] at (0.48,-0.4) {$ \scriptscriptstyle +,1 $};
\end{scope}
\node[anchor=east] at (-1.3, -4.73) {$S=-++-$};
\begin{scope}[shift={(0,-4.26)}]
    \node at (0,0.32) {$P_{4}$};
    \node[labelnode] at (-0.72,0.12) {$ \scriptscriptstyle -,1 $};
    \node[labelnode] at (-0.24,0.12) {$ \scriptscriptstyle +,1 $};
    \node[labelnode] at (0.24,0.12) {$ \scriptscriptstyle +,1 $};
    \node[labelnode] at (0.72,0.12) {$ \scriptscriptstyle -,0 $};
    \draw [style=none] (-0.72,0) -- (-0.72,-0.28);
    \draw [style=none] (-0.24,-0.28) -- (-0.24,0);
    \draw [style=mid arrow] (0.25,-0.148) -- (0.24,0);
    \draw [style=none] (0.25,-0.148) -- (0.24,-0.28);
    \draw [style=mid arrow] (0.25,-0.148) -- (0.71,-0.148);
    \draw [style=mid arrow] (0.72,0) -- (0.71,-0.148);
    \draw [style=none] (0.72,-0.28) -- (0.71,-0.148);
    \draw [style=none] (-0.72,-0.28) -- (-0.72,-0.56);
    \draw [style=none] (0.72,-0.56) -- (0.72,-0.28);
    \draw [style=mid arrow] (-0.23,-0.428) -- (-0.24,-0.28);
    \draw [style=none] (-0.23,-0.428) -- (-0.24,-0.56);
    \draw [style=mid arrow] (-0.23,-0.428) -- (0.23,-0.428);
    \draw [style=mid arrow] (0.24,-0.28) -- (0.23,-0.428);
    \draw [style=none] (0.24,-0.56) -- (0.23,-0.428);
    \draw [style=mid arrow] (0,-0.84) -- (0.24,-0.56);
    \draw [style=mid arrow] (0.48,-0.84) -- (0.72,-0.56);
    \draw [style=mid arrow] (-0.72,-0.56) -- (-0.48,-0.708);
    \draw [style=mid arrow] (-0.24,-0.56) -- (-0.48,-0.708);
    \draw [style=mid arrow] (-0.48,-0.84) -- (-0.48,-0.708);
    \node[labelnode] at (-0.48,-0.96) {$ \scriptscriptstyle +,1 $};
    \node[labelnode] at (0,-0.96) {$ \scriptscriptstyle +,1 $};
    \node[labelnode] at (0.48,-0.96) {$ \scriptscriptstyle +,1 $};
\end{scope}
\node[anchor=east] at (-1.3, -6.15) {$S=-+-+$};
\begin{scope}[shift={(0,-5.68)}]
    \node at (0,0.32) {$P_{5}$};
    \node[labelnode] at (-0.72,0.12) {$ \scriptscriptstyle -,1 $};
    \node[labelnode] at (-0.24,0.12) {$ \scriptscriptstyle +,1 $};
    \node[labelnode] at (0.24,0.12) {$ \scriptscriptstyle -,0 $};
    \node[labelnode] at (0.72,0.12) {$ \scriptscriptstyle +,1 $};
    \draw [style=none] (-0.72,0) -- (-0.72,-0.28);
    \draw [style=none] (0.72,-0.28) -- (0.72,0);
    \draw [style=mid arrow] (-0.23,-0.148) -- (-0.24,0);
    \draw [style=none] (-0.23,-0.148) -- (-0.24,-0.28);
    \draw [style=mid arrow] (-0.23,-0.148) -- (0.23,-0.148);
    \draw [style=mid arrow] (0.24,0) -- (0.23,-0.148);
    \draw [style=none] (0.24,-0.28) -- (0.23,-0.148);
    \draw [style=mid arrow] (0,-0.56) -- (0.24,-0.28);
    \draw [style=mid arrow] (0.48,-0.56) -- (0.72,-0.28);
    \draw [style=mid arrow] (-0.72,-0.28) -- (-0.48,-0.428);
    \draw [style=mid arrow] (-0.24,-0.28) -- (-0.48,-0.428);
    \draw [style=mid arrow] (-0.48,-0.56) -- (-0.48,-0.428);
    \node[labelnode] at (-0.48,-0.68) {$ \scriptscriptstyle +,1 $};
    \node[labelnode] at (0,-0.68) {$ \scriptscriptstyle +,1 $};
    \node[labelnode] at (0.48,-0.68) {$ \scriptscriptstyle +,1 $};
\end{scope}
\node[anchor=east] at (-1.3, -7.57) {$S=--++$};
\begin{scope}[shift={(0,-7.1)}]
    \node at (0,0.32) {$P_{6}$};
    \node[labelnode] at (-0.72,0.12) {$ \scriptscriptstyle -,1 $};
    \node[labelnode] at (-0.24,0.12) {$ \scriptscriptstyle -,0 $};
    \node[labelnode] at (0.24,0.12) {$ \scriptscriptstyle +,1 $};
    \node[labelnode] at (0.72,0.12) {$ \scriptscriptstyle +,1 $};
    \draw [style=mid arrow] (0,-0.28) -- (0.24,0);
    \draw [style=mid arrow] (0.48,-0.28) -- (0.72,0);
    \draw [style=mid arrow] (-0.72,0) -- (-0.48,-0.148);
    \draw [style=mid arrow] (-0.24,0) -- (-0.48,-0.148);
    \draw [style=mid arrow] (-0.48,-0.28) -- (-0.48,-0.148);
    \node[labelnode] at (-0.48,-0.4) {$ \scriptscriptstyle +,1 $};
    \node[labelnode] at (0,-0.4) {$ \scriptscriptstyle +,1 $};
    \node[labelnode] at (0.48,-0.4) {$ \scriptscriptstyle +,1 $};
\end{scope}
\end{tikzpicture}}
    \end{subfigure}\hfill
    \begin{subfigure}[t]{0.60\linewidth}
        \centering
        \caption{$\Delta_4(2,1)$}
        \label{fig: cell module 4 21}
        \scalebox{0.92}{\input{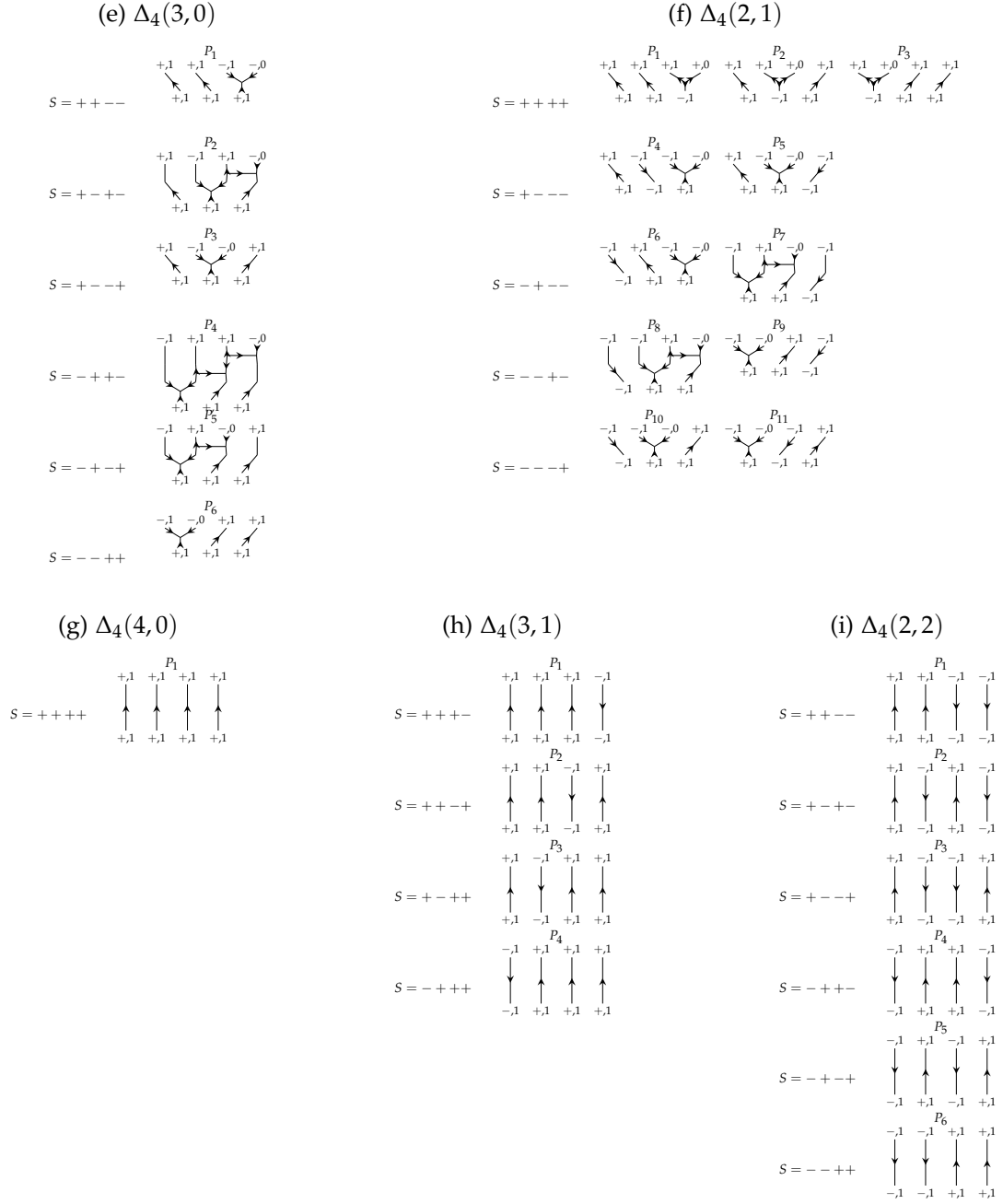}}
    \end{subfigure}

    \vspace{0.65cm}

    \begin{subfigure}[t]{0.31\linewidth}
        \centering
        \caption{$\Delta_4(4,0)$}
        \label{fig: cell module 4 40}
        \scalebox{0.92}{
\tikzset{
    style/.style={#1},
    none/.style={},
    mid arrow/.style={postaction={decorate,decoration={
            markings, mark=at position 0.55 with {\arrow{stealth}}
        }}
    },
    dotted/.style={dashed}
}
\begin{tikzpicture}[
    x=1cm, y=1cm,
    every node/.style={font=\tiny, inner sep=0.5pt},
    labelnode/.style={font=\tiny, inner sep=0.3pt}
]
\node[anchor=east] at (-1.3, -0.47) {$S=++++$};
\begin{scope}[shift={(0,0)}]
    \node at (0,0.32) {$P_{1}$};
    \node[labelnode] at (-0.72,0.12) {$ \scriptscriptstyle +,1 $};
    \node[labelnode] at (-0.24,0.12) {$ \scriptscriptstyle +,1 $};
    \node[labelnode] at (0.24,0.12) {$ \scriptscriptstyle +,1 $};
    \node[labelnode] at (0.72,0.12) {$ \scriptscriptstyle +,1 $};
    \draw [style=mid arrow] (-0.72,-0.72) -- (-0.72,0);
    \draw [style=mid arrow] (-0.24,-0.72) -- (-0.24,0);
    \draw [style=mid arrow] (0.24,-0.72) -- (0.24,0);
    \draw [style=mid arrow] (0.72,-0.72) -- (0.72,0);
    \node[labelnode] at (-0.72,-0.84) {$ \scriptscriptstyle +,1 $};
    \node[labelnode] at (-0.24,-0.84) {$ \scriptscriptstyle +,1 $};
    \node[labelnode] at (0.24,-0.84) {$ \scriptscriptstyle +,1 $};
    \node[labelnode] at (0.72,-0.84) {$ \scriptscriptstyle +,1 $};
\end{scope}
\end{tikzpicture}}
    \end{subfigure}\hfill
    \begin{subfigure}[t]{0.31\linewidth}
        \centering
        \caption{$\Delta_4(3,1)$}
        \label{fig: cell module 4 31}
        \scalebox{0.92}{
\tikzset{
    style/.style={#1},
    none/.style={},
    mid arrow/.style={postaction={decorate,decoration={
            markings, mark=at position 0.55 with {\arrow{stealth}}
        }}
    },
    dotted/.style={dashed}
}
\begin{tikzpicture}[
    x=1cm, y=1cm,
    every node/.style={font=\tiny, inner sep=0.5pt},
    labelnode/.style={font=\tiny, inner sep=0.3pt}
]
\node[anchor=east] at (-1.3, -0.47) {$S=+++-$};
\begin{scope}[shift={(0,0)}]
    \node at (0,0.32) {$P_{1}$};
    \node[labelnode] at (-0.72,0.12) {$ \scriptscriptstyle +,1 $};
    \node[labelnode] at (-0.24,0.12) {$ \scriptscriptstyle +,1 $};
    \node[labelnode] at (0.24,0.12) {$ \scriptscriptstyle +,1 $};
    \node[labelnode] at (0.72,0.12) {$ \scriptscriptstyle -,1 $};
    \draw [style=mid arrow] (-0.72,-0.72) -- (-0.72,0);
    \draw [style=mid arrow] (-0.24,-0.72) -- (-0.24,0);
    \draw [style=mid arrow] (0.24,-0.72) -- (0.24,0);
    \draw [style=mid arrow] (0.72,0) -- (0.72,-0.72);
    \node[labelnode] at (-0.72,-0.84) {$ \scriptscriptstyle +,1 $};
    \node[labelnode] at (-0.24,-0.84) {$ \scriptscriptstyle +,1 $};
    \node[labelnode] at (0.24,-0.84) {$ \scriptscriptstyle +,1 $};
    \node[labelnode] at (0.72,-0.84) {$ \scriptscriptstyle -,1 $};
\end{scope}
\node[anchor=east] at (-1.3, -1.89) {$S=++-+$};
\begin{scope}[shift={(0,-1.42)}]
    \node at (0,0.32) {$P_{2}$};
    \node[labelnode] at (-0.72,0.12) {$ \scriptscriptstyle +,1 $};
    \node[labelnode] at (-0.24,0.12) {$ \scriptscriptstyle +,1 $};
    \node[labelnode] at (0.24,0.12) {$ \scriptscriptstyle -,1 $};
    \node[labelnode] at (0.72,0.12) {$ \scriptscriptstyle +,1 $};
    \draw [style=mid arrow] (-0.72,-0.72) -- (-0.72,0);
    \draw [style=mid arrow] (-0.24,-0.72) -- (-0.24,0);
    \draw [style=mid arrow] (0.24,0) -- (0.24,-0.72);
    \draw [style=mid arrow] (0.72,-0.72) -- (0.72,0);
    \node[labelnode] at (-0.72,-0.84) {$ \scriptscriptstyle +,1 $};
    \node[labelnode] at (-0.24,-0.84) {$ \scriptscriptstyle +,1 $};
    \node[labelnode] at (0.24,-0.84) {$ \scriptscriptstyle -,1 $};
    \node[labelnode] at (0.72,-0.84) {$ \scriptscriptstyle +,1 $};
\end{scope}
\node[anchor=east] at (-1.3, -3.31) {$S=+-++$};
\begin{scope}[shift={(0,-2.84)}]
    \node at (0,0.32) {$P_{3}$};
    \node[labelnode] at (-0.72,0.12) {$ \scriptscriptstyle +,1 $};
    \node[labelnode] at (-0.24,0.12) {$ \scriptscriptstyle -,1 $};
    \node[labelnode] at (0.24,0.12) {$ \scriptscriptstyle +,1 $};
    \node[labelnode] at (0.72,0.12) {$ \scriptscriptstyle +,1 $};
    \draw [style=mid arrow] (-0.72,-0.72) -- (-0.72,0);
    \draw [style=mid arrow] (-0.24,0) -- (-0.24,-0.72);
    \draw [style=mid arrow] (0.24,-0.72) -- (0.24,0);
    \draw [style=mid arrow] (0.72,-0.72) -- (0.72,0);
    \node[labelnode] at (-0.72,-0.84) {$ \scriptscriptstyle +,1 $};
    \node[labelnode] at (-0.24,-0.84) {$ \scriptscriptstyle -,1 $};
    \node[labelnode] at (0.24,-0.84) {$ \scriptscriptstyle +,1 $};
    \node[labelnode] at (0.72,-0.84) {$ \scriptscriptstyle +,1 $};
\end{scope}
\node[anchor=east] at (-1.3, -4.73) {$S=-+++$};
\begin{scope}[shift={(0,-4.26)}]
    \node at (0,0.32) {$P_{4}$};
    \node[labelnode] at (-0.72,0.12) {$ \scriptscriptstyle -,1 $};
    \node[labelnode] at (-0.24,0.12) {$ \scriptscriptstyle +,1 $};
    \node[labelnode] at (0.24,0.12) {$ \scriptscriptstyle +,1 $};
    \node[labelnode] at (0.72,0.12) {$ \scriptscriptstyle +,1 $};
    \draw [style=mid arrow] (-0.72,0) -- (-0.72,-0.72);
    \draw [style=mid arrow] (-0.24,-0.72) -- (-0.24,0);
    \draw [style=mid arrow] (0.24,-0.72) -- (0.24,0);
    \draw [style=mid arrow] (0.72,-0.72) -- (0.72,0);
    \node[labelnode] at (-0.72,-0.84) {$ \scriptscriptstyle -,1 $};
    \node[labelnode] at (-0.24,-0.84) {$ \scriptscriptstyle +,1 $};
    \node[labelnode] at (0.24,-0.84) {$ \scriptscriptstyle +,1 $};
    \node[labelnode] at (0.72,-0.84) {$ \scriptscriptstyle +,1 $};
\end{scope}
\end{tikzpicture}}
    \end{subfigure}\hfill
    \begin{subfigure}[t]{0.31\linewidth}
        \centering
        \caption{$\Delta_4(2,2)$}
        \label{fig: cell module 4 22}
        \scalebox{0.92}{
\tikzset{
    style/.style={#1},
    none/.style={},
    mid arrow/.style={postaction={decorate,decoration={
            markings, mark=at position 0.55 with {\arrow{stealth}}
        }}
    },
    dotted/.style={dashed}
}
\begin{tikzpicture}[
    x=1cm, y=1cm,
    every node/.style={font=\tiny, inner sep=0.5pt},
    labelnode/.style={font=\tiny, inner sep=0.3pt}
]
\node[anchor=east] at (-1.3, -0.47) {$S=++--$};
\begin{scope}[shift={(0,0)}]
    \node at (0,0.32) {$P_{1}$};
    \node[labelnode] at (-0.72,0.12) {$ \scriptscriptstyle +,1 $};
    \node[labelnode] at (-0.24,0.12) {$ \scriptscriptstyle +,1 $};
    \node[labelnode] at (0.24,0.12) {$ \scriptscriptstyle -,1 $};
    \node[labelnode] at (0.72,0.12) {$ \scriptscriptstyle -,1 $};
    \draw [style=mid arrow] (-0.72,-0.72) -- (-0.72,0);
    \draw [style=mid arrow] (-0.24,-0.72) -- (-0.24,0);
    \draw [style=mid arrow] (0.24,0) -- (0.24,-0.72);
    \draw [style=mid arrow] (0.72,0) -- (0.72,-0.72);
    \node[labelnode] at (-0.72,-0.84) {$ \scriptscriptstyle +,1 $};
    \node[labelnode] at (-0.24,-0.84) {$ \scriptscriptstyle +,1 $};
    \node[labelnode] at (0.24,-0.84) {$ \scriptscriptstyle -,1 $};
    \node[labelnode] at (0.72,-0.84) {$ \scriptscriptstyle -,1 $};
\end{scope}
\node[anchor=east] at (-1.3, -1.89) {$S=+-+-$};
\begin{scope}[shift={(0,-1.42)}]
    \node at (0,0.32) {$P_{2}$};
    \node[labelnode] at (-0.72,0.12) {$ \scriptscriptstyle +,1 $};
    \node[labelnode] at (-0.24,0.12) {$ \scriptscriptstyle -,1 $};
    \node[labelnode] at (0.24,0.12) {$ \scriptscriptstyle +,1 $};
    \node[labelnode] at (0.72,0.12) {$ \scriptscriptstyle -,1 $};
    \draw [style=mid arrow] (-0.72,-0.72) -- (-0.72,0);
    \draw [style=mid arrow] (-0.24,0) -- (-0.24,-0.72);
    \draw [style=mid arrow] (0.24,-0.72) -- (0.24,0);
    \draw [style=mid arrow] (0.72,0) -- (0.72,-0.72);
    \node[labelnode] at (-0.72,-0.84) {$ \scriptscriptstyle +,1 $};
    \node[labelnode] at (-0.24,-0.84) {$ \scriptscriptstyle -,1 $};
    \node[labelnode] at (0.24,-0.84) {$ \scriptscriptstyle +,1 $};
    \node[labelnode] at (0.72,-0.84) {$ \scriptscriptstyle -,1 $};
\end{scope}
\node[anchor=east] at (-1.3, -3.31) {$S=+--+$};
\begin{scope}[shift={(0,-2.84)}]
    \node at (0,0.32) {$P_{3}$};
    \node[labelnode] at (-0.72,0.12) {$ \scriptscriptstyle +,1 $};
    \node[labelnode] at (-0.24,0.12) {$ \scriptscriptstyle -,1 $};
    \node[labelnode] at (0.24,0.12) {$ \scriptscriptstyle -,1 $};
    \node[labelnode] at (0.72,0.12) {$ \scriptscriptstyle +,1 $};
    \draw [style=mid arrow] (-0.72,-0.72) -- (-0.72,0);
    \draw [style=mid arrow] (-0.24,0) -- (-0.24,-0.72);
    \draw [style=mid arrow] (0.24,0) -- (0.24,-0.72);
    \draw [style=mid arrow] (0.72,-0.72) -- (0.72,0);
    \node[labelnode] at (-0.72,-0.84) {$ \scriptscriptstyle +,1 $};
    \node[labelnode] at (-0.24,-0.84) {$ \scriptscriptstyle -,1 $};
    \node[labelnode] at (0.24,-0.84) {$ \scriptscriptstyle -,1 $};
    \node[labelnode] at (0.72,-0.84) {$ \scriptscriptstyle +,1 $};
\end{scope}
\node[anchor=east] at (-1.3, -4.73) {$S=-++-$};
\begin{scope}[shift={(0,-4.26)}]
    \node at (0,0.32) {$P_{4}$};
    \node[labelnode] at (-0.72,0.12) {$ \scriptscriptstyle -,1 $};
    \node[labelnode] at (-0.24,0.12) {$ \scriptscriptstyle +,1 $};
    \node[labelnode] at (0.24,0.12) {$ \scriptscriptstyle +,1 $};
    \node[labelnode] at (0.72,0.12) {$ \scriptscriptstyle -,1 $};
    \draw [style=mid arrow] (-0.72,0) -- (-0.72,-0.72);
    \draw [style=mid arrow] (-0.24,-0.72) -- (-0.24,0);
    \draw [style=mid arrow] (0.24,-0.72) -- (0.24,0);
    \draw [style=mid arrow] (0.72,0) -- (0.72,-0.72);
    \node[labelnode] at (-0.72,-0.84) {$ \scriptscriptstyle -,1 $};
    \node[labelnode] at (-0.24,-0.84) {$ \scriptscriptstyle +,1 $};
    \node[labelnode] at (0.24,-0.84) {$ \scriptscriptstyle +,1 $};
    \node[labelnode] at (0.72,-0.84) {$ \scriptscriptstyle -,1 $};
\end{scope}
\node[anchor=east] at (-1.3, -6.15) {$S=-+-+$};
\begin{scope}[shift={(0,-5.68)}]
    \node at (0,0.32) {$P_{5}$};
    \node[labelnode] at (-0.72,0.12) {$ \scriptscriptstyle -,1 $};
    \node[labelnode] at (-0.24,0.12) {$ \scriptscriptstyle +,1 $};
    \node[labelnode] at (0.24,0.12) {$ \scriptscriptstyle -,1 $};
    \node[labelnode] at (0.72,0.12) {$ \scriptscriptstyle +,1 $};
    \draw [style=mid arrow] (-0.72,0) -- (-0.72,-0.72);
    \draw [style=mid arrow] (-0.24,-0.72) -- (-0.24,0);
    \draw [style=mid arrow] (0.24,0) -- (0.24,-0.72);
    \draw [style=mid arrow] (0.72,-0.72) -- (0.72,0);
    \node[labelnode] at (-0.72,-0.84) {$ \scriptscriptstyle -,1 $};
    \node[labelnode] at (-0.24,-0.84) {$ \scriptscriptstyle +,1 $};
    \node[labelnode] at (0.24,-0.84) {$ \scriptscriptstyle -,1 $};
    \node[labelnode] at (0.72,-0.84) {$ \scriptscriptstyle +,1 $};
\end{scope}
\node[anchor=east] at (-1.3, -7.57) {$S=--++$};
\begin{scope}[shift={(0,-7.1)}]
    \node at (0,0.32) {$P_{6}$};
    \node[labelnode] at (-0.72,0.12) {$ \scriptscriptstyle -,1 $};
    \node[labelnode] at (-0.24,0.12) {$ \scriptscriptstyle -,1 $};
    \node[labelnode] at (0.24,0.12) {$ \scriptscriptstyle +,1 $};
    \node[labelnode] at (0.72,0.12) {$ \scriptscriptstyle +,1 $};
    \draw [style=mid arrow] (-0.72,0) -- (-0.72,-0.72);
    \draw [style=mid arrow] (-0.24,0) -- (-0.24,-0.72);
    \draw [style=mid arrow] (0.24,-0.72) -- (0.24,0);
    \draw [style=mid arrow] (0.72,-0.72) -- (0.72,0);
    \node[labelnode] at (-0.72,-0.84) {$ \scriptscriptstyle -,1 $};
    \node[labelnode] at (-0.24,-0.84) {$ \scriptscriptstyle -,1 $};
    \node[labelnode] at (0.24,-0.84) {$ \scriptscriptstyle +,1 $};
    \node[labelnode] at (0.72,-0.84) {$ \scriptscriptstyle +,1 $};
\end{scope}
\end{tikzpicture}}
    \end{subfigure}
    \caption{Basis vectors for the cell modules of $\mathcal K_4$ with $x\geq y$ (continued).
    The diagram scale and row labelling are the same as on the preceding page.}
\end{figure}
\clearpage
\endgroup

\end{appendices}

\newpage
\printbibliography

\end{document}